\documentclass[10pt]{amsart}
\usepackage{fullpage}
\usepackage{times}
\usepackage{graphicx}

\usepackage{amsmath,amsthm,amsfonts,amscd,amssymb}
\usepackage{a4}
\usepackage{array}
\usepackage{booktabs}
\usepackage{multirow}
\usepackage{hhline}
\usepackage{xcolor}
\usepackage{pst-node}
\usepackage{pst-plot}
\usepackage[all]{xy}

\newtheorem{thm}{Theorem}[section]
\newtheorem{cor}[thm]{Corollary}
\newtheorem{defn}[thm]{Definition}
\newtheorem{lem}[thm]{Lemma}
\newtheorem{rem}[thm]{Remark}
\newtheorem{prop}[thm]{Proposition}
\newtheorem{ex}[thm]{Example}
{ \theoremstyle{remark} }
\newtheorem{question}[thm]{Question}

\numberwithin{thm}{section}
\numberwithin{equation}{section}

\newcommand{\myop}[1]{\operatorname{#1}}

\newcommand{\id}{\myop{id}}

\newcommand{\F}{\mathcal F}

\newcommand{\Z}{\mathbb Z}

\newcommand{\T}{{\mathbb{T}}}

\newcommand{\mc}{\mathcal}

\def\g{\mathfrak{g}}
\def\h{\mathfrak{h}}

\newcommand{\Real}{\mathbb R}

\newcommand{\abs}[1]{\left\vert#1\right\vert}

\newcommand{\eps}{\varepsilon}

\newcommand{\la}{\langle}
\newcommand{\ra}{\rangle}

\newcommand{\A}{\mathcal{A}}

\newcommand{\B}{\mathcal{B}}

\newcommand{\Comp}{\mathbb{C}}

\newcommand{\D}{\mathcal{D}}
\newcommand{\Hi}{\mathcal{H}}
\newcommand{\I}{\mathcal{I}}

\newcommand{\M}{\mathcal{M}}
\newcommand{\n}{\mathbb{N}}

\newcommand{\tor}{\mathbb{T}}

\newcommand{\z}{\mathbb{Z}}

\newcommand{\Om}{\Omega}
\newcommand{\om}{\omega}

\newcommand{\prt}{\widehat{\otimes}}

\newcommand{\wbar}[1]{\overline{#1}}
\newcommand{\what}[1]{\widehat{#1}}
\newcommand{\wtil}[1]{\widetilde{#1}}
\newcommand{\til}[1]{\tilde{#1}}

\newcommand{\aut}{\mathrm{Aut}}

\newcommand{\ran}{\operatorname{ran}}
\newcommand{\supp}{\operatorname{supp}}
\newcommand{\spn}{\operatorname{span}}
\newcommand{\spec}{\operatorname{Spec}}

\newcommand{\fA}{\mathcal{A}}
\newcommand{\fB}{\mathcal{B}}

\newcommand{\fD}{\mathcal{D}}
\newcommand{\fE}{\mathcal{E}}

\newcommand{\fI}{\mathcal{I}}

\newcommand{\fM}{\mathcal{M}}
\newcommand{\fN}{\mathcal{N}}

\newcommand{\fe}{\mathfrak{e}}

\newcommand{\fh}{\mathfrak{heis}}

\newcommand{\Cee}{\mathbb{C}}
\newcommand{\Ree}{\mathbb{R}}
\newcommand{\Tee}{\mathbb{T}}

\newcommand{\Hee}{\mathbb{H}}

\newcommand{\alp}{\alpha}

\newcommand{\Gam}{\Gamma}
\newcommand{\lam}{\lambda}

\newcommand{\ome}{\omega}
\newcommand{\Ome}{\Omega}
\newcommand{\sig}{\sigma}
\newcommand{\Sig}{\Sigma}

\begin{document}

\title[Spectra of Beurling-Fourier algebras]{Beurling-Fourier algebras on Lie groups and their spectra}

\author{Mahya Ghandehari}
\address{Department of Mathematical Sciences, University of Delaware,
501 Ewing Hall,
Newark, DE 19716,  USA}
\curraddr{}
\email{mahya@udel.edu}
\thanks{M. Ghandehari was partially supported by University of Delaware Research Foundation, and partially by NSF grant
DMS-1902301, while this work was being completed.}

\author{Hun Hee Lee}
\address{Department of Mathematical Sciences and the Research Institute of Mathematics, Seoul National University, Gwanak-ro 1, Gwanak-gu, Seoul 08826, Republic of Korea}
\curraddr{}
\email{hunheelee@snu.ac.kr}
\thanks{H. H. Lee was supported by the Basic Science Research Program through the National Research Foundation of Korea (NRF) Grant NRF-2017R1E1A1A03070510 and the National Research Foundation of Korea (NRF) Grant funded by the Korean Government (MSIT) (Grant No.2017R1A5A1015626).}

\author{Jean Ludwig}
\address{Institut \'{E}lie Cartan de Lorraine,
Universit\'{e} de Lorraine -- Metz,
B\^{a}timent A, Ile du Saulcy, F-57045 Metz, France}
\curraddr{}
\email{jean.ludwig@univ-lorraine.fr}
\thanks{}

\author{Nico Spronk}
\address{Pure Mathematics, University of Waterloo, 200 University Avenue West
Waterloo, Ontario, Canada}
\curraddr{}
\email{nspronk@uwaterloo.ca}
\thanks{N. Spronk was partially supported by NSERC grant 312515-2015.}

\author{Lyudmila Turowska}
\address{Department of Mathematical Sciences, Chalmers University of Technology and University of Gothenburg, Gothenburg SE-412 96, Sweden}
\curraddr{}
\email{turowska@chalmers.se}
\thanks{L. Turowska was partially supported by "Stiftelsen G S Magnussons Fond" and the Department of Mathematical Scinces, Chalmers University of Technology through a guest research program.}


\begin{abstract}
We investigate Beurling-Fourier algebras, a weighted version of Fourier algebras, on various Lie groups focusing on their spectral analysis. We will introduce a refined general definition of weights on the dual of locally compact groups and their associated Beurling-Fourier algebras. Constructions of nontrivial weights will be presented focusing on the cases of representative examples of Lie groups, namely $SU(n)$, the Heisenberg group $\mathbb{H}$, the reduced Heisenberg group $\mathbb{H}_r$, the Euclidean motion group $E(2)$ and its simply connected cover $\widetilde{E}(2)$. We will determine the spectrum of Beurling-Fourier algebras on each of the aforementioned groups emphasizing its connection to the complexification of underlying Lie groups. We also demonstrate ``polynomially growing'' weights does not change the spectrum and show the associated regularity of the resulting Beurling-Fourier algebras.
\end{abstract}

\maketitle

\tableofcontents

\section{Introduction}
Let $G$ be a locally compact group and $A(G)$ its Fourier algebra as defined in \cite{Eym}.  
We recall that $A(G)$ is the predual of the von Neumann algebra $VN(G)$ generated by the 
left regular representation $\lambda:G\to\B(L^2(G))$, and, 
moreover, is a dense subalgebra of $C_0(G)$.  In a sense which is specified in the theory of 
locally compact quantum groups, the pair $(A(G),VN(G))$ is the Pontryagin dual object to 
$(L^1(G),L^\infty(G))$, where we purposely suppress mention of the  Haar weights.  In fact, if $G$ 
is abelian, then $(A(G),VN(G))\cong(L^1(\widehat{G}),L^\infty(\widehat{G}))$ where 
$\widehat{G}$ is the dual group.  
Hence we expect aspects of the theory of the convolution 
algebra $L^1(G)$, and related convolution algebras, to be reflected in the theory of $A(G)$.

Beurling algebras, $L^1(G,\omega)$ for submultiplicative weights, $\omega:G\to(0,\infty)$ satisfying $\omega(xy)\leq\omega(x)\omega(y)$
for $x,y$ in $G$,
are important variants of $L^1(G)$, in particular when $G$ is abelian or is compactly 
generated.  Recently, weighted versions of the Fourier algebra have been introduced and 
investigated by various subsets of the present authors
\cite{GLSS, LS, LST} under the name of Beurling-Fourier algebras; 
see also \cite{ORS}.  These have proved to admit a rich theory.
In \cite{LS} properties such as Arens regularity are studied.  This is followed up by \cite{GLSS} where it is shown that Beurling-Fourier algebras can often be isomorphic to algebras of operators on Hilbert spaces, a property which stands in contrast to $A(G)$ (\cite[Prop.\ 3.1]{LSS}).  In \cite{LST}, spectral theory and associated properties such as regularity are studied.  The present article is really
a continuation on the theme of the latter.

The first goal of the present note is to give a unified treatment
to all known examples.  In broad terms, we proceed as follows.
\begin{itemize}
\item[(1)] We formulate a general definition of a weight $W$ on the dual of $G$, which 
allows the definition of the Beurling-Fourier algebra $A(G,W)$ which is commutative.
See Section \ref{sec:def-weights}.
\end{itemize}
The definition is in terms of unbounded positive operators.
In order to show that the definition is meaningful, we must proceed to
\begin{itemize}
\item[(2)] construct examples of weights $W$ on the dual of $G$.
\end{itemize}
To do this, we have three fundamental strategies: we construct central weights,
typically on connected compact groups (Section \ref{subsection-central-weights}); we 
extend certain weights from subgroups (Section \ref{subsection-ext-subgroup}); and we 
use the Laplacian on certain connected Lie groups, which, in particular, gives us the polynomial
weights (Section \ref{ssec:Weights-Laplacian-general}).
The first step in the analysis of a commutative Banach algebra
is to understand its spectral theory.  Hence for every example that we devise, we
\begin{itemize}
\item[(3)] compute the spectrum of $A(G,W)$.
\end{itemize}
This turns out to be the most difficult aspect of our theory, and our approach is outlined below.

Beginning with the influential work of Wiener \cite{wiener}, spectral theory has proved to 
be an essential part of understanding a commutative Banach algebra.  Of course,
our modern understanding of spectral theory arises from the revolutionary work of
Gelfand \cite{gelfand}.  Eymard \cite{Eym}, Sait\^{o} \cite{saito} 
and Herz \cite{herz} have all given different proofs that
the spectrum of $A(G)$ is identifiable with $G$.  In both \cite{Eym,herz}, spectral synthesis at points plays a key role in 
the determination of the spectrum.

Let us consider the case of $G$ abelian.  We let 
$$
\what{G}_\Comp=\{\chi:G\to\Comp^\times\;|\;\chi\text{ is continuous and mutiplicative}\},
$$
where $\Comp^\times$ is the multiplicative group of non-zero complex numbers.
Then it is straightforward to see that
the spectrum of $L^1(G,\omega)$ is the set of {\it $\omega$-bounded characters},
those $\chi$ in $\what{G}_\Comp$ such that $|\chi(x)|\leq\omega(x)$ for all $x$ in $G$ (see \cite[Section 2.8]{Kan}).
Notice for $\chi$ in $\what{G}_\Comp$,  the range of any compact subgroup
is in $\T$.  Hence we expect interesting theory of the spectrum only for $G$ with
no compact subgroups, i.e.\ those groups for which $\what{G}$ is connected.  Notice, in
the case of connected Lie dual group, i.e.\ $G=\Real^n\times\Z^m$ so
$\what{G}\cong\Real^n\times\T^m$, and we have
$\what{G}_\Comp\cong\Comp^n\times(\Comp^\times)^m$ is the complexification of 
$\what{G}$.
We note that especially in the case that $G=\Real^n$ or $\Z^n$, and
$\omega$ bounded away from $0$, the set of Fourier transforms
$A(G,\omega)=\{\hat{f}:f\in L^1(G,\omega)\}$ have been important in the study of spectral properties of commuting operators on Banach spaces and (systems of) operator equations through $A(G,\omega)$-functional calculus. Here the analytic structure and spectral synthesis in $A(G,\omega)$ have been significant issues.
See, for example, \cite{mcintoshp,schulmant}.

When we consider the Beurling-Fourier algebras $A(G,W)$, we simultaneously lose
any (a priori) notion of spectral synthesis of points and a straightforward notion
of ``generalized $W$-bounded character".  Hence this task of understanding
the structure of the spectrum requires a novel approach.
We require a means of allowing  complexifications of (in the present 
paper) a connected Lie group $G$ to act as operators on $L^2(G)$, 
affiliated with $VN(G)$, in a manner that interacts nicely with $W$.  
Regrettably, there are no general means of doing so in the literature, so we are forced to devise one ourselves.
Our approach entails a large amount of hard analysis, specific to 
examples.  Our mix of unbounded operator theory, harmonic analysis, and Lie theory appears to be novel and hence an interesting contribution itself.

\subsection{Basic strategy}
In \cite{LST}, three of the present authors considered central weights on compact groups.  
Here the dual object $\widehat{G}$ is the set of equivalence classes of irreducible
representations.  
They considered functions $w:\widehat{G}\to(\delta,\infty)$, $\delta>0$, which satisfy
the submultiplicativity condition 
	\begin{equation}\label{eq-submultiplicativity-compact}
		w(\rho) \le w(\pi)w(\pi')
	\end{equation}
for any $\pi, \pi', \rho \in \widehat{G}$ such that $\rho \subseteq \pi\otimes \pi'$, i.e.\
$\rho$ is a subrepresentation of $\pi\otimes \pi'$.  They, then, let
	$$A(G,w) = \left\{f\in C(G): \sum_{\pi \in \widehat{G}}w(\pi) d_\pi \|\widehat{f}^G(\pi)\|_1 <\infty \right\}$$
where the norm is given by a weighted sum of the trace norms of the matricial Fourier coefficients
\begin{equation}\label{eq-fourier-compact}
	\widehat{f}^G(\pi) = \int_G f(g) \pi(g)dg.
\end{equation}
In the case that $w$ is the constant function $1$, this gives $A(G)$.  In this case
the subalgebra ${\rm Trig}(G)$ of finite linear combinations of matrix coefficients
of elements of $\widehat{G}$ forms a dense subspace.  McKennon \cite{McK1} defined the 
{\it abstract complexification} $G_\Comp$ of $G$ as the set of non-zero multiplicative functionals
${\rm Spec}\,{\rm Trig}(G)$, and a full analysis was done by Cartwright and 
McMullen \cite{CMc}.  
For a connected compact Lie group, this is exactly the universal complexification 
due to Chevalley (\cite[III.8]{BtD}), and generally is a pro-Lie group which is an
extension of the complexification of the connected component of the identity $G_0$,
by the totally disconnected quotient $G/G_0$.  Hence computing
${\rm Spec}\,A(G,w)$ reduces to the task of determining which elements of $G_\Comp$
lie naturally in the dual of $A(G,w)$.  This task is interesting only on connected groups
and can be easily reduced to connected Lie groups.  See Section \ref{subsec:central-compact} for some sample computations in this context.

For the reason just mentioned and other considerations below, we restrict our analysis
to certain connected Lie groups. The main theme of this paper is to extend the above scheme to the case of general connected possibly non-compact Lie groups. However, we immediately face an obstacle, namely the absence of the abstract Lie theory applicable for non-compact locally compact groups. There is one model of abstract Lie theory for locally compact groups suggested by McKennon (\cite{McK2}). However, the authors were not able to find any direct connection between our Beurling-Fourier algebras $A(G, W)$ and McKennon's model (see (2) of Remark \ref{rem-weights-BFalg} for more details) at the time of this writing.
For this reason, we are forced to establish an ``abstract Lie model'' suitable for Beurling-Fourier algebras from scratch.

As a motivation for this ``abstract Lie model'' we consider the case of $G = \Real$. Using Pontryagin duality we begin with a weight function (i.e.\ sub-multiplicative and Borel measurable) $w: \widehat{\Real} \to (0,\infty)$. We assume that $w$ is bounded below (i.e. $\inf_{x\in \Real} w(x)>0$) to ensure that $L^1(\widehat{\Real}, w) \subseteq  L^1(\widehat{\Real})$. Here, $\widehat{\Real}$ is the dual group of $\Real$. Then the Beurling-Fourier algebra $A(\Real, w)$ can be identified with the Beurling algebra $L^1(\widehat{\Real}, w)$ via the Fourier transform, $\F^{\widehat{\Real}}$, on the dual group $\widehat{\Real}$. Even though we know $\widehat{\Real}\cong \Real$ we will keep the notation $\widehat{\Real}$ to emphasize the distinction of the two groups. As mentioned above, the spectrum ${\rm Spec}L^1(\widehat{\Real}, w)$ is well-understood via the concept of $w$-bounded characters. However, we require a more subtle route starting with a dense subalgebra $\A = \F^{\widehat{\Real}}(C^\infty_c(\widehat{\Real}))$ of $A(\Real, w) \cong L^1(\widehat{\Real}, w)$, which plays an important role replacing ${\rm Trig}(G)$ for a compact group $G$. An element $\varphi \in \text{Spec} A(\Real,w)$  is determined by its restriction $\varphi|_\A$ by the density of $\A \subseteq A(\Real, w)$, and its transferred version $\psi := \varphi|_\A \circ \F^{\widehat{\Real}} : C^\infty_c(\widehat{\Real}) \to \Comp$ is now a continuous multiplicative linear functional with respect to convolution product on $\widehat{\Real}$. This is illustrated in the diagram below.
\[
\xymatrix{
A(\Real, w) \ar[d]_{\varphi}
& \A \ar@{_{(}->}[l] \ar[ld]_{\varphi|_\A}
& C^\infty_c(\widehat{\Real}) \ar[l]_(.6){\F^{\widehat{\Real}}} \ar@/^/[lld]^{\psi}
\\
\mathbb{C} & &
}
\]
Since $\psi$ arises from a locally integrable function in $L^\infty(\widehat{\Real}, w^{-1})$, it is continuous on $C^\infty_c(\widehat{\Real})$.  Hence the Cauchy functional equation for distributions, which will be discussed in Section \ref{sssec:CFE-Euclidean}, shows that (up to normalization of
the Lebesgue measure)
$$
\psi(f)=\int_{\widehat{\Real}}f(x)e^{-icx}\,dx
\text{ for some }c\text{ in }\Comp.
$$
Notice that the Paley-Wiener Theorem tells us that $\A$ is an algebra of
analytic functions, and hence solving the Cauchy functional equation amounts to saying
that such point evaluations comprise ${\rm Spec}\fA$.
Returning to  $\text{Spec} A(\Real,w)$, we simply need to determine for which $c$ in $\Comp$ does $\sup_{x\in\widehat{\Real}} \frac{|e^{-icx}|}{w(x)}<\infty$.   Notice that $\Comp$
is the universal complexification $\Real_\Comp$.

In this paper, we aim to determine $\text{Spec}A(G,W)$ for a connected Lie group $G$ by extending the above approach as follows. 
\begin{itemize}
\item[(Step 1)] Any dense subalgebra $\A$ of $A(G, W)$ gives an injective embedding 
$${\rm Spec}A(G, W) \subseteq {\rm Spec}\A.$$ 
We require $\A$ to satisfy that ${\rm Spec}\A \cong G_\Comp$ through an appropriate ``abstract Lie'' theory.  Hence, we require that any element in $\A$ extends analytically to $G_\Comp$, so that we can identify a point $x\in G_\Comp$ and the point evaluation functional $\varphi_x \in {\rm Spec}\A$ at $x$.  
\item[(Step 2)] We check which points in $G_\Comp$ give rise to a linear functional bounded in the $A(G,W)$-norm.
\end{itemize}
Both of these steps are much more involved than in the abelian case illustrated above.  In particular, choosing the sublalgebra $\A$ is a highly non-trivial task since we need to ensure its density in $A(G,W)$ for any general weight $W$, possibly of ``exponential growth". Thus, for example, an immediate candidate $C^\infty_c(G)$, the space of test functions, is not enough for that purpose, as indicated in Remark \ref{rem-ChoiceSubalgebra} below.
To find $\A$ in (Step 1) we borrow the ``background'' Euclidean structure of the given Lie group. This trick drives us to a suitable modification of the Cauchy functional equation, leading us to the points on $G_\Comp$. Thus, the procedure just described could be understood as the ``abstract Lie'' theory we needed.  Moreover, we use the concept of entire vectors for unitary representations to guarantee their analytic extendability.

The technicality of choosing a dense subalgebra $\A$ and the associated Cauchy type functional equation forces us not to attempt to establish a general theory applicable for any connected Lie group in this paper. Instead, we will focus on representative examples, namely $SU(n)$ among compact connected Lie groups, the Heisenberg group $\mathbb{H}$ among simply connected nilpotent Lie groups and its reduced version $\mathbb{H}_r$, the Euclidean motion group $E(2)$ acting on $\Real^2$ among solvable non-nilpotent Lie groups and its simply connected cover $\widetilde{E}(2)$.  These are groups which are known to have
sufficiently many entire vectors.

\subsection{Organization}
In Section \ref{chap-preliminaries} we summarize some basic materials we need in this paper.
In Section \ref{sec-unbdd-op} we cover basics on unbounded operators, including Borel functional calculus for strongly commuting self-adjoint operators and a general treatment on an extension of $*$-homomorphisms to certain unbounded operators.
In Section \ref{sec-Lie} we provide materials about Lie groups, Lie algebras, and related operators, including complexification models of Lie groups and entire vectors of unitary representations.
We include a short Section \ref{ssec:ChoiceFourier} on the choice of Fourier transforms since we use various versions of them.

In Section \ref{chap-Def-BF-alg} we will provide a general definition of Beurling-Fourier algebras on locally compact groups, and the associated weights on their dual, which replaces the definition in \cite{LS}.
We begin with motivation from the case of abelian groups in Section \ref{sec:abelian1}.
In Section \ref{sec:def-weights} we give a rigorous definition of weights on the dual of locally compact groups and define associated Beurling-Fourier algebras based on it. A more concrete interpretation of Beurling-Fourier algebras is given in the following two subsections \ref{ssec:bddbelow} and \ref{ssec:separable-type I}.
In Section \ref{sec-construction-weight} we introduce three fundamental ways of constructing weights, namely the central weights, the weights extended from subgroups, and the weights obtained from Laplacian on the group.

Starting from Section \ref{chap-cpt} we examine concrete examples, with the first being compact connected Lie groups, in particular $SU(n)$, the $n\times n$ special unitary group.
We first provide details of weights on the dual of compact connected Lie groups in Section \ref{sec-weights-compact}.
We then determine the spectrum of Beurling-Fourier algebras in Section \ref{sec-SU(n)}. The cases of central weights and the weights extended from subgroups are fundamentally different, so the corresponding approaches also differ.

In Section \ref{chap-Heisenberg} we analyze the case of the Heisenberg group $\mathbb{H}$. The technical key observation here is that we borrow the ``background'' Euclidean structure of $\mathbb{H}$, namely $\Real^3$ for the choice of dense subalgebra $\A$, playing the role which ${\rm Trig}(G)$ does in compact theory. Then, we continue to the Cauchy functional equation for distributions on $\Real^3$ to provide a substitute for the Chevalley style of complexification model. In Section \ref{chap-reduced-Heisenberg} we continue the case of the reduced Heisenberg group $\mathbb{H}_r$, which shares most of the technical details of the Heisenberg group case.

In Section \ref{chap-E(2)} we focus on the case of the Euclidean motion group $E(2)$. The choice of the dense subalgebra becomes more involved, reflecting the structure of the representation theory using the polar form on $\Real^2$. Moreover, the corresponding Cauchy functional equation is also more involved. In Section \ref{chap-E-infty(2)} we continue the case of the simply connected cover $\widetilde{E}(2)$ of $E(2)$, which shares most of the technicalities.

Up to this point, we mainly focused on the case of ``exponentially growing'' weights providing the spectrum of the corresponding Beurling-Fourier algebras strictly larger than the original group. In Section \ref{chap-spec-poly-regularity}, however, we consider the case of ``polynomially growing weights''. The main result is that polynomially growing weights do not change the spectrum of the Beurling-Fourier algebras. The proof also provides the regularity of the corresponding Beurling-Fourier algebras. We also provide some non-regular Beurling-Fourier algebras at the end of this section.

In the final section, we collect some questions remaining from our analysis.


\section{Preliminaries}\label{chap-preliminaries}

\subsection{Unbounded operators}\label{sec-unbdd-op}
We collect some of the basic materials on unbounded operators. Our main reference on this matter is the text of
Schm\"{u}dgen, \cite{Sch}, in particular Chapters 4 and 5.

A linear map $T$ defined on a subspace ${\rm dom} T$, which we call the domain of $T$, of a Hilbert space $H$ into another Hilbert space $K$ is called {\it closed} if  the graph of $T$, $\{(h,Th): h\in {\rm dom}T\}$, is closed in $H\oplus K$. We say that another linear map $S:{\rm dom}S \subseteq H \to K$ is an {\it extension} of $T$ if ${\rm dom}T \subseteq {\rm dom}S$ and $S|_{{\rm dom}T} = T$. In this case we write $T\subseteq S$. We say that $T$ is {\it closable} if it has a closed extension. In this case we denote the smallest (with respect to the above inclusion) closed extension by $\overline{T}$, the {\it closure} of $T$.

We consider an unconventional unitary ``equivalence'' notation for  unbounded operators, which we introduce for the notational convenience. We consider two unbounded operators $S$ and $T$ acting on the Hilbert spaces $H$ and $K$, respectively. Let $U :H \to K$ be a unitary operator. Then we write
	\begin{equation}\label{eq-unitary-eq-notation}
	S \stackrel{U}{\sim} T\;\;\text{on}\;\; \fD
	\end{equation}
if for some $\|\cdot\|_H$-dense subspace $\fD \subseteq {\rm dom}S$ we have
$USh = TUh$ for each $h$ in $\fD$. Note that the relation $\stackrel{U}{\sim}$ is not an actual equivalence relation since it need not be transitive.

If an operator $T: {\rm dom} T\subseteq H \to H$ is densely defined (i.e.\ ${\rm dom} T$ is dense in $H$) then its adjoint operator $T^*: {\rm dom}(T^*)\subseteq H \to H$ is well-defined as a linear map and closed. We say that such a $T$ is {\it self-adjoint} if $T = T^*$, in which case $T$ is closed as well. We say that such a $T$ is {\it essentially self-adjoint} if $\overline{T}$ is self-adjoint. A self-adjoint (possibly unbounded) operator on a Hilbert space is very well-understood through the spectral integral, which we review below.

Given a measurable space $(\Omega,\A)$ and a Hilbert space $H$, a {\it spectral measure}
is a map $E:\A\to\B(H)$ which is projection-valued, countably additive in the strong operator sense
and for which $E(\Omega)=I$.  If $f:\Ome\to\Ree$ is $\A$-measurable then we may consider the spectral integral
\begin{equation}\label{eq:specint}
T=\int_\Ome f\, dE=\int_\Ree t\;d(E\circ f^{-1})(t),
\end{equation}
which is a densely defined self-adjoint operator on $H$. Conversely, let $T$ be a densely defined self-adjoint operator on $H$. The spectral theorem tells us that there exists a spectral measure $E_T$ on the Borel $\sig$-algebra
$\fB_\Ree$ on $\Ree$, for which
\[
\mathrm{dom}T=\left\{x\in H:\int_\Real t^2\,d\| E_T(t)x\|^2<\infty\right\}
\]
and for which
\[
T=\int_\Real t\,dE_T(t).
\]
Furthermore for
$T$ as in (\ref{eq:specint}), we have $E_T = E \circ f^{-1}$, where $E \circ f^{-1}(B) = E(f^{-1}(B))$, $B\in \mc B_\Real$.

A family of densely defined self-adjoint operators $T_1,\dots,T_n$ is said to {\it strongly commute} if each pair of
elements $E_{T_i}(A)$ and $E_{T_j}(B)$ commute, where $A,B\in \fB_\Ree$.
In this case we can define a product spectral measure on the product $\sig$-algebra,
	$$E_{T_1}\times \dots\times E_{T_n}:\fB_{\Ree^n}=
\fB_\Ree \otimes \cdots \otimes \fB_\Ree \to\fB(H),$$
taking the natural definition on measurable rectangles:
	$$E_{T_1}\times \dots\times E_{T_n}(B_1\times\dots\times B_n)
:=E_{T_1}(B_1)\dots E_{T_n}(B_n).$$
We shall write $E_{T_1,\dots,T_n}=E_{T_1}\times \dots\times E_{T_n}$.
Using this spectral measure, we obtain a functional
calculus:  for any real-valued measurable $g$ on $\supp(E_{T_1}\times \dots\times E_{T_n})$,  we may define
\begin{align*}
g(T_1,\dots,T_n)&=\int_{\Ree^n}g(t_1,\dots,t_n)\,d[E_{T_1}\times \dots\times E_{T_n}](t_1,\dots,t_n) \\
&=\int_\Ree s \,d[(E_{T_1}\times \dots\times E_{T_n})\circ g^{-1}](s)
\end{align*}
with $E_{g(T_1,\dots,T_n)}=(E_{T_1}\times \dots\times E_{T_n})\circ g^{-1}$.
Notice that if $g$ is {\it locally bounded} (i.e.\ bounded on any compact set) then
\begin{equation}\label{eq:calcdom}
\fD_{T_1,\dots,T_n}=\bigcup_{k=1}^\infty \ran E_{T_1}\times \dots\times E_{T_n}([-k,k]^n)
\end{equation}
is  a {\it core} for each operator $S=T_1,\dots,T_n$ or $g(T_1,\dots,T_n)$,
i.e.\ is dense in ($\mathrm{dom}S,\|\cdot\|_S)$.
If $T$ is given in (\ref{eq:specint}), then we have product and
composition rules: if $S=\int_\Ome f'\,dE$ for $f':\Ome\to\Ree$ measurable and $g$ is a measurable function
on $\ran(f)$ we have
\[
\wbar{ST}=\int_\Ome ff'\,dE\text{ and }g(T)=\int_\Omega g\circ f\,dE.
\]

Notice that if $\mu_n(t_1,\dots,t_n)=t_1\dots t_n$, then we get the closure of $T_1\dots T_n$ and its spectral measure given by
\[
\mu_n(T_1,\dots,T_n)=\wbar{T_1\dots T_n}\text{ and }
(E_{T_1}\times \dots \times E_{T_n})\circ \mu_n^{-1}=E_{\wbar{T_1\dots T_n}}.
\]

This functional calculus allows us to easily define the two concepts, below.

\subsubsection{Tensor products} Let $T_k$ be a densely defined, self-adjoint operator on $H_k$, $k=1,2$.
Then let
	$$T_1\otimes I :=\int_\Ree s\,d[E_{T_1}\otimes I](s), \;\text{and}\; I\otimes T_2 :=\int_\Ree t\,d[I\otimes E_{T_2}](t).$$
These operators are each densely defined and self-adjoint, and
they both strongly commute on the Hilbertian tensor product $H_1\otimes^2 H_2$.  We may thus define
\[
T_1\otimes T_2 :=\mu_2(T_1\otimes I,I\otimes T_2)=\wbar{(T_1\otimes I)(I\otimes T_2)}.
\]
It is sufficient, in practice, to consider this operator on $\fD_{T_1\otimes I,I\otimes T_2}$, as defined above.

\subsubsection{Homomorphisms}\label{ssec:homomorphisms}
One of the main technicalities in the general construc\-tion of we\-ights in \cite{LS} was to extend a $*$-homomorphism to unbounded operators, which we clarify here.

We say that a densely defined self-adjoint operator $T$ as given in (\ref{eq:specint}) is {\it affiliated} with
a von Neumann subalgebra $\fM$ of $\fB(H)$ if $UT \subset TU$ for every unitary $U \in \fM'$, or equivalently if $E_T(B)=E(f^{-1}(B))\in\fM$ for each $B$ in $\fB_\Ree$.
Note that this condition is implied by the condition on the spectral measure that $E(A)\in\fM$ for each $A$ in $\fA$.  Let
$\pi:\fM\subseteq \fB(H)\to \fN\subseteq \fB(H')$ be a normal
$*$-homomorphism between von Neumann algebras, and  $T$ be as defined above. We define $\pi(T)$ to be
\[
\pi(T) :=\int_\Ree t\, d[\pi \circ E_T](t)=\int_\Omega f\,d[\pi\circ E]
\]
where $\pi \circ E_T$ and $\pi \circ E$ are evidently spectral measures. Then $\pi(T)$ is a self-adjoint operator on $H'$ affiliated with $\fN$.

We observe that if $T_1,\dots,T_n$ is a strongly commuting family of
densely defined self-adjoint operators, each affiliated with $\fM$,
then $E_{T_1}\times \dots \times E_{T_n}(B)\in \fM$ for $B$ in $\fB_{\Ree^n}$ (as may be checked on measurable
rectangles and extended).  Hence
$\pi \circ E_{\wbar{T_1\cdots T_n}}=[(\pi \circ E_{T_1}) \times \dots\times (\pi \circ E_{T_n})]\circ\mu_n^{-1}$ takes values in $\fN$.
Furthermore we have
\[
\pi(\wbar{T_1\cdots T_n})=\wbar{\pi(T_1)\cdots\pi(T_n)}
\]
which, on the common dense domain $\fD_{\pi(T_1),\cdots,\pi(T_n)}$, is equal to $\pi(T_1)\cdots\pi(T_n)$.
In the sequel we shall take liberty to simply write
\[
\pi(T_1\cdots T_n)=\pi(T_1)\cdots\pi(T_n)
\]
where this operator is understood to act on a common dense domain such as $\fD_{\pi(T_1),\cdots,\pi(T_n)}$.
Furthermore, for any real-valued measurable $g$ on
	$$\supp(E_{T_1}\times \cdots\times E_{T_n})
\supseteq\supp[(\pi \circ E_{T_1})\times \cdots\times (\pi \circ E_{T_n})]$$
we have
\[
g(\pi(T_1),\cdots,\pi(T_n))=\pi(g(T_1,\cdots,T_n)).
\]

\subsubsection{Homomorphisms for non-commuting pairs}
We shall frequently make use of the following, which will typically apply to particular cases of products of
non-self adjoint bounded operators with unbounded self-adjoint operators. No assumption of strong
commuting will be made.  For unbounded $S$ and bounded $A$ we recall that $\mathrm{dom}SA
=\{h\in H:Ah\in\mathrm{dom}S\}$, whereas $\mathrm{dom}AS=\mathrm{dom}S$.

\begin{prop}\label{prop-extension-variant}
Let $T$, $\fM$ and $\pi$ be as above, and $A\in\fM$.
	\begin{enumerate}
	\item If $TA$ is densely defined and bounded, then
$TA\in\fM$, $\ran \pi(A)\subset\mathrm{dom}(\pi(T))$ and $\pi(TA)=\pi(T)\pi(A)$.

	\item If $\pi$ is injective and $\pi(T)\pi(A)$ is bounded, then $TA\in\fM$.

	\item If $AT$ is bounded, then $\pi(AT)=\pi(A)\pi(T)$.
	\end{enumerate}
\end{prop}

\begin{proof}
(1)  Let $P_n=E_T([-n,n])$. Then $\lim_{n}P_n=I$ and $\lim_{n}TP_n=T$
in the strong operator topology on $\mathrm{dom}T$.
The same is true for limits involving  the projections $\pi(P_n)$ relative to
$\mathrm{dom}(\pi(T))$.

First, $TA$ is closed; see Exercise 1.5.9 of \cite{Sch}.  A closed bounded operator has full domain
$H$.  Hence $TA=\lim_{n}TP_nA$, where the limit is in the strong operator topology
on $H$, so $TA\in\fM$.  In fact, the sequence of operators $TP_nA$ is uniformly bounded and hence
converges in the $\sig$-strong operator topology on $H$.  Thus $\lim_{n}\pi(TP_nA)=\pi(TA)$ in the
$\sig$-strong operator topology on $H'$.  We have $\pi(TP_nA)=\pi(TP_n)\pi(A)$ while
$\pi(TP_n)=\pi(T)\pi(P_n)$, as $P_n$ strongly commutes with $T$.  If $h\in\ran\pi(A)$,
$h=\lim_{n}\pi(P_nA)x$ for some $x$ in $H'$, and $\pi(T)\pi(P_nA)x=\pi(TP_nA)x$ converges
to $y=\pi(TA)x$, so $(h,y)=\lim_{n}(\pi(P_nA)x,\pi(T)\pi(P_nA)x)$
is in the closed graph of $\pi(T)$.  Thus $\ran \pi(A)\subseteq\mathrm{dom}(\pi(T))$.
But then $\mathrm{dom}(\pi(T)\pi(A))=H'$ and $\pi(T)\pi(A)=\lim_{n}\pi(TP_nA)=\pi(TA)$, where
the limit is in the strong operator topology on $H'$.

(2)  We simply apply $\pi^{-1}$ to the von Neumann algebra $\pi(\fM)$ in $\fB(H')$, and appeal
to part (i).

(3)  We use simple facts about the adjoint as shown in Proposition 1.7 of \cite{Sch}.
If $AT$ is bounded then $TA^*=(AT)^*$  is bounded.
Then we  see that $\pi(AT)=\pi(TA^*)^*=[\pi(T)\pi(A^*)]^*$, where the factorization is from (i).  The latter
contains $\pi(A)\pi(T)$, so $\pi(A)\pi(T)$ is bounded with dense domain.  On this domain
we see that $\pi(AT)=\pi(A)\pi(T)$.
\end{proof}

\subsection{Lie groups, Lie algebras and related operators}\label{sec-Lie}

We collect some materials on Lie theory which we will use frequently in this paper. The symbols $G$ and $\mathfrak{g}$ will be reserved for a connected real Lie group and its associated Lie algebra with the exponential map $\exp : \g \to G$ throughout the paper unless specified otherwise. We also fix the symbol $H$ and $\mathfrak{h}$ for a connected closed Lie subgroup of $G$ and its associated Lie algebra.

\subsection{Complexification of Lie groups}\label{subsubsec-complexification}
We say that a complex Lie group $G_\Comp$ together with a Lie group homomorphism $\psi : G \to G_\Comp$ is the {\it universal complexification} of $G$ if for any Lie group homomorphism $\varphi: G\to H$ for a complex Lie group $H$ there is a complex Lie group homomorphism $\tilde{\varphi}:G_\Comp \to H$ such that $\tilde{\varphi} \circ \psi = \varphi$. It may be the case that the Lie algebra $\tilde{\mathfrak{g}}_\Comp$ of $G_\Comp$ is a proper quotient of the Lie algebra complexification $\mathfrak{g}_\Comp$ of $\mathfrak{g}$. However, all the examples in this paper satisfy that
$\tilde{\mathfrak{g}}_\Comp=\mathfrak{g}_\Comp$, so that we will simply call $G_\Comp$ the {\it complexification}.

When $G$ is compact, we have a concrete construction of the universal complexification $G_\Comp$ due to Chevalley (\cite[III. 8]{BtD}). We define
	$$G_\Comp := \{T\in \text{Trig}(G)^\dagger : m^\dagger(T) = T\otimes T,\; T\ne 0\}$$
where $m : \text{Trig}(G) \odot \text{Trig}(G) \to \text{Trig}(G)$ is the pointwise multiplication defined on the algebraic tensor product and $m^\dagger$ is the algebraic adjoint. In other words, $G_\Comp$ is the set of non-zero multiplicative functionals on $\text{Trig}(G)$. Using the canonical duality
	$$\text{Trig}(G)^\dagger \cong \prod_{\pi\in \widehat{G}}M_{d_\pi}$$
where $\text{Trig}(G)^\dagger$ is the algebraic dual of ${\rm Trig}(G)$ and $\widehat{G}$ is the unitary dual of $G$, we can regard elements in $G_\Comp$ as sequences of matrices indexed by $\widehat{G}$. This justifies the injective embedding (or the group homomorphism)
	$$J : G \to G_\Comp, \; g \mapsto J(g) = (J(g)(\pi))_{\pi \in \widehat{G}}$$
given by
	$$J(g)(\pi):=[\pi_{ij}(g)]^{d_\pi}_{i,j=1}$$
which is just $\pi(g)$. Here $\pi_{ij}$ is the coefficient function of the unitary representation $\pi:G \to \B(H_\pi)$ given by $\pi_{ij}(g) = \la \pi(g)e_j, e_i  \ra$ for a fixed orthonormal basis $(e_j)^{d_\pi}_{j=1}$ of $H_\pi$. It is well known that $(G_\Comp, J)$ is the universal complexification of $G$, and  the Lie algebra associated to $G_\Comp$ is the Lie algebra complexification $\mathfrak{g}_\Comp$ of $\g$. From the universality of $G_\Comp$, the representation $\pi$ extends to a complex representation $ \pi_\Comp  : G_\Comp \to GL_{d_\pi}(\Comp)$, defined as $\pi_\Comp (T)=[T(\pi_{ij})]$. Thus, to any $T \in G_\Comp$ we can associate the sequence $( \pi_\Comp (T))_{\pi \in \widehat{G}}$ in the following way.
	$$T(\pi) = [T(\pi_{ij})] = [ (\pi_\Comp)_{ij}(T)] =  \pi_\Comp (T).$$
The matrix $ \pi_\Comp (T)$ has the polar decomposition $ \pi_\Comp (T) = U_\pi \exp(i X_\pi)$ with $U_\pi \in U_{d_\pi}(\Comp)$, the group of unitary $d_\pi\times d_\pi$ matrices, and $X_\pi = -X^*_\pi \in M_{d_\pi}$ and it has been proved that there is a uniquely determined $g\in G$ and $X\in \g$ such that
	$$\pi(g) = U_\pi\;\;\text{and}\;\; d\pi(X) = X_\pi$$
where $d\pi : \g \to M_{d_\pi}$ is the Lie algebra representation derived of $\pi$. This leads us to the Cartan decomposition:
	\begin{equation}\label{eq-Cartan-cpt}
	G_\Comp \cong G \cdot G^+_\Comp = G \cdot \exp(i\mathfrak{g}).
	\end{equation}
Note that the above concept of complexification has been extended to general compact groups by McKennon \cite{McK1} and Cartwright/McMullen \cite{CMc}.

\subsubsection{Operators associated to certain elements of the universal enveloping algebra and entire vectors}

We fix a basis $\{X_1,\ldots,X_n\}$ of $\g$ and we consider the following norm on the complexification $\g_\Comp$ given by
	$$\Big\|\sum^n_{j=1}a_j X_j\Big\| := \sum^n_{j=1}|a_j|,\; (a_j)^n_{j=1} \subseteq \Comp.$$
Let $U(\g)$ denote the universal enveloping algebra of $\g_{\mathbb C}$. We simply refer to $U(\g)$ as the enveloping algebra of $\g$. The map $X\mapsto -X$ is an anti-isomorphism of $\g$. Its unique extension to an anti-isomorphism of $U(\g)$ is an involution for the algebra $U(\g)$. Let $\pi : G \to \B(H_\pi)$ be a unitary representation of $G$. A vector $ v \in H_\pi$ is called a {\it $C^\infty$-vector for $\pi$} if the mapping $g\mapsto\pi(g) v $  from the $C^\infty$-manifold $G$ into the Hilbert space $H_\pi$ is a $C^\infty$-mapping. Let $\D^\infty(\pi)$ be the space of $C^\infty$-vectors for $\pi$, which is known to be dense in $H_\pi$ and invariant under $\pi(g)$, $g\in G$. For $X\in\g$, we define the operator $d \pi(X)$ with domain $\D^\infty(\pi)$ by
	$$d \pi(X) v =\frac{d}{dt}\pi(\exp(tX)) v |_{t=0}.$$
The mapping $X\mapsto d \pi(X)$ satisfies the following. For any $\alpha, \beta\in {\mathbb R}$, $X, Y\in\g$ and $v, w \in \D^\infty(\pi)$ we have
\begin{itemize}
	\item The space $\D^\infty(\pi)$ is invariant under $d \pi(X)$,
	\item
	$d \pi(\alpha X+\beta Y) v =\alpha d \pi(X) v +\beta d \pi(Y) v $,
	\item
	$d \pi([X,Y]) v =d \pi(X) d \pi(Y) v - d \pi(Y)d \pi(X) v $,
	\item
	$\la d \pi(X) v , w\ra = -\la v , d \pi(X)w\ra$.
\end{itemize}
In other words, $d \pi$ is a $*$-homomorphism of the Lie algebra $\g$ on $\D^\infty(\pi)$, which extends uniquely to a $*$-homomorphism $d \pi$ from the $*$-algebra $U(\g)$ to the algebra of operators on $\D^\infty(\pi)$.

For some elements in $U(\g)$ we have a better understanding of $d\pi(X)$. For $X\in\g$, $i\, d\pi(X)$ is known to be essentially self-adjoint on $\Hi_\pi$. We denote its self-adjoint extension $\overline{i\, d\pi(X)}$ by
	$$i\, \partial \pi(X)$$
which is actually the infinitesimal generator of the strongly continuous one-parameter unitary group $t\mapsto \pi(\exp(tX))$ on $H_\pi$, i.e.  we have
	$$\pi(\exp(tX))=\exp(t\partial \pi(X)).$$
The case of {\it (Nelson) Laplacian} $\Delta=X_1^2+\ldots+ X_n^2 \in U(\g)$ is more involved (\cite[Corollary 10.2.5]{Sch}), but still the same is true, namely $d \pi(\Delta)$ is essentially self-adjoint with the self-adjoint extension $\partial \pi(\Delta)$, which is a negative self-adjoint operator on $H_\pi$.

The space $\D^\infty(\pi)$ has a natural locally convex topology given by the following family of seminorms $(\rho_m)_{m\geq 0}$ where
	$$\rho_m( v ) = \max_{1\le j_k \le n}\|d\pi(X_{j_1}\cdots X_{j_m}) v \|$$
	with the understanding that $\rho_0(v)=\|v\|$.
The above family can be used to define analytic vectors and entire vectors for $\pi$. We say that a $C^\infty$-vector $ v  \in H_\pi$ is an {\it analytic vector for $\pi$} if the mapping $g\mapsto\pi(g) v $ is a real analytic function on $G$, or equivalently (\cite[Lemma 7.1]{Nel}),
	$$E_s( v ) := \sum^\infty_{m=1}\frac{s^m}{m!}\rho_m( v )<\infty$$
for some $s>0$. We say that a $C^\infty$-vector $ v  \in H_\pi$ is an {\it entire vector for $\pi$} if $E_s( v )<\infty$ for all $s>0$. We denote the space of all entire vectors for $\pi$ by $\D^\infty_\Comp(\pi)$. For $v\in \D^\infty_\Comp(\pi)$ we can readily check that $\displaystyle \pi(\exp X)v = \sum^\infty_{n=0}\frac{1}{n!}(d\pi(X))^n v$ for $X \in \mathfrak{g}$. Motivated by this fact we define for $X \in \mathfrak{g}_\mathbb{C}$ that
	$$\pi_\Comp(\exp_\Comp X)v := \sum^\infty_{n=0}\frac{1}{n!}(d\pi(X))^n v$$
where $\exp_\Comp : \g_\Comp \to G_\Comp$ is the exponential map of the complexification $G_\Comp$ of $G$. We will use the notation $\exp$ for $\exp_\Comp$ by a slight abuse of notation since $\exp_\Comp$ actually is an extension of $\exp$. Then, by \cite[Proposition 2.2, 2.3]{Good69} we know that $\pi_\Comp$ is a holomorphic representation of $G_\Comp$ on $\D^\infty_\Comp(\pi)$ in the following sense: for any $v\in \D^\infty_\Comp(\pi)$ the map $X \in \g_\Comp \mapsto \pi_\Comp(\exp X)v \in \D^\infty_\Comp(\pi)$ is holomorphic and
	$$\pi_\Comp(\exp X)\pi_\Comp(\exp Y) v = \pi_\Comp(\exp X \exp Y)v.$$
for any  $X, Y\in \g_\Comp$.

\begin{rem}
When $G$ is compact and $\pi \in \widehat{G}$, then it is easy to see that any element in $H_\pi\cong \Comp^{d_\pi}$ is an entire vector for $\pi$. Moreover, the above definitions of $d\pi$ and $\pi_\Comp$ coincide with the corresponding symbols from Section \ref{subsubsec-complexification}.
\end{rem}

We are mostly interested in the case of $\pi$ being the left regular representation $\lambda$ of $G$ or an irreducible unitary representation appearing in the decomposition of $\lambda$. For the left regular representation we have a criterion for identifying entire vectors by Goodman \cite{Good71}. We assume that $G$ is separable, type I and unimodular, so that we have a clear Plancherel picture of $G$ as follows. The unitary dual $\widehat{G}$ becomes a standard Borel space and there is a unique Borel measure $\mu$ on $\widehat{G}$ with the following property: for a fixed $\mu$-measurable cross-section $\xi \to \pi^\xi$ from $\widehat{G}$ to concrete irreducible unitary representations acting on $H_\xi$ we have
	$$\langle f_1, f_2\rangle = \int_{\widehat{G}}\text{Tr}(\widehat{f_1}^G(\xi) \widehat{f_2}^G(\xi)^*) d\mu(\xi),\;\; f_1, f_2 \in L^1(G)\cap L^2(G)$$
where
	\begin{equation}\label{eq-group-Fourier-transform}
	\widehat{f}^G(\xi) = \F^{G}(f)(\xi) := \int_G f(g)\pi^\xi(g)dg \in \B(H_\xi),\;\; f\in L^1(G)\cap L^2(G).
	\end{equation}
Thus, the group Fourier transform
	$$\F^G : L^1(G) \to L^\infty(\widehat{G}, d\mu; \B(H_\xi)),\; f\mapsto \F^G(f)$$ with $\F^G(f) = (\F^G(f)(\xi))_{\xi \in \widehat{G}} = (\widehat{f}^G(\xi))_{\xi \in \widehat{G}}$ extends to a unitary
	$$\F^G : L^2(G) \to L^2(\widehat{G}, d\mu; S^2(H_\xi)),\; f\mapsto \F^G(f).$$
Here, $S^2(H)$ is the space of Hilbert-Schmidt operators on a Hilbert space $H$.	
We further assume that $G$ is solvable, so that we may suppose that our fixed basis $\{X_j\}$ for $\mathfrak{g}$ is a {\it Jordan-H\"{o}lder} basis (\cite{Good71} or \cite{Pen}), i.e. for $\mathfrak{h}_k = {\rm span}\{X_1, \cdots, X_k\}$, $1\le k\le n$ we have
	$$[\mathfrak{g}, \mathfrak{h}_k]\subseteq \mathfrak{h}_{k-1},\; 1\le k\le n.$$
For ease of reference recall some of the main results in \cite{Good71, Pen}.
	\begin{thm}\label{thm-Goodman}
	Let $G$ be a connected solvable Lie group which is separable, type I and unimodular.
		\begin{enumerate}				
			\item A function $f\in L^2(G)$ is an entire vector for $\lambda$ if and only if $$\begin{cases}{\ran}\widehat{f}^G(\xi) \subseteq \D^\infty_\Comp(\pi^\xi)\; \mu\text{-almost every $\xi$ and}\\
				\displaystyle \int_{\widehat{G}}\sup_{\gamma \in \Om_t} \|\pi^\xi_\Comp(\gamma^{-1})\widehat{f}^G(\xi)\|^2_2d\mu(\xi)<\infty \;\text{for any $t>0$,}\end{cases}$$
			where the set $\Om_t$ is given by $\Om_t = \{ {\rm exp}X : X\in \mathfrak{g}_\mathbb{C},\; \|X\| <t\}$ and $\|\cdot\|_2$ is the Hilbert Schmidt norm. Moreover, we have
				\begin{equation}\label{eq-regular-decomposition}
				\lambda_\Comp(\gamma) \stackrel{\F^{G}}{\sim} (\pi^\xi_\Comp(\gamma))_{\xi\in\widehat{G}}\;\;\text{ on \, $\D^\infty_\Comp(\lambda)$}
				\end{equation}	
in the sense of \eqref{eq-unitary-eq-notation} for any $\gamma\in G_\Comp$.

			\item Let $f\in L^2(G)$ be an entire vector for $\lambda$, then we have
	$$\int_{\widehat{G}} \sup_{\gamma \in \Om_t}\|\pi^\xi_\Comp(\gamma^{-1})\widehat{f}^G(\xi)\|_1 d\mu(\xi) <\infty$$
for any $t>0$, where $\|\cdot\|_1$ is the trace class norm. Moreover, $f$ is analytically extended to $G_\mathbb{C}$ with the analytic continuation $f_\Comp$ given by the absolutely convergent integral
				$$f_\Comp(\gamma) = \int_{\widehat{G}} \text{\rm Tr}(\pi^\xi_\Comp(\gamma^{-1})\widehat{f}^G(\xi))d\mu(\xi),\; \gamma\in G_\Comp.$$
		\end{enumerate}	
	\end{thm}
\begin{proof}
The first statement is from \cite[Theorem 3.1]{Good71} and the second one is from \cite[Theorem 4.1]{Good71}, where $G$ is assumed to be a simply connected nilpotent Lie group. The nilpotency condition on $G$ could be relaxed into the solvability condition by replacing the technical step \cite[Theorem 1.1]{Good71} with \cite[Corollary I.5]{Pen}. The simple connectedness can be easily removed by a careful examination of the proofs, where most of the part are involved with the norm estimates of the elements of the universal enveloping algebra.
\end{proof}

 We end this section by recording a theorem of Paley-Wiener which characterizes entire functions for the left regular representation of $\Real^n$.
	\begin{prop}\label{prop-Paley-Weiner-super-exp-decay}
		Let $F \in L^2(\Real^n)$, $n\in \n$ be a function. Then, $F$  satisfies
			$$e^{t(|\xi_1| + \cdots + |\xi_n|)}\widehat{F}^{\Real^n}(\xi_1, \cdots, \xi_n) \in L^2(\Real^n)$$
		for any $t>0$ if and only if $F$ extends to an entire function on $\Comp^n$ and satisfies
	$$\sup_{|y_1|, \cdots, |y_n| \le s} \int_{\Real^n} |F(x_1+iy_1, \cdots, x_n+iy_n)|^2dx_1\cdots dx_n < \infty$$
for any $s>0$. In this case we have
	$$ \int_{\Real^n} \widehat{F}^{\Real^n}(\xi_1, \cdots, \xi_n) e^{i(z_1\xi_1 + \cdots + z_n\xi_n)}d\xi_1\cdots d\xi_n = F_\Comp(z_1, \cdots, z_n)$$
where $F_\Comp$ is the analytic continuation of $F$ and $(z_1, \cdots, z_n) \in \Comp^n$.
	\end{prop}
\begin{proof}
The case of $n=1$ is presented in \cite[Section 7.1]{Kat} and the case of $n\ge 2$ can be done by a similar argument.
\end{proof}

\subsubsection{The choice of Fourier transforms}\label{ssec:ChoiceFourier}
We take a moment to record the choices of Fourier transforms we made in this paper. We will use group Fourier transforms on various Lie groups. When the group is non-abelian (separable, type I) as in the previous section, then we choose \eqref{eq-group-Fourier-transform} as the definition.

For the abelian case, we follow the canonical choice as follows. For $f\in L^1(\Real^n)$ we define
	$$\widehat{f}^{\Real^n}(s_1,\cdots, s_n) := \frac{1}{(2\pi)^{\frac{n}{2}}}\int_{\Real^n}f(x_1,\cdots, x_n)e^{-i (x_1s_1 + \cdots + x_ns_n)}dx_1\cdots dx_n$$
where $dx_1\cdots dx_n$ is the Lebesgue measure on $\Real^n$. For $f\in L^1(\tor^n)$ we define
	$$\widehat{f}^{\tor^n}(d_1,\cdots, d_n) := \int_{\tor^n}f(y_1,\cdots, y_n)e^{-i (y_1d_1 + \cdots + y_nd_n)}dy_1\cdots dy_n$$	
where $dy_1\cdots dy_n$ is the normalized Haar measure on $\tor^n$. Finally, for $f\in \ell^1(\z^n)$ we define
	$$\widehat{f}^{\z^n}(t_1,\cdots, t_n) := \sum_{(d_1,\cdots, d_n)\in \z^n}f(d_1,\cdots, d_n)e^{-i (t_1d_1 + \cdots + t_nd_n)}.$$
These choices provide the following consequences:
\begin{itemize}
	\item $L^2(\Real^n) \to L^2(\Real^n),\; f\mapsto \widehat{f}^{\Real^n}$ and $L^2(\tor^n) \to \ell^2(\z^n),\; f\mapsto \widehat{f}^{\tor^n}$ are unitaries.
	\item $\widehat{f'}^\Real(x) = ix\widehat{f}^\Real(x)$, $x\in \Real$ and $\widehat{f'}^\tor(m) =  im\widehat{f}^\tor(m)$, $m\in \z$.
	\item $\widehat{f*g}^{\Real^n} = (2\pi)^{\frac{n}{2}}\widehat{f}^{\Real^n}\widehat{g}^{\Real^n}$. Thus, $(2\pi)^{\frac{n}{2}}\F^{\Real^n}$ is multiplicative with respect to the $\Real^n$-convolution, whilst $\widehat{f*g}^{\tor^n} = \widehat{f}^{\tor^n}\widehat{g}^{\tor^n}$, so that $\F^{\tor^n}$ is multiplicative with respect to the $\tor^n$-convolution.
	\item $\widehat{(\widehat{f}^\z)}^\tor(m) = f(-m)$, $m\in \z$.
\end{itemize}
Note that we made a choice of the Fourier transforms for the abelian case different from the non-abelian case.


\section{A refined definition for Beurling-Fourier algebras}\label{chap-Def-BF-alg}

In \cite{LS} the authors suggested a model for a weight $W$ on the dual of a locally compact group $G$. Regrettably, the suggested set of axioms in \cite{LS} was slightly misleading and had limitations in covering variety of examples beyond the ones already covered there, namely the case of ``central weights" on compact groups (extending the case of \cite{LST}) and the Heisenberg group.
Here, we introduce a refined definition of weights and the associated Beurling-Fourier algebras extending the previous definitions in \cite{LS}. Note that a closely related model of weighted Fourier algebras has been introduced in \cite{ORS}. See Remark \ref{rem-weights-BFalg} below for the detailed comparison of the definitions.

\subsection{Motivation: review of weights on abelian groups}\label{sec:abelian1}

\begin{defn}\label{def:weightfunc}
Given a locally compact abelian group $G$, a {\it weight function} is a Borel measurable function $w:G\to(0,\infty)$ which
satisfies
\begin{equation}\label{eq-submultiplicativity}
w(xy)\leq w(x)w(y)\text{ for almost every}\; x,y\in G.
\end{equation}

\end{defn}

\begin{rem}\label{rem:weightfunc}
It is shown in \cite{dzinotyiweyi} (see also Lemma 1.3.3 of \cite{Kan}) that a weight function is
always locally bounded:  given compact $K$ in $G$ there are constants $a$ and $b$ so
$0<a\leq w(x)\leq b$ for all $x$ in $K$.  It is then shown in Section 3.7 of \cite{reiters} that $w$ is
{\it equivalent} to a continuous weight function $\ome:G\to(0,\infty)$; i.e.\ there is $M>0$ so $\frac{1}{M}\ome\leq w\leq M\ome$.
\end{rem}

The {\it Beurling algebra} with the weight function $w$ is given by
\begin{equation}\label{eq:Balg}
L^1(G,w)=\{f:G \to \Comp\, |\, \text{$f$ Borel measurable},\; fw\in L^1(G)\}
\end{equation}
equipped with the norm $\|f\|_{L^1(G,w)}:=\|fw\|_{L^1(G)}$. It is well-known that this space is a Banach algebra under convolution.  Moreover if $w$ is equivalent
to a weight function $\ome$, then $L^1(G,w)$ is isomorphic to $L^1(G,\ome)$ as a Banach algebra. We observe that this space has dual space
\begin{equation}\label{eq:Balgdual}
L^\infty(G,\tfrac{1}{w})=\{f:G \to \Comp \, |\, \text{$f$ Borel measurable},\; \tfrac{f}{w}\in L^\infty(G)\}
\end{equation}
with the dual norm given by $\|f\|_{L^\infty(G,\frac{1}{w})}
:=\|\tfrac{f}{w}\|_{L^\infty(G)}$.

Let $\Psi:L^1(G,w)\to L^1(G)$ be the surjective isometry given by $\Psi(f)=fw$.  Let $\Ome:G\times G\to(0,1]$
be given by $\Ome(s,t)=\frac{w(st)}{w(s)w(t)}$, and define for $f,g$ in $L^1(G)$ the twisted
convolution $f\ast_\Ome g$ by
\begin{equation}\label{eq:twist}
f\ast_\Ome g(y)=\int_G f(x)g(x^{-1}y)\Ome(x,x^{-1}y)\,dx\text{ for a.e.\ }y\text{ in }G.
\end{equation}
Then we have that $\Psi(f)\ast_\Ome\Psi(g)=\Psi(f\ast g)$, showing that $(L^1(G),\ast_\Ome)$ is a
Banach algebra, isometrically isomorphic to $L^1(G,w)$.

The submultiplicativity condition \eqref{eq-submultiplicativity} can be rephrased as the following function inequality
	\begin{equation}\label{eq-function-inequality}
	\Gamma(w) \le w\times w \Leftrightarrow \Om = \Gamma(w)(w^{-1}\times w^{-1}) \le 1,
	\end{equation}
where $\Gamma(w): G\times G \to \Comp$ is given by $\Gamma(w)(x,y) = w(xy)$, $x,y\in G$ and $w\times w (x,y) = w(x)w(y)$, $x,y\in G$. Note that the map $\Gamma$ is in fact  the obvious extension (to unbounded functions) of the canonical coproduct $\Gamma: L^\infty(G) \to L^\infty(G\times G)$. Now we would like to transfer the above inequality to a condition on operators using the canonical embedding $L^\infty(G) \to B(L^2(G)), \phi \mapsto M_\phi$, where $ M_\phi$ denotes the multiplication operator with respect to the function $\phi$.
In the case of an unbounded Borel measurable function $\phi$, the operator $M_\phi$ can be concretely understood as an unbounded operator. Thus, the statement  $\Om \le 1$ in \eqref{eq-function-inequality} is equivalent to $M_\Om$ being a contraction (since $\Om\ge 0$). Now we observe that the operator $M_\Om$ is actually an extension of $M_{\Gamma(w)} M_{w^{-1} \times w^{-1}} = M_{\Gamma(w)}(M_{w}^{-1} \otimes M_{w}^{-1})$, which is the composition  of two unbounded operators, namely $M_{\Gamma(w)}$ and $M_{w}^{-1} \otimes M_{w}^{-1}$. Finally, we note that $M_w$ is affiliated with the commutative von Neumann algebra $L^\infty(G) \subseteq B(L^2(G))$. This motivates the general definition we formulate below.

\subsection{Weights on the dual of $G$ and Beurling-Fourier algebras}\label{sec:def-weights}
We now give a very general definition of a weight which encompasses all examples we have.
We shall make liberal use of concepts surrounding unbounded operators
discussed in Section \ref{sec-unbdd-op}.

We recall that the {\it coproduct} is the unique normal $*$-homomorphism
	$$\Gam:VN(G)\to VN(G)\bar{\otimes}VN(G)$$
satisfying
	$$\Gam(\lam(s))=\lam(s)\otimes\lam(s), \; \forall s\in G.$$  This probably appeared first in Section 9 of \cite{Stine}. Here, we use the same symbol as in the case of $L^\infty(G)$ by abuse of notation.  The coproduct $\Gamma$ is cocommutative (i.e. $\Sigma \circ \Gamma = \Gamma$ for the tensor flipping map $\Sigma :  VN(G\times G) \to VN(G\times G),\; A\otimes B\mapsto B\otimes A$) and satisfies the co-associativity law
	$$(\Gam\otimes\id)\circ\Gam=(\id\otimes \Gam)\circ\Gam.$$

\begin{defn}\label{def:weight}
A {\it weight} on the dual of $G$ is a densely defined operator $W$ on $L^2(G)$, which satisfies
\begin{itemize}
\item[(a)] $W$ is positive, affiliated with $VN(G)$, and is injective on its domain (hence admits positive
inverse);
\item[(b)] $\Gam(W)(W^{-1}\otimes W^{-1})$ is defined and contractive on a dense subspace, hence extends
to a contraction $X_W$ on $L^2(G\times G)$; and
\item[(c)] the contraction $X_W$ satisfies the {\it 2-cocycle} condition:
\[
((\Gam\otimes\id) (X_W))(X_W\otimes I)=((\id\otimes\Gam)(X_W))(I\otimes X_W).
\]
\end{itemize}
Moreover, we say that $W$ is a {\it strong weight} if it is a weight with the additional condition that
\begin{itemize}
\item[(d)] $\Gam(W)$ and $W\otimes W$ strongly commute.
\end{itemize}

\end{defn}

\begin{rem}\label{rem:weights}
\begin{enumerate}
	\item
	Let $W$ a weight on the dual of $G$. Since $W$ is invertible, the appropriate analogue of the core (\ref{eq:calcdom}) is
\begin{equation}\label{eq:calcdom1}
\fD_W=\bigcup_{n=1}^\infty{\ran}E_W([\tfrac{1}{n},n]).
\end{equation}
    It is also
	given by functional calculus: $W^{-1}=\int_{(0,\infty)}\frac{1}{t}\,dE_W(t)$.  Further
	 we see that $WW^{-1}E_W([\frac{1}{n},n])=E_W([\frac{1}{n},n])$, so
	 the core $\fD_W$ is common to the domains and ranges of both $W$ and $W^{-1}$.

	\item It is clear that $W\otimes W$ is a weight on the dual of $G\times G$.
	
	\item The extended contraction $X_W$ actually belongs to $VN(G\times G)$. Indeed, we know that both of $\Gamma(W)$ and $W^{-1}\otimes W^{-1}$ are affiliated with $VN(G\times G)$, so that for any unitary $U \in VN(G\times G)'$ we have $U\Gamma(W) \subset \Gamma(W)U$ and $U(W^{-1}\otimes W^{-1}) \subset (W^{-1}\otimes W^{-1})U$. Thus, for any $\xi \in \D(W^{-1} \otimes W^{-1})$ such that $(W^{-1} \otimes W^{-1})\xi \in \D(\Gamma(W))$ we have
		\begin{align*}
		U\Gamma(W)(W^{-1}\otimes W^{-1}) \xi
		& = \Gamma(W) U (W^{-1}\otimes W^{-1}) \xi\\
		& = \Gamma(W)(W^{-1}\otimes W^{-1})U\xi.
		\end{align*}
	Such elements $\xi$ form a dense subspace of $L^2(G\times G)$, so that we know $X_W$ belongs to $VN(G\times G)$.
	
	\item In the quantum group literature the 2-cocycle condition is often different from ours in (c) as follows:
	\[
	(X_W\otimes I)((\Gam\otimes\id) (X_W))=(I\otimes X_W)((\id\otimes\Gam)(X_W)).
	\]
	These conditions coincide in the case of a strong weight.
	
	\item  We will see examples of strong weights in Proposition \ref{prop-extended-weights-abelian}.
	Definition \ref{def-poly-weight-laplacian} will provide some examples of weights which are not strongly commuting.
	
\end{enumerate}
\end{rem}

\begin{ex}\label{ex:abelian-weight}
Let $G$ be an abelian group, and fix a weight function $w:\widehat{G}\rightarrow (0,\infty)$ on the dual of $G$.
Then, the associated multiplication operator $M_w$ acting on $L^2(\widehat{G})$ is a densely defined positive operator affiliated with the von Neumann algebra $L^\infty(\widehat{G})$ as one can see in \cite[p.342]{KaR} for instance.
Recall that in the case of an abelian group, $VN(G)$ can be identified with $L^\infty(\widehat{G})$ via the group Fourier transform $\F^G: L^2(G) \to L^2(\widehat{G})$, i.e. $VN(G) = (\F^G)^{-1} L^\infty(\widehat{G}) \F^G$. Thus, we get a densely defined positive operator $\widetilde{M}_w$ affiliated with the von Neumann algebra $VN(G)$ using this unitary conjugation, i.e. $\widetilde{M}_w := (\F^G)^{-1} \circ M_w \circ \F^G$. Now the discussion in Section \ref{sec:abelian1} tells us that $\widetilde{M}_w$ is a weight on the dual of $G$.
\end{ex}
We say that a positive operator $T$ is {\it bounded below} if it is injective and $T^{-1}$ is bounded.
The upshot of the following is that for none of our examples will we have to manually verify the 2-cocycle condition.

\begin{prop}\label{prop:swisw}
Let $W$ be a densely defined operator on $L^2(G)$.
	\begin{enumerate}
		\item If $W$ satisfies the conditions (a), (b) and (d) in Definition \ref{def:weight}, then it is a strong weight.
		\item If $W$ is bounded below and satisfies (a) and (b) in Definition \ref{def:weight}, then it is a weight.
	\end{enumerate}
\end{prop}

\begin{proof}
(1) Using functional calculus as it applies to homomorphisms and tensor products of commuting operators, we have on a common core (guarranteed by the functional calculus theory) that
\begin{align*}
((\Gam\otimes\id) &(X_W))(X_W\otimes I) \\
&=(\Gam\otimes\id)\left(\Gam(W)(W^{-1}\otimes W^{-1})\right)(\Gam(W)(W^{-1}\otimes W^{-1})\otimes I) \\
&=(((\Gam\otimes \id)\circ\Gam) (W))(\Gam(W^{-1})\otimes W^{-1})((\Gam(W)(W^{-1}\otimes W^{-1}))\otimes I) \\
&=(((\Gam\otimes \id)\circ\Gam) (W))(W^{-1}\otimes W^{-1}\otimes W^{-1}).
\end{align*}
Likewise, on the same core, we have
\[
((\id\otimes\Gam)(X_W))(I\otimes X_W)=(((\id\otimes\Gam)\circ\Gam)(W))(W^{-1}\otimes W^{-1}\otimes W^{-1})
\]
and we appeal to the coassociativity of $\Gam$.

(2) Given assumption (a) and (b), Proposition \ref{prop-extension-variant}, applied to the homomorphisms
$\Gam\otimes\id$ and $\id\otimes \Gam$, justifies the computations of the last paragraph to give (c).
\end{proof}

Now let us consider our {\it first model of the Beurling-Fourier algebra}, defined in the
proposition below, which is analogous to
(\ref{eq:twist}), above.

\begin{prop}\label{prop-BF-alg-asso-com}
Let $W$ be a weight on the dual of $G$.
The map
	$$\Gam_W:VN(G)\to VN(G\times G)\cong VN(G)\bar{\otimes}VN(G)$$
given by
\[
\Gam_W(A)=\Gam(A)X_W
\]
 is weak*-weak* continuous, contractive, co-associative and co-commutative, and hence
induces a product $\cdot_W$ on $A(G)$, making $(A(G),\cdot_W)$ a commutative Banach algebra.
Furthermore, the spectrum of $(A(G),\cdot_W)$ is given by
\[
\spec(A(G),\cdot_W)=\{U\in VN(G):\Gam_W(U)=U\otimes U,\;\; U\ne 0\}.
\]
\end{prop}

\begin{proof}
Note that $X_W$ belongs to $VN(G\times G)$ as we have seen in Remark \ref{rem:weights}.
Weak*-weak*-continuity and contractivity of $\Gam_W$ are trivial.
We use co-associativity of $\Gam$ and the 2-cocycle condition on $X_W$
to see that for $A$ in $VN(G)$ we have
\begin{align*}
((\Gam_W\otimes\id)\circ\Gam_W)(A)
&=((\Gam\otimes\id)(\Gam(A)X_W))(X_W\otimes I) \\
&=(((\Gam\otimes\id)\circ\Gam)(A))[((\Gam\otimes\id)(X_W))(X_W\otimes I)] \\
&=(((\id\otimes\Gam)\circ\Gam)(A))[((\id\otimes\Gam)(X_W))(I\otimes X_W)] \\
&=((\id\otimes \Gam_W)\circ\Gam_W)(A)
\end{align*}
which is the co-associativity condition for $\Gam_W$.

The flip map $\Sig$ on $VN(G)\bar{\otimes}VN(G)$ is a $*$-automorphism
which satisfies $\Sig\circ\Gam=\Gam$, and is
given by $\Sig(A)=FAF$ where $F=F^*$ is the flip unitary on $L^2(G\times G)$.
Note that we have $F(W^{-1}\otimes W^{-1})F=W^{-1}\otimes W^{-1}$ on $\mathcal D=\mathcal D(W^{-1})\otimes \mathcal D(W^{-1})$, which is a core for the operator $W^{-1}\otimes W^{-1}$. Hence
	$$W^{-1}\otimes W^{-1}=\overline{F(W^{-1}\otimes W^{-1})F|_{\mathcal D}}\subset F(W^{-1}\otimes W^{-1}) F$$
which implies $W^{-1}\otimes W^{-1}=F(W^{-1}\otimes W^{-1})F$, as both operators are selfadjoint.
We have also $F\Gamma(W)F=\Gamma(W)$ as $F\Gamma(\lambda(s))F=\Gamma(\lambda(s))$ giving that $F\Gamma(A)F=\Gamma(A)$ for any $A\in VN(G)$ and hence for any self-adjoint operator affiliated with $VN(G)$. From the definition of weight we know that $\mathcal D(\Gamma(W)(W^{-1}\otimes W^{-1}))$ is dense in $L^2(G\times G)$ and so is $F(\mathcal D(\Gamma(W)(W^{-1}\otimes W^{-1})))$. Now for $\xi \in F(\mathcal D(\Gamma(W)(W^{-1}\otimes W^{-1})))$ we have
	\begin{align*}
	F\Gamma(W)(W^{-1}\otimes W^{-1})F(\xi)
	& = (F\Gamma(W)F)(F(W^{-1}\otimes W^{-1})F)(\xi)\\
	& = \Gamma(W)(W^{-1}\otimes W^{-1})(\xi),
	\end{align*}
so that $FX_WF(\xi) = X_W(\xi)$, and the density of the subspace $F(\mathcal D(\Gamma(W)(W^{-1}\otimes W^{-1})))$ in $L^2(G\times G)$ tells us that $FX_WF = X_W$.
Consequently, we have $\Sig\circ\Gam_W=\Gam_W$, i.e. $\Gam_W$ is co-commutative.

For the last assertion we note that $u\cdot_Wv = (\Gam_W)_*(u\otimes v)$ for $u,v$ in $A(G)$. Thus for $U$ in $VN(G)$,
$U\in\spec(A(G),\cdot_W)$ exactly when for any $u,v$ in $A(G)$
\[
\langle \Gam_W(U),u\otimes v\rangle=
\langle U,u\cdot_W v\rangle=
\langle U,u\rangle\langle U,v\rangle=
\langle U\otimes U,u\otimes v\rangle
\]
which is when $ \Gam_W(U)=U\otimes U$.
\end{proof}

\begin{rem}
The coproduct $\Gam_W$ on $VN(G)$ is completely contractive, so $(A(G),\cdot_W)$ is a completely	contractive Banach algebra with respect to the canonical operator space structure on $A(G) = VN(G)_*$. See \cite{ER} for the details of operator spaces and completely	contractive Banach algebras.
\end{rem}

In computational practice, it is more convenient to consider {\it the second model of a Beurling-Fourier algebra}, which is equivalent to the first model.
First, we define weighted spaces following \eqref{eq:Balgdual}.

\begin{defn}
We define the weighted space $VN(G,W^{-1})$ by
	$$VN(G,W^{-1}) := \{AW: A\in VN(G)\}$$
with the norm
	$$\|AW\|_{VN(G,W^{-1})} := \|A\|_{VN(G)},$$
which gives us a natural isometry:
	$$\Phi : VN(G) \to VN(G,W^{-1}),\; A\mapsto AW.$$
We also endow $VN(G,W)$ with the operator space structure that  makes $\Phi$ a complete isometry.

The weighted space $VN(G\times G, W^{-1}\otimes W^{-1})$ is similarly defined.
\end{defn}

\begin{rem}
	Note that $VN(G,W^{-1})$ is well-defined as a set. Indeed, if we have $AW = BW$ for $A,B\in VN(G)$ we know that $A$ and $B$ coincide on $\ran W$, which is dense from the definition of weight $W$. Thus, we have $A=B$.
\end{rem}

We continue to define weighted space and the second model of a Beurling-Fourier algebra following \eqref{eq:Balg}.

\begin{defn}\label{2nddef-BFalg}
We define the weighted space
	$$A(G, W) := (\Phi^{-1})^*(A(G)),$$
which can be understood as a predual $VN(G,W^{-1})_*$ of $VN(G,W^{-1})$ with an obvious duality bracket. Then, the Banach algebra structure of $(A(G), \cdot_W)$ can be transferred to $A(G,W)$ via the isometry $\Phi^{-1}_* = (\Phi^{-1})^*|_{A(G)}$. In other words, the algebra multiplication on $A(G,W)$ is given by
	$$\Gam^W_* :=
\Phi_*^{-1}\circ (\Gam_W)_*\circ(\Phi_*\otimes \Phi_*):A(G,W)\prt A(G,W)\to A(G,W),$$
where $\prt$ is the operator space projective tensor product. For $u,v \in A(G,W)$ we use the notation
	$$u\cdot v := \Gam^W_*(u\otimes v).$$
\end{defn}

The following is straightforward from Proposition \ref{prop-BF-alg-asso-com}.

\begin{prop}\label{prop:2ndspec}
We have
\[
\spec(A(G,W),\cdot)=\{UW:U\in\spec(A(G),\cdot_W)\}.
\]
\end{prop}

The space $A(G,W)$ is defined as an abstract predual of $VN(G,W^{-1})$. However, in many situations, we could give a concrete model for $A(G,W)$ which justifies the title ``weighted space'' as follows.

\subsubsection{When $W$ is bounded below}\label{ssec:bddbelow}

Recall that for $u\in A(G)$ and $S\in VN(G)$, we let $Su$ in $A(G)$ be given by $\langle Su,A\rangle=\langle u,AS\rangle$, for every $A\in VN(G)$.

\begin{prop}\label{prop:boundedlyinvertible}
If $W$ is bounded below, then there is a natural continuous injective algebra homomorphism
	$$j_*:A(G,W)\to A(G),\; \Phi^{-1}_*(u)\mapsto W^{-1}u,$$
whose adjoint is the formal embedding
	$$j:VN(G)\to VN(G,W^{-1}),\; A \mapsto AW^{-1}W.$$
\end{prop}
\begin{proof}
We consider the elements $W_n=WE_W([\frac{1}{n},n])\in VN(G)$ which satisfy
$W_nW^{-1}=E_W([\frac{1}{n},n])$, by (1) of Remark \ref{rem:weights}.
It follows that $$\{AW^{-1}:A\in VN(G)\}\supseteq \{AE_W([\frac{1}{n},n]): A\in VN(G)\},$$ and the latter set is weak* dense in $VN(G)$, which shows that $j_*$ is injective.

Let us consider the adjoint map $j=(j_*)^*$. Given $A$ in $VN(G)$ and $v\in A(G,W)$, we have
\[
\langle v,j(A)\rangle=\langle j_*(v),A\rangle=\langle W^{-1}\Phi_*(v),A\rangle=\langle \Phi_*(v),AW^{-1}\rangle=\langle v,AW^{-1}W\rangle,
\]
so  $j(A)=AW^{-1}W$.
It is immediate that $\Phi^{-1}\circ j(A)=AW^{-1}$ for $A$ in $VN(G)$.

We wish to see that $j_*$ is an algebra homomorphism, which shows that its range is a subalgebra of $A(G)$. Hence we must show that the identity $\Gamma_*\circ (j_*\otimes j_*)=j_*\circ\Gam_*^W=j_*\circ \Phi_*^{-1}\circ (\Gamma_W)_*\circ (\Phi_*\otimes \Phi_*)$
holds on $A(G,W)\otimes A(G,W)$, which is equivalent to having
\[
\Gamma_*\circ (j_*\circ\Phi^{-1}_*\otimes j_*\circ\Phi^{-1}_*)=j_*\circ \Phi_*^{-1}\circ (\Gamma_W)_*
\]
on $A(G)\otimes A(G)$.  If $A\in VN(G)$ we compute
\[
\Gam_W\circ\Phi^{-1}\circ j(A)=\Gam_W(AW^{-1})=\Gam(A)W^{-1}\otimes W^{-1}=(\Phi^{-1}\circ j\otimes \Phi^{-1}\circ j)\circ \Gamma (A)
\]
which gives the desired equation by duality.
\end{proof}

\begin{rem}
The above embedding $j_*$ allows us to identify $A(G,W)$ with the space
	$$\tilde{A}(G,W) := \{W^{-1}u : u\in A(G)\}\subseteq A(G)$$ with the norm
	$$\|W^{-1}u\|_{\tilde{A}(G,W)} := \| u\|_{A(G)}.$$
\end{rem}

It is well-known that $\mathrm{Spec}(A(G))=\{\lam(s):s\in G\}\cong G$, where the last identification is a homeomorphism.

\begin{cor}
If $W$ is bounded below, then $\mathrm{Spec}(A(G,W),\cdot)$ contains the set $$\{\lam(s)W^{-1}W:s\in G\},$$ which in weak*
topology is homeomorphic to $G$.
\end{cor}

\begin{proof}
We appeal to Proposition \ref{prop:2ndspec}, the fact above, and the fact that $j$ is weak*-weak* continuous
and injective.  Thus, if $G$ is non-compact,
$j(\{\lam(s):s\in G\}\cup\{0\})$ is compact and homeomorphic to $\{\lam(s):s\in G\}\cup\{0\}$,
which in turn is homeomorphic to the one-point compactification of $G$.
\end{proof}

\begin{rem}\label{rem-weights-BFalg}
\begin{enumerate}
    \item Let us compare the definitions of weights on the dual of a general locally compact group $G$ from \cite[Definition 2.4]{LS} and Definition \ref{def:weight}.
    
    First of all, the restrictions on a weight $W$ being boundedly invertible and on the operator $\Gamma(W)(W^{-1}\otimes W^{-1})$ being selfadjoint in \cite[Definition 2.4]{LS} has been removed in Definition \ref{def:weight}, which allows us to include more examples of  weights. Moreover, the techniques of exhibiting non-trivial examples of weights in Section \ref{sec-construction-weight} can easily be verified with Definition \ref{def:weight}, whilst the case of \cite[Definition 2.4]{LS} is not clear for the moment.
    
    Secondly, \cite[Definition 2.4]{LS} is slightly misleading as follows. A closed positive operator $W$ affiliated with $VN(G)$ is called a weight on the dual of $G$ in the sense of \cite[Definition 2.4]{LS} when the operator $\Gamma(W)$ satisfies certain conditions. However, the understanding of $\Gamma(W)$ in \cite{LS} is based on the extension of $*$-homomorphism given by \cite[Lemma 2.1]{LS}, which depends on the choice of a net of projections in the Lemma. This forces us to assume a fixed choice of a net of projections for $\Gamma(W)$ and the validity of the rest of the results in \cite[Section 2.1]{LS} is maintained when we make this implicit assumption.
    
    In contrast, we understand $\Gamma(W)$, in this article, through spectral integrals as is explained in Section \ref{ssec:homomorphisms}, which is free of the above mentioned ambiguity. Note that the spaces $\mc D:= \cup^\infty_{n=1} {\ran} E_W([\frac{1}{n},n])$ and $\mc D':= \cup^\infty_{n=1} {\ran} \Gamma ( E_W([\frac{1}{n},n]))$ are cores for $W$ and $\Gamma(W)$, respectively, thanks to the property of spectral integrals. This means that the projections $E_W([\frac{1}{n},n])$, $n\ge 1$, satisfy all the conditions of \cite[Lemma 2.1]{LS}, and $\Gamma(W)$ there coincides with the one from Section \ref{ssec:homomorphisms}. Thus, we can say, thanks to Proposition \ref{prop:swisw}, that the weight $W$ on the dual of $G$ in \cite[Definition 2.4]{LS} is a weight in the sense of Definition \ref{def:weight} as long as we take the choice of the above sequence of projections. Note that the examples in \cite[Section 2.2 -- 2.3]{LS} are based on ``canonical" choices of projections different from the above sequence $E_W([\frac{1}{n},n])$, $n\ge 1$. We will explain that they are still included in the framework of Definition \ref{def:weight} in Remark \ref{rem-coincide-cpt} and Remark \ref{rem-Heisenberg-central-weights}.
    
    \item
     In \cite{ORS} the authors defined a weight inverse $w^{-1}$ as an element of $\M(C^*_r(G))$, the multiplier algebra of the reduced group $C^*$-algebra of $G$, which is expected to be a replacement of $W^{-1}\in VN(G)$ in \cite[Definition 2.4]{LS}. Since $\M(C^*_r(G))$ and $VN(G)$ are different in general, a direct comparison is not possible for a general locally compact group $G$. However, for a compact group $G$ we do have $\M(C^*_r(G)) = VN(G)$, and \cite[Examples]{ORS} (before Theorem 2.6 there) explains that the central weights in \cite[Definition 2.4]{LS} produce weight inverses in \cite{ORS} in this case.

	\item Proposition \ref{prop:boundedlyinvertible} shows the truth of Remark 2.10 (1) of \cite{LS}, effectively replacing
	Theorem 2.8 of that article.

	\item
The elements of $VN(G,W^{-1})$ are poorly behaving as operators. For example, the operator $AW$, $A\in VN(G)$ is not even closable in general. Indeed, we know that $(AW)^* = W^*A^* = WA^*$ (\cite[Proposition 1.7(ii)]{Sch}), so that $AW$ is closable if and only if $\mathrm{dom}(WA^*)$ is dense in $L^2(G)$. This cannot be true for the operator $A$ with ${\ran}A^*\cap \mathrm{dom}(W) = \{0\}$. So we need to be very careful when we deal with general elements in weighted spaces and this is one of the obstacles to connect $\varphi \in {\rm Spec}A(G,W)$ to McKennon's model (\cite{McK2}) of complexifications for locally compact groups. More precisely, $\varphi \in {\rm Spec}A(G,W)$ should be understood as an element $AW \in VN(G,W^{-1})$, which, in general, is not even closable whilst any element of $G_\Comp$ in McKennon's model is a closed operator affiliated to $VN(G)$ acting on $L^2(G)$ satisfying a certain property.

\end{enumerate}
\end{rem}

\subsubsection{When $G$ is separable and type I}\label{ssec:separable-type I}

In this case the unitary dual $\widehat{G}$ is a standard Borel space. Moreover, by \cite[Theorem 3.24]{Fuhr}, we have a standard measure $\mu$ on $\widehat{G}$ and a $\mu$-measurable cross-section $\xi \to \pi^\xi$ from $\widehat{G}$ to concrete irreducible unitary representations acting on $H_\xi$ such that $\lambda$ is quasi-equivalent to $\int^\oplus_{\widehat{G}} \pi^\xi d\mu(\xi)$ so that we have
	$$VN(G) \cong L^\infty(\widehat{G},d\mu(\xi);\B(H_\xi)).$$
Thus, in turn we get	
	$$A(G) \cong L^1(\widehat{G},d\mu(\xi);S^1(H_\xi)).$$
With the above identifications in mind, we would like to focus on a general weight $W$ on the dual of $G$, which is not necessarily bounded below. Since the weight $W$ is a closed operator affiliated with $VN(G)$, it is automatically decomposable (see \cite[Proposition 4.4]{DNSZ}) with the decomposition
	$$W = \int^\oplus_{\widehat{G}}W_\xi d\mu(\xi)$$
such that $W_\xi$ is closed and densely defined on $H_\xi$ for almost every $\xi$.
This decomposition allows us to make a concrete realization of certain elements of $A(G,W)$ as in the case of $W$ being bounded below. We first consider the following map.
	$$j_*: A(G,W) \to \A,\; \Phi^{-1}_*(\phi) \mapsto W^{-1}\phi = (W^{-1}_\xi\phi_\xi)_{\xi\in \widehat{G}},$$
where $\A$ is the space of all decomposable closed operators acting on $L^2(G)$. Note that $\phi_\xi$ is a bounded operator, so that $W^{-1}_\xi\phi_\xi$ is a closed operator for almost every $\xi$. 

Now we consider a dense subspace $\mc S$ of $A(G,W)$ such that for any $\phi\in A(G)$ with $\Phi^{-1}_*(\phi)\in \mc S$, the operator $W^{-1}_\xi\phi_\xi$ is densely defined for almost every $\xi$. Note that the elements of $j_*(\mc S)$ are exactly those decomposable closed operators $(\psi_\xi)_{\xi\in \widehat{G}}$ acting on $L^2(G)$ such that (1) $\psi_\xi$ is densely defined and ${\ran}(\psi_\xi) \subseteq {\rm dom}(W_\xi)$ for almost every $\xi$; (2) and the field of operators $(W_\xi\psi_\xi)_{\xi\in \widehat{G}}$ belongs to $\Phi_*(\mc S) \subseteq L^1(\widehat{G},d\mu(\xi);S^1(H_\xi))$; this observation will be used frequently later with the choice of the subspace $\mc S$ in each case.

Moreover, we can see that the map $j_*$ is an embedding on $\mc S$, i.e. $(j_*)|_\mc S$ is injective. Indeed, if we have $W^{-1}\phi = W^{-1}\tilde{\phi}$ for $\phi, \tilde{\phi}\in A(G)$ with $\Phi^{-1}_*(\phi), \Phi^{-1}_*(\tilde{\phi}) \in \mc S$, then for almost every $\xi$ we have $W^{-1}_\xi\phi_\xi = W^{-1}_\xi\tilde{\phi}_\xi$, which are densely defined closed operators. By \cite[Proposition 1.7(i)]{Sch} we have $\phi^*_\xi W^{-1}_\xi \subset (W^{-1}_\xi\phi_\xi)^* = (W^{-1}_\xi\tilde{\phi}_\xi)^* \supset \tilde{\phi}^*_\xi W^{-1}_\xi$. Since ${\rm dom}(\phi^*_\xi W^{-1}_\xi) = {\rm ran}W_\xi = {\rm dom}(\tilde{\phi}^*_\xi W^{-1}_\xi)$ we have $\phi^*_\xi W^{-1}_\xi = \tilde{\phi}^*_\xi W^{-1}_\xi$. So, we conclude that $\phi^*_\xi$ and $\tilde{\phi}^*_\xi$ coincide on ${\rm dom}W_\xi$, which is a dense subspace of $H_\xi$ by \cite[Proposition 12.1.8]{Sch2}, and we get $\phi_\xi = \tilde{\phi}_\xi$ a.e. $\xi$.

As before the above embedding $j_*$ allows us to identify $\mc S\subseteq A(G,W)$ with the space
	$$\tilde{\mc S} := \{W^{-1}\phi : \phi\in A(G),\; \Phi^{-1}_* \phi \in \mc S\}\subseteq \A$$ with the norm
	$$\|W^{-1}\phi \|_{\tilde{\mc S}} := \|\phi\|_{A(G)}.$$	
As we shall frequently deal with ``central'' weights in the case of separable type I groups, we introduce the definition here.
	\begin{defn}\label{def-central-weights}
		We call a weight $W$ on the dual of a locally compact group $G$ {\it central} if it is affiliated with the centre of $VN(G)$.
	\end{defn}
Hence if $G$ is separable and type I, then a central weight $W$ admits the decomposition $W = \int^\oplus_{\widehat{G}}W_\xi d\mu(\xi)$ with $W_\xi$ being a constant multiple of the identity for almost every $\xi$.
\begin{rem}
When $G$ is compact, we can actually drop the condition of separability on $G$ in the above.
\end{rem}

\subsection{Examples of weights}\label{sec-construction-weight}

In this section we review three fundamental ways of constructing weights on the dual of $G$, namely central weights as in Definition \ref{def-central-weights}, weights extended from subgroups and the weights using Laplacian on the group. Regardless of the choice of the methods they are all based on classical examples of weight functions on $\Real^n$ or $\z^n$, which we begin with.

\subsubsection{A list of weight functions on $\Real^k \times \z^{n-k}$}

We first exhibit canonical examples of weight functions on $\Real$ or $\z$.

\begin{ex} For $a,s,t\ge 0$ and $0\le b\le 1$ we have a family of weight functions on $\Real$ as follows:
	$$w_{a,b,s,t}(x) := e^{a|x|^b}(1+|x|)^s(\log(e+|x|))^t, \; x\in \Real.$$
The above includes polynomial weights ($a=t=0$), exponential weights ($b=1$, $s = t=0$), and sub-exponential weights ($b<1$). The restrictions $w_{a,b,s,t}|_\z$ are clearly weight functions on $\z$.
\end{ex}

We may extend the above weights to $\Real^k\times \z^{n-k}$ by tensoring.
\begin{ex}\label{ex-weights-higher-dim-abelian}
For $a_j,s_j,t_j\ge 0$ and $0\le b_j\le 1$, where $1\le j \le n$, we have a family of weight functions on $\Real^k\times \z^{n-k}$ as follows:
	$$w_{(a_j,b_j,s_j,t_j)_{1\le j\leq n}}(x_1, \cdots, x_n) := \prod^n_{j=1}w_{a_j,b_j,s_j,t_j}(x_j)$$
for $x_1, \cdots, x_k\in \Real,\; x_{k+1}, \cdots, x_n \in \z$.
\end{ex}

Of course, there are many other weights, but any weights on $\Real^k\times \z^{n-k}$ are at most exponentially growing.
\begin{prop}\label{prop-at-most-exp}
Let $w$ be a weight function on $\Real^k\times \z^{n-k}$. Then, there are positive numbers $C,\rho_1, \cdots, \rho_n$ such that
	$$w(x_1, \cdots, x_n) \le C\rho_1^{|x_1|} \cdots \rho_n^{|x_n|}, \;\; (x_1, \cdots, x_n) \in \Real^k\times \z^{n-k}.$$
\end{prop}
\begin{proof}
We recall, as noted in Remark \ref{rem:weightfunc}, any weight function is locally bounded.
Then it is straightforward to check that
	$$w(x_1, \cdots, x_n) \le m^{\lceil|x_1|\rceil}_1 \cdots  m^{\lceil|x_n|\rceil}_k
	\le m_1^{|x_1|+1}\cdots m_n^{|x_n|+1}$$
where $m_i=\sup_{t\in [-1,1]}w(\cdots,0,t,0,\cdots)$ ($i$th position) for
$1\le i \le n$.
\end{proof}

\begin{ex} There is a family of multiplicative weight functions on $\Real$. For $c\in \Real$ we have
	$$w_c(x) = e^{cx},\; x\in \Real.$$
The restriction $w_c|_\z$ is a multiplicative weight function on $\z$. This class of examples provide weights which are not bounded below.
\end{ex}

\subsubsection{Central weights}\label{subsection-central-weights}

When the group $G$ is compact or more generally separable and type I, it is natural to expect weights to be central in the sense of Definition \ref{def-central-weights}. Indeed, this is what happened in the case of compact groups as we have seen in the Introduction. Here, we present a different approach more suitable to Definition \ref{def:weight}.

Let $G$ be a compact group. It is well-known that the left regular representation $\lambda$ on $G$ is a direct sum of irreducible ones, i.e. we have the quasi-equivalence
	\begin{equation}\label{eq-cpt-regular-decomp}
	\lambda \cong \oplus_{\pi \in \widehat{G}}\, \pi,
	\end{equation}
which gives us
	$$VN(G) \cong \bigoplus_{\pi \in \widehat{G}}M_{d_\pi}\;\;\text{and}\;\; A(G) \cong \ell^1\text{-}\bigoplus_{\pi \in \widehat{G}} d_\pi S^1_{d_\pi}.$$
Here, $S^p_n$, $1\le p \le \infty$ refers to Schatten $p$-class on $\ell^2_n$. In the above, the latter identification uses a standard duality coming from the Plancherel theorem
	$$L^2(G) \cong \ell^2\text{-}\bigoplus_{\pi \in \widehat{G}} \sqrt{d_\pi} S^2_{d_\pi}.$$	

Let $w: \widehat{G} \to (0,\infty)$ be a function. We associate $w$ with the operator
	$$W = \bigoplus_{\pi \in \widehat{G}}w(\pi)I_\pi$$
which can be understood as a spectral integral. More precisely, we consider the spectral measure
	$$E:\mathcal{P}(\widehat{G}) \to \B(H),\; J \mapsto \oplus_{\pi \in J}I_\pi$$
where $\mathcal{P}(\widehat{G})$ is the $\sigma$-algebra of all subsets of $\widehat{G}$ and $H$ is given by
	$$H = \ell^2 \text{-}\bigoplus_{\pi}\sqrt{d_\pi}S^2_{d_\pi}.$$
Indeed, the scalar valued map $\la E(\cdot) X, X \ra : \mathcal{P}(\widehat{G}) \to \Real$ for any $X=(X(\pi))_\pi \in H$ is given by
	$$\la E(J) X, X \ra = \sum_{\pi \in J} d_\pi \|X(\pi)\|^2_2, \;\; J\subseteq \widehat{G}.$$
so that $\la E(\cdot) X, X \ra$ is a measure for any $X$. Thus, \cite[Lemma 4.4]{Sch} tells us that $E$ is a spectral measure, and we define $W$ to be the spectral integral $\displaystyle \int_{\widehat{G}}w\, dE$, which we denote by $\bigoplus_{\pi \in \widehat{G}}w(\pi)I_\pi.$

Recall that for $A = (A(\pi))_{\pi\in \widehat{G}} \in VN(G)$ we have
	$$\Gamma(A) = \oplus_{\pi, \pi'}\left[ U^*_{\pi,\pi'}(\oplus_{\sigma \subseteq \pi\otimes \pi'} A(\sigma))U_{\pi, \pi'} \right]$$
where for $\sigma, \pi,\pi'\in\widehat{G}$, the notation $\sigma \subseteq \pi\otimes \pi'$ means that $\sigma$ is a subrepresentation of $\pi\otimes \pi'$, and $U_{\pi, \pi'}$ is the unitary appearing in the irreducible decomposition of $\pi \otimes \pi'$. Since $\displaystyle \Gamma(W) = \int_{\widehat{G}}w\, d(\Gamma \circ E)$, we have for $\pi, \pi' \in \widehat{G}$ that
	\begin{align*}
	\Gamma(W)(W^{-1}\otimes W^{-1}) (\pi, \pi')
	& = U^*_{\pi,\pi'}(\oplus_{\sigma \subseteq \pi \otimes \pi'}w(\sigma)w(\pi)^{-1}w(\pi')^{-1} I_\sigma)U_{\pi,\pi'},
	\end{align*}
which explains that $\Gamma(W)(W^{-1}\otimes W^{-1})$ is a contraction if and only if
	\begin{equation}\label{eq-weight-function-compact}
	    w(\sigma) \le w(\pi)w(\pi')
	\end{equation}
for any $\sigma, \pi, \pi' \in \widehat{G}$ such that $\sigma \subseteq \pi \otimes \pi'$. This is exactly the sub-multiplicativity condition \eqref{eq-submultiplicativity-compact}. For the 2-cocycle condition we note that $W\otimes W$ is central (i.e. at each components we have constant multiples of the identity), so that $\Gamma(W)$ and $W\otimes W$ are strongly commuting.

From the above description of the central weight $W = \bigoplus_{\pi \in \widehat{G}}w(\pi)I_\pi$ and the discussion in Section \ref{ssec:separable-type I} with $\mc S = A(G,W)$ we recover the Beurling-Fourier algebra $A(G, W) = A(G,w)$ on compact groups in \cite{LST}
    \begin{equation}\label{eq-cpt-central-A(G, W)}
	A(G, W) = \Big\{X = (X_\pi) \in \prod_{\pi \in \widehat{G}}M_{d_\pi}: \sum_{\pi\in \widehat{G}}d_\pi w(\pi) \|X_\pi\|_1<\infty\Big\}.
	\end{equation}
Considering the embedding ${\rm Trig}(G) \hookrightarrow \prod_{\pi \in \widehat{G}}M_{d_\pi}$, $f \mapsto (\widehat{f}(\pi))_{\pi \in \widehat{G}}$ we get
    $$\|f\|_{A(G,W)} = \sum_{\pi\in \widehat{G}}d_\pi w(\pi) \|\widehat{f}(\pi)\|_1,\;\; f\in {\rm Trig}(G) \subseteq C(G)$$
as before.

\begin{rem}\label{rem-coincide-cpt}
    \begin{enumerate}
        \item Note that we are not assuming $w$ to be bounded away from zero unlike \cite{LS}.
        
        \item In \cite[Section 2.2]{LS} the net of projections $\{E(J): J \subseteq \widehat{G},\; |J| <\infty\}$ was considered to define $\Gamma(W)$. However, it is immediate to see that both of the definitions of $\Gamma(W)$, the one from \cite[Section 2.2]{LS} and the one from Section \ref{ssec:homomorphisms}, coincide. 
    \end{enumerate}
\end{rem}

We call the above $A(G,W)$ a {\it central} Beurling-Fourier algebra on compact groups as in \cite{LS}. More detailed examples of central weights on the dual of compact groups will be presented later in Section \ref{sec-weights-compact}.
	\begin{rem}
	In \cite{LS} a model for central weights on the dual of the Heisenberg group has been investigated, which can be explained in the model we described in the next section, namely the weights extended from subgroups. See Remark \ref{rem-Heisenberg-weights} for more details. Note also that central weights on the dual of the reduced Heisenberg group and the dual of the Euclidean motion group are completely described in Remark \ref{rem-reduced-Heisenberg-central-weights} and Remark \ref{rem-E(2)-central-weights}, respectively.
	\end{rem}

\subsubsection{Extension from closed subgroups}\label{subsection-ext-subgroup}

In this subsection we provide a fundamental construction of weights by extending from closed subgroups under mild assumptions. Let $H$ be a closed subgroup of a locally compact group $G$ and we consider the restriction map $R_H:A(G)\to A(H),\; f \mapsto f|_H$, which is a surjective quotient homomoprhism thanks
to Herz restriction theorem, see for example \cite{herz,mcmullan}.
Then, the adjoint $\iota=R_H^*$ is an injective $*$-homomorphism satisfying
	\begin{equation}\label{eq-embedding-subgroup}
	    \iota : VN(H) \hookrightarrow VN(G),\; \lambda_H(x) \mapsto \lambda_G(x),
	\end{equation}
where $\lambda_H$ and $\lambda_G$ are left regular representations of $H$ and $G$, respectively. From this point on the symbol $\iota$ will be reserved for the above particular embedding.

	\begin{prop}\label{prop-ext-subgroup}
	Let $H$ be a closed subgroup of a locally compact group $G$ and $W_H$ be a weight on the dual of $H$. Then the operator $W_G = \iota(W_H)$ is a weight on the dual of $G$ provided that any of the following holds:
\begin{enumerate}	
	\item $H$ is abelian,
	\item $W_H$ is central, or
	\item $W_H$ is bounded below.
\end{enumerate}	
	\end{prop}
\begin{proof}
First, we observe that $\Gamma_G \circ \iota = (\iota \otimes \iota)\circ \Gamma_H.$ Indeed, we can easily check the identity for the elements $\lambda_H(x)$, $x\in H$ and we can use $w^*$-density to get the general result.
In cases (1) and (2) we have that $W_H$ is a strong weight, i.e. $\Gamma_H(W_H)$ and $W_H\otimes W_H$ are strongly commuting. Then, we can easily see that $\Gamma_G(W_G) = \Gamma_G \circ \iota (W_H) = (\iota \otimes \iota) (\Gamma_H(W_H))$ and $W_G\otimes W_G = (\iota \otimes \iota) (W_H\otimes W_H)$ are also strongly commuting from the construction in Section \ref{ssec:homomorphisms}. Thus, we may apply functional calculus for the above operators, so that we get $X_{W_G} = (\iota \otimes \iota)(X_{W_H})$ a contraction, which immediately implies that $W_G$ is even a strong weight on the dual of $G$ by Proposition \ref{prop:swisw}.

For the case (3) we appeal to Proposition \ref{prop-extension-variant} to get $X_{W_G} = (\iota \otimes \iota)(X_{W_H})$ and to verify the 2-cocycle condition, which establishes (b) and (c) of Definition \ref{def:weight}.
\end{proof}

When $G$ is a connected Lie group and $H$ is a connected abelian Lie subgroup, then all the extended weights are obtained from the Lie derivatives via functional calculus.
Note that if $H$ is a connected abelian Lie group then it is isomorphic as a Lie group to  $\Real^j \times \tor^{n-j}$, for some $0\le j\le n$.  In particular, we can arrange a basis $\{X_1, \cdots, X_n\}$ for the Lie algebra $\h$
of $H$ for which $\exp(\Ree X_1+\dots+\Ree X_j)\cong\Ree^j$, and each $\exp(X_k)\cong\tor$ for
$k=j+1,\dots,n$.

\begin{prop}\label{prop-extended-weights-abelian}
Let $H$ be a closed connected abelian Lie subgroup of a connected Lie group $G$, with basis
$\{X_1, \cdots, X_n\}$ for its Lie algebra $\h$ arranged as above.  A weight function
$w: \widehat{H} \to (0,\infty)$, induces a weight on $G$ by functional calculus as follows:
\[
W_G=w(i\partial\lambda_G(X_1), \cdots, i\partial\lambda_G(X_n)).
\]
\end{prop}

\begin{proof}
First, we let $H=\Real^j \times \tor^{n-j}$.  The Fourier transform of a derivative gives the formula
\[
i\partial \lambda_H(X_k)=(\F^H)^{-1}\circ M_{x_k}\circ \F^H
\]
where $M_{x_k}f(x_1,\dots,x_n)=x_kf(x_1,\dots,x_n)$.  Notice that $x_1,\dots,x_j$ are real variables, whereas
$x_{j+1},\dots,x_n$ are integer variables. Hence functional calculus, and its interchangability with homomorphisms,
provides that
\[
w(i\partial\lambda_H(X_1), \cdots, i\partial\lambda_H(X_n))=\widetilde{M}_w,
\]
where $\widetilde{M}_w$ is the unitary conjugation of a multiplication operator as in Example \ref{ex:abelian-weight}.
Since $H$ is abelian, we may apply Proposition \ref{prop-ext-subgroup}, the fact that $\iota(\partial\lambda_H(X_k))
=\partial\lambda_G(X_k)$ for each $k$, and again functional calculus, to obtain the desired formula.
\end{proof}

When we have two subgroups isomorphic by an automorphism on the group, then there is a natural connection between the corresponding spectrums. Let $\alpha: G \to G$ be a continuous group automorphism. If $G$ is a connected Lie group, then from the definition of universal complexification we have a uniquely determined analytic extension $\alpha_\Comp : G_\Comp \to G_\Comp$ of $\alpha$, where $G_\Comp$ is the universal complexification of $G$. It is straightforward to see that $\alpha_\Comp$ is also a group automorphism. Note that the transferred measure $d(\alpha(x))$  is left invariant again, so that there is a constant $C>0$ such that $d(\alpha(x)) = C\cdot dx$. This implies that we have a unitary $U = \sqrt{C}\alpha_{L^2}$ on $L^2(G)$, where $\alpha_{L^2}: L^2(G) \to L^2(G),\; f\mapsto f\circ \alpha$. This map transfers to the level of $A(G)$ and $VN(G)$ as follows.
	$$\alpha_A : A(G) \to A(G),\; f \mapsto f\circ \alpha.$$
Indeed, for $f(x) = \la\lambda(x)\xi, \eta\ra$, $\xi,\eta \in L^2(G)$ we have
	\begin{align*}
	f\circ \alpha(x)
	& =  \la\lambda(\alpha(x))\xi, \eta \ra\\
	& = \int_G \xi(\alpha(x)^{-1}y)\overline{\eta(y)}dy\\
	& = \int_G \xi(\alpha(x^{-1}\alpha^{-1}(y))\overline{\eta(y)}dy\\
	& = C\cdot \int_G \xi(\alpha(x^{-1}z)\overline{\eta(\alpha(z))}dz\\
	& = C\cdot \la\lambda(x)\xi\circ \alpha, \eta\circ \alpha \ra\\
	& = \la\lambda(x) U(\xi), U(\eta)\ra ,
	\end{align*}
which means $\alpha_A$ is an isometry. Then, its adjoint $\alpha_{VN} = \alpha^*_A$ is given by
	$$\alpha_{VN} : VN(G) \to VN(G),\; \lambda(x) \mapsto \lambda(\alpha(x)) = U^*\lambda(x)U.$$
Now we can observe that $\alpha_{VN}$ is an inner normal $*$-isomorphism.	

\begin{thm}\label{thm-automorphism-principle}
Let $\alpha: G \to G$ be a continuous automorphism and $W$ be a weight on the dual of $G$ which is either a strong weight or is bounded below. Then $\alpha_{VN}(W)$ is also a weight on the dual of $G$ and we have
	$${\rm Spec} A(G,\alpha_{VN}(W)) \cong {\rm Spec} A(G, W)$$
which is implemented by resticting  the isometry
	$$\Pi: VN(G,W^{-1}) \to VN(G, \alpha_{VN}(W^{-1})),\; AW \mapsto \alpha_{VN}(A)\alpha_{VN}(W).$$
\end{thm}

\begin{proof}
We first note that $\alpha_{VN}(W)$ is also a weight on the dual of $G$. Indeed, we have the identity
	\begin{equation}\label{eq-automorphism}
	(\alpha_{VN} \otimes \alpha_{VN}) \circ\Gamma = \Gamma\circ\alpha_{VN}.
	\end{equation}
It is straightforward to verify conditions (a) and (b) of Definition \ref{def:weight}
using properties of normal automorphism $\alpha_{VN}$, which is implemented by unitary,
and appeal to Proposition \ref{prop:swisw} to see that $\alpha_{VN}(W)$ is a weight.  Likewise, on
$VN(G)$ we have
\[
(\alpha_{VN} \otimes \alpha_{VN})\circ\Gamma_{\alpha_{VN}(W)} = \Gamma_W \circ \alpha_{VN}
\]
which, by Proposition \ref{prop-BF-alg-asso-com}, shows that $\alpha_{VN*}:(A(G),\cdot_W)\to (A(G),\cdot_{\alp_{VN}(W)})$ is an isometric isomorphism, and $\alp_{VN}$ carries
$\mathrm{Spec}(A(G),\cdot_{\alp_{VN}(W)})$ onto $\mathrm{Spec}(A(G),\cdot_W)$.

We let $\Phi^\alpha:VN(G)\to VN(G,\alpha_{VN}(W)^{-1})$
be given in analogy to $\Phi$ in the definition of the space $VN(G,W^{-1})$.  We then let
$\Pi_*=({\Phi^{\alpha}_*})^{-1}\circ\alpha_{VN*}\circ \Phi_*:A(G,W)\to A(G,\alp_{VN}(W))$.  The adjoint
$\Pi={\Pi_*}^*$ satisfies $$\Pi(A\alp_{VN}(W))=\Phi\circ \alp_{VN}\circ (\Phi^\alp)^{-1}(A\alp_{VN}(W))=\alp_{VN}(A)W.$$
It follows from Proposition \ref{prop:2ndspec} that $\Pi$ maps ${\rm Spec} A(G,\alpha_{VN}(W))$ onto
${\rm Spec} A(G, W)$.
\end{proof}

\begin{rem}
The above theorem tells us that we could focus on a representative choice of subgroup among the subgroups in the similarity class.
\end{rem}

We end this subsection with the following functorial result about restrictions.
\begin{prop}\label{prop-restriction}
Let $H$ be a closed subgroup of $G$, $W_H$ be a weight on the dual of $H$ which is bounded below,
and $W_G=\iota(W_H)$ be the weight on the dual of $G$ specified by Proposition \ref{prop-ext-subgroup}.
Then the restriction map
\[
u\mapsto u|_H:A(G,W_G)\to A(H,W_H)
\]
is a surjective quotient homomorphism.
\end{prop}
\begin{proof}
Recall the embedding $\iota:VN(H)\to VN(G)$ from \eqref{eq-embedding-subgroup}.
It is shown in the proof of Proposition \ref{prop-ext-subgroup} that $\iota\otimes\iota(X_{W_H})=X_{W_G}$.
Since, also, $(\iota\otimes\iota)\circ\Gamma_H=\Gamma_G\circ\iota$,
for $A$ in $VN(H)$ we have that
\[
(\iota\otimes\iota)\circ\Gamma_{W_H}(A)=\iota\otimes\iota(\Gamma_H(A)X_{W_H})=\Gamma_{W_G} \circ \iota (A).
\]
Thus, by duality, the quotient map $R_H$ is an algebra homomorphism from
$(A(G),\cdot_{W_G})$ onto $(A(H),\cdot_{W_H})$.  Then the desired restriction map is given by
\[
(\Phi^{-1}_H)_*\circ R_H\circ (\Phi_G)_*:A(G,W_G)\to A(H,W_H)
\]
where $(\Phi_G^{-1})_*:A(G)\to A(G,W_G)$ is the isometry indicated in Definition \ref{2nddef-BFalg}, for example.
\end{proof}


\subsubsection{Construction of weights using Laplacian}\label{ssec:Weights-Laplacian-general}

Measuring the growth rate of a function has always been a central theme in various fields of mathematics. In \cite{LST} the authors suggested that the Laplacian could provide such a measure on the dual of a compact connected Lie group $G$. For such a group it is well-known that the operator $\partial\lambda(\Delta)$ is central, i.e. we have a function $\Om: \widehat{G} \to (0,\infty)$ such that
	$$-\partial\lambda(\Delta) = \oplus_{\pi\in \widehat{G}} \Om(\pi)I_\pi.$$	
Note that the function $\Om$ is independent of the choice of the basis $\{X_1,\cdots, X_n\}$ of the associated Lie algebra $\g$ and the function $\Om$ plays the role of homogeneous monomial of order 2. From this we can define a family of polynomially/exponentially growing weights on the dual of a compact group $G$ as follows.

	\begin{defn}\label{def-poly-exp-weights-compact-Lie}
For $\alpha\ge 0$, we define the weight $w^\Delta_\alpha$ by
	$$w^\Delta_\alpha(\pi) := (1+\Om(\pi))^{\frac{\alpha}{2}}.$$
The above weight  $w^\Delta_\alpha$ is said to have polynomial growth rate of order $\alpha$.\\
Similarly, for $\beta\ge 1$ we define the weight $w^\Delta_\beta$ by
	$$w^\Delta_\beta(\pi) := \beta^{\sqrt{\Om(\pi)}},\;\; \pi\in\widehat{G}.$$
The above weight  $w^\Delta_\beta$ is said to have exponential growth rate of order $\beta$.
	\end{defn}
We refer to \cite[Lemma 5.3]{LST} to see that $w^\Delta_\alpha$ and $w^\Delta_\beta$ are indeed weights.

The above idea can be extended to a general (possibly non-compact) connected Lie group $G$, with a partial success of defining polynomially growing weights. We begin with some basic facts on the operators $\partial \lambda(X)$, $X\in \mathfrak{g}$ and $\partial \lambda(\Delta)$ with $\Delta = X_1^2 + \cdots  + X_n^2$  for a given basis $\{X_1 ,\cdots , X_n\}$ of $\g$. Recall that $\partial \lambda(X)$ is the infinitesimal generator of the one-parameter unitary group $(\lambda(\exp(tX)))_{t\in \Real}$, i.e. $\lambda(\exp(tX))=\exp(t\partial \lambda(X))$. We also have $\Gamma(\lambda(\exp(tX))) = \lambda(\exp(tX)) \otimes \lambda(\exp(tX))$, $t\in \Real$. By comparing  the associated infinitesimal generators on both sides (or taking derivative at $t=0$) we get
	$$\Gamma(\partial\lambda(X)) = \partial\lambda(X)\otimes I + I \otimes \partial\lambda(X).$$
Now on the common invariant subspace $\D^\infty(\lambda) \otimes \D^\infty(\lambda)$ we compute
	\begin{align}\label{eq-comulti-Laplacian}
	\Gamma(\partial\lambda(\Delta))
	& = \sum^n_{k=1}(\partial\lambda(X_k)\otimes I + I \otimes \partial\lambda(X_k))^2\\
	& = \partial\lambda(\Delta)\otimes I + I \otimes \partial\lambda(\Delta) + 2\sum^n_{k=1}\partial\lambda(X_k)\otimes \partial\lambda(X_k). \nonumber
	\end{align}
We recall that each operator $i\partial\lambda(X_k)$ is self-adjoint  and affiliated with $VN(G)$ so its square $-\partial\lam(X_k)^2$ is positive. 

Our starting point for the weights is the operator 
$A = I - \partial\lambda(\Delta)$ 
and its powers $A^m$, $m\ge 1$. Note that $A$ is positive,
affiliated with $VN(G)$, and bounded
below.  We require the following domination result.
	\begin{prop}\label{prop-domination} (\cite[Lemma 6.3]{Nel})
	 For any element $Y \in U(\mathfrak{g})$ with order $\le 2m$ (i.e. elements in the span of the elements $X_{i_1}\cdots X_{i_k}$, $k\le 2m$) we have a constant $C$ depending only on $m$ such that
	 	$$\|d\lambda(Y) h\| \le C \|(I - d\lambda(\Delta))^mh\|,\; h\in \D^\infty(\lambda).$$
	\end{prop}

\begin{thm}
There is a constant $C_m>0$ for which the operator 
$$C_mA^m = C_m(I - \partial\lambda(\Delta))^m,\; m\ge 1$$ 
is a weight on the dual of $G$.
\end{thm}
\begin{proof}
We first check that $\Gamma(A^m)(A^{-m}\otimes A^{-m})$ is densely defined and bounded. By \eqref{eq-comulti-Laplacian}, on $\D:=\D^\infty(\lambda)\otimes \D^\infty(\lambda)$ we have
	\begin{align*}
	\Gamma(A^m)|_{\D}
	& = \Gamma(I - \partial\lambda(\Delta))^m|_{\D}\\
	& = \left(I\otimes I - \partial\lambda(\Delta)\otimes I - I \otimes \partial\lambda(\Delta) - 2\sum^n_{k=1}\partial\lambda(X_k)\otimes \partial\lambda(X_k)\right)^m|_{\D}\\
	&= \left(I\otimes I - d\lambda(\Delta)\otimes I - I \otimes d\lambda(\Delta) - 2\sum^n_{k=1}d\lambda(X_k)\otimes d\lambda(X_k)\right)^m,
	\end{align*}
which is a linear combination of operators of the form $d\lambda(S) \otimes d\lambda(T)$ where $S,T\in U(\mathfrak{g})$
are of order $\le 2m$.
Now we observe that $A^m(\D^\infty(\lambda))$ is dense in $L^2(G)$. Indeed, we begin with $k\in (A^m(\D^\infty(\lambda)))^\perp$. Then since $\D^\infty(\lambda)$ is a core for $A^m$ (\cite[Corollary 10.2.7]{Sch2}) we can easily see that $k$ actually belongs to ${\ran}(A^m)^\perp = \{0\}$, since $A^m$ is bounded below. This observation combined with Proposition \ref{prop-domination} tells us that $d\lambda(S)A^{-m}$ and $d\lambda(T)A^{-m}$ are bounded operators on the common dense subspace $A^m(\D^\infty(\lambda))$, so that $(d\lambda(S)\otimes d\lambda(T))(A^{-m} \otimes A^{-m})$ is a bounded operator on $A^m(\D^\infty(\lambda)) \otimes A^m(\D^\infty(\lambda))$, which is dense in $L^2(G\times G)$. Consequently, we know that $\Gamma(A^m)(A^{-m}\otimes A^{-m})$ is bounded on the dense subspace $A^m(\D^\infty(\lambda)) \otimes A^m(\D^\infty(\lambda))$ with $\|\Gamma(A^m)(A^{-m}\otimes A^{-m})\|\le C_m$, or equivalently
	$$\|\Gamma(C_mA^m)(C^{-1}_mA^{-m}\otimes C^{-1}_mA^{-m})\|\le 1.$$
Since $A^{m}$ is bounded below, Proposition \ref{prop:swisw} shows that $C_mA^m$ is a weight on the dual of $G$.
\end{proof}

\begin{defn}\label{def-poly-weight-laplacian}
We call the weight $W_m := C_mA^m$, $m\ge 1$, the {\it polynomial weight} on the dual of $G$ of order $2m$.
\end{defn}

\begin{rem}
\begin{enumerate}
\item The fact that $\sqrt{1+t} \le 1 + \sqrt{t} \le \sqrt{2}\sqrt{1+t}$ for $t\ge 0$, and functional calculus show that the operator $A^{1/2} = (I - \partial\lambda(\Delta))^{1/2}$ is comparable up to constant with the operator $I + \sqrt{- \partial\lambda(\Delta)}$. Thus, we may use $W_m'=C_m'(I + \sqrt{- \partial\lambda(\Delta)})^m$ (for some $C_m'>0$) in place of $W_m$.

\item For $G=\Real$, the operator $W'_m$ corresponds to a multiple of the weight $w:\widehat{\Real}\to (0,\infty)$, given by
$w(\xi)=(1+|\xi|)^{2m}$.  This justifies our present terminology.

\item Let $G$ be a compact connected Lie group. For $\alpha = 2m$ the weight $w^\Delta_\alpha$ from Definition \ref{def-poly-exp-weights-compact-Lie} is nothing but the above weight $W_m$ with $C_m=1$.

\item We were not able to define exponentially growing weights for non-compact Lie groups with only two exceptions of Euclidean motion group $E(2)$ and its simply connected cover $\widetilde{E}(2)$. See Section \ref{ssec:exp-weight-E(2)} for the details.
\end{enumerate}
\end{rem}


\section{Compact connected Lie groups and the special unitary group $SU(n)$}\label{chap-cpt}

\subsection{More on representations of compact connected Lie groups}
In this section we recall some representation theory of compact connected Lie groups starting with the highest weight theory from \cite{Wall} and \cite[Section 5]{LST}. We consider the decomposition $\mathfrak{g} = \mathfrak{z} + \mathfrak{g}_1$,
where $\mathfrak{z}$ is the center of $\mathfrak{g}$ and $\mathfrak{g}_1 = [\mathfrak{g}, \mathfrak{g}]$.
Let $\mathfrak{t}$ be a maximal abelian subalgebra of $\mathfrak{g}_1$ and $T = \text{exp}\mathfrak{t}$.
Then there are fundamental weights $\lambda_1, \cdots, \lambda_r, \Lambda_1,\cdots,\Lambda_l \in \mathfrak{g}^*$ with $r = \text{dim}\mathfrak{z}$ and $l=\text{dim}\mathfrak{t}$ such that
any $\pi \in \widehat{G}$ is in one-to-one correspondence with its associated highest weight $\Lambda_\pi = \sum^r_{i=1}a_i\lambda_i + \sum^l_{j=1}b_j \Lambda_j$ with
$(a_i)^r_{i=1}\in \z^r$ and $(b_j)^l_{j=1}\in \z^l_+$. Let $\chi_i$ be the character of $G$ associated to the highest weight $\lambda_i$ and $\pi_j$ be the irreducible representation associated to the weight $\Lambda_j$. Then,
	$$S=\Big\{\pm \chi_i, \pi_j : 1\le i\le r, 1\le j \le l\Big\}$$
is known to generate $\widehat{G}$. More precisely, if we denote for every $k\ge 1$,
	$$S^{\otimes k} =\Big\{\pi \in \widehat{G} : \pi \subset \sigma_1 \otimes \cdots \otimes \sigma_k \;\;
	\text{where}\;\; \sigma_1, \cdots, \sigma_k \in S\cup\{1\} \Big\}$$
then we have
	$$\bigcup_{k\ge 1} S^{\otimes k} = \widehat{G}.$$
	Now we define $\tau_S : \widehat{G} \to \mathbb{N}\cup\{0\}$, {\it the length function on $\widehat{G}$ associated to $S$}, by
	\begin{equation}\label{eq-length-function-def}
	\tau_S(\pi) := k,\;\;\text{if}\;\; \pi\in S^{\otimes k}\backslash S^{\otimes(k-1)}.
	\end{equation}
From the definition, we clearly have
	\begin{equation}\label{eq-length-subadditive}
	\tau_S(\sigma) \le \tau_S(\pi) + \tau_S(\pi')
	\end{equation}
for any $\pi, \pi'\in \widehat{G}$ and $\sigma \subset \pi\otimes \pi'$.

We know that any $\pi\in \widehat{G}$ uniquely extends to a holomorphic representation ${\pi_{\mathbb C}}$ on $G_\Comp$, the universal complexification of $G$. Moreover,
each coefficient function $\pi_{ij}$, $1\le i,j\le d_\pi$, clearly belongs to $\D^\infty_\Comp(\lambda)$ from the definition of entire vectors.

Now we move our attention to a representative example of compact connected Lie groups, namely $SU(n)$, the special unitary group of degree $n$. The associated Lie algebra  $\mathfrak{su}(n)$ is given by
	$$\mathfrak{su}(n) = \Big\{X = -X^* \in M_n: {\rm Tr}(X) = 0\Big\}$$ with the regular commutator as the Lie bracket and the exponential map is the usual matrix exponential. Denoting by $E_{ij}$ the $(i,j)$-matrix unit we have that the abelian subalgebra
	$$\mathfrak{t} = \langle X_{jj}:=i(E_{jj}-E_{nn}) :1\le j\le n-1 \rangle$$
with
	$$\exp(\mathfrak{t}) = H = \Big\{D = {\rm diag}(x_1, \cdots, x_n): |x_1| = \cdots = |x_n| = 1,\; x_1\cdots x_n = 1\Big\}$$
which is the canonical maximal torus consisting of diagonal matrices in $SU(n)$.

The unitary dual $\widehat{SU(n)}$ of $SU(n)$ can be identified with $\z^{n-1}_+$, which is in 1-1 correspondence with the index set
	$$\I_n = \Big\{\lambda=(\lambda_1,\cdots, \lambda_n)\in \z^n_+: \lambda_1 \ge \lambda_2 \ge \cdots \ge \lambda_{n-1} \ge \lambda_n=0\Big\}$$
via the map
	$$\mathsf{a} = (a_1, \cdots, a_{n-1}) \in \z^{n-1}_+ \Leftrightarrow \lambda=(\lambda_1\cdots, \lambda_n) \in \I_n$$
given by
	$$\lambda_1 = a_1 + \cdots + a_{n-1}, \; \lambda_2 = a_2 + \cdots + a_{n-1}, \; \cdots,\; \lambda_{n-1} = a_{n-1},\; \lambda_n = 0.$$
We will denote the corresponding representation in $\widehat{SU(n)}$ by $\pi_\mathsf{a}$ or $\pi_\lambda$ according to our preferences. We have a canonical generating set for $\widehat{SU(n)}$ given by
	$$S = \Big\{(1, 0, \cdots, 0), \cdots, (0, \cdots, 0, 1) \Big\}\subset \z^{n-1}_+$$
which gives us a canonical word length function, denoted by $\tau_S$, given by
	\begin{equation}\label{eq-length-su(n)}
	\tau_S(\pi_\lambda) = a_1 + \cdots + a_{n-1} = \lambda_1,
	\end{equation}

The representation $\pi_\lambda$, $\lambda \in \I_n$, is acting on a vector space $V_\lambda$ with a basis $\{v_T\}$, where $T$ runs through all the semistandard Young tableaux of shape $\lambda$. Every tableaux $T$ has a parameter set $\{t_k: 1\le k \le n\}$, where $t_k$ is the number of times $k$ appears in the tableau $T$. Moreover, each vector $v_T$ is an eigenvector of the matrix $\pi_\lambda(D)$, $D = {\rm diag}(x_1, \cdots, x_n) \in H$, with the eigenvalue $x_1^{t_1}\ldots x_n^{t_n}$  (see \cite[Problem 6.15]{FH}). From the construction of representations $\pi_\lambda$ and $v_T$, for which  the usual symmetric and anti-symmetric tensor products are used, we can see that the vectors $v_T$ are orthogonal with respect to the natural inner product on the representation space $V_\lambda$. Let $\tilde{v}_T$ be the normalization of $v_T$, then we have an orthonormal basis $\{\tilde{v}_T\}$ for $V_\lambda$ with the same eigenvalues. Thus, we can say that $\pi_\lambda(D)$ is also a diagonal matrix with entries $x_1^{t_1}\ldots x_n^{t_n} = x_1^{t_1-t_n}\ldots x_{n-1}^{t_{n-1} - t_n}$ with respect to the orthonormal basis we described above. In other words, we have
	\begin{equation}\label{eq-SU(n)-diag}
	\pi_\lambda(D)\tilde{v}_T = x_1^{t_1}\ldots x_n^{t_n}\cdot \tilde{v}_T.
	\end{equation}
In particular, for any $t\in \Real$ we have a diagonal matrix $\pi_\lambda(\exp(tX_{jj}))$ given by
	$$\pi_\lambda(\exp(tX_{jj}))\tilde{v}_T = e^{it(t_j-t_n)}\tilde{v}_T, \; 1\le j\le n-1.$$
Thus, $\partial \pi_\lambda(X_{jj})$ is also the diagonal matrix given by
	\begin{equation}\label{eq-su(n)-Lie-Der}
	\partial \pi_\lambda(X_{jj})\tilde{v}_T = i(t_j-t_n)\tilde{v}_T,\; 1\le j\le n-1.
	\end{equation}

The universal complexification of $SU(n)$ is the special linear group $SL(n, \Comp)$ and the above representation $\pi_\lambda$ is known to  extend to the holomorphic representation $({\pi}_\lambda)_{\Comp}$ on $SL(n, \Comp)$ with the same formula \eqref{eq-SU(n)-diag} for $D = {\rm diag}(x_1, \cdots, x_n) \in SL(n, \Comp)$. To avoid clutter in notation, we denote the holomorphic representation $({\pi}_\lambda)_{\Comp}$ by $\tilde{\pi}_\lambda$.

\section{Weights on the dual of compact connected Lie group $G$}\label{sec-weights-compact}
In this section, we present three types of weights on the dual of a compact connected Lie group $G$: central weights, weights extended from subgroups, and weights derived from the Laplacian. Note first that a weight $W$ on the dual of $G$ can be understood as a collection of matrices $(W(\pi))_{\pi \in \widehat{G}} \in \prod_{\pi \in \widehat{G}}M_{d_\pi}$ with the dense subspace $\mc S = \Phi_*^{-1}({\rm Trig}(G))$ in Section \ref{ssec:separable-type I}.
We begin with central weights.

\begin{ex}\label{ex-central-weights}
For $\alpha\ge 0$ we have dimension weights $w_\alpha$ given by
	$$w_\alpha(\pi) := (d_\pi)^\alpha,\;\; \pi\in\widehat{G}.$$
The sub-multiplicativity comes from the fact that $\sigma \subseteq \pi \otimes \pi'$, $\sigma, \pi, \pi' \in\widehat{G}$ implies that $d_\sigma \le d_\pi d_{\pi'}$.	

For $\alpha\ge 0$ and $\beta\ge 1$, we have weights $w^S_\alpha$ and $w^S_\beta$ given by
	$$w^S_\alpha(\pi) := (1+\tau_S(\pi))^\alpha,\;\; w^S_\beta(\pi) := \beta^{\tau_S(\pi)},\;\; \pi\in\widehat{G}.$$
Here, $\tau_S$ is the word length function on $\widehat{G}$ with respect to the generating set $S$ given in (\ref{eq-length-function-def}). The submultiplicativity this time comes from the subadditivity \eqref{eq-length-subadditive}.

In particular, when $G = SU(2)$ we have $\widehat{SU(2)} \cong \{\pi_n: n\ge 0 \}\cong \z_+$ with $d_{\pi_n} = n+1$, $n\ge 0$. Then for $S = \{\pi_1\}$, $\alpha \ge 0$ and $\beta\ge 1$ we have
	$$w_\alpha(\pi_n) = w^S_\alpha(\pi_n) = (n+1)^\alpha,\;\; w^S_\beta(\pi_n) = \beta^n,\; n\ge 0.$$
\end{ex}

We already have a general principle of extending weights from subgroups, namely Proposition \ref{prop-ext-subgroup} and Proposition \ref{prop-extended-weights-abelian}. However, precise description of such weights requires more details on the representation theory of the underlying group. For this reason we focus on the case of $SU(n)$ here. We begin with the case of an abelian subgroup.

	\begin{ex}\label{ex-SU(n)fromMaxTorus}
	Let $G = SU(n)$ and $H\cong \tor^{n-1}$ be the canonical maximal torus. Then Proposition \ref{prop-ext-subgroup} and \eqref{eq-su(n)-Lie-Der} immediately lead us to the following: for a weight function $w: \z^{n-1} \to (0,\infty)$ the extended weight $W = \iota(\widetilde{M}_w)$ is a direct sum of matrices
		$$W = \oplus_{\lambda \in \I_n}W(\pi_\lambda),\;\; W(\pi_\lambda) \in M_{d_\lambda}$$
	where $d_\lambda = {\rm dim}\pi_\lambda$. Moreover, each matrix $W(\pi_\lambda)$ is a diagonal one given by
		$$W(\pi_\lambda) \tilde{v}_T= w(t_n-t_1, \cdots, t_{n}-t_{n-1})\tilde{v}_T$$
where $T$ is a  semistandard Young tableaux of shape $\lambda$ with parameters $t_1, \cdots, t_n$. Note that the resulting weight $W$ is not central unless $w$ is the constant $1$ function.
	\end{ex}
	
Among many non-abelian subgroups of $SU(n)$ we check the case of $SU(n-1)$, which would be one of the easiest such choices.
	\begin{ex}\label{ex-SU(n)fromSU(n-1)}
	Let $G = SU(n)$ and $H = SU(n-1)$ embedded in $G$ as the left upper corner. We first consider the embedding $*$-homomorphism
		$$\iota: VN(H) \to VN(G),\; \lambda_H(f) \mapsto \int_H
f(g)\lambda_G(g)dg \cong \oplus_{\lambda \in \I_n} \int_H
f(g)\pi_\lambda(g)dg.$$
	Recall that the restriction $\pi_\lambda|_H$ is clearly a finite dimensional representation of $H$ and its irreducible decomposition is well-known as follows (\cite[Chap. IX]{Kna}):
		$$\pi_\lambda|_{SU(n-1)} \cong \oplus_\mu \pi_\mu$$
	where the direct sum is taken over all $\mu\in \I_{n-1}$ satisfying the interlacing condition
		$$\lambda_1 \ge \mu_1 \ge \lambda_2 \ge \mu_2 \ge  \cdots \ge \mu_{n-2} \ge \lambda_{n-1}.$$
	Now we fix a weight $W_H = \oplus_{\mu\in \I_{n-1}} W_H(\mu)$ on the dual of $H$, which is central or bounded below, so that it satisfies one of the conditions in Proposition \ref{prop-ext-subgroup}. Then the extended weight $W_G = \iota(W_H(\mu)) = \oplus_{\lambda \in \I_n} W_G(\lambda)$ is given by
		$$W_G(\lambda) \cong \oplus_\mu W_H(\mu)$$
	where the direct sum is over all $\mu\in \I_{n-1}$ satisfying the above interlacing condition with respect to $\lambda$.
	\end{ex}

As we have seen in Section \ref{ssec:Weights-Laplacian-general} we already have two types of weights derived from the Laplacian, namely
	$$w^\Delta_\alpha(\pi) := (1+\Om(\pi))^{\frac{\alpha}{2}},\;\; w^\Delta_\beta(\pi) := \beta^{\sqrt{\Om(\pi)}},\;\; \pi\in\widehat{G}$$
for $\alpha\ge 0$ and $\beta \ge 1$, where $\Om: \widehat{G} \to (0,\infty)$ is the function satisfying
	$$-\partial\lambda(\Delta) = \oplus_{\pi\in \widehat{G}} \Om(\pi)I_\pi.$$
\begin{rem}
One canonical way of measuring ``growth'' is to use length functions, so that we can define ``growth rate'' on $\widehat{G}$ for a compact connected Lie group $G$ using the length function $\tau_S$ in \eqref{eq-length-function-def}. Note that the function $\Om$ and the length function $\tau_S$ are equivalent (\cite[(5.3) and (5.5)]{LST}), so that the resulting ``polynomially growing'' weights $w^\Delta_\alpha$ and $w^S_\alpha$ are equivalent. However, this equivalence does not carry over to ``exponentially growing'' weights thanks to the presence of the equivalence constant. Note also that the concept of length functions on $\widehat{G}$ for a non-compact Lie group $G$ is not well-understood at this moment. For this reason we would like to focus on the approach of using Laplacian for ``growth rate on the dual of $G$''.
\end{rem}

\subsection{Description of spectrum of $A(G,W)$ for a compact connected Lie group $G$ and $SU(n)$}\label{sec-SU(n)}

In this section, we describe the spectrum of $A(G,W)$ for a compact connected Lie group $G$, where the weight $W$ is one of the three types described in the previous section.

\subsubsection{The case of central weights}\label{subsec:central-compact}

We begin with the case of central weights. The general strategy goes as follows. We fix a weight function $w: \widehat{G} \to (0,\infty)$ and recall the observation ${\rm Spec} A(G,w) \subseteq {\rm Spec} {\rm Trig}(G)$ and we pick an element $\varphi \in {\rm Spec} {\rm Trig}(G)$.

\[
\xymatrix{
A(G, w) \ar[d]_{\varphi}
& \;{\rm Trig}(G) \ar@{_{(}->}[l] \ar[ld]^{\varphi|_{{\rm Trig}(G)}}
\\
\mathbb{C} &
}
\]

By the ``abstract Lie'' theory we know that $\varphi$ actually comes from a point $x_\varphi \in G_\Comp$. 
To determine whether $\varphi$ belongs to ${\rm Spec} A(G,w)$, we need to check
the norm condition $\varphi \in A(G,w)^* \cong VN(G, w^{-1})$, i.e.
	$$\sup_{\pi \in \widehat{G}}\frac{\|\pi_\Comp(x_\varphi)\|}{w(\pi)}<\infty,$$ which suggests that we need more details of the representation theory of $G$. For this reason we focus only on the case of $G=SU(n)$. We recall a part of Littlewood-Richardson rule \cite[Proposition 15.25 (ii)]{FH}.
\begin{prop}\label{prop-LR-simple}
Let $\mathsf{a} = (a_1, \cdots, a_{n-1}) \in \z^{n-1}_+$ and $\pi_\mathsf{a} = \pi_{(a_1, \cdots, a_{n-1})}$ be the associated representation. We also let $\pi_k = \pi_{(0, \cdots, 0,1,0,\cdots, 0)}$, where we have 1 only at the $k$-th coordinate. Then we have
	$$\pi_{(a_1, \cdots, a_{n-1})}\otimes \pi_k \cong \oplus_\mathsf{b} \pi_\mathsf{b}$$
where the direct sum is over all $\mathsf{b} = (b_1,\cdots,b_{n-1})\in \z^{n-1}_+$ satisfying the following: there is a subset $J \subseteq \{1,\cdots,n\}$ with $k$ elements such that if $i\notin J$ and  $i+1\in J$ then $a_i>0$, and
	$$b_i = \begin{cases} a_i - 1 & \text{if}\;\; i\notin J,\; i+1\in J\\ a_i + 1 & \text{if}\;\; i\in J, \; i+1\notin J\\ a_i & \text{otherwise} \end{cases}.$$
\end{prop}

\begin{prop}\label{prop-character-diagonal}
Let $D = {\rm diag}(x_1, \cdots, x_n) \in SL(n, \Comp)$. Then, we have
	$$\|(\pi_{(a_1, \cdots, a_{n-1})})_\Comp(D)\| = \|(\pi_1)_\Comp(D)\|^{a_1} \cdots \|(\pi_{n-1})_\Comp(D)\|^{a_{n-1}}.$$
In particular, for the case that $|x_1|\ge |x_2| \ge \cdots \ge |x_n|$ we actually have
	$$\|(\pi_{(a_1, \cdots, a_{n-1})})_\Comp(D)\| =  |x_1|^{a_1}\cdots |x_1\cdots x_{n-1}|^{a_{n-1}}.$$
\end{prop}
\begin{proof}
We will apply Proposition \ref{prop-LR-simple} with $J = \{1,\cdots,k\}$. Then we can see that
	$$\pi_{(a_1, \cdots,a_k + 1,\cdots,a_{n-1})}\; \text{(increased by 1 only at the $k$-th coordinate)}$$
appears in the decomposition of $\pi_{(a_1, \cdots, a_{n-1})}\otimes \pi_k$. This explains that
	$$\|(\pi_{(a_1, \cdots,a_k + 1,\cdots,a_{n-1})})_\Comp(D)\| \le \|(\pi_{(a_1, \cdots, a_{n-1})})_\Comp(D)\| \cdot \|(\pi_k)_\Comp(D)\|$$
which gives us the upper bound
	$$\|(\pi_{(a_1, \cdots, a_{n-1})})_\Comp(D)\| \le \|(\pi_1)_\Comp(D)\|^{a_1} \cdots \|(\pi_{n-1})_\Comp(D)\|^{a_{n-1}}.$$

For the lower bound we first assume that $D$ is rearranged in the way that $|x_1|\ge |x_2| \ge \cdots \ge |x_n|$. Note that we do not lose generality from this assumption. When we consider the representation $\pi_k$ we are lead to the shape $\lambda$ with $\lambda_1 = \cdots  = \lambda_k =1$, $\lambda_{k+1} = \cdots  = \lambda_n = 0$. Then, the allowed tableaux $T$ can only have one or zero of each numbers between 1 and $n$, so that from \eqref{eq-SU(n)-diag} we have
	$$\|(\pi_k)_\Comp(D)\| = \max_T |x^{t_1}_1\cdots x^{t_n}_n| = \max_{1\le i_1<\cdots< i_k\le n}|x_{i_1}\cdots x_{i_k}| = |x_1 \cdots x_k|.$$
On the other hand for the general shape $\lambda$ we pick a specific choice of tableaux $T$ with $k$ filling the whole $k$-row for $1\le k\le n-1$. Then, we have $t_k = \lambda_k = a_k+\cdots + a_{n-1}$,  $1\le k\le n-1$ and $t_n=0$. Thus we have
	$$|x^{t_1}_1\cdots x^{t_n}_n| = |x^{a_1+\cdots+a_{n-1}}_1\cdots x^{a_{n-1}}_{n-1}| = |x_1|^{a_1}\cdots |x_1\cdots x_{n-1}|^{a_{n-1}}$$
which gives us the lower bound we wanted.

\end{proof}

A direct application of the above gives us the following results.
\begin{ex} Let $G=SU(n)$ and $w^S_\beta$, $\beta\ge 1$ be the exponentially growing weight function on $\widehat{SU(n)}$ from Example \ref{ex-central-weights}. Then we have
\begin{align}\label{eq-SU(n)-central-exp-spectrum}
\lefteqn{{\rm Spec}A(SU(n), w^S_\beta)}\\
& \cong \Big\{UDV: U,V\in SU(n), D={\rm diag}(x_1, \cdots, x_n)\in SL(n,\Comp)\nonumber \\
& \;\;\;\;\;\;\;\;\;\;\;\;\;\;\;\;\;\;\; \text{with}\; |x_1|\ge \cdots \ge |x_{n-1}|\; \text{and}\; |x_1\cdots x_k|\le \beta\;\text{for all}\;\; 1\le k\le n-1 \Big\}. \nonumber
\end{align}
The above can be shown as follows. We first recall the KAK-decomposition of $SL(n, \Comp)$, namely for any $A\in SL(n, \Comp)$ we have $A = UDV$ for some $U,V\in SU(n), D={\rm diag}(x_1, \cdots, x_n) \in SL(n,\Comp)$. This decomposition can be transferred to the level of ${\rm Trig}(G)^\dagger \cong \prod_{\pi\in \widehat{G}}M_{d_\pi}$ and since our weight is central, the unitary matrices corresponding to the evaluation functional at $U$ and $V$ do not contribute to the operator norm. Thus, it is enough to focus on the diagonal part $D$, for which we get the conclusion  directly from Proposition \ref{prop-character-diagonal} and the fact that
	$$w^S_\beta(\pi_{(a_1, \cdots, a_{n-1})}) = \beta^{\tau_S(\pi_{(a_1, \cdots, a_{n-1})})} = \beta^{a_1 + \cdots + a_{n-1}}.$$
In particular, for the case $n=2$ we actually recover the following result from \cite[Example 4.5]{LST}.
	\begin{equation}\label{eq-spec-SU(2)-central}
	{\rm Spec}A(SU(2), w^S_\beta) \cong \left\{U\begin{bmatrix} \rho & 0 \\ 0 & \rho^{-1}\end{bmatrix}V: U,V \in SU(2),\; \frac{1}{\beta}\le \rho \le \beta\right\}.
	\end{equation}
\end{ex}

\begin{ex}
Let $G= SU(2)$ and $w^\Delta_\beta$ , $\beta\ge 1$ be the exponentially growing weight function on $\widehat{SU(2)}$ from Definition \ref{def-poly-exp-weights-compact-Lie}. Note that we have (\cite[Proposition 8.3.2]{Far})
	$$\Om(\pi_n) = n(n+2),\;\; n\ge 0.$$
Thus, we have $w^\Delta_\beta(\pi_n) = \beta^{\sqrt{n(n+2)}}$, so that the same argument as in the previous example leads us to the following.
	$${\rm Spec}A(SU(2), w^\Delta_\beta) \cong \left\{U\begin{bmatrix} \rho & 0 \\ 0 & \rho^{-1}\end{bmatrix}V: U,V \in SU(2),\; \frac{1}{\beta}\le \rho \le \beta\right\},$$
which produces the same spectrum as in \eqref{eq-spec-SU(2)-central}.
	
It is natural to be interested in the other exponentially growing weight function $w^\Delta_\beta$ on $\widehat{SU(n)}$, $n\ge 3$. However, the computation of $\Om(\pi_{(a_1, \cdots, a_{n-1})})$ is quite complicated, so that we do not pursue this direction in this paper.
\end{ex}

\subsubsection{The case of weights coming from closed subgroups}\label{subsubsec:subgroupweight}

In this subsection we focus on the weight $W_G$ extended from a weight $W_H$ on the dual of a closed Lie subgroup $H$ of $G$ (i.e. $W_G = \iota(W_H)$) as in Proposition \ref{prop-ext-subgroup}. We actually have the following general result.

\begin{thm}\label{thm-spec-extended-compact}
The spectrum of the Beurling-Fourier algebra $A(G,W_G)$ is determined as follows:
	$${\rm Spec} A(G,W_G) \cong \Big\{ g \cdot \exp(iX) \in G_\Comp : g \in G, X \in \mathfrak{h},\; \exp(iX) \in {\rm Spec} A(H,W_H)\Big\}.$$
\end{thm}

Since ${\rm Spec} A(G,W_G) \subseteq  {\rm Spec} {\rm Trig}(G)$, we start the proof of the above theorem  with examining an element $\varphi \in  {\rm Spec} {\rm Trig}(G)$. By the Cartan decomposition \eqref{eq-Cartan-cpt} we know that
	$$\varphi = \lambda(g)\exp(\partial\lambda(iX)) \cong (\pi(g)\exp(id\pi(X)))_{\pi\in \widehat{G}}$$
for uniquely determined $g\in G$ and $X\in \mathfrak{g}$. Now in order to check whether $\varphi \in VN(G, W^{-1}_G)$ (or equivalently whether $\exp(\partial\lambda(iX))W^{-1}_G$ is bounded), we prove that for every $X \in \mathfrak{g}\backslash \mathfrak{h}$, the operator $\exp(\partial\lambda(iX))W^{-1}_G$ is actually unbounded under a mild assumption on $W_H$. This job requires an integral formula associated to a modified exponential map.

\begin{prop}\label{prop-local-formulas}
Suppose that $G$ is a unimodular (possibly non-compact) connected Lie group. Let $\dim \mathfrak{g} = n \ge d = \dim \mathfrak{h}$ and $\{X_1, \cdots, X_n\}$ be a basis for $\mathfrak{g}$ such that $\{X_1, \cdots, X_d\}$ is a basis for $\mathfrak{h}$. Then, there is a neighborhood $W$ of the origin of $\mathfrak{m} = {\rm span}\{X_{d+1}, \cdots, X_n\} \cong \Real^{n-d}$ such that
			$$E: H \times W \to U,\; (h,x_{d+1},\cdots, x_n) \mapsto h\cdot \exp(x_{d+1}X_{d+1})\cdots \exp(x_nX_n)$$
		is a diffeomorphism onto an open set $U\subseteq G$. Moreover, for $f \in C_c(G)$ supported on $\ran E$ we have
			$$\int_G f(x)dx = \int_H \int_W f(E(h,v)) D(h,v)\, dv\, dh$$
		where $D$ is a smooth function with $D(e,0)>0$. Here, $dx$ and $dh$ are left Haar measures on $G$ and $H$, respectively and $dv$ is the Lebesgue measure on $W \subseteq \mathfrak{m}$.
\end{prop}
\begin{proof}
This is a standard procedure. We set $\mathfrak{m} = \text{span} \{X_{d+1},\cdots, X_n\} \subseteq \mathfrak{g}$. Then the tangent space $T_{H\cdot e}(H\backslash G)$ of $H\backslash G$ at $H\cdot e$ can be identified with $\mathfrak{m}$. More precisely, there is a neighborhood $W'$ of $H \cdot e$ in $H\backslash G$ such that the map
	$$\Phi: \mathfrak{m} \to H\backslash G,\; (x_{d+1}, \cdots, x_n) \mapsto H\cdot \exp (x_{d+1} X_{d+1}) \cdots \exp (x_nX_n)$$
is a local diffeomorphism from a neighborhood $W$ of the origin in $\mathfrak{m}$ onto $W'$.
Moreover, the map $\tau: W' \subseteq H\backslash G \to G$, defined as
$$\tau(H\cdot \exp (x_{d+1} X_{d+1}) \cdots \exp (x_nX_n))=\exp (x_{d+1} X_{d+1}) \cdots \exp (x_nX_n)$$
is a smooth section map such that
$q\circ \tau = \id_{W'}$, where $q:G \to H\backslash G$ is the canonical quotient map.
Then, the map
	$$H \times W \stackrel{id_H \times \Phi}{\longrightarrow} H \times W' \to G, \; (h,x_{d+1},\cdots, x_n) \mapsto h\cdot \exp(x_{d+1}X_{d+1})\cdots \exp(x_nX_n)$$
is the map $E$ we were looking for, which is a diffeomorphism onto the image of the map.

The integral formula follows directly from the above, and the local form of the pull-back volume form $E^*\om$ on $H \times W$ for the volume form $\om$ on $G$ associated with the left Haar measure $\mu_G$. Note that $\mu_G$ is also right invariant due to the unimodularity of $G$ and we need right invariance since we use right cosets.
\end{proof}

\begin{thm}\label{thm-unbdd-Lie-compact}
Let $G$ be a compact connected Lie group. Let $W_G$ be a weight on the dual of $G$ extended from a bounded below weight $W_H$. Then for any $X\in \mathfrak{g} \backslash \mathfrak{h}$ the operator $\exp(i\partial \lambda(X))W^{-1}_G$ is unbounded, whenever it is densely defined.
\end{thm}
\begin{proof}
We split the proof into three steps:\\

[Step 1] We may assume that $X = X_{d+1}$ in Proposition \ref{prop-local-formulas}. We can also assume that the neighborhood $W$ in Proposition \ref{prop-local-formulas} is of the form $J \times V$ for some precompact neighborhoods $J$ and $V$ of the origins of $\Real \cdot X_{d+1}\cong \Real$ and ${\rm span}\{X_{d+2}, \cdots, X_n\} \cong \Real^{n-d-1}$, respectively, so that we have the modified exponential map
	\begin{equation}\label{eq-modified-exp}
	E: H\times J \times V \to {\ran}E\subseteq G,\;(h,t,v)\mapsto E(h,t,v).
	\end{equation}
We choose a small enough $\eps>0$ such that
	$$\exp(sX)E(h,t,v) \in {\ran}E\;\; \text{for any}\;\;s\in (-\eps, \eps)\;\text{and}\; (h,t,v)\in H \times \frac{1}{2}J \times \frac{1}{2}V.$$
Note that we used the compactness of $H \times \overline{\frac{1}{2}J \times \frac{1}{2}V}$ in the open set ${\ran}E$. Then by composing the above maps with $E^{-1}$ we can find smooth maps $\theta$, $\alpha$ and $\beta$ from $(-\eps, \eps)\times H \times \frac{1}{2}J \times \frac{1}{2}V$ to $H$, $J$ and $V$, respectively such that
	\begin{equation}\label{eq-3-maps}
	\exp(-sX)E(h,t,v) = E(\theta(s,h,t,v),\alpha(s,h,t,v),\beta(s,h,t,v))
	\end{equation}
for any $(s,h,t,v)\in (-\eps, \eps)\times H \times \frac{1}{2}J \times \frac{1}{2}V$. This relationship immediately tells us that
	\begin{equation}\label{eq-3-maps2}
	\theta(0,h,t,v) = h,\; \alpha(0,h,t,v) = t,\; \beta(0,h,t,v) = v
	\end{equation}
and
	$$\alpha(s,e,0,0) = -s $$
for any $s\in (-\eps, \eps)$. The latter implies $\partial_s\alpha(0,e,0,0) = -1$, so we can find an open neighborhood $U$ of the origin of $\mathfrak{h}$, open subsets $J'\subseteq \tfrac{1}{2} J$, $V'\subseteq \tfrac{1}{2} V$ containing zero such that
	\begin{equation}\label{eq-alpha-map}
	|\partial_s\alpha(0,h,t,v)|\ge \tfrac{1}{2},\;\; (h,t,v) \in \exp(U) \times J' \times V'
	\end{equation}

\vspace{0.5cm}

[Step 2]
In this step, we show that $\partial\lambda(X)W_G^{-1}$ is an unbounded operator on $L^2(G)$.
We pick non-zero real-valued functions $\varphi \in C^\infty_c(\Real)$ and $\gamma \in C^\infty_c(\Real^{n-d-1})$ such that $\overline{\text{supp}\, \varphi} \subseteq J'$ and $\overline{\text{supp}\,\gamma} \subseteq V'$. Now we consider the function $\psi$ on $G$ given by
	$$\psi(E(h,t,v)) := \varphi(t)\gamma(v)$$ for $(h,t,v)\in H \times J' \times V'$ and zero elsewhere. By combining \eqref{eq-3-maps} and \eqref{eq-3-maps2} we have for any $(h,t,v)\in H \times J' \times V'$ that
	\begin{align*}
	\partial\lambda(X)\psi(E(h,t,v))
	& = \frac{d}{ds}\psi(\exp(-sX)E(h,t,v))|_{s=0}\\
	& = \frac{d}{ds}\psi(E(\theta(s,h,t,v),\alpha(s,h,t,v),\beta(s,h,t,v)))|_{s=0}\\
	& = \frac{d}{ds}\varphi(\alpha(s,h,t,v))\gamma(\beta(s,h,t,v))|_{s=0}\\
	& = \varphi'(t)\partial_s\alpha(0,h,t,v)\gamma(v) + \varphi(t)\partial_s(\gamma\circ \beta)(0,h,t,v)\\
	& = A(E(h,t,v)) + B(E(h,t,v)).
	\end{align*}
To estimate the norm  $\|\partial\lambda(X)\psi\|_{L^2(G)}$, we consider each of the above terms separately. Note that both functions $A$ and $B$ are supported in the compact set $H\times J'\times V'$.
By Proposition \ref{prop-local-formulas} and \eqref{eq-alpha-map} we have,
	\begin{align*}
	\|A\|^2_{L^2(G)}
	& = \int_H \int_{J'} \int_{V'} | \varphi'(t)\partial_s\alpha(0,h,t,v)\gamma(v) |^2 D(h,t,v)\, dhdtdv\\
	& \ge \tfrac{1}{2} C \cdot \mu_H(U) \cdot \|\varphi'\|^2_2 \cdot \|\gamma\|^2_2,
	\end{align*}
where $C = \inf\{|D(h,t,v)| : (h,t,v)\in H \times J' \times V'\}$. For the function $B$ we have
	\begin{align*}
	\|B\|^2_{L^2(G)}
	& = \int_H \int_{J'} \int_{V'} | \varphi(t)\partial_s(\gamma\circ \beta)(0,h,t,v) |^2 \cdot D(h,t,v)\, dhdtdv\\
	& \le |V'|\, \|D\|_{C(H \times \overline{J'} \times \overline{V'})}\cdot \|\varphi\|^2_2 \cdot \|\partial_s(\gamma\circ \beta)\|^2_{C([-\eps/2, \eps/2] \times  H \times \overline{J'} \times \overline{V'})},
	\end{align*}
	where $|V'|$ denotes the volume of $V'$.

Now we replace $\varphi(t)$ by its dilated version $\sqrt{N}\varphi(Nt)$ for $N\ge 1$ to define
	$$\psi_N(E(h,t,v)) := \sqrt{N}\varphi(Nt)\gamma(v)$$ for $(h,t,v)\in H \times \frac{1}{N}J' \times V'$ and zero elsewhere. Then combining all the above estimates we get
	\begin{align*}
	\|\partial\lambda(X)\psi_N\|_{L^2(G)} \ge \|A_N\|_{L^2(G)} - \|B_N\|_{L^2(G)}
	\gtrsim N
	\end{align*}
by a simple change of variables, where $\partial\lambda(X)\psi_N = A_N + B_N$ as in the above computation. On the other hand we can similarly check
	$$\|\psi_N\|_{L^2(G)} \lesssim \|\varphi\|_2 \cdot \|\gamma\|_2$$
which is independent of $N$.

Finally, we observe that $\psi_N$ satisfies $W^{-1}_G\psi_N = c \psi_N$ for some constant $c\neq 0$ independent of $N$. Indeed, the operator $W^{-1}_G = \iota(W^{-1}_H) \in \iota(VN(H))\subseteq VN(G)$ is a limit of operators of the form $\iota(\lambda_H(f))$, $f\in L^1(H)$ in the strong operator topology. From the definition of $\psi_N$, it is easy to check that $\lambda_G(h)\psi_N = \psi_N$, for every $h\in H$. Thus, we have
	$$\iota(\lambda_H(f))\psi_N=\int_H f(h)(\lambda_G(h)\psi_N) dh=(\int_H f(h) dh)\psi_N=\langle\lambda_H(f)1_H, 1_H\rangle\psi_N,$$
where $1_H$ is the constant function 1 on $H$. Taking limits of both sides in the strong operator topology, we conclude that $W_G^{-1}\psi_N=\langle W_H^{-1}1_H, 1_H\rangle \psi_N$. Now, we let $c=\langle W_H^{-1}1_H, 1_H\rangle$, which is nonzero as $W_H^{-1}$ is positive and injective.
	
\vspace{0.5cm}
	
[Step 3] We shall now lift the above lower bound for $\|\partial\lambda(X)\psi_N\|_{L^2(G)}$ to the case of $\|\exp(i\partial\lambda(X))\psi_N\|_{L^2(G)}$. Let $F(\cdot)$ denote the spectral measure of the self-adjoint operator $i\partial\lambda(X)$. Suppose on the contrary that $\exp(i\partial\lambda(X))W^{-1}_G$ is bounded and densely defined. So ${\rm dom}(\exp(i\partial\lambda(X))W^{-1}_G)=L^2(G) $, which implies that $W^{-1}_G\psi_N$ is in the domain of $\exp(i\partial\lambda(X))$. Now let  $C>0$ be a constant such that
	$$\int_\Real x^2 d\la F(x)\psi_N, \psi_N  \ra = \|i\partial\lambda(X)\psi_N\|^2_2 \ge  CN^2$$
so that we have either $\displaystyle \int_{[0,\infty)} x^2 d\la F(x)\psi_N, \psi_N  \ra$ or $\displaystyle \int_{(-\infty, 0]} x^2 d\la F(x)\psi_N, \psi_N  \ra$ is at least  $CN^2/2$. In the first case we have
	$$\|\exp(i\partial\lambda(X))\psi_N\|^2_2 = \int_\Real e^{2x} d\la F(x)\psi_N, \psi_N  \ra \ge \int_{[0,\infty)} x^2 d\la F(x)\psi_N, \psi_N  \ra \ge CN^2/2.$$
In the latter case we have
	\begin{align*}
\|\exp(i\partial\lambda(X)){\psi_N}\|^2_2
         & = \|\exp(i\partial\lambda(X))\overline{\psi_N}\|^2_2\\
	& = \int_\Real e^{2x} d\la F(x)\overline{\psi_N}, \overline{\psi_N}  \ra\\
	& = \int_\Real e^{2x} d{\la JF(x)J\psi_N, \psi_N \ra}\\
	& = \int_\Real e^{2x} d{\la F(-x)\psi_N, \psi_N \ra}\\
	& \ge \int_{(-\infty, 0]} x^2 d\la F(x)\psi_N, \psi_N  \ra \ge CN^2/2,
	\end{align*}
where $J$ denotes the usual complex conjugation on $L^2(G)$. Here, we used the fact that $F(B) = JF(-B)J$ for any Borel set $B\subseteq \Real$. Indeed, note that $\bar{\lambda} = \lambda$, where $\bar{\lambda}$ is the complex conjugate representation of $\lambda$ given by $\bar{\lambda}(\cdot) = J\lambda(\cdot)J$. This implies that $\partial \bar{\lambda}(X)=J\partial{\lambda}(X)J$ and $i\partial \bar{\lambda}(X)=-J(i\partial{\lambda}(X))J$.  We can now conclude that $\exp(i\partial\lambda(X))W^{-1}_G$ is unbounded in either of the cases.
\end{proof}

\begin{rem}\label{rem-unbdd-Lie-compact}
The proof of Theorem \ref{thm-unbdd-Lie-compact} actually works without any modification for a non-compact connected Lie group $G$ as long as the subgroup $H$ is compact.
\end{rem}

Now we finish the proof of Theorem \ref{thm-spec-extended-compact}.

\begin{proof}[Proof of Theorem \ref{thm-spec-extended-compact}]
Let us go back to the starting point, namely the Cartan decomposition for $\varphi \in {\rm Spec} A(G,W_G) \subseteq  {\rm Spec} {\rm Trig}(G)$
	$$\varphi = \lambda(g)\exp(\partial\lambda(iX)) \cong (\pi(g)\exp(id\pi(X)))_{\pi\in \widehat{G}},$$
where $g\in G$ and $X\in \mathfrak{g}$ are uniquely determined. We need to check whether the operator $\exp(\partial\lambda(iX))W^{-1}_G$ is bounded, and Theorem \ref{thm-unbdd-Lie-compact}  allows us to eliminate ``wrong directions'' $X \in \mathfrak{g}\backslash  \mathfrak{h}$.
Indeed, for the truncated $\widetilde{W}_H = W_H \lor 1$ obtained by functional calculus we get $\widetilde{W}_G = \iota(\widetilde{W}_H) = W_G \lor 1$ by functional calculus again, which makes the operator $W_G \widetilde{W}^{-1}_G$ contractive. Here, we use the convention $a\lor b = \max\{a,b\}$, $a,b\in \Real$. We first need to check that $\widetilde{W}_G$ is still a weight on the dual of $G$. Thanks to the restrictions on $W$ in Proposition \ref{prop-ext-subgroup} we only need to focus on the cases where $H$ is abelian or $W_H$ is central. For both of the cases the weight $W_H$ is coming from a weight function $w$ on $\widehat{H}$, so that $\widetilde{W}_H$ is associated with the truncated function $w \lor 1$, which can easily be checked to be a weight function. Note that we are referring to the condition \eqref{eq-weight-function-compact} when $H$ is non-abelian. This means that the operator $\widetilde{W}_H$ is a strong weight on the dual of $H$, so that $\widetilde{W}_G = \iota(\widetilde{W}_H)$ is a weight on the dual of $G$.

Now, for any $A\in VN(G)$ we have $AW_G = (AW_G \widetilde{W}^{-1}_G) \widetilde{W}_G$, which means that $VN(G,W^{-1}_G) \subseteq VN(G,\widetilde{W}^{-1}_G)$. We can even see that ${\rm Spec} A(G,W_G)\subseteq {\rm Spec} A(G,\widetilde{W}_G)$ by the density of ${\rm Trig}(G)$ in Beurling-Fourier algebras. We, then, can appeal to Theorem \ref{thm-unbdd-Lie-compact}  since $\widetilde{W}_G$ is bounded below.

Now we focus on the case $X\in \mathfrak{h}$. We first observe a further detail on the embedding
	$$\iota : VN(H) \to VN(G), \lambda_H(x) \mapsto \lambda_G(x)$$
whose pre-adjoint map is the restriction map $R : A(G) \to A(H),\; f\mapsto f|_H$. Since $G$ is compact, we actually know that $R|_{{\rm Trig}(G)} : {\rm Trig}(G) \to {\rm Trig}(H)$ is a surjective homomorphism. This implies that $\iota$ extends to an injective $*$-homomorphism
	$$(R|_{{\rm Trig}(G)})^\dagger : {\rm Trig}(H)^\dagger \to {\rm Trig}(G)^\dagger$$
which we still denote by $\iota$ by abuse of notation. Moreover, the quasi-equivalence \eqref{eq-cpt-regular-decomp} gives us
	$$\iota((T(\sigma))_{\sigma \in  \widehat{H}}) \cong \left(\oplus_{\sigma \subseteq \pi|_H, \sigma \in \widehat{H}} T(\sigma)\right)_{\pi\in \widehat{G}}.$$
Indeed, for $f\in {\rm Trig}(G) \subseteq L^1(H)$ we have
	\begin{align*}
	\iota((\widehat{f}^H(\sigma)_{\sigma \in \widehat{H}})
	& \cong \iota\left(\int_H f(x)\lambda_H(x)d\mu_H(x)\right)\\
	& \cong \int_H f(x)\lambda_G(x)d\mu_H(x)\\
	& \cong \left(\int_Hf(x)\pi|_H(x)d\mu_H(x)\right)_{\pi\in \widehat{G}}\\
	& \cong \left(\oplus_{\sigma \subseteq \pi|_H, \sigma \in \widehat{H}} \int_Hf(x)\sigma(x)d\mu_H(x)\right)_{\pi\in \widehat{G}}\\
	& = \left(\oplus_{\sigma \subseteq \pi|_H, \sigma \in \widehat{H}} \widehat{f}^H(\sigma)\right)_{\pi\in \widehat{G}},
	\end{align*}
which proves the claim. This particular quasi-equivalence tells us that for $T\in {\rm Trig}(H)^\dagger$, the operator $T$ is bounded if and only if $\iota(T)$ is bounded.

Recall that $\partial\lambda_G(X) = \iota(\partial\lambda_H(X))$, so that $\exp(\partial\lambda_G(iX)) = \iota(\exp(\partial\lambda_H(iX)))$ by functional calculus and consequently
	$$\exp(\partial\lambda_G(iX))W^{-1}_G = \iota(\exp(\partial\lambda_H(iX))W^{-1}_H).$$
For the above equality we actually need to recall that we are dealing with 3 scenarios of (1) $H$ abelian; (2) $W_H$ central; (3) $W_H$ bounded below. For the first 2 cases we apply joint functional calculus and for the third we use Proposition \ref{prop-extension-variant}.
Now $\exp(\partial\lambda_H(iX))\in {\rm Spec}A(H,W_H)$ implies that $\exp(\partial\lambda_H(iX))W^{-1}_H$ is bounded and consequently we get that $\exp(\partial\lambda_G(iX))W^{-1}_G$ is bounded from the above observation, which completes the proof.
\end{proof}

We close this subsection with examples of the case $G = SU(n)$.
\begin{ex}
Let $G=SU(n)$ and $H = \exp \mathfrak{t} \cong \tor^{n-1}$ be  the canonical maximal torus. Let
	$$w_{\beta_1, \cdots, \beta_{n-1}}(j_1, \cdots, j_{n-1}) = \beta_1^{|j_1|} \cdots \beta_{n-1}^{|j_{n-1}|},\; (j_1,\cdots, j_{n-1})\in \z^{n-1}$$
for $\beta_1, \cdots, \beta_{n-1}\ge 1$ as in Example \ref{ex-weights-higher-dim-abelian} and $W^\tor_{\beta_1, \cdots, \beta_{n-1}} = \iota(\widetilde{M}_{w_{\beta_1, \cdots, \beta_{n-1}}})$, the extended weight. Now we observe that $\exp_H (\sum^{n-1}_{j=1}x_j X_{jj}) = (e^{ix_j})^{n-1}_{j=1} \in \tor^{n-1}$. Then directly from Theorem \ref{thm-spec-extended-compact} we have
	\begin{align*}
	\lefteqn{{\rm Spec} A(SU(n), W^\tor_{\beta_1, \cdots, \beta_{n-1}})}\\
	& = \Big\{UD: U\in SU(n), D = {\rm diag}(x_1, \cdots, x_n), \\
	&\;\;\;\;\;\;\;\;\;\;\;\;\;\;\;\; x_1\cdots x_n =1,\; \frac{1}{\beta_j} \le |x_j| \le \beta_j,\; 1\le j\le n-1 \Big\}.
	\end{align*}
\end{ex}

\begin{ex}
Let $G=SU(n)$ and $H = SU(n-1)$ embedded as the left upper corner. Let $W^{SU(n-1)}_\beta = \iota(W_H)$ with $W_H = w^S_\beta$, $\beta\ge 1$ be the exponentially growing weight function on $\widehat{SU(n-1)}$ from Example \ref{ex-central-weights}. Then by Theorem \ref{thm-spec-extended-compact} and \eqref{eq-SU(n)-central-exp-spectrum} we have
	\begin{align*}
	\lefteqn{{\rm Spec} A(SU(n), W^{SU(n-1)}_\beta)}\\
	& = \Big\{UDV \in SL(n,\Comp): U\in SU(n), V\in SU(n-1),\\
	&\;\;\;\;\;\;\;\;\;\;\;\;\;\;\;\; \;\;\;\;\;\;\;\;\;\;\;\;\;\;\;\; \;\;\;\;\; D = {\rm diag}(x_1, \cdots, x_{n-1},1) \in SL(n,\Comp)\; \text{with}\\
	&\;\;\;\;\;\;\;\;\;\;\;\;\;\;\;\; \;\;\;\;\;\;\;\;\;\;\;\;\;\;\;\; \;\;\;\;\; |x_1|\ge \cdots \ge |x_{n-2}|\; \text{and}\; |x_1 \cdots x_k| \le \beta, \; 1\le k\le n-2 \Big\}.
	\end{align*}	
\end{ex}


\section{The Heisenberg group}\label{chap-Heisenberg}

The Heisenberg group is
	$$\mathbb{H} = \left\{(y,z,x) = \begin{bmatrix}1&x&z\\&1&y\\&&1\end{bmatrix}: x,y,z \in \Real\right\} = (\Real\times \Real)\rtimes \Real$$
with the group law being the matrix multiplication or equivalently
	$$(y,z,x)\cdot (y', z', x') = (y+y', z+z'+xy', x+x').$$	
Here, we use the notation $(y,z,x)$ instead of $(x,y,z)$ in order to keep the semi-direct product structure of $\mathbb{H}$. The Haar measure of ${\mathbb H}$ is given by $d(y,z,x)=dx\, dy\, dz$, where $dx$, $dy$, and $dz$ denote the Lebesgue measure on ${\mathbb R}$.
	
For any $a \in \Real^*$, we have an irreducible unitary representation given by
	$$\pi^a(y,z,x) \xi(t) = e^{-i a(ty-z)} \xi(-x+t),\; \xi \in L^2(\Real).$$
Now the left regular representation $\lambda$ is given by the decomposition
	$$\lambda \cong \int^\oplus_{\Real^*} \pi^a |a|da.$$
This quasi-equivalence tells us that
	$$VN(\mathbb{H}) \cong L^\infty(\Real^*, |a|da; \B(L^2(\Real)))$$
and
	$$A(\mathbb{H}) \cong L^1(\Real^*, |a|da; S^1(L^2(\Real)))$$
where $\Real^* = \Real\backslash \{0\}$ and $S^p(H)$, $1\le p <\infty$, is the Schatten $p$-class acting on a Hilbert space $H$.
	
The above isomorphism is given through the inverse Fourier transform, which we describe now. For $f\in L^1(\mathbb{H})$, we define the group Fourier transform on $\mathbb{H}$ by
	$$\F^{\mathbb{H}}(f) = (\F^{\mathbb{H}}(f)(a))_{a\in \Real^*} = (\widehat{f}^{\mathbb{H}}(a))_{a\in \Real^*} \in L^\infty(\Real^*,  |a|da; \B(L^2(\Real)))$$
and
	$$\widehat{f}^{\mathbb{H}}(a) = \int_{\mathbb{H}} f(g)\pi^a(g) dg.$$
Note that $\F^{\mathbb{H}}(f)$ takes values in $L^2(\Real^*, |a|da;  S_2(L^2(\Real)))$ if $f\in  L^1(\mathbb{H})\cap  L^2(\mathbb{H})$, and we have
$\|f\|_{L^2(\mathbb{H})}=\|\F^{\mathbb{H}}(f)\|_{2}$. So $\F^{\mathbb{H}}|_{L^1\cap L^2}$ extends to a unitary isomorphism from $L^2(\mathbb{H})$ onto $L^2(\Real ^*, |a|da; S^2(L^2(\Real)))$, which we again denote by $\F^{\mathbb{H}}(f)$. A partial inverse for $\F^{\mathbb{H}}$, which is denote by $(\F^{\mathbb{H}})^{-1}$, is given by the formula
        $$(\F^{\mathbb{H}})^{-1}(F)(x) = \int_{\Real^*} {\text{Tr}}(F(a){\pi^a(x)}^*) |a|da,\;\; F \in L^1(\Real^*, |a|da; S^1(L^2(\Real))).$$
The inverse Fourier transform $(\F^{\mathbb{H}})^{-1}$ is an isometric Banach space isomorphism from $L^1(\Real^*, |a|da; S^1(L^2(\Real)))$ onto $A(\mathbb{H})$. (See \cite{KT-book} for representation theory of the Heisenberg group, and \cite{Fuhr} for its Fourier analysis.)

The Lie algebra of ${\mathbb H}$ is $\mathfrak{heis} = \la X,Y,Z : [X,Y]=Z, [Y,Z]=0=[Z,X] \ra \cong \Real^3$, which is called the Heisenberg Lie algebra, with the exponential map
	$$\text{exp}: \mathfrak{heis} \to \mathbb{H}, \; xX + yY + zZ \mapsto (y,z+\frac{1}{2}xy,x).$$
The complexifications of $\mathfrak{heis}$ and $\mathbb{H}$ are particularly easy to describe, namely $\mathfrak{heis}_\Comp\cong \Comp^3$ with the same basis and
	$$\mathbb{H}_\Comp = \left\{(y,z,x) = \begin{bmatrix}1&x&z\\&1&y\\&&1\end{bmatrix}: x,y,z \in \Comp\right\}.$$
Since $\mathbb{H}_\Comp \cong \Comp^3$ is simply connected, it is straightforward to see that it actually is the universal complexification with the inclusion $\mathbb{H} \hookrightarrow \mathbb{H}_\Comp$. Moreover, we clearly have the following Cartan decomposition
	\begin{equation}\label{eq-Cartan-heis}
	\mathbb{H}_\Comp \cong \mathbb{H} \cdot \exp(i \, \mathfrak{heis}).
	\end{equation}

We describe the subgroup structure on $\mathbb{H}$ for the convenience of the readers.
	\begin{prop}\label{prop-subgp-structure-Heisenberg}
	The proper closed connected Lie subgroups of $\Hee$ are $H_Y=\{(t, 0, 0):t\in\Ree\}$, $H_Z=\{(0, t, 0):t\in\Ree\}$ and $H_{Y,Z}=\{(s, t, 0):s,t\in\Ree\}$ up to automorphisms of ${\mathbb H}$.
	\end{prop}
\begin{proof}
From the simple connectivity of $\Hee$ we can only focus on the structure of $\fh$. It is straightforward to check that one-dimensional subspaces of $\fh$ are $\Real Y$ or $\Real Z$ up to Lie algebra automorphisms, which gives us the subgroups $H_Y$ and $H_Z$. The two-dimensional Lie subalgebras of $\fh$ are each of the form $\spn\{A X+B Y, C Z\}$ where $(A^2+B^2)C\not=0$, since any pair of linearly independent elements of $\spn\{X,Y\}$ necessarily generate all of $\fh$. Then, they are all coming from the Lie subalgebra $\la Y, Z\ra$ via a Lie algebra automorphism, which gives us the subgroup $H_{Y,Z}=\{(y,z,0):y,z\in\Ree\}$.	
\end{proof}	

For $a\in \Real^*$ we need to understand $\partial \pi^a(X)$, $\partial \pi^a(Y)$ and $\partial \pi^a(Z)$ in a concrete way. Indeed, we can easily check  for $\xi\in C^\infty_c(\Real)$ that
	\begin{align}\label{eq-Lie-derivatives-Hei}
		\begin{cases}\partial \pi^a(X) \xi = -\xi',\\ (\partial \pi^a(Y)\xi)(t) = -ia t \xi(t),\\ \partial \pi^a(Z)\xi = ia \xi.
		\end{cases}
	\end{align}	
A characterization of entire vectors for $\pi^a$ is given in \cite[p.388]{Good71} as follows. A function $f\in L^2(\Real)$ is an entire vector for $\pi^a$ if and only if $f$ extends to an entire function on $\Comp$ and satisfies
	$$\sup_{|\text{\rm Im} z| <t}e^{t|z|}|f(z)| < \infty$$
for any $t>0$. Note that the above condition is independent of the parameter $a$. For example, the $n$-th Hermite function $\varphi_n$, $n\ge 0$, is an entire vector for $\pi^a$ for any $a\in \Real^*$.

\begin{rem}\label{rem-Heisenberg-weights}
We can now present all the weights on the dual of ${\mathbb H}$ which are extended from its closed subgroups.
By Proposition \ref{prop-subgp-structure-Heisenberg} and Theorem \ref{thm-automorphism-principle} we only need to consider the subgroup $H = H_{Y,Z} \cong \Real^2$. For a weight function $w: \widehat{H_{Y,Z}} \cong \Real^2 \to (0,\infty)$,  we consider the extended weight $W = \iota(\widetilde{M}_w)$ affiliated with $VN({\mathbb H})$, or its equivalent form  $(W(a))_{a\in \Real^*}$ affiliated with $L^\infty(\Real ^*, |a|da; \B(L^2(\Real)))$. Recall that
$\widetilde{M}_w$ is the weight on the dual of $H$ defined as in Example \ref{ex:abelian-weight}.
By \eqref{eq-Lie-derivatives-Hei} and Proposition \ref{prop-extended-weights-abelian} we have
	\begin{equation}\label{eq-Weights-Heisenberg}
	W(a)\xi(t) = w(at,-a)\xi(t).
	\end{equation}
for appropriate $\xi \in L^2(\Real)$, which is a multiplication operator with the parameter $a\in \Real^*$.
\end{rem}

\begin{rem}\label{rem-Heisenberg-central-weights}
In \cite{LS} central weights on the dual of $\mathbb{H}$ have been considered. The centrality forces us to begin with $W = (w(a)I_{\B(L^2(\Real))})_{a\in \Real^*}$ for some function $w: \Real^* \to (0,\infty)$ and it has been shown in \cite[Theorem 2.17]{LS} that $w$ should satisfy the usual sub-mutiplicativity on $\Real$ with additional assumptions that $w$ extends to a continuous function on $\Real$ and is bounded below.
The above \eqref{eq-Weights-Heisenberg} shows that this case is included in the case of extended weights from the 1-dimensional subgroup $H_Z = \{(0,t,0): t\in \Real\}$ of $\mathbb{H}$. In \cite[Section 2.3]{LS} the sequence of projections $\{E_m = 1_{[-m,m]} \otimes 1_{B(L^2(\Real))}: m\ge 1\}$ was considered to define $\Gamma(W)$. However, it is immediate to see that both of the definitions of $\Gamma(W)$, the one from \cite[Section 2.3]{LS} and the one from Section \ref{ssec:homomorphisms}, coincide.
\end{rem}
		
\subsection{Description of ${\rm Spec} A(\mathbb{H},W)$} \label{sec-Heisenberg}
In this section we only consider the weight $W$ extended from the subgroup $H=H_{Y,Z}$. More precisely, we fix a weight function $w: \widehat{H_{Y,Z}} \cong \Real^2 \to (0,\infty)$ and we set $W_H = \widetilde{M}_w$ and $W = \iota(W_H)$ (with the equivalent form $(W(a))_{a\in \Real^*}$) as in Remark \ref{rem-Heisenberg-weights}.
By Remark \ref{rem:weightfunc}, we can assume, without loss of generality,  that our weight function $w$ is locally integrable.
Moreover,  we assume that  all of its weak derivatives are at most exponentially growing, i.e. for any multi-index $\alpha$, there are constants $C, D>0$ such that
	\begin{equation}\label{eq-derivatives-exp-growth}
		|\partial^\alpha w(x,y)| \le C e^{D(|x|+|y|)}
	\end{equation}
for a.e. $(x,y)\in \Real^2$. 
Note that important examples of weights on $\Real^2$ such as polynomial weights, exponential weights, and sub-exponential weights, satisfy Condition \eqref{eq-derivatives-exp-growth}.
This condition guarantees the existence of a suitable candidate for the subspace $\mc S$ in \ref{ssec:separable-type I}; 
namely the subalgebra $\B$ from Definition~\ref{def-spaces-Heisenberg}, which is defined as the Fourier image of functions whose partial derivatives have a super-exponential decay.
We will fix the symbols $w$ and $W$ throughout this section.

The Heisenberg group $\mathbb{H}$ actually has a ``background'' Euclidean structure $\widehat{\Real}^3$, whose Haar measure, namely the Lebesgue measure, is identical with the Haar measures of $\mathbb{H}$.
This motivates us to begin with the space of test functions $C^\infty_c(\Real^3)$ and its $\Real^3$-Fourier transform image as a function algebra $\A$ on $\mathbb{H}$. We will show that the subalgebra $\A$ sits inside of $A(\mathbb{H},W)$ densely regardless of the choice of $W$, which is a highly non-trivial fact. Thus, for any $\varphi \in \text{Spec}A(\mathbb{H},W)$ we know that the restriction $\varphi|_\A$ is multiplicative with respect to the pointwise multiplication. By composing $(2\pi)^{\frac{3}{2}}\F^{\Real^3}$ we end up with the map $\psi = \varphi \circ ((2\pi)^{\frac{3}{2}} \F^{\Real^3}) : C^\infty_c(\Real^3) \to \Comp$ which is multiplicative with respect to $\Real^3$-convolution.

\[
\xymatrix{
A(\mathbb{H},W) \ar[d]_{\varphi}
& \A \ar@{_{(}->}[l] \ar[ld]_{\varphi|_\A} &
& C^\infty_c(\Real^3) \ar[ll]_(.6){(2\pi)^{\frac{3}{2}} \F^{\Real^3}} \ar@/^/[llld]^{\psi}
 \\
\mathbb{C} & & &
}
\]

This leads us to solving a Cauchy type functional equation on $\Real^3$ in distribution sense. Thus, the problem of understanding elements in Spec$A(\mathbb{H},W)$ reduces to solving a Cauchy type functional equation, which is the beauty of borrowing the ``background" Euclidean structure of $\mathbb{H}$.

\subsubsection{A dense subalgebra and dense subspaces of $A(\mathbb{H},W)$}

Now we define dense subalgebras $\A$ and $\B$ of $A(\mathbb{H},W)$, and a dense subspace $\mathcal{D}$ of
	$$\F^{\mathbb{H}}(A(\mathbb{H})) = L^1(\Real^*, |a|da; S_1(L^2(\Real))).$$
This triple $(\A, \B, \D)$ will replace the role of $\text{Trig}(G)$ for a compact group $G$.

\begin{defn}\label{def-spaces-Heisenberg}
We define
	$$\A := \F^{\Real^3}(C^\infty_c(\Real^3))\subseteq C^\infty(\mathbb{H})\;\; \text{and}\;\;\B := \F^{\Real^3}(\B_0)\subseteq C^\infty(\mathbb{H})$$
where
	$$\B_0 := \{ f\in L^1_{\rm loc}(\Real^3) : e^{t(|x|+|y|+|z|)}(\partial^\alpha f)(x,y,z) \in L^2(\Real^3),\, \forall t>0,\, \forall \text{multi-index}\; \alpha \}$$
where $\partial^\alpha$ refers to the partial derivative in the weak sense, and $L^1_{\rm loc}(\Real^3)$ refers to the space of locally integrable fucntions on $\Real^3$. We endow $\B_0$ with  natural locally convex topology given by the family of semi-norms $\{\psi^\alpha_t : t>0,\,  \text{multi-index}\; \alpha\}$, where
	\begin{equation}\label{eq-psi}
	\psi^\alpha_t(f) = \|e^{t(|x|+|y|+|z|)}(\partial^\alpha f)(x,y,z)\|_{L^2(\Real^3)}.
	\end{equation}
	We equip $\B$ with the topology coming from $\B_0$.
	Finally, we define the space $\mathcal{D}$ by	
	$$\mathcal{D} := \text{span}\{h \otimes P_{mn} : m,n \in \z^{\geq 0},\; h\in C^\infty_c(\Real^*)\}\subseteq C^\infty_c(\Real^*; S^1(L^2(\Real)))$$
where	$P_{mn}$ is the rank 1 operator on $\B(L^2(\Real))$ given by $P_{mn} \xi = \la \xi, \varphi_n \ra \varphi_m$ with respect to the basis $\{\varphi_n\}_{n\ge 0}$ consisting of Hermite functions. We will fix the symbols $\A$, $\B$, $\B_0$, $\D$ throughout this section.
\end{defn}

\begin{rem}\label{rem-ChoiceSubalgebra}
	\begin{enumerate}
		\item
		From the isometric identification
	$$A(\mathbb{H}) = ({\F}^{{\mathbb H}})^{-1}L^1(\Real^*, |a|da; S^1(L^2(\Real)))$$
it is clear that the space $(\F^\mathbb{H})^{-1}\mathcal{D}$ is dense in $A(\mathbb{H})$.
		
		\item
		In the above an immediate candidate $C^\infty_c(\mathbb{H})$ of a dense subalgebra of $A(\mathbb{H},W)$ is not enough for our purpose. This claim can be examined even in the simplest non-compact case of $G = \Real$. Indeed, we can check that
			$$(\cap_{\beta \ge 1}A(\Real, w_\beta))\cap C^\infty_c(\Real) = \{0\},$$
		where $w_\beta(\xi) = \beta^{|\xi|}$, $\xi \in \widehat{\Real}$ is a weight function on $\widehat{\Real}$. Suppose we have $f\in (\cap_{\beta \ge 1}A(\Real, w_\beta))\cap C^\infty_c(\Real)$ with ${\rm supp}f \subseteq [-A, A]$ for some $A>0$. Then the Paley-Wiener theorem says that $\widehat{f}^{\Real}$ extends to an entire function $F$ on $\Comp$ with the growth condition $|F(z)| \le e^{A|z|}$ for any $z\in \Comp$. Then, the condition $f\in \cap_{\beta \ge 1}A(\Real, w_\beta) \Leftrightarrow \widehat{f}^{\Real} \in \cap_{\beta \ge 1}L^1(\widehat{\Real}, w_\beta)$ and the fact that $\widehat{f}^{\Real}$ belongs to the Schwarz class on $\Real$ imply that $\widehat{f}^{\Real} \in \cap_{\beta \ge 1}L^2(\widehat{\Real}, w_\beta)$, so that another theorem by Paley-Wiener (see Proposition \ref{prop-Paley-Weiner-super-exp-decay}) tells us that $f$ itself should be extendable to an entire function on $\Comp$, which forces $f=0$.
	\end{enumerate}
\end{rem}

We collect some basic properties of the above spaces.

\begin{prop}\label{prop-properties-subspaces}
	\begin{enumerate}
		\item The spaces $\A$ and $\B$ are algebras with respect to the pointwise multiplication.
		\item The space $\B_0$ is a Fr\'{e}chet space.
		\item The inclusion $C^\infty_c(\Real^3) \subseteq \B_0$ is continuous with dense range.
	\end{enumerate}
\end{prop}
\begin{proof}
(1) This is clear from a standard Euclidean theory and the definition of the space $\B$.

(2) This is a routine procedure.

(3) Continuity of the inclusion is clear. For the density we use the same argument for the smooth approximation of Sobolev functions (\cite[Theorem 3.3]{HH} for example).
\end{proof}

\begin{rem}\label{rem-space-K}
The space $\B_0$ can be called as the space of functions whose partial derivatives have a ``super-exponential'' decay. Note that a slightly different version of the space $\B_0$ has already been introduced in \cite{Jor} under the name of ``hyper-Schwartz space''.
\end{rem}

We have several reasons for the specific choice of $\B_0$. First,  the super-exponential decay property allows us to ``absorb" the effect of the weight $W$ which is possibly ``exponentially growing".

\begin{prop}\label{prop-absorb-weights}
For $f\in \B_0$ we have 	
	$$W\F^\mathbb{H}(\widehat{f}^{\Real^3}) = \F^\mathbb{H}(\widehat{g}^{\Real^3}),\;\;\text{where}\;\; g(t,u,s) = w(-t,-u)f(t,u,s),\; s,t,u\in \Real.$$
Moreover, $g\in \B_0$ and, the map $\B_0 \to \B_0,\; f\mapsto g$ is continuous.
\end{prop}
\begin{proof}
We begin with the following detailed understanding of the operator $\F^\mathbb{H}(\widehat{f}^{\Real^3})$ for any $f\in \B_0$ and $a\in \Real^*$, which is well-known and straightforward to check.
	\begin{align*}
	\F^\mathbb{H}(\widehat{f}^{\Real^3})(a)\xi(t)
	& = \int_{\Real^3}\widehat{f}^{\Real^3}(y,z,x) e^{-ia(ty-z)}\xi(-x+t)dydzdx\\
	& = 2\pi \int_\Real \widehat{f}^{\Real}_3(-at,a, t-x) \xi(x)dx,
	\end{align*}
where $\widehat{f}^{\Real}_3$ is used to denote the Fourier transform of $f$ with respect to the third variable only. Thus, $\F^\mathbb{H}(\widehat{f}^{\Real^3})(a)$ is an integral operator with the kernel
	\begin{equation}\label{eq-kernel-Heis}
	K_f(t,x) = 2\pi\widehat{f}^{\Real}_3(-at,a,t-x).
	\end{equation}
From \eqref{eq-Weights-Heisenberg} and the above formula it is clear that we have $W\F^\mathbb{H}(\widehat{f}^{\Real^3}) = \F^\mathbb{H}(\widehat{g}^{\Real^3})$ for the function $g$ given by $g(t,u,s) = w(-t,-u)f(t,u,s)$.
The fact that we have $g\in \B_0$ is directly from the condition \eqref{eq-derivatives-exp-growth}. Moreover, the continuity of the given map is trivial.
\end{proof}
\begin{rem}
For $f$, $g$, and $W$ as in Proposition  \ref{prop-absorb-weights}, we have
	$$\|\widehat{f}^{\Real^3}\|_{A({\mathbb H}, W)}=\|\widehat{g}^{\Real^3}\|_{A({\mathbb H})}.$$
Moreover, this implies that the space $\mc B$ can be used as the subspace $\mc S$ in \ref{ssec:separable-type I}.	
\end{rem}

We need a Fourier algebra norm estimate on the Heisenberg group. Roughly speaking Fourier transform of functions on the Heisenberg group with enough regularity belongs to the Fourier algebra. We record a general result for later use.

\begin{lem}\label{lem-Fourier-norm-estimate-general}
Let $G$ be a connected Lie group with the real dimension $d(G)$. For $m>\frac{d(G)}{4}$ there is a constant $C>0$ such that
	$$\|f\|_{A(G)} \le C \|(I- \partial\lambda(\Delta))^m f\|_{L^2(G)},$$
	for every $f\in {\rm dom}(I- \partial\lambda(\Delta))^m$.
\end{lem}
\begin{proof}
This is direct from \cite[Lemma 3.3 and (3.8)]{LT}.
\end{proof}

We continue to collect more properties of the spaces $\A$, $\B$ and $\D$.

\begin{prop}\label{prop-embeddings-Heisenberg} The spaces $\A$, $\B$ and $\D$ satisfy the following.
	\begin{enumerate}
		\item The space $\B$ is continuously embedded in $A(\mathbb{H}, W)$.
		\item The space $(\F^\mathbb{H})^{-1}\D$ is a subspace of $\B$ which is dense in $A(\mathbb{H}, W)$.
		\item The algebra $\A$ is dense in $A(\mathbb{H}, W)$.
	\end{enumerate}
\end{prop}
\begin{proof}
(1) As $\partial \lambda(X)$, $\partial \lambda(Y)$ and $\partial \lambda(Z)$ are infinitesimal generators for the one-parameter subgroups
$\lambda(\exp(tX))$, $\lambda(\exp(tY))$, and $\lambda(\exp(tZ))$ respectively, one can easily verify the formulas
	$$\begin{cases}\partial \lambda(X) F(y,z,x) = \partial_z F(y,z,x)\cdot(-y) - \partial_x F(y,z,x)\\ \partial \lambda(Y) F(y,z,x) = -\partial_y F(y,z,x) \\ \partial \lambda(Z) F(y,z,x) = -\partial_z F(y,z,x)\end{cases}$$
for any $F\in \mathcal{S}(\Real^3)$ as a function on $\mathbb{H}$. Here, $\mathcal{S}(\Real^3)$ refers to the Schwartz class on $\Real^3$. This implies that $I- \partial\lambda(\Delta) = I- \partial \lambda(X)^2 - \partial \lambda(Y)^2 - \partial \lambda(Z)^2$ is a linear partial differential operator with polynomial coefficients on $\Real^3$, the background Euclidean structure of $\mathbb{H}$. Now for $f\in C^\infty_c({\Real}^3)$ we have by Lemma \ref{lem-Fourier-norm-estimate-general} that
	$$\|\widehat{f}^{\Real^3}\|_{A(\mathbb{H})} \le C \|(I- \partial\lambda(\Delta)) \widehat{f}^{\Real^3}\|_{L^2(\Real^3)}.$$
Note that we can take $m = 1$ in Lemma \ref{lem-Fourier-norm-estimate-general} since $d(\mathbb{H}) = 3$. By taking the $\Real^3$-Fourier inversion we get
	$$(\F^{\Real^3})^{-1}((I- \partial\lambda(\Delta)) \widehat{f}^{\Real^3}) = D f$$
for some linear partial differential operator $D$ on $\Real^3$ with polynomial coefficients. Then the Plancherel theorem on $\Real^3$ tells us that for any $t>0$ we have
	$$\|\widehat{f}^{\Real^3}\|_{A(\mathbb{H})} \le C_t \cdot \sum_{|\alpha| \le M}\psi^\alpha_t(f)$$
for some constant $C_t$ depending only on $t$, where $M\in \n$ is determined by the order of $D$. This explains the continuity of the embedding $\B \subseteq A(\mathbb{H})$ and consequently of the embedding $\B \subseteq A(\mathbb{H},W)$ by Proposition \ref{prop-absorb-weights}.

\vspace{0.3cm}

(2) We first show that $(\F^\mathbb{H})^{-1}\D \subseteq \B$. For $m,n \in \z^{\geq 0},\; h\in C^\infty_c(\Real^*)$ we focus on the value of $h \otimes P_{mn}$ at the parameter $a\in \Real^*$, namely the operator $h(a)P_{mn}$, which is an integral operator with the kernel
	$$K(t,x) = {\varphi_m(t)} \varphi_n(x)h(a).$$
We would like to find a function $f\in \B_0$ such that $K = K_f$ in \eqref{eq-kernel-Heis}, which will imply that $P_{mn} \otimes h=\F^\mathbb{H}(\widehat{f}^{\Real^3})$. Indeed, it is straightforward to check that
	$$f(y,z,x) = \frac{1}{2\pi}i^n\varphi_n(x)e^{-i \frac{xy}{z}}\varphi_m(-\frac{y}{z}) h(z),\; z\ne 0$$
satisfies $K = K_f$. Concerning the condition $f\in \B_0$ we first consider the function $e^{t(|x|+|y|+|z|)}f(x,y,z)$ for a fixed $t>0$. Then, we have
	\begin{align*}
	\lefteqn{\int_{\Real^3}e^{2t(|y|+|z|+|x|)}|f(y,z,x)|^2 dydzdx}\\
	& = \frac{1}{(2\pi)^2}\int_{\Real^3}e^{2t(|y|+|z|+|x|)}|\varphi_n(x) \varphi_m(-\frac{y}{z}) h(z)|^2 dydzdx\\
	& =  \frac{1}{(2\pi)^2}\int_{\Real^3} e^{2t(|yz|+|z|+|x|)}|\varphi_n(x)\varphi_m(y)h(z)|^2 |z|\, dydzdx < \infty,
	\end{align*}
where we use the fact that $h$ is compactly supported on $\Real^*$. The $L^2$-conditions for the function $e^{t(|y|+|z|+|x|)}\partial^\alpha f(y,z,x)$ can be similarly checked. This gives the claim $(\F^\mathbb{H})^{-1}\D \subseteq \B$.

Secondly, we prove that $(\F^\mathbb{H})^{-1}\D$ is dense in $A(\mathbb{H},W)$, or equivalently the modified space
	\begin{equation}\label{eq-space-D-tilde}
	\widetilde{\mathcal{D}} := \text{span}\{W(h \otimes P_{mn}) : m,n \in \z^{\geq 0},\; h\in C^\infty_c(\Real^*)\}
	\end{equation}
is dense in $L^1(\Real^*, |a|da; S^1(L^2(\Real))) \cong A(\mathbb{H})$. Now we follow a standard argument for the completeness of $\{\varphi_n\}_{n\ge 0}$ in $L^2(\Real)$. Note that $W(a)P_{mn}$ is also a rank 1 operator $\xi \mapsto \la \xi, \varphi_n \ra \psi^a_m$, where
$\psi^a_m(t) = w(at, -a)\varphi_m(t)$. Since we have dependence on the parameter $a$, we need to consider the whole family $(\psi^a_m)_{m\ge 0, a\ne 0}$ altogether. Moreover, recall that
	$$L^1(\Real^*, |a|da; S^1(L^2(\Real))) \cong L^1(\Real^*, |a|da; L^2(\Real)) \otimes_\gamma L^2(\Real)$$
where $\otimes_\gamma$ is the projective tensor product of Banach spaces, so that it is enough to check that the subspace
	$\text{span}\{h \otimes \psi^a_m : m\ge 0, h\in C^\infty_c(\Real^*)\}$
is dense in $L^1(\Real^*, |a|da; L^2(\Real))$. Indeed, for any function $F(a,t) \in L^\infty(\Real^*, |a|da; L^2(\Real))$ such that
	$$\int_{\Real^*} \int_\Real F(a,t) h(a) \psi^a_m(t) dt\, |a| da = 0$$
for every $m\ge 0, h\in C^\infty_c(\Real^*)$, we have
	$$\int_{\Real^*} \int_\Real F(a,t)w(ta,-a) e^{-\frac{1}{2}t^2}P(t) |a|h(a) dtda = 0$$
for any polynomial $P$. Let ${\rm supp}_h$ denote the characteristic function of the support of  $h$, and set
	$$H(a,t) = F(a,2t)w(2ta,-a) e^{-\frac{3}{2}t^2} |a| \cdot {\rm supp}_h(a) \in L^\infty(\Real^*; L^2(\Real))$$
where we use the fact that $w$ is at most of exponential growth (Propositon \ref{prop-at-most-exp}) and $h$ is compactly supported. Moreover,  $H\in L^2(\Real^*;L^2(\Real))$ as well, since it has compact support in $\Real ^*$.
We have defined $H$ so that $\displaystyle \int_{\Real^*}\int_{\Real} H(a,t) e^{-\frac{1}{2}t^2}P(t) h(a) dtda=0$ for every polynomial $P(t)$ and every $h\in C_c(\Real^*)$, which
implies that $H(a,t)=0$  for almost every $(a,t) \in \Real^* \times \Real$. Here we use the facts that $C_c(\Real^*)$ is dense in $L^2(\Real)$, and Hermite functions form a basis for $L^2(\Real)$. Thus $F(a,t) = 0$ for almost every $(a,t) \in \Real^* \times \Real$. This explains the density of
$(\F^\mathbb{H})^{-1}(\D)$ in $A(\mathbb{H},W)$.

\vspace{0.3cm}

(3) From (3) of Proposition \ref{prop-properties-subspaces} we know that $\A$ is dense in $\B$. Then the expected density follows easily from a standard argument together with the result in (2).
\end{proof}

The density of $\A$ in $A(\mathbb{H}, W)$ is the first link we need for determining ${\rm Spec} A(\mathbb{H}, W)$, which we could only prove through the additional spaces $\B$ and $\D$.
The choice of the space $\D$ also provides us a big enough source of entire functions for the left regular representation $\lambda$ on $G$. The need for entire vectors will be clarified later in Section \ref{subsec-final-Heisenberg}.

	\begin{prop}\label{prop-density-entire}
		We have the inclusion $(\F^{\mathbb{H}})^{-1}\mathcal{D} \subseteq \D^\infty_\Comp(\lambda)$.
	\end{prop}
\begin{proof}
By a simple modification of the arguments of \cite[p. 388]{Good71} we can check that $(\F^{\mathbb{H}})^{-1}(h \otimes P_{mn})$, $m,n \geq 0,\; h\in C^\infty_c(\Real^*)$ is an entire vector for $\lambda$ of $\mathbb{H}$. The only difference is that we are not using the modified version $\varphi^a_n$ given by $\varphi^a_n(x) = |a|^{\frac{1}{4}}\varphi_n(|a|^{\tfrac{1}{2}}x)$ as in \cite{Good71}.
\end{proof}

Note that $(\F^{\mathbb{H}})^{-1}\mathcal{D} \cap \A = \{0\}$ whilst $(\F^{\mathbb{H}})^{-1}\mathcal{D} \subseteq \B$. This is another reason
that the subalgebra $\A$ alone is not sufficient for our purposes, and we do need the bigger subalgebra $\B$.

\subsubsection{Solving Cauchy functional equations on $\Real^n$}\label{sssec:CFE-Euclidean}
In this section we explain how we solve the Cauchy functional equation on $\Real^n$ for distributions, which we denote by $({\rm CFE}_{\Real^n})$.

	$$({\rm CFE}_{\Real^n}) \;\; T\in C^\infty_c(\Real^n)^* \;\text{satisfying}\; \la T, f*g \ra = \la T, f\ra \la T, g\ra, \; f,g,\in C^\infty_c(\Real^n).$$
We remark that if the distribution $T$ is coming from a locally integrable function $\psi$ on $\Real^n$, then the above condition can be easily shown to be equivalent to the usual form
	$$\psi(x+y) = \psi(x)\psi(y)\;\; \text{for a.e.}\;\; x,y\in \Real^n.$$
	 	
\begin{thm}\label{thm-Cauchy-Rn}
Let $T\in C^\infty_c(\Real^n)^*$ be a solution of $({\rm CFE}_{\Real^n})$, then there are uniquely determined $c_1, \cdots, c_n\in \Comp$ such that
	$$\la T, f\ra = \int_{\Real^n} f(x_1, \cdots, x_n) e^{-i(c_1x_1 + \cdots + c_nx_n)}dx_1\cdots dx_n, \; f\in C^\infty_c(\Real^n).$$
In other words, the distribution $T$ is actually a function of exponential type
	$$e^{-i(c_1x_1 + \cdots + c_nx_n)}.$$
\end{thm}	
\begin{proof}
This must be standard, but we include the proof for the convenience of the readers. Indeed, we will repeat the same steps of the solution later in a much more involved form for the case of $E(2)$ and $\widetilde{E}(2)$.

\vspace{0.3cm}

{\bf (The case of $n=1$)} Let $D : C^\infty_c(\Real) \to C^\infty_c(\Real),\; f\mapsto f'$ be the usual differentiation and $D^*: C^\infty_c(\Real)^* \to C^\infty_c(\Real)^*$ be the adjoint map. Note that we are purposely using $D^*$, which is actually the negative sign of the usual convention of differentiation in distribution theory. The proof splits into three steps.

[Step 1] If $T\in C^\infty_c(\Real)^*$ is a solution of $({\rm CFE}_{\Real^n})$, then $T$ satisfies a first order linear differential equation, namely
	$$D^*T = cT$$
for some $c\in \Comp$. Indeed, for any $f, g\in C^\infty_c(\Real)$ we have $\la T, (Df)*g \ra = \la T, f*(Dg)\ra$; by multiplicativity of $T$, we have $\la D^*T, f\ra \la T, g \ra = \la T, Df \ra \la T, g \ra = \la T, f \ra \la T, Dg \ra$.
Assuming $T\neq 0$ and  taking $g$ such that $\la T, g \ra \ne 0$, we have
	$$\la D^*T, f\ra = \frac{\la T, Dg \ra}{\la T, g \ra}\cdot \la T, f\ra.$$

[Step 2] If $T\in C^\infty_c(\Real)^*$ satisfies $D^*T = cT$ for some $c\in \Comp$, then $D^*(T\cdot e^{cx}) = 0$. Note that this is a direct consequence of Leibniz's rule. Here we are using the easy fact that the test function space $C^\infty_c(\Real)$ is closed under multiplying exponential functions.

[Step 3] If $S\in C^\infty_c(\Real)^*$ satisfies $D^* S = 0$, then $S$ is actually a constant function. Here is a standard argument excerpted from \cite[Theorem 4.3]{DK}. We choose a function $h\in C^\infty_c(\Real)$ with $\int_\Real h(x)\,dx=1$ and define the operator $I : C^\infty_c(\Real) \to  C^\infty_c(\Real)$ by
	$$I(\phi)(x) = \int^x_{-\infty}\phi(t)dt - (\int_\Real \phi(x)dx)\left( \int^x_{-\infty} h(t) dt \right).$$
Note that we have $D(I(\phi)) = \phi - (\int_\Real \phi(x)dx) h,$ so that $\la S, \phi\ra - (\int_\Real \phi(x)dx)\la S, h\ra = \la S, D(I(\phi))\ra = \la D^*S, I(\phi)\ra = 0.$
Thus, we get
	$$\la S, \phi\ra = (\int_\Real \phi(x)dx)\la S, h\ra$$
which implies that $S$ is actually a constant function with value $\la S, h\ra$. Now we just need to observe that the original condition $({\rm CFE}_{\Real^n})$ forces the constant value of $S = T\cdot e^{cx}$ to be equal to 1.

\vspace{0.3cm}

{\bf (The case of $n\ge 2$)} For simplicity we focus on the case $n=2$. Higher dimensional cases are similar. Let $\partial_x$ and $\partial_y$ be the partial derivatives on $\Real^2$.

[Step 1] With a similar arguments we get $\partial^*_x T = c_1 T$ and $\partial^*_y T = c_2 T$ for some $c_1, c_2 \in \Comp$.

[Step 2] By Leibniz's rule we get $\partial^*_x(T \cdot e^{c_1x + c_2y}) = \partial^*_y(T \cdot e^{c_1x + c_2y}) = 0$.

[Step 3] Suppose that $S\in C^\infty_c(\Real^2)^*$ satisfies $\partial^*_xS = \partial^*_yS = 0$. We define partial integrations $E_x, E_y: C^\infty_c(\Real^2) \to  C^\infty_c(\Real)$ by
	$$E_x(\phi)(y) = \int_\Real \phi(x,y)dx,\;\; E_y(\phi)(x) = \int_\Real \phi(x,y)dy.$$
We use the chosen function $h\in C^\infty_c(\Real)$ again, and define the operator $I : C^\infty_c(\Real^2) \to  C^\infty_c(\Real^2)$ by
	$$I(\phi)(x,y) = \int^x_{-\infty}\phi(t,y)dt - E_x(\phi)(y)\left( \int^x_{-\infty} h(t) dt \right).$$
Note that we have $\partial_x(I(\phi))(x,y) = \phi(x,y) - E_x(\phi)(y) h(x)$, so that we have
	$$\la S, \phi\ra - \la S, h(x)E_x(\phi)(y)\ra = \la S, \partial_x(I(\phi))\ra = \la \partial^*_xS, I(\phi)\ra = 0.$$
This tells us that $\la S, \phi\ra = \la S, h \times E_x(\phi)\ra.$ By applying the same argument on the second variable we get
	$$\la S, \phi\ra = \la S, h \times E_x(\phi)\ra = \la S, h \times h \ra \cdot \int_{\Real^2} \phi(x,y) dxdy$$
which implies that $S$ is a constant function. Finally, by setting $S = T \cdot e^{c_1x + c_2y}$ we can easily conclude that the constant value must be 1 by recalling the original condition $({\rm CFE}_{\Real^n})$.
\end{proof}

\subsubsection{The final step for the Heisenberg group}\label{subsec-final-Heisenberg}

We begin with a realization of the spectrum $\text{Spec}A(\mathbb{H},W)$ in $\mathbb{H}_\Comp.$

	\begin{prop}
	Every character $\varphi \in {\rm Spec}A(\mathbb{H},W)$ is uniquely determined by a point $(y,z,x) \in \mathbb{H}_\Comp \cong \Comp^3$, which is nothing but the evaluation at the point $(y,z,x)$ on $\B$ (and consequently on $\A$).
	\end{prop}
\begin{proof}
Let $\varphi \in \text{Spec} A(\mathbb{H},W)$ and consider a continuous composition $\psi = \varphi \circ ((2\pi)^{\frac{3}{2}}\F^{\Real^3})$ as in the following diagram.
\[
\xymatrix{
A(\mathbb{H},W) \ar[d]_{\varphi}
& \B \ar@{_{(}->}[l]
& \A \ar@{_{(}->}[l] \ar[lld]_{\varphi|_\A} &
& C^\infty_c(\Real^3) \ar[ll]_(.5){(2\pi)^{\frac{3}{2}}\F^{\Real^3}} \ar@/^/[lllld]^{\psi}
\\
\mathbb{C} & & &
}
\]
	
Since $\varphi$ is multiplicative with respect to pointwise multiplication, the functional $\psi$ is multiplicative with respect to $\Real^3$-convolution. Now by Theorem \ref{thm-Cauchy-Rn}, or equivalently by solving the Cauchy functional equation for $\psi$ we know that there is $(y,z,x) \in \Comp^3$ such that for $f\in C^\infty_c(\Real^3)$ we have
	$$\varphi((2\pi)^{\frac{3}{2}}\widehat{f}^{\Real^3}) = \int_{\Real^3}f(t,u,s)e^{- i (yt + zu + xs)}dtduds.$$
By the density of $C^\infty_c(\Real^3)$ in $\B_0$ and the continuity of the above two functionals $f\mapsto \varphi((2\pi)^{\frac{3}{2}}\widehat{f}^{\Real^3})$ and $f\mapsto \int_{\Real^3}f(t,u,s)e^{- i (yt + zu + xs)}dtduds$ on $\B_0$ we can see that the above equality is actually true for all $f\in \B_0$. Thus, by Proposition \ref{prop-Paley-Weiner-super-exp-decay} we have
	\begin{equation}\label{eq-Cauchy-sol}
	\varphi(F) = F_\Comp(y, z, x)
	\end{equation}	
for any $F\in \B$, where $F_\Comp$ is the analytic continuation of $F$. Finally, the density of $\B$ in $A(\mathbb{H},W)$ finishes the proof.
\end{proof}	

Secondly, we will show that the above realization of ${\rm Spec}A(\mathbb{H},W)$ in $\mathbb{H}_\Comp$ respects the Cartan decomposition \eqref{eq-Cartan-heis}.

\begin{prop}\label{Prop-respect-Cartan-heis}
We have ${\rm Spec}A(\mathbb{H},W)\subseteq \mathbb{H} \cdot \exp(i\,\mathfrak{heis})$ in the sense that for any $\varphi\in {\rm Spec}A(\mathbb{H},W)$ there are uniquely determined $g\in \mathbb{H}$ and $X'\in \mathfrak{heis}$ such that $\lambda_\Comp(\exp(iX'))W^{-1}$ is bounded on a dense subspace of $L^2(\mathbb{H})$ and
	$$\varphi = \lambda(g) \overline{\lambda_\Comp(\exp(iX'))W^{-1}}W.$$
\end{prop}
\begin{proof}
We first recall that for $X'\in \mathfrak{heis}_\Comp$ we have
	$$\lambda_\Comp(\exp(X')) \stackrel{\F^\mathbb{H}}{\sim} (\pi^a_\Comp(\exp(X')))_{a\in \Real^*}\;\; \text{on}\;\; \D^\infty_\Comp(\lambda).$$
The decomposition $W = \int^\oplus_{\Real^*}W(a) |a|da$ tells us that $W^{-1} \stackrel{\F^\mathbb{H}}{\sim} (W^{-1}(a))_{a\in \Real^*}$ on ${\ran}W$, so that by composition we get
	\begin{equation}\label{eq-equiv-Heis-op}
	    \lambda_\Comp(\exp(X'))W^{-1} \stackrel{\F^\mathbb{H}}{\sim} (\pi^a_\Comp(\exp(X'))W^{-1}(a))_{a\in \Real^*}\;\; \text{on}\;\; (\F^{\mathbb{H}})^{-1}(\widetilde{\D}),
	\end{equation}
where the space $\widetilde{\D} = W\D$ is the one introduced in \eqref{eq-space-D-tilde}. Note that we can show that $\widetilde{\D}$ is also dense in $L^2(\Real^*, |a|da; S^2(L^2(\Real)))$ by repeating the same argument for the density of $\widetilde{\D}$ in $L^1(\Real^*, |a|da; S^1(L^2(\Real)))$ in the proof of Proposition \ref{prop-embeddings-Heisenberg}.

Now we consider $\varphi\in {\rm Spec}A(\mathbb{H},W)$ a character associated with the point $(y,z,x) = \exp(-X')\in \mathbb{H}_\Comp$, $X'\in \mathfrak{heis}_\Comp$.
For $F\in (\F^\mathbb{H})^{-1}\mathcal{D}$ we have by (2) of Theorem \ref{thm-Goodman}, \eqref{eq-Cauchy-sol} and Proposition \ref{prop-density-entire} that
	\begin{align}\label{eq-uniform-bdd}
	\varphi(F)
	& = F_\Comp(y, z, x) = \int_{\Real^*} \text{\rm Tr}(\pi^a_\Comp((y, z, x)^{-1})\widehat{F}^\mathbb{H}(a))\,|a|da\\
	& = \int_{\Real^*} \text{Tr}(\pi^a_\Comp(\exp(X'))\widehat{F}^\mathbb{H}(a))|a|da \nonumber\\
	& = \int_{\Real^*} \text{Tr}(\pi^a_\Comp(\exp(X'))W^{-1}(a) W(a)\widehat{F}^\mathbb{H}(a))|a|da. \nonumber
	\end{align}
Since $|\varphi(F)| \le \|\varphi\|_{VN(\mathbb{H},W^{-1})} \cdot \|F\|_{A(\mathbb{H},W)} = \|\varphi\|_{VN(\mathbb{H},W^{-1})} \cdot \|WF\|_{A(\mathbb{H})}$, the density of $\widetilde{\D}$ in $L^1(\Real^*, |a|da; S^1(L^2(\Real)))$ tells us that $\pi^a_\Comp(\exp(X'))W^{-1}(a)$ is uniformly bounded with respect to $a \in \Real^*$. From condition \eqref{eq-equiv-Heis-op} we now know that $\lambda_\Comp(\exp(X'))W^{-1}$ is bounded and we get the conclusion
	$$\varphi = \overline{\lambda_\Comp(\exp(X'))W^{-1}}W \in VN(\mathbb{H},W^{-1}).$$
Indeed, for $F\in (\F^\mathbb{H})^{-1} \mathcal{D}$ we have
	\begin{align*}
	\lefteqn{\la \overline{\lambda_\Comp(\exp(X'))W^{-1}}W, F \ra_{(VN(\mathbb{H},W^{-1}), A(\mathbb{H}, W))}}\\
		& = \la \overline{\lambda_\Comp(\exp(X'))W^{-1}}, WF \ra_{(VN(\mathbb{H}), A(\mathbb{H}))}\\
		& = \int_{\Real^*} \text{Tr}(\overline{\pi^a_\Comp(\exp(X'))W^{-1}(a)} W(a)\widehat{F}^\mathbb{H}(a))|a|da = \varphi(F),
	\end{align*}	
where we use \eqref{eq-uniform-bdd} for the last equality. The density of $(\F^\mathbb{H})^{-1}\D$ in $A(\mathbb{H}, W)$ gives us the desired conclusion.

Finally we recall the Cartan decomposition \eqref{eq-Cartan-heis} and the fact that $\pi^a_\Comp$ is a (local) representation on $\D^\infty_\Comp(\lambda)$. Then in combination with the above observations we get that 
	${\rm Spec}A(\mathbb{H},W)$ is contained in $$\Big\{ \lambda(g) \overline{\lambda_\Comp(\exp(iX'))W^{-1}}W: X'\in \mathfrak{heis},\; \lambda_\Comp(\exp(iX'))W^{-1}\; \text{bounded} \Big\}.$$
\end{proof}

\begin{rem}
It is tempting to claim the equality $\varphi = \lambda(g)\lambda_\Comp(\exp(iX'))$, but we stick to the equality $\varphi = \lambda(g) \overline{\lambda_\Comp(\exp(iX'))W^{-1}}W$ reflecting the precise structure of the weighted space $VN(\mathbb{H},W)$.
\end{rem}

We continue to determine the Gelfand spectrum of $A(\mathbb{H},W)$ as follows.

	\begin{thm}\label{thm-Heisenberg-Spec}
	Let $\mathfrak{h}$ be the Lie subalgebra of $\g$ corresponding to the subgroup $H = H_{Y,Z}$ of $\mathbb{H}$. Suppose that $W_H$ is a weight on the dual of $H$ and $W= \iota(W_H)$ is the extended weight on the dual of $G$ as in Section \ref{subsection-ext-subgroup}. Then we have
		$${\rm Spec}A(\mathbb{H},W) \cong \Big\{ g\cdot \exp(iX'): g\in \mathbb{H},\; X'\in \mathfrak{h},\; \exp(iX') \in {\rm Spec} A(H, W_H)\Big\}.$$
	\end{thm}

The proof for the above theorem begins with excluding elements $X\in \mathfrak{heis}\backslash \mathfrak{h}$. This requires an analogue of Theorem \ref{thm-unbdd-Lie-compact}. We will establish a more general result for the possible future extension of the theory to the case of nilpotent Lie groups.

\begin{thm}\label{thm-unbdd-Lie-noncompact}
Let $G$ be a connected and simply connected nilpotent Lie group and $H = \exp \mathfrak{h}$ be a closed connected abelian Lie subgroup of $G$. Suppose that $W_H$ is a bounded below weight on the dual of $H$ and $W_G = \iota(W_H)$ is the extended weight on the dual of $G$ as in Section \ref{subsection-ext-subgroup}. Then for any $X\in \mathfrak{g} \backslash \mathfrak{h}$ the operator $\exp(i\partial \lambda(X))W^{-1}_G$ is unbounded, whenever it is densely defined.
\end{thm}
\begin{proof}
[Step 1] We first assume that our basis $\{X_1, \cdots, X_n\}$ of $\g$ is a {\it weak Malcev} basis of $\g$ satisfying that (1) $\{X_1, \cdots, X_d\}$ is a basis of $\h$ and (2) $X = X_{d+1}$. Recall that $\{X_1, \cdots, X_n\}$ being a weak Malcev basis of $\g$ means that $\h_k = {\rm span}\{X_1, \cdots, X_k\}$ is a Lie subalgebra of $\g$ for all $1\le k \le n$. It is well known that such a basis with condition (1) exists for a connected and simply connected nilpotent Lie group $G$ by \cite[Theorem 1.1.13]{CG}. The condition (2) is not true in general, but we will handle that case later.

We will basically follow the proof of Theorem \ref{thm-unbdd-Lie-compact}. The modification begins at the second step there. Moreover, we assume that $H \cong \Real^n$ for simplicity. The case $H \cong \Real^{n-k}\times \tor^{k}$ can be done similarly.

\vspace{0.3cm}

[Step 2] We pick the functions $\varphi \in C^\infty_c(\Real)$ and $\gamma \in C^\infty_c(\Real^{n-d-1})$ as before. We further pick a non-zero function $\delta \in \mathcal{S}(H)$, which does not have to be real-valued this time and will be specified later. Here, $\mathcal{S}(H)$ refers to the Schwarz class on $H\cong \Real^n$. We similarly define the function $\psi$ on $G$ by
	$$\psi(E(h,t,v)) := \delta(h)\varphi(t)\gamma(v)$$ for $(h,t,v)\in H \times J' \times V'$ and zero elsewhere. Then, for $(h,t,v)\in H \times J' \times V'$ we have
	\begin{align*}
	\lefteqn{\partial\lambda(X)\psi(E(h,t,v))}\\
	& = \partial_s (\delta \circ \theta)(0,h,t,v) \varphi(t)\gamma(v) +  \delta(h)\varphi'(t)\partial_s\alpha(0,h,t,v)\gamma(v)\\
	&\;\;\;\; + \delta(h)\varphi(t)\partial_s(\gamma\circ \beta)(0,h,t,v)\\
	& = A(E(h,t,v)) + B(E(h,t,v)) + C(E(h,t,v)).
	\end{align*}
The estimates for $B$ is the same as before and we have
	$$\|B\|^2_{L^2(E(H \times J' \times V'))} \gtrsim \|\varphi'\|_2^2.$$	
For the estimate of $A$ and $C$ we use the fact that the functions $\theta$ and $\beta$ are {``polynomials''}. Recall that a map $f:E\to F$ between two finite dimensional  vector spaces is called {\it polynomial} if it is a polynomial with respect to some (and for any) pair of bases. We use the fact that the exponential map is a global diffeomorphism between $\g$ and $G$ to transfer the concept of polynomial maps to the level of $G$. We say that a map $f:G\to G$ is  a {\it polynomial} if the corresponding map  $\exp^{-1} \circ f\circ\exp:\g \to \g$ is a polynomial. This definition can be easily extended to the case of a map $f:G\times G\to G$ and for example, it is known that the group multiplication $G\times G \to G, (g_1, g_2)\to g_1 g_2$ is a polynomial by \cite[Theorem 6.13]{FL}. We know from \cite[Proposition 1.2.8]{CG} that we can choose $J = \Real$ and $V = \Real^{n-d-1}$ in the definition of the modified exponential map $E$ in \eqref{eq-modified-exp}, and $E$ and $E^{-1}$ both are polynomials. Then, the map $\theta$ and $\beta$ from \eqref{eq-3-maps} are obtained by the composition of group multiplication, $E^{-1}$ and the corresponding projection, so that they also are polynomials.

Thus, we can see that
	$$\partial_s (\delta \circ \theta)(0,h,t,v) = \partial \delta(h) \cdot P(h,t,v);\;\text{and}\;\;\partial_s(\gamma\circ \beta)(0,h,t,v)) = \partial \gamma(v) \cdot Q(h,t,v)$$
for some polynomials $P$ and $Q$. From the choice of $\delta$ and $\gamma$ we know that
	$$\|A\|^2_{L^2(E(H \times J' \times V'))} \lesssim \|\varphi\|_2\;\;\text{and}\;\; \|C\|^2_{L^2(E(H \times J' \times V'))} \lesssim \|\varphi\|_2.$$

Now we replace $\varphi(t)$ by its dilated version $\sqrt{N}\varphi(Nt)$ for $N\ge 1$ to define
	$$\psi_N(E(h,t,v)) := \delta(h) \sqrt{N}\varphi(Nt)\gamma(v)$$ for $(h,t,v)\in H \times \frac{1}{N}J' \times V'$ and zero elsewhere. Then combining all the above estimates we get
	$$\|\partial\lambda(X)\psi_N\|_{L^2(G)} \gtrsim N\;\; \text{and}\;\; \|\psi_N\|_{L^2(G)} \lesssim 1$$
as before.

Finally, we observe that $W^{-1}_G\psi_N(E(h,t,v)) = W_H^{-1}\delta(h) \sqrt{N}\varphi(Nt)\gamma(v)$. Indeed, thanks to the boundedness of $W^{-1}_G = \iota(W^{-1}_H) \in \iota(VN(H))\subseteq VN(G)$, it is enough to check $\lambda_G(k)\psi_N(E(h,t,v)) = \lambda_H(k)\delta(h) \sqrt{N}\varphi(Nt)\gamma(v)$, $k\in H$ which is clear from the definition of $\psi_N$.
	
\vspace{0.3cm}
	
[Step 3] For this part we repeat the same argument for  $\widetilde{\psi_N} = W^{-1}_G\psi_N$ to get
	$$\|\partial\lambda(X)\widetilde{\psi_N}\|_{L^2(G)} \gtrsim N.$$
If we choose a non-zero real-valued $\tilde{\delta} = W_H^{-1}\delta \in \mathcal{S}(H)$, then we have the same estimate and we have
	$$\|\exp(i\partial\lambda(X))W^{-1}_G\psi_N\|^2_2 = \|\exp(i\partial\lambda(X))\widetilde{\psi_N}\|^2_2 \gtrsim N^2.$$
This leads us to the conclusion that $\exp(i\partial\lambda(X))W^{-1}_G$ is unbounded provided that $\tilde{\delta} \in \text{dom}(W_H)$. Now recall that $W_H = (\F^H)^{-1} \circ M_w \circ \F^H$, so that a Gaussian function on $H \cong \Real^n$ would be an appropriate choice for $\tilde{\delta}$, since $w$ is at most exponentially growing.

\vspace{0.3cm}

Now we consider the case that our chosen vector $X$ is not the $(d+1)$-th element $X_{d+1}$ in the given weak Malcev basis. Fortunately, we may choose our Malcev basis $\{X_1, \cdots, X_n\}$ such that (1) $\{X_1, \cdots, X_d\}$ is a basis of $\h$ and (2) $X = X_j$ for some $d+1\le j\le n$. Indeed, the existence of a weak Malcev basis $\{X_1, \cdots, X_n\}$ with condition (1) is already guarranteed by \cite[Theorem 1.1.13]{CG}. Now we write $X = x_1X_1 + \cdots + x_nX_n$, and let $1\le j \le n$ be the largest index so that $x_j\ne 0$. Then, we know that $X_j = aX + Y$, $a\ne 0$ and $Y\in \h_{j-1}$, so that $\h_j = {\rm span}\{X_1, \cdots, X_{j-1}, X\}$. Thus, we know that ${\rm span}\{X_1, \cdots, X_{j-1}, X, X_{j+1}, \cdots, X_n\}$ is another weak Malcev basis. Moreover, we should have $j>d$ since $X\in\g\setminus \h$, which explains the claim. Now we consider a slightly changed version of the modified exponential map in Proposition \ref{prop-local-formulas} as follows.
	$$E: H\times J \times V \to U,\; (h, t, x_{d+1},\cdots, \check{x}_j, \cdots, x_n)=(h,t,v)\mapsto E(h,t,v)$$
where $E(h,t,v) = h\cdot \exp(x_{d+1}X_{d+1})\cdots \exp(tX_j) \cdots  \exp(x_nX_n)$ and the notation $\check{x}_j$ means that we are skipping $j$-th variable in the expression. Then, all the results in Proposition \ref{prop-local-formulas} are still valid and all the above arguments works as before.
\end{proof}

\begin{proof}[Proof of Theorem \ref{thm-Heisenberg-Spec}]
As in the proof of Theorem \ref{thm-spec-extended-compact} we may suppose that $X'\in \mathfrak{h}$ by Theorem \ref{thm-unbdd-Lie-noncompact} and a similar truncation argument.
Now we claim that the operator $(\lambda_H)_\Comp(\exp(iX'))W^{-1}_H$ is bounded if and only if $\lambda_\Comp(\exp(iX'))W^{-1}$ is bounded.
Under the assumption that $W^{-1}_H$ is bounded, this claim follows directly from Proposition \ref{prop-extension-variant}. However, we can also prove the claim without using this assumption.
 Indeed, for $X' = (y,z,0) \in \mathfrak{h}$, we can readily check that
 	$$\F^H\circ (\lambda_H)_\Comp(\exp(iX')) \circ (\F^H)^{-1} = M_{\Phi_1}\;\;\text{and}\;\; \F^H\circ W^{-1}_H \circ (\F^H)^{-1} = M_{\Phi_2},$$
where $\Phi_1(a,b)=e^{ya+zb}$ and $\Phi_2(b,a)=\frac{1}{w(a,b)}$. Thus, $\F^H\circ (\lambda_H)_\Comp(\exp(iX')) \circ (\F^H)^{-1}$ and  $\F^H\circ W^{-1}_H \circ (\F^H)^{-1}$ are strongly commuting, so we can use joint functional calculus to see that
	$$\lambda_\Comp(\exp(iX'))W^{-1}=\iota((\lambda_H)_\Comp(\exp(iX'))W^{-1}_H).$$
On the other hand, we observe that
	\begin{equation}\label{eq-YZ-subgroup}\F^H\circ (\lambda_H)_\Comp(\exp(iX'))W^{-1}_H \circ (\F^H)^{-1} = M_\Phi
	\end{equation}
where $\Phi(a,b) = \frac{e^{ya+zb}}{w(a,b)}$. Note that the exponential map restricted to $\mathfrak{h}$ is nothing but the identity map.
We now use \eqref{eq-Lie-derivatives-Hei} and Proposition \ref{prop-extended-weights-abelian}, to obtain the exact formula for $\iota((\lambda_H)_\Comp(\exp(iX'))W^{-1}_H)=\{Y(a)\}_{a\in \Real^*}$. Indeed, we have
	\begin{equation*}
	Y(a)\xi(t) = \Phi(at,-a)\xi(t).
	\end{equation*}
for appropriate $\xi \in L^2(\Real)$, which is a multiplication operator with the parameter $a\in \Real^*$. Clearly,  $\iota((\lambda_H)_\Comp(\exp(iX'))W^{-1}_H)$ is bounded precisely when $\Phi$ is bounded, which is equivalent to the boundedness of  $(\lambda_H)_\Comp(\exp(iX'))W^{-1}_H$.
\end{proof}

\begin{rem}\label{rem-1D-subgp-Heis}
The same arguments can be applied to $H=H_Y$ and $H=H_Z$ to get the statement of Theorem \ref{thm-Heisenberg-Spec} with those subgroups instead of $H=H_{Y,Z}$.
\end{rem}

\begin{ex}
	For $X' = (y',z',0) \in \mathfrak{h}$ the condition $\exp(iX') \in {\rm Spec} A(H,W_H)$ is equivalent to the existence of a constant $C>0$ such that
		$$e^{ay'}e^{bz'} \le C w(a,b),\; \text{for almost all}\; (a,b)\in \Real^2$$
by \eqref{eq-YZ-subgroup}.
In particular, for specific choices of the weight function $w$, we have the following, by Theorem \ref{thm-Heisenberg-Spec}.
	\begin{enumerate}
		\item
		When $w(a,b) = \beta^{|a|}_1 \beta^{|b|}_2$, $(a,b)\in \Real^2$ for some $\beta_1, \beta_2 \ge 1$, we have
\begin{align*}	
{\rm Spec} A(\mathbb{H},W) \cong \{g\cdot(iy',iz',0) \in \mathbb{H}_\Comp \cong \Comp^3:&\  g\in {\mathbb H}, \ y',z'\in {\mathbb R}, \\& |y'| \le \log \beta_1, |z'| \le \log \beta_2\}.
\end{align*}
	Especially, when $\beta_2=1$ we have
	$${\rm Spec} A(\mathbb{H},W) \cong \{(y,z,x) \in \mathbb{H}_\Comp \cong \Comp^3: |{\rm Im}y| \le \log \beta_1, |{\rm Im}z| = |{\rm Im}x| =0\}.$$
		\item
		When $w(a,b) = \beta^{\sqrt{a^2 + b^2}}$, $(a,b)\in \Real^2$ for some $\beta \ge 1$, we have
\begin{align*}	{\rm Spec} A(\mathbb{H},W) \cong \{g\cdot(iy',iz',0) \in \mathbb{H}_\Comp \cong \Comp^3: &\ g\in {\mathbb H}, \ y',z'\in {\mathbb R}, \\& (y')^2+(z')^2 \le (\log \beta)^2\}.
\end{align*}
	\end{enumerate}
\end{ex}

We end this section with a description of the symmetry given by automorphisms on $\mathbb{H}$ more precise than Theorem \ref{thm-automorphism-principle}.

	\begin{thm}\label{thm-Heisenberg-Spec-automorphism}
	Let $\alpha: \mathbb{H}\to \mathbb{H}$ be a Lie group automorphism. We recall the notations $\alpha_{VN}$ and $\alpha_\Comp$ stated before Theorem \ref{thm-automorphism-principle}. Then we have
		$${\rm Spec}A(\mathbb{H},\alpha_{VN}(W)) \cong \alpha_\Comp({\rm Spec}A(\mathbb{H},W)) \subseteq \mathbb{H}_\Comp.$$
	\end{thm}
\begin{proof}
By Theorem \ref{thm-Heisenberg-Spec} for any $\varphi\in {\rm Spec}A(\mathbb{H},W)$ we have
	$$\varphi = \lambda(g) \overline{\lambda_\Comp(\exp(iX))W^{-1}}W$$
for some $g\in \mathbb{H}$ and $X\in \mathfrak{h}$. Now Theorem \ref{thm-automorphism-principle} tell us that the associated element in ${\rm Spec}A(\mathbb{H},\alpha_{VN}(W))$ is
	$$\alpha_{VN}(\lambda(g) \overline{\lambda_\Comp(\exp(iX))W^{-1}})\alpha_{VN}(W) = \alpha_{VN}(\lambda(g))\alpha_{VN}( \overline{\lambda_\Comp(\exp(iX))W^{-1}})\alpha_{VN}(W).$$
Since $\alpha_{VN}$ is an inner automorphism via a unitary conjugation we have
	$$\alpha_{VN}(\lambda_\Comp(\exp(iX))W^{-1}) = \alpha_{VN}(\lambda_\Comp(\exp(iX)))\alpha_{VN}(W^{-1}).$$
Now we know that $\alpha_{VN}(\lambda(g)) = \lambda(\alpha(g))$ and $\alpha_{VN}(\lambda_\Comp(\exp(iX))) =  \lambda_\Comp(\alpha_\Comp(\exp(iX)))$. Indeed, the automorphism $\alpha$ can be lifted to the Lie algebra level $\alpha_\mathfrak{h} : \mathfrak{h}\to \mathfrak{h}$ such that $\alpha_\Comp(\exp(X + iY)) = \exp(\alpha_\mathfrak{h}(X) + i \alpha_\mathfrak{h}(Y))$, $X, Y\in \mathfrak{h}$. Recall that $\partial\lambda(X)$ is the infinitesimal generator of the one-parameter group $\lambda(\exp(tX))$, $t\in \Real$, so that we have $\alpha_{VN}(\partial\lambda(X))$ is the infinitesimal generator of the one-parameter group $$\alpha_{VN}(\lambda(\exp(tX))) = \lambda(\exp(t\alpha_\mathfrak{h}(X))), t\in \Real,$$ which means that $\alpha_{VN}(\partial\lambda(X)) = \partial \lambda(\alpha_\mathfrak{h}(X))$. Thus, we have
	\begin{align*}
	\alpha_{VN}(\lambda_\Comp(\exp(iX)))
	& = \alpha_{VN}(\exp(\partial \lambda(iX)))
	= \exp(\alpha_{VN}(\partial \lambda(iX)))\\
	& = \exp(\partial \lambda(i\alpha_\mathfrak{h}(X)))
	= \lambda_\Comp(\exp(i\alpha_\mathfrak{h}(X)))\\
	& = \lambda_\Comp(\alpha_\Comp(\exp(iX))).
	\end{align*}
This justifies our claim.
\end{proof}


\section{The reduced Heisenberg group}\label{chap-reduced-Heisenberg}

The reduced Heisenberg group $\mathbb{H}_r$ is 	
	$$\mathbb{H}_r = (\Real \times \tor)\rtimes \Real$$
with the group law
	$$(y,z,x) \cdot (y',z',x') = (y+y', zz'e^{ixy'}, x+x').$$	
The order of variables $(y,z,x)$ is again from the semi-direct product structure of $\mathbb{H}_r$.
Using Mackey machinery, we can completely describe the unitary dual of ${\mathbb H}_r$ as follows:
$$\widehat{{\mathbb H}_r}=\big\{\chi_{r,s}: (r,s)\in {\mathbb R}^2\big\}\cup\big\{\pi^n: n\in {\mathbb Z}\backslash\{0\}\big\},$$
where for every $n \in \z\backslash \{0\}$, the irreducible unitary representation $\pi^n$ is defined on the Hilbert space $L^2({\mathbb R})$, whereas $\chi_{r,s}$
is a one-dimensional representation for every $(r,s)\in {\mathbb R}^2$. The precise formulas of these representations are given as
\begin{eqnarray*}
\chi_{r,s}(y,z,x)&=& e^{i(ry+sx)},\\
\pi^n(y,z,x) \xi(t) &=& z^n e^{-inty} \xi(-x+t),\; \xi \in L^2(\Real).
\end{eqnarray*}
For $f\in L^1(\mathbb{H}_r)$, we define the group Fourier transform on $\mathbb{H}_r$ by
	$$\F^{\mathbb{H}_r}(f) = (\F^{\mathbb{H}_r}(f)(\pi))_{\pi\in \widehat{{\mathbb H}_r}}=(\widehat{f}^{\mathbb{H}_r}(\pi))_{\pi\in \widehat{{\mathbb H}_r}}$$
where $H_\pi=L^2(\Real)$ when $\pi=\pi^n$, and $H_\pi={\mathbb C}$ otherwise, and
	$$\widehat{f}^{\mathbb{H}_r}(\pi)=\int_{\mathbb{H}_r} f(g)\pi(g) dg.$$
We sometimes identify $\widehat{{\mathbb H}_r}$ with $\Real^2\sqcup \z\backslash\{0\},$ where $\sqcup$ denotes the disjoint union of sets. Unlike the Heisenberg group, the Plancherel measure of ${\mathbb H}_r$ takes the one-dimensional representations into account, as shown in the following lemma.
\begin{lem}
The Plancherel measure $\mu$ of ${\mathbb H}_r$ is given by $\frac{dydx}{(2\pi)^2}$ when it is restricted to $\Real^2$, and it is given by  $\mu(\{\pi^n\})=\frac{|n|}{2\pi}$ on $\z\backslash\{0\}$.
\end{lem}
\begin{proof}
Since ${\mathbb H}_r$ is a second countable unimodular Type I group, it admits a unique measure, called the Plancherel measure, which satisfies the Parseval identity for the map $f\in L^2({\mathbb H}_r)\mapsto (\widehat{f}(\pi))_{\pi\in \widehat{{\mathbb H}_r}}$. So, we only need to check the Parseval identity for the above measure. Let $f\in (L^1\cap L^2)({\mathbb H}_r)$. For $n\in \z\backslash\{0\}$ and $\xi\in L^2(\Real)$, we have
\begin{eqnarray}
\widehat{f}^{{\mathbb H}_r}(\pi^n)\xi(t)&=&\int_{\mathbb{T}}\int_{\Real}\int_\Real f(y,z,x)z^ne^{-inty} \xi(-x+t)dxdydz\label{eq:hat-F(n)}\\
&=&\int_\Real K^f(t,x)\xi(x)dx,\nonumber
\end{eqnarray}
where
$$K^f(t,x)=\int_{\mathbb{T}}\int_{\Real} f(y,z,t-x)z^ne^{-inty}dydz=\sqrt{2\pi}\widehat{f}^{\Real \times \tor}_{1,2}(nt,-n,t-x),\ (t,x)\in \Real^2.$$
Here, $\widehat{f}^{\Real \times \tor}_{1,2}$ denotes the $\Real\times \tor$-Fourier transform in the first two variables. Clearly, $\widehat{f}^{{\mathbb H}_r}(\pi^n)$ is a Hilbert-Schmidt integral operator with norm given by
\begin{eqnarray*}
\|\widehat{f}^{{\mathbb H}_r}(\pi^n)\|_2^2&=&\|K^f\|_2^2=2\pi\int_\Real\int_\Real|\widehat{f}^{\Real \times \tor}_{1,2}(nt,-n,t-x)|^2 dxdt\\
&=&\frac{2\pi}{|n|}\int_\Real\int_\Real |\widehat{f}^{\tor}_2(y,-n,x)|^2dydx,
\end{eqnarray*}
where $\widehat{f}^{\tor}_2$ denotes the $\tor$-Fourier transform in the second variable.
Let ${\mathbb E}$ be the orthogonal projection on $L^2({\mathbb H}_r)$ defined as
	$${\mathbb E}f(y,z,x) := \int_{{\mathbb T}}f(y,u,x)du.$$
Clearly, $h\in L^1({\mathbb H}_r)\cap{\rm Ker}({\mathbb E})$ precisely when
$\widehat{h}^{\tor}_2(y,0,x)=0$ for every $x,y\in \Real$. For any $h\in L^1({\mathbb H}_r)\cap {\rm Ker}({\mathbb E})$, we have
\begin{equation*}
\|h\|_2^2=\sum_{n\in{\mathbb Z}\backslash\{0\}}\int_{\Real^2} |\widehat{h}^{\tor}_2(y,n,x)|^2dydx=\frac{|n|}{2\pi}\sum_{n\in{\mathbb Z}\backslash\{0\}} \|\widehat{h}^{{\mathbb H}_r}(\pi^n)\|_2^2.
\end{equation*}
On the other hand, it is easy to see that if $g\in (L^1\cap L^2)({\mathbb H}_r)$ is constant when restricted to ${\mathbb T}$, then $\widehat{g}^{{\mathbb H}_r}(\pi^n)=0$ for every $n\in {\mathbb Z}\backslash \{0\}$. So for such an element $g$, we have
\begin{eqnarray*}
\|g\|_2^2&=&\int_{\Real^2}\int_{{\mathbb T}} |g(y,z,x)|^2 dydzdx=\int_{\Real^2} \Big|g|_{\Real^2}(y,x)\Big|^2 dydx\\
&=&\int_{\Real^2}\Big|\widehat{g|_{\Real^2}}^{\Real^2}(r,s)\Big|^2 drds=\int_{\Real^2}\Big|\frac{1}{2\pi}\int_{\Real^2}g(y,x)e^{-i(ry+sx)}dydx\Big|^2 drds\\
&=&\frac{1}{(2\pi)^2}\int_{\Real^2}\Big|\widehat{g}^{{\mathbb H}_r}(\chi_{r,s})\Big|^2 drds.\\
\end{eqnarray*}
Now consider an arbitrary $f\in (L^1\cap L^2)({\mathbb H}_r)$, and note that $f$ can be orthogonally decomposed into $f={\mathbb E}(f)+(f-{\mathbb E}(f))$, with the property that ${\mathbb E}(f)|_{{\mathbb T}}$ is constant and $f-{\mathbb E}(f)\in {\rm Ker}({\mathbb E})$. Thus, we have
\begin{eqnarray*}
\|f\|_2^2&=&\|{\mathbb E}(f)\|_2^2+\|f-{\mathbb E}(f)\|_2^2\\
&=&\frac{1}{(2\pi)^2}\int_{\Real^2}\Big|\widehat{{\mathbb E}(f)}^{{\mathbb H}_r}(\chi_{r,s})\Big|^2 drds+\frac{|n|}{2\pi}\sum_{n\in{\mathbb Z}\backslash\{0\}} \|\widehat{f-{\mathbb E}(f)}^{{\mathbb H}_r}(\pi^n)\|_2^2\\
&=&\frac{1}{(2\pi)^2}\int_{\Real^2}\Big|\widehat{f}^{{\mathbb H}_r}(\chi_{r,s})\Big|^2 drds+\frac{|n|}{2\pi}\sum_{n\in{\mathbb Z}\backslash\{0\}} \|\widehat{f}^{{\mathbb H}_r}(\pi^n)\|_2^2.
\end{eqnarray*}
\end{proof}

The Plancherel measure gives us the quasi-equivalent decomposition of the  left regular representation $\lambda$ as follows:
	$$\lambda \cong \int^\oplus_{\Real^2}\chi_{r,s}\frac{drds}{(2\pi)^2} \oplus \bigoplus_{n\in \z \backslash \{0\}} \frac{|n|}{2\pi} \pi^n.$$
When $\Real^2$ and $\z\backslash\{0\}$ are equipped with the corresponding restrictions of the Plancherel measure on ${\mathbb H}_r$,
 the Fourier transform on ${\mathbb H}_r$ gives us the following decomposition.	
 \begin{eqnarray}
&&\F^{\mathbb{H}_r}: A({\mathbb{H}_r})\rightarrow L^1\Big(\widehat{\Real^2},\frac{drds}{(2\pi)^2}\Big)\oplus_1 L^1\Big(\z\backslash\{0\},\mu_d; S^1(L^2(\Real))\Big)\label{eq:FT}\\
&&f\mapsto \big(\widehat{f}^{\mathbb{H}_r}(\chi_{r,s})\big)_{(r,s)\in \Real^2} \oplus \bigoplus_{ n\in {\mathbb Z}\backslash \{0\}}\widehat{f}^{{\mathbb H}_r}(\pi^n),\nonumber
\end{eqnarray}
where we denote the restriction of $\mu$ to $\z\backslash\{0\}$ by $\mu_d$. The formula for $\widehat{f}^{\mathbb{H}_r}(\pi^n)$, $f\in L^1({\mathbb{H}_r})$ has been given in Equation (\ref{eq:hat-F(n)}). The case of one-dimensional representations goes as follows. We define
$ \widetilde{f}:={\mathbb E}(f)|_{\Real^2}$, then we have
	$$\widehat{f}^{\mathbb{H}_r}(\chi_{r,s})=2\pi\widehat{\widetilde{f}}^{\Real^2}(-r,-s).$$
Taking the dual of the above decomposition, we have
$$VN(\mathbb{H}_r) \cong L^\infty\Big(\widehat{\Real^2},\frac{drds}{(2\pi)^2}\Big)\oplus_{\infty} L^\infty\Big(\z\backslash\{0\},\mu_d; \B(L^2(\Real))\Big).$$
 For the rest of this section, we use $\widehat{f}^{{\mathbb H}_r}(n)$ and $\widehat{f}^{{\mathbb H}_r}(r,s)$ to denote $\widehat{f}^{{\mathbb H}_r}(\pi^n)$ and $\widehat{f}^{{\mathbb H}_r}(\chi_{r,s})$ respectively. The Fourier transform in (\ref{eq:FT}) forms an isometry, i.e. for $f\in A({\mathbb H}_r)\cap L^1({\mathbb H}_r)$, we have
	$$\|f\|_{A({\mathbb H}_r)}= \frac{1}{(2\pi)^2}\int_{\Real^2}|\widehat{f}^{{\mathbb H}_r}(r,s)|drds+\frac{1}{2\pi}\sum_{ {\mathbb Z}\backslash \{0\}}|n|\cdot\|\widehat{f}^{{\mathbb H}_r}(n)\|_1.$$
To avoid clutter, we use a reducible representation $\pi^0$ of ${\mathbb H}_r$ to encode all the one-dimen\-sio\-nal representations $\chi_{r,s}$. Indeed, we define $\pi^0$ to be  the left regular representation of ${\mathbb H}_r$ on $L^2(\Real^2)$, that is
\begin{equation}\label{eq-pi-zero}
\pi^0(y,z,x) := \lambda_{{\Real^2}}(y,x) \in \B(L^2(\Real)).
\end{equation}
It is easy to see that for $F\in L^2(\widehat{\Real^2})$ and $f\in L^1({\mathbb H}_r)$ we have
\begin{equation*}
\F^{\Real^2}\circ \pi^0(f)\circ  ({\F^{\Real^2}})^{-1}(F) (r,s)=\widehat{f}^{\mathbb{H}_r}(-r,-s)F(r,s),
\end{equation*}
i.e. $\F^{\Real^2}\circ \pi^0(f)\circ  ({\F^{\Real^2}})^{-1}$ is a multiplication operator, which is affiliated with the von Neumann algebra $L^\infty(\widehat{\Real^2}) \cong VN(\Real^2)$. Moreover
	$$\|\pi^0(f)\|_{A(\Real^2)}=\| \widehat{f}^{\mathbb{H}_r}|_{\widehat{\Real^2}} \|_{L^1(\widehat{\Real}^2)}.$$
This allows us to define a modified version of the Fourier transform as below, which is again an isometry.
\begin{eqnarray}
&&\F^{\mathbb{H}_r}: A({\mathbb{H}_r})\longrightarrow \frac{1}{(2\pi)^2}A(\Real^2)\oplus_1 L^1\Big(\z\backslash\{0\},\mu_d; S^1(L^2(\Real))\Big)\label{eq:MFT}\\
&&f\mapsto \widehat{f}^{{\mathbb H}_r}(0)\oplus \bigoplus_{n\in {\mathbb Z}\backslash \{0\}}\widehat{f}^{{\mathbb H}_r}(n),\nonumber
\end{eqnarray}
where we have used the convention $\widehat{f}^{{\mathbb H}_r}(0) := \pi^0(f)$. Note that $L^2(\Real)$ and $L^2(\Real^2)$ are equipped with their Lebesgue measures.

The corresponding Lie algebra for ${\mathbb H}_r$ is $\mathfrak{heis}$, the Heisenberg Lie algebra, with the exponential map
	$$\text{exp}: \mathfrak{heis} \to \mathbb{H}_r, \;  xX + yY + zZ \mapsto (y,\exp(iz+\frac{i}{2}xy),x).$$
We fix a complexification of $\mathbb{H}_r$ given by $(\Comp \times \Comp^*)\rtimes \Comp$ with the same group law, which we will denote by $(\mathbb{H}_r)_\Comp.$ Here we use the symbol $\Comp^* = \Comp\backslash \{0\}$. Note that $(\mathbb{H}_r)_\Comp \cong \Comp \times \Comp^* \times \Comp$, which is not simply connected. However, by considering the canonical covering maps $\mathbb{H} \to \mathbb{H}_r$ and $\mathbb{H}_\Comp \to (\mathbb{H}_r)_\Comp$ we can easily check that $ (\mathbb{H}_r)_\Comp$ together with the inclusion $\mathbb{H}_r \hookrightarrow (\mathbb{H}_r)_\Comp$ is the universal complexification. Moreover, we clearly have the following Cartan decomposition
	\begin{equation}\label{eq-Cartan-heis-reduced}
	(\mathbb{H}_r)_\Comp \cong \mathbb{H}_r \cdot \exp(i \, \mathfrak{heis}).
	\end{equation}

	\begin{prop}\label{prop-subgp-structure-Heisenberg-reduced}
	The proper closed Lie subgroups of $\Hee_r$ are $H_Y=\{(t, 0, 0):t\in\Ree\}$, $H_Z=\{(0, z, 0):z\in\tor \}$ and $H_{Y,Z}=\{(t, z, 0):t\in\Ree, z\in\tor\}$ up to automorphisms.
	\end{prop}
\begin{proof}
This is direct from the description of one or two dimensional subalgebras of $\mathfrak{heis}$ in the proof of Proposition \ref{prop-subgp-structure-Heisenberg}. We only need to observe that the resulting subgroups are closed in $\mathbb{H}_r$.
\end{proof}	

For $n \ne 0$ and $\xi\in C^\infty_c(\Real)$ we have that
	\begin{align}\label{eq-Lie-derivatives-reducedHei}
		\begin{cases}\partial \pi^n(X) \xi = -\xi',\\ (\partial \pi^n(Y)\xi)(t) = -int \xi(t),\\ \partial \pi^n(Z)\xi = in \xi.
		\end{cases}
	\end{align}
For $n=0$ and $\eta \in C^\infty_c(\Real^2)$ we have	
	\begin{align}
		\begin{cases}\partial \pi^0(X) \eta = -\partial_x\eta,\\ \partial \pi^0(Y)\eta = -\partial_y\eta,\\ \partial \pi^0(Z)\eta = 0.
		\end{cases}
	\end{align}
				
Entire vectors for $\pi^n$, $n\ne 0$ are the same as in the case of Heisenberg group; that is, $f\in L^2(\Real)$ is an entire vector for $\pi^n$, $0\ne n \in \z$ if and only if $f$ extends to an entire function on $\Comp$ and satisfies
	$$\sup_{|\text{\rm Im} z| <t}e^{t|z|}|f(z)| < \infty$$
for any $t>0$. For $n=0$ we recall that $\pi^0$ is the left regular representation on $\Real^2$, so that by \cite[Proposition 4.1]{Good69} and the argument in the beginning of Section 5 of \cite{Good69} we have $F\in L^2(\Real^2)$ is an entire vector for $\pi^0$ if and only if $e^{t(|x| + |y|)}\widehat{F}^{\Real^2}(x,y) \in L^2(\Real^2)$ for any $t>0$. Recall that Proposition \ref{prop-Paley-Weiner-super-exp-decay} characterizes the latter condition, from which
we can easily see that the functions $\varphi_k \otimes \varphi_l$, $k,l\ge 0$, are entire vectors for $\pi^0$, where $\varphi_k$ is the $k$-th Hermite function.

\subsection{Description of ${\rm Spec} A(\mathbb{H}_r,W)$}\label{sec-reduced-Heisenberg}

In this section, we only consider the weight $W$ extended from the subgroup $H=H_{Y,Z}\cong \Real\times \tor$ according to Proposition \ref{prop-subgp-structure-Heisenberg-reduced} and Theorem \ref{thm-automorphism-principle}.  More precisely, we fix a weight function $w: \widehat{H_{Y,Z}} \cong \Real \times \z \to (0,\infty)$ and we set $W_H = \widetilde{M}_w$ and $W = \iota(W_H)=(W(n))_{n\in \z}$.
Precisely speaking we have
	\begin{equation}\label{eq-Weights-reduced1}
	W(n)\xi(t) = w(nt,-n)\xi(t),\;\; 0\ne n \in \z
	\end{equation}
by Proposition \ref{prop-extended-weights-abelian} and
	\begin{equation}\label{eq-Weights-reduced2}
	(\F^{\Real^2}\circ W(0) \circ (\F^{\Real^2})^{-1}) F(b,a) = w(b, 0)F(b,a)
	\end{equation}
for appropriate $\xi \in L^2(\Real)$ and $F\in L^2(\Real^2)$.

\begin{rem}\label{rem-reduced-Heisenberg-central-weights}
The above shows that the weights extended from the subgroup $H_{Y,Z}\cong \Real\times \tor$ with the weight function $w: \widehat{H_{Y,Z}} \cong \Real \times \z \to (0,\infty)$ satisfying the condition that the value $w(t,n)$, $(t,n)\in \Real \times \z$, depends only on $n\in \z$ are all central weights. In this case we actually get a weight function $w|_\z$ on $\z$, which means that we could say that the weights extended from the 1-dimensional subgroup $H_Z = \{(0,z,0): z\in \tor\}$ of $\mathbb{H}_r$ are all central weights.

It turns out that this covers all the central weights on the dual of $\mathbb{H}_r$ up to ``equivalence''. More precisely, let $W$ be a central weight on the dual of $\mathbb{H}_r$, which means we have a measurable function $w: \widehat{\Real}^2 \sqcup \z \backslash \{0\} \to (0,\infty)$, where $\sqcup$ denotes the disjoint union of sets, satisfying $W = w|_{\widehat{\Real}^2} \oplus \bigoplus_{n\in \z \backslash \{0\}}w(n)I_{L^2(\Real)}$. Now we establish the following ``fusion rule'' for $\widehat{\mathbb{H}}_r$.
	\begin{itemize}
		\item $\pi^n \otimes \pi^m \cong \pi^{n+m}$, $n,m \in \z \backslash \{0\}$, $n\ne -m$;
		\item $\pi^n \otimes \pi^{-n} \cong \pi^0$, $n \in \z \backslash \{0\}$;
		\item $\pi^n \otimes \chi_{r,s} \cong \pi^n$, $n \in \z \backslash \{0\}$, $r,s \in \Real$;
		\item $\chi_{r,s} \otimes \chi_{r',s'} \cong \chi_{r+r',s+s'}$, $r,s,r',s' \in \Real$,
	\end{itemize}
where $\pi^0$ is the representation from \eqref{eq-pi-zero}. Indeed, the first quasi-equivalence and the third equivalence are directly from the Stone-von Neumann theorem (\cite[(6.49)]{Fol-book}) after combining the canonical quotient map from the usual Heisenberg group. The fourth identity is from the usual character formula on $\Real^2$. The second unitary equivalence requires some elaboration as follows. We first consider the following unitary map $U: L^2(\Real^2) \to L^2(\Real^2),\; \eta \mapsto U\eta$, where $U\eta(s,t) = \eta(s, t+s)$, $s,t\in \Real$ for $\eta \in L^2(\Real^2)$. Then it is straightforward to check the   following relation.
	$$U \Big(\pi^n \otimes \pi^{-n} (y,z,x)\Big) U^*\eta(s,t) = e^{inty}\eta(s-x, t),\;\; s,t \in \Real, n \in \z \backslash \{0\}.$$
Thus, by taking conjugation with respect to the $\Real$-Fourier transform on the second variable we get the unitary equivalence we wanted.

The above fusion rule tells us the following corresponding formula for the co-multi\-pli\-ca\-tion, namely for $A \in L^\infty\Big(\widehat{\Real^2},\frac{drds}{(2\pi)^2}\Big)\oplus_{\infty} L^\infty\Big(\z\backslash\{0\},\mu_d; \B(L^2(\Real))\Big)$ we have
	\begin{itemize}
		\item $\Gamma(A)(n, m) \cong A(n+m)$, $n,m \in \z \backslash \{0\}$, $n\ne -m$;
		\item $\displaystyle \Gamma(A)(n, -n) \cong \int^\oplus_{\Real^2}A(r,s)\,drds$, $n \in \z \backslash \{0\}$;
		\item $\Gamma(A)(n, (r,s)) \cong A(n)$, $n \in \z \backslash \{0\}$, $r,s \in \Real$;
		\item $\Gamma(A)((r,s), (r',s')) \cong A(r+r',s+s')$, $r,s,r',s' \in \Real$.
	\end{itemize}
Now the condition $\Gamma(W)(W^{-1}\otimes W^{-1})$ being a contraction becomes
	\begin{itemize}
		\item $w(n+m) \le w(n)w(m)$, $n,m \in \z \backslash \{0\}$, $n\ne -m$;
		\item $\sup_{r,s \in \Real}w(r,s) \le w(n)w(-n)$, $n \in \z \backslash \{0\}$;
		\item $w(n) \le w(r,s)w(n)$, $n \in \z \backslash \{0\}$, $r,s \in \Real$;
		\item $w(r+r',s+s')\le w(r,s)w(r',s')$, $r,s,r',s' \in \Real$.
	\end{itemize}
If we extend $w$ to $w: \widehat{\Real}^2 \sqcup \z \to (0,\infty)$ by assigning $w(0) := \sup_{r,s \in \Real}w(r,s)$, then we can see that the extended weight $w$ consists of two independent parts $w|_\z$ and $w|_{\widehat{\Real}^2}$, both sub-multiplicative on each domain and satisfying the only additional condition $1\le w|_{\widehat{\Real}^2}(r,s) \le w|_\z(0)$. The last condition means that $w|_{\widehat{\Real}^2}$ is a bounded and bounded below weight function on $\widehat{\Real}^2$, so that it is equivalent to the constant 1 weight function. This justifies the claim that all the central weights on the dual of $\mathbb{H}_r$ are exactly the weights extended from the 1-dimensional subgroup $H_Z = \{(0,z,0): z\in \tor\}$ up to ``equivalence''.
\end{rem}

In general we can apply the same strategy as in the case of the Heisenberg group $\mathbb{H}$ with less technicalities. By Remark \ref{rem:weightfunc} again, we can assume, without loss of generality,  that our weight function $w = (w_n)_{n\in \z}$ is consisting of locally integrable $w_n$, $n\in \z$. Moreover, we assume for technical reasons that all of their weak derivatives are at most exponentially growing, i.e. for any $k\ge 0$, there are constants $C_n, D_n>0$, $n\in \z$ such that
	\begin{equation}\label{eq-derivatives-exp-growth-reduced}
		|(w_n)^{(k)}(x)| \le C_n e^{D_n|x|}
	\end{equation}
for a.e. $x\in \Real$ and for all $n\in \z$. We will fix the symbols $w$ and $W$ throughout this section.

\subsubsection{A dense subalgebra and dense subspaces of $A(\mathbb{H}_r,W)$}

We define dense subalgebras $\A$ and $\B$ and a dense subspace $\mathcal{D}$ of $A(\mathbb{H}_r,W)$. By analogy with the case of $\mathbb{H}$ it is natural to begin with the test function space on $\Real \times \z \times \Real$, which we denote by $C^\infty_c(\Real \times \z \times \Real)$. This space is given by
	$$\Big\{f = (f_n)_{n\in \z} : f_n \equiv 0 \;\text{except for finitely many}\; n\in \z,\; f_n \in C^\infty_c(\Real^2)\Big\}.$$
Here, we use the convention $f_n(y, x) = f(y, n, x)$ for $n\in \z$, $(y,x)\in \Real^2$.

\begin{defn}
We define
	$$\A := \F^{\Real \times \z \times \Real}(C^\infty_c(\Real \times \z \times \Real))\subseteq C^\infty(\mathbb{H}_r)\;\; \text{and}\;\;\B := \F^{\Real \times \z \times \Real}(\B_0)\subseteq C^\infty(\mathbb{H}_r)$$
where
	\begin{align*}
	\B_0 := \Big\{ f = (f_n)_{n\in \z} \subseteq L^1_{\rm loc}(\Real^2) : & \; f_n \equiv 0 \;\text{except for finitely many}\; n\in \z,\\ & \; e^{t(|y|+|x|)}(\partial^\alpha f_n)(y,x) \in L^2(\Real^2),\\ & \;\forall t>0,\, \forall \text{multi-index}\; \alpha \Big\}.
	\end{align*}
 We endow a natural locally convex topology on $\B_0$ given by the family of semi-norms $\{\psi^\alpha_{t,n} : t>0,  n\in \z, \, \text{multi-index}\; \alpha\}$, where
	\begin{equation}\label{eq-psi-reduced}
	\psi^\alpha_{t,n}(f) = \|e^{t(|x|+|y|)}(\partial^\alpha f_n)(y, x)\|_{L^2(\Real^2)}.
	\end{equation}
	Finally, we define the space $\mathcal{D}\subseteq L^1\Big(\widehat{\Real^2},\frac{drds}{(2\pi)^2}\Big)\oplus_1 L^1\Big(\z\backslash\{0\},\mu_d; S^1(L^2(\Real))\Big)$ by	
		\begin{align*}
		\D := {\rm span} \Big\{F \oplus_1 P^l_{mn} : &\; m,n \in \z^{\geq 0},\;  F\in L^1(\widehat{\Real^2}),\\ &\; e^{t(|y|+|x|)}(\partial^\alpha (\widehat{F}^{\Real^2}))(y,x) \in L^2(\Real^2),\\ & \;\forall t>0,\, \forall \text{multi-index}\; \alpha\Big\},
		\end{align*}
where $P^l_{mn}\in L^1\Big(\z\backslash\{0\},\mu_d; S^1(L^2(\Real))\Big)$ attains $P_{mn}$ at the $l$-th coordinate, and is $0$ everywhere else. Recall that $P_{mn}$ is the rank 1 operator on $B(L^2(\Real))$ as in Definition \ref{def-spaces-Heisenberg}.
We will fix the symbols $\A$, $\B$, $\B_0$ and $\D$ throughout this section.
\end{defn}

The basic properties of the spaces $\A$, $\B$, $\B_0$ and $\D$ are obtained in a similar way as in the case of $\mathbb{H}$, so that we omit the proof.

\begin{prop}\label{prop-properties-subspaces-reduced} The spaces $\A$, $\B$ and $\B_0$ satisfy the following.
	\begin{enumerate}
		\item The spaces $\A$ and $\B$ are algebras with respect to the pointwise multiplication.
		\item The space $\B_0$ is a Fr\'{e}chet space.
		\item The inclusion $C^\infty_c(\Real \times \z \times \Real) \subseteq \B_0$ is continuous with dense range.
	\end{enumerate}
\end{prop}

More properties are coming in the same order as in the case of $\mathbb{H}$.

\begin{prop}\label{prop-absorb-weights-reduced}
For $f\in \B_0$ we have 	
 	$$W\F^{\mathbb{H}_r}(\widehat{f}^{\Real \times \z \times \Real}) = \F^{\mathbb{H}_r}(\widehat{g}^{\Real \times \z \times \Real}),\;\text{where}\; \begin{cases}g(t,n,s) = w(-t,-n)f(t,n,s), \; n\ne 0 \\  g(t,0,s) = \frac{1}{2\pi}w(-t,0)f(t,0,s),\, n=0.\end{cases}$$
Moreover, $g\in \B_0$ and, the map $\B_0 \to \B_0,\; f\mapsto g$ is continuous.
\end{prop}
\begin{proof}
Let $n\in \z\backslash\{0\}$ and $f\in \B_0$. Then we have
	\begin{align*}
	\F^{\mathbb{H}_r}(\widehat{f}^{\Real \times \z \times \Real})(n)\xi(t)
	& = \int_\tor \int_{\Real^2}\widehat{f}^{\Real \times \z \times \Real}(y,z,x) z^n e^{-inty} \xi(-x+t) dxdydz\\
	& = \sqrt{2\pi}\int_\Real \widehat{f}^{\Real}_3(-nt,n, t-x) \xi(x)dx,
	\end{align*}
where $\widehat{f}^{\Real}_3$ implies that we take Fourier transform with respect to the third variable only. Thus, $\F^{\mathbb{H}_r}(\widehat{f}^{\Real \times \z \times \Real})(n)$ is an integral operator with the kernel
	$$K_{f,n}(t,x) = \sqrt{2\pi}\widehat{f}^{\Real}_3(-nt, n, t-x).$$

For the case of $n=0$ we have
	\begin{align*}
	\F^{\mathbb{H}_r}(\widehat{f}^{\Real \times \z \times \Real})(0)\eta(s,t)
	& = \int_\tor \int_{\Real^2}\widehat{f}^{\Real \times \z \times \Real}(y,z,x) \eta(s-y, t-x) dxdydz\\
	& =  \int_{\Real^2}\widehat{f}^{\Real^2}_{1,3}(y,0,x) \eta(s-y, t-x) dxdy\\
	& = \int_{\Real^2} \widehat{f}^{\Real^2}_{1,3}(s-y, 0, t-x) \eta(y,x)dydx,
	\end{align*}
so that
	\begin{equation}\label{eq-n=0}
	\F^{\Real^2}\circ \F^{\mathbb{H}_r}(\widehat{f}^{\Real \times \z \times \Real})(0)\circ (\F^{\Real^2})^{-1} F(b,a) =2\pi f(-b, 0, -a)F(b,a)
	\end{equation}
for $F\in L^2(\Real^2)$. From this point on we can repeat the same argument as in the case of $\mathbb{H}$ for the conclusion.
\end{proof}

We remark that the space $\mc B$ can be used as the subspace $\mc S$ in \ref{ssec:separable-type I} thanks to the above proposition.

\begin{prop}\label{prop-embeddings-reduced-Heisenberg} The space $\A$, $\B$ and $\D$ satisfy the following.
	\begin{enumerate}
		\item The space $\B$ is continuously embedded in $A(\mathbb{H}_r, W)$.
		\item The space $(\F^{{\mathbb H}_r})^{-1}\D$ is a subspace of $\B$ which is dense in $A(\mathbb{H}_r, W)$.
		\item The algebra $\A$ is dense in $A(\mathbb{H}_r, W)$.
	\end{enumerate}
\end{prop}
\begin{proof}
(1) As $\partial \lambda(X)$, $\partial \lambda(Y)$ and $\partial \lambda(Z)$ are infinitesimal generators for one-parameter subgroups $\lambda(\exp(tX))$, $\lambda(\exp(tY))$, and $\lambda(\exp(tZ))$ respectively, one can easily verify the formulas
	$$\begin{cases}\partial \lambda(X) F(y,z,x) = D_z F(y,z,x)\cdot(-y) - \partial_x F(y,z,x)\\ \partial \lambda(Y) F(y,z,x) = -\partial_y F(y,z,x) \\ \partial \lambda(Z) F(y,z,x) = - D_z F(y,z,x)\end{cases}$$
for any $F\in \mathcal{S}(\Real \times \tor \times \Real)$ as a function on $\mathbb{H}_r$, which is straightforward to check. Here, the operator $D_z$ is the derivative with respect to the complex variable $z$ given by $D_z(z^n) =\textcolor{blue}{i} nz^n$, $n\in \z$. In other words, $D_z$ is nothing but the differentiation $\frac{\partial}{\partial \theta}$ for the functions of $z=e^{i\theta}$. Moreover, we have $(\F^\z)^{-1} \circ  D_z \circ \F^\z$ is the multiplication operator $(a_n)_{n\in \z} \mapsto (-i na_n)_{n\in \z}$. From this point we can follow the proof of (1) of Proposition \ref{prop-embeddings-Heisenberg}.

\vspace{0.3cm}

(2) From the decomposition
	$$\F^{{\mathbb H}_r}(A(\mathbb{H}_r)) = L^1\Big(\widehat{\Real^2},\frac{drds}{(2\pi)^2}\Big)\oplus_1 L^1\Big(\z\backslash\{0\},\mu_d; S^1(L^2(\Real))\Big)$$
it is clear that $(\F^{{\mathbb H}_r})^{-1}\mathcal{D}$ is dense in $A(\mathbb{H}_r)$. We first claim that $(\F^{{\mathbb H}_r})^{-1}\mathcal{D} \subseteq \B$. For $l\in \z\backslash\{0\}$ we focus on the operator $0\oplus P^l_{mn}$. We have that $P^l_{mn}$ is an integral operator with the kernel
	$$K(t,x) = \varphi_m(t) \varphi_n(x).$$
We would like to find a function $f\in \B_0$ such that $K_{f,m} = \delta_{l,m}K$, $m\in \z\backslash\{0\}$, which will imply that $0\oplus P^l_{mn}=\F^{\mathbb{H}_r}(\widehat{f}^{\Real \times \z \times \Real})$. Indeed, it is straightforward to check that for $l\ne 0 \in \z$
	$$f(t,k,s) = \frac{1}{2\pi}\delta_{k,l}(-i)^n\varphi_n(s)e^{-i \frac{st}{l}}\varphi_m(-\frac{t}{l}), \; k\in \z, s,t\in \Real$$
satisfies $K = K_f$. For $l=0$ there is nothing to check. This explains the claim that $(\F^{\mathbb{H}_r})^{-1}\mathcal{D}\subseteq \B$.

Secondly, we show that $(\F^{\mathbb{H}_r})^{-1}\mathcal{D}$ is also dense in $A(\mathbb{H}_r,W)$. This is equivalent to the modified space
	\begin{align*}
\widetilde{\D} :=  & \;{\rm span} \Big\{(\F^{\Real^2}\circ W(0)\circ (\F^{\Real^2})^{-1})F \oplus \bigoplus_{l\in \z\backslash\{0\}} W(l)P^l_{mn}\\
& \;\;\;\;\;\;\;\;\;\; : m,n \in \z^{\geq 0},\; F\in L^1(\widehat{\Real^2}), \; e^{t(|y|+|x|)}(\partial^\alpha \widehat{F}^{\Real^2})(y,x) \in L^2(\Real^2),\\
& \;\;\;\;\;\;\;\;\;\;\;\;\; \forall t>0,\, \forall \text{multi-index}\; \alpha\Big\},
		\end{align*}
being dense in $\F^{\mathbb{H}_r}(A(\mathbb{H}_r))$. Indeed, for $l=0$, we know that the functions of the form $w(b, 0)F(b,a)$ are dense in $L^1(\Real^2)$. Moreover, for $l\in \z\backslash\{0\}$ we note that the linear span of $\{w(lt, -l)\varphi_m(t): m\in \z\}$ can be shown to be dense in $L^2(\Real)$ as in the proof of (2) of Proposition \ref{prop-embeddings-Heisenberg}.

\vspace{0.3cm}

(3) The same as in the case of $\mathbb{H}$.

\end{proof}

	\begin{prop}\label{prop-density-entire-reduced}
		We have the inclusion $(\F^{\mathbb{H}_r})^{-1}\mathcal{D} \subseteq \D^\infty_\Comp(\lambda)$.
	\end{prop}
\begin{proof}
If the support of $f\in L^2(\mathbb{H}_r)$ is finite and lies in the discrete part of $\widehat{\mathbb{H}_r}$, then
we only need to check the condition of $\widehat{f}^{\mathbb{H}_r}(n)$ being an entire vector  of $\pi^n$ for each $n\in \z\backslash\{0\}$, which is the case for every element in $\mathcal{D}$. If the support of $f\in L^2(\mathbb{H}_r)$ lies in the
continuous portion of  $\widehat{\mathbb{H}_r}$, then $f$ belongs to the first summand in the direct sum decomposition of $\mathcal{D}$, i.e. $f\in L^1(\widehat{\Real^2})$ with $e^{t(|y|+|x|)}(\partial^\alpha \widehat{F}^{\Real^2})(y,x) \in L^2(\Real^2)$. By Proposition \ref{prop-Paley-Weiner-super-exp-decay}, such an $f$ is an entire function for the left regular representation of $\Real^2$. Combining these facts finishes the proof.
\end{proof}

\subsubsection{Solving Cauchy functional equations on $\Real \times \z \times \Real$ and the final step}

	\begin{align*}
	({\rm CFE}_{\Real \times \z \times \Real}) & \;\; T\in C^\infty_c(\Real \times \z \times \Real)^* \;\text{satisfying}\\
	&\;\; \la T, f*g \ra = \la T, f\ra \la T, g\ra, \; f,g\in C^\infty_c(\Real \times \z \times \Real).
	\end{align*}

In the above, $*$ refers to the convolution on the group $\Real \times \z \times \Real$. We omit the proof of the following theorem, which is an obvious modification of the case of $({\rm CFE}_{\Real^n})$.
				
\begin{thm}\label{thm-Cauchy-Rn-reduced}
Let $T\in C^\infty_c(\Real \times \z \times \Real)^*$ be a solution of $({\rm CFE}_{\Real \times \z \times \Real})$, then there are uniquely determined $c_1, c_2\in \Comp$ and $c_3\in \Comp^*$ such that
	$$\la T, f\ra = \int_{\Real \times \Real} \sum_{n\in \z} f(t, n, s) e^{-i(c_1t + c_2 s)}c_3^{-n}dtds, \; f\in C^\infty_c(\Real \times \z \times \Real).$$
In other words, the distribution $T$ is actually a function of exponential type
	$$T(t,n,s) = e^{-i(c_1t + c_2 s)}c_3^{-n}.$$
\end{thm}

We continue to a realization of $\text{Spec}A(\mathbb{H}_r,W)$ in $(\mathbb{H}_r)_\Comp,$ whose proof is similar to the case of $\mathbb{H}$.

	\begin{prop}
	Every character $\varphi \in {\rm Spec}A(\mathbb{H}_r,W)$ is uniquely determined by a point $(y,z,x) \in (\mathbb{H}_r)_\Comp = (\Comp \times \Comp^*) \rtimes \Comp$, which is nothing but the evaluation at the point $(y,z,x)$ on $\B$ (and consequently on $\A$).
	\end{prop}

Here comes our final result.

\begin{thm}\label{thm-Heisenberg-Spec-reduced}
	Let $\mathfrak{h}$ be the Lie subalgebra corresponding to the subgroup $H = H_{Y,Z}$ of $\mathbb{H}_r$.
	Suppose $W$ is a weight on the dual of ${\mathbb H}_r$ which is extended from a weight {\color{blue} $W_H$} on $H_{Y,Z}$.
	Then we have
		$${\rm Spec}A(\mathbb{H}_r,W) \cong \Big\{ g\cdot \exp(iX'): g\in \mathbb{H}_r,\; X'\in \mathfrak{h},\; \exp(iX') \in {\rm Spec} A(H, W_H)\Big\}.$$
	\end{thm}
\begin{proof}
In general we have a similar but easier proof with a different pattern of the case $l=0$. Note that we need Theorem \ref{thm-unbdd-reduced-Heisenberg} below as a replacement for Theorem \ref{thm-unbdd-Lie-noncompact}.
\end{proof}

\begin{thm}\label{thm-unbdd-reduced-Heisenberg}
Suppose $W$ is a weight on the dual of ${\mathbb H}_r$ which is extended from a bounded below weight on $H_{Y,Z}$.
For any $X'\in \mathfrak{heis} \backslash \mathfrak{h}$ the operator $\exp (i\partial \lambda(X'))W^{-1}$ is unbounded, whenever it is densely defined.
\end{thm}
\begin{proof}
Since $\mathbb{H}_r$ is not simply connected, we cannot apply Theorem \ref{thm-unbdd-Lie-noncompact}. Instead, we can describe the maps $\theta$ and $\alpha$ precisely in this case. We first write $X' = aX + bY + cZ$ for some $a,b,c\in \Real$ with $a\ne 0$. Indeed, we have $E: H_{Y,Z} \times \Real \to \mathbb{H}_r$ given by $E((y,z,0), t) = (y,z,0) e^{tX'}$, from which we get
	$$E((y,z,0), t) = (y+bt, ze^{i(ct+\frac{1}{2}abt^2)}, at),\; y, t\in \Real, z\in \tor$$
with the inverse
	$$E^{-1}(y', z', x') = ((y'-\frac{b}{a}x', z'e^{-i(\frac{c}{a}x'+\frac{b}{2a}(x')^2)}, 0), \frac{x'}{a}), \; (y', z', x') \in \mathbb{H}_r.$$
Note that $E$ is a global diffeomorphism with $E$ and $E^{-1}$ being polynomial, so that the maps $\theta$ and $\alpha$ are also polynomial.
\end{proof}

\begin{rem}\label{rem-1D-subgp-reduced-Heis}
The statements of Theorem \ref{thm-unbdd-reduced-Heisenberg} and Theorem \ref{thm-Heisenberg-Spec-reduced} hold true for $H=H_Y$ and $H=H_Z$ in place of $H_{Y,Z}$, by applying similar arguments.   
\end{rem}

\begin{ex}
	For $X' = (y',z',0) \in \mathfrak{h}$ the condition $\exp(iX') \in {\rm Spec} A(H,W_H)$ is equivalent to the existence of a constant $C>0$ such that
		$$e^{ay'}e^{nz'} \le C w(a,b),\; \text{for almost all}\; (a,n)\in \Real \times \z$$
by an immediate analogue of \eqref{eq-YZ-subgroup}.
In particular, for specific choices of the weight function $w$, we have the following, by Theorem \ref{thm-Heisenberg-Spec-reduced}.
When $w(a,n) = \beta^{|a|}_1 \beta^{|n|}_2$, $(a,n)\in \Real \times \z$ for some $\beta_1, \beta_2 \ge 1$, we have
	\begin{align*}
		{\rm Spec} A(\mathbb{H}_r,W)
		& \cong \{g\cdot(iy',e^{iz'},0) \in (\mathbb{H}_r)_\Comp \cong \Comp \times \Comp^* \times \Comp\\
		& \;\;\;\;\;\; : g\in {\mathbb H}_r, \ y',z'\in {\mathbb R},  \; |y'| \le \log \beta_1, |z'| \le \log \beta_2\}.
	\end{align*}
Especially, when $\beta_2=1$ we have
	$${\rm Spec} A(\mathbb{H}_r,W) \cong \{(y,z,x) \in (\mathbb{H}_r)_\Comp \cong \Comp \times \Comp^* \times \Comp : |{\rm Im}y| \le \log \beta_1, |z|=1, |{\rm Im}x|=0\}.$$

\end{ex}

We end this section with the symmetry given by automorphisms on $\mathbb{H}_r$, whose proof is the same as in the case of $\mathbb{H}$.

	\begin{thm}\label{thm-reducedHeisenberg-Spec-automorphism}
	Let $\alpha: \mathbb{H}_r\to \mathbb{H}_r$ be a Lie group automorphism. Then we have
		$${\rm Spec}A(\mathbb{H}_r,\alpha_{VN}(W)) = \alpha_\Comp({\rm Spec}A(\mathbb{H}_r,W))\subseteq (\mathbb{H}_r)_\Comp.$$
	\end{thm}


\section{The Euclidean motion group on $\Real^2$}\label{chap-E(2)}

The Euclidean motion group on $\Real^2$ is
	$$E(2) = \left\{ (x,y,z) = \begin{bmatrix} z & x+iy \\ 0 & 1 \end{bmatrix}: x,y \in \Real, z\in \tor\right\} = \Real^2 \rtimes \tor$$
with the group law by the matrix multiplication or equivalently
	$$(x,y,z)\cdot (x',y',z') = ((x,y)^T + \rho(z) (x',y')^T, zz')$$
where $\rho(z) = \begin{bmatrix} {\rm Re}z & -{\rm Im}z \\ {\rm Im}z & {\rm Re}z \end{bmatrix}$ and $(x',y')^T$ refers to the transposed column vector. Here, we use the notations $(x,y,z) = (\begin{bmatrix} x\\ y \end{bmatrix}, z) = ((x,y)^T, z)$.

The unitary dual $\widehat{E(2)}$ of $E(2)$ can be described as follows. For any $r>0$ we define an irreducible unitary representation $\pi^r$ acting on $L^2[0,2\pi]$ by
	\begin{equation}\label{eq-Irr-E(2)}
	\pi^r(x,y,z)F(\theta) = e^{i r(x \cos \theta + y \sin \theta)} F(\theta - t),\; F\in L^2[0,2\pi], z=e^{it}.
	\end{equation}
Here, we are using the identification $[0,2\pi] \cong \tor$ via $t \mapsto e^{it}$.
	
The representations $(\pi^r)_{r>0}$ are all of the irreducible unitary representations appearing in the Plancherel picture, and we have
	$$\lambda \cong \int^\oplus_{\Real^+} \pi^r rdr.$$
This quasi-equivalence tells us that
	$$VN(E(2)) \cong L^\infty(\Real^+, rdr; \B(L^2[0,2\pi]))$$
and
	$$A(E(2)) \cong L^1(\Real^+, rdr; S^1(L^2[0,2\pi])).$$

For $f\in L^1(E(2))$ we define the group Fourier transform on $E(2)$ by
	$$\F^{E(2)}(f) = (\F^{E(2)}(f)(r))_{r>0} = (\widehat{f}^{E(2)}(r))_{r>0} \in L^\infty(\Real^+, rdr; \B(L^2[0,2\pi]))$$
and
	$$\widehat{f}^{E(2)}(r) = \int_{E(2)} f(g)\pi^r(g) dg = \int_\tor \int_{\Real^2} f(\alpha, z) \pi^r(\alpha, z) d\alpha dz.$$

The Lie algebra of $E(2)$ is $\mathfrak{e}(2) = \la S, X, Y : [S, X] =Y, [S, Y]=-X, [X, Y] = 0 \ra \cong \Real^3$ with the exponential map
	$$\exp: \mathfrak{e}(2) \to E(2)$$
given by
	\begin{align}\label{eq-exp-E(2)}
	\lefteqn{\exp(s S + xX + yY)}\\
	& = (\tfrac{1}{s}(\sin s)x + \tfrac{1}{s}(\cos s-1)y, \tfrac{1}{s}(1-\cos s)x + \tfrac{1}{s}(\sin s) y,  e^{is}), \nonumber
	\end{align}
where the value at $s=0$ is defined by taking the limit $s \to 0$, i.e.
	$$\exp(0 S + xX + yY)=(x,y,1).$$
We can see that $\mathfrak{e}(2)_\Comp \cong \Comp^3$ and we consider a complexification of $E(2)$ given by $\Comp^2 \rtimes \Comp^*$ with the same group law, which we denote by $E(2)_\Comp$. Note that we may use the identification $\{s \in \Comp: 0\le \text{Re}\, s < 2\pi\}  \cong \Comp^*$ via $s \mapsto e^{is}$. We can actually check that $E(2)_\Comp$ with the canonical inclusion $E(2) \hookrightarrow E(2)_\Comp$ is the universal complexification. See Remark \ref{rem-universal-comp-E(2)-cover} below. Moreover, we clearly have the following Cartan decomposition
	\begin{equation}\label{eq-Cartan-E(2)}
	E(2)_\Comp \cong E(2) \cdot \exp(i \, \mathfrak{e}(2)).
	\end{equation}

For $r>0$ we need to understand $\partial \pi^r(S)$, $\partial \pi^r(X)$ and $\partial \pi^r(Y)$ in a concrete way. Indeed, we can easily check  for $F\in C^\infty(\tor)$ that
	\begin{align*}
		\begin{cases}\partial \pi^r(S) F = - F',\\ (\partial \pi^r(X)F)(\theta) = ir \cos \theta \cdot F(\theta),\\ (\partial \pi^r(Y)F)(\theta) = ir \sin \theta \cdot F(\theta).
		\end{cases}
	\end{align*}
Identifying $L^2(\tor)$ with $\ell^2(\z)$ via the Plancherel transform, we get the following operators on $\ell^2(\z)$:
	\begin{align}\label{eq-Lie-derivatives-E(2)}
		\begin{cases}\partial \pi^r(S) e_n = -ine_{n},\\ \partial \pi^r(X)e_n = \frac{ir}{2} (e_{n-1} + e_{n+1}),\\ \partial \pi^r(Y) e_n =\frac{r}{2} (e_{n+1} - e_{n-1}),
		\end{cases}
	\end{align}
where $\{e_n: n\in \z\}$ is the canonical orthonormal basis of $\ell^2(\z)$. Moreover, it is also straightforward to check that
	\begin{equation}\label{eq-Laplacian-E(2)}
		\partial \pi^r(-\Delta)e_n = (n^2+r^2)e_n,\;\; n\in \z.
	\end{equation}

Finally, we record that any trigonometric polynomial is clearly an entire vector for $\pi^r$, $r>0$.

\subsection{Weights on the dual of $E(2)$}\label{sec-weights-E(2)}

\subsubsection{Weights from subgroups}\label{subsub-weights-E(2)-subgroup}
As in the previous cases, we first identify all closed Lie subgroups of $E(2)$.

	\begin{prop}\label{prop-subgp-structure-E(2)}
The proper closed Lie subgroups of $E(2)$ are $H_S = \{(0, 0, z): z\in \tor\} \cong \tor$, $H_Y=\{(0,y,1): y\in \Real\}\cong \Real$,   $H_{r}=\{(x,rx,1):x\in\Real\}\cong \Real$ for every $r\in{\mathbb R}^{\geq 0}$, and $H_{X,Y} =\{(x,y,1): x,y\in \Real\} \cong \Real^2$ up to automorphism.
	\end{prop}
\begin{proof}
We begin with the description of $\aut(\fe(2))$. By examining the Lie bracket relations of the basis $\{S,X,Y\}$ we can easily conclude that any automorphism $\alpha : \fe(2) \to \fe(2)$ is of the form
	$$\alp(S)=aS+bX+cY,\,\alp(X)= dX\text{ and }\alp(Y)=adY$$
for some $b,c,d\in \Real$ with $a=\pm1$ and $d\ne 0$. Now, we observe that  any one-dimensional subspace of $\fe(2)$ is of the form $\Real(S + v_1X +v_2 Y)$ for $v = (v_1,v_2)\in \Real^2$ or $\Real(u_1X +u_2 Y)$ for $u = (u_1,u_2)\in \Real^2$ with $|u|_2 = 1$.  The classification up to automorphisms are straightforward from the description of $\aut(\fe(2))$. Indeed, the first one gives us the one parameter subgroup $H_S$ up to automorphism, while the second one gives us a family of subgroups, up to automorphism,  indexed by $r\in{\mathbb R}^{\geq 0}$ defined as $H_{r}=\{(x,rx,1):x\in\Real\}$, or the subgroup $H_Y=\{(0,y,1): y\in \Real\}$.

It is an easy exercise in linear algebra to see that if $X_1$ and $X_2$ are linearly independent in $\fe(2)$ that either
$X_1,X_2\in\fe(2)'$, and hence we get $\langle X_1,X_2\rangle=\fe(2)'$ (derived ideal), or we find that
$\langle X_1,X_2\rangle=\fe(2)$. This gives us the result for the two dimensional subgroups.

\end{proof}

 By Proposition \ref{prop-subgp-structure-E(2)} and Theorem \ref{thm-automorphism-principle}, we only need to consider the weights extended from the subgroups $H_S\cong \tor$ and $H_{X,Y}\cong \Real^2$. For a weight function $w: \widehat{H_S} \cong \z \to (0,\infty)$ or $w: \widehat{H_{X,Y}} \cong \Real^2 \to (0,\infty)$ we consider the extended weight $W = \iota(\widetilde{M}_w) = (W(r))_{r>0}$, which is given as follows.
	
	\vspace{0.3cm}
(The case of $H_S$)
	\begin{equation}\label{eq-WeightsK}
	(\F^\tor \circ W(r) \circ (\F^\tor)^{-1})e_n = w(n)e_n,\;\; n\in \z.
	\end{equation}
In other words, $W(r)$ is a Fourier multiplier on $L^2(\tor)$ with respect to the symbol $w$, which is independent of the parameter $r$.

\vspace{0.3cm}
(The case of $H_{X,Y}$)
	$$W(r)F(\theta) = w(-r\cos\theta, -r\sin\theta)F(\theta),\;\; F\in C^\infty(\tor)$$
or equivalently
	\begin{equation}\label{eq-WeightsH}
		W(r)F(z) = w(-rz)F(z),\;\; z\in \tor,
	\end{equation}
with the identification $z=e^{i\theta}$. When $w$ is a radial function, the operator $W(r)$ is a multiple of identity operator with the constant depending on the parameter $r$.

\begin{rem}\label{rem-E(2)-central-weights}
The above shows that the weights extended from the 2-dimensional subgroup $H_{X,Y}$ of $E(2)$ using the radial weight function on the dual of $H_{X,Y}$ are central. Moreover, we can actually prove that all the central weights on the dual of $E(2)$ are of the above form. Indeed, the centrality of a weight $W$ on the dual of $E(2)$ forces us to begin with $W = (w(r)I_{\B(L^2(\tor))})_{r>0}$ for some measurable function $w: \Real_+ \to (0,\infty)$. Now we recall the following fusion rule of $E(2)$
	\begin{equation}\label{eq-fusion-E(2)}
	\pi^r \otimes \pi^s \cong \int^\oplus_{\abs{r-s} \le a \le r+s} \pi^a\; da,
	\end{equation}
where we have a quasi-equivalence. For the convenience of the reader we provide a proof for the above fusion rule below.	
We first note that the formula \eqref{eq-Irr-E(2)} for the irreducible representation $\pi^r$ can be rewritten as follows.
	$$\pi^r(x,y,z)F(z')=e^{i\langle w, rz'\rangle}F(\bar{z}z'),\; F\in L^2(\tor),\; z'\in \tor,$$
where $w = x+iy$ and $\langle w, z'\rangle = {\rm Re} w \cdot {\rm Re} z'+ {\rm Im} w \cdot {\rm Im} z'$, i.e.\ the ``dot" product when $\Cee$ is identified with $\Ree^2$. In the above representation we may replace the parameter $r>0$ with $a\in \Comp\backslash\{0\}$ to get an irreducible unitary representation $\pi^a$ and it is straightforward to check that for any $a, b\in \Comp\backslash\{0\}$ we have a unitary equivalence
\begin{equation}\label{eq:ueq}
\pi^a \cong \pi^b \quad\Leftrightarrow\quad |a|=|b|,
\end{equation}
where the intertwining map is given by $\lambda_\tor (z)$ for $z\in \tor$ such that $za=b$.
Now we set a unitary map $U : L^2(\Tee\times\Tee) \to L^2(\Tee\times\Tee)$ given by
	$$UG(z_1,z_2)=G(z_1,z_1z_2), \; G\in L^2(\Tee\times\Tee), \; z_1,z_2 \in \tor$$
so that we have $U^*G(z_1,z_2)=G(z_1,\bar{z}_1z_2)$.
Then, for $r,s>0$ we have
	\begin{align*}
	U\Big(\pi^r\otimes\pi^s(x,y,z)\Big)U^*G(z_1,z_2)
	& = \pi^r\otimes\pi^s(x,y,z)U^*G(z_1,z_1z_2) \\
	& = e^{i\langle w, rz_1 + sz_1z_2 \rangle)}U^*G(\bar{z}z_1,\bar{z}z_1z_2) \\
	& = e^{i\langle w, (r+sz_2)z_1\rangle }G(\bar{z}z_1,z_2).
\end{align*}
Since $L^2(\mathbb{T}^2)\cong L^2(\mathbb{T};L^2(\mathbb{T}))=\int^\oplus_{\mathbb{T}}L^2(\mathbb{T})_y\,dy$, where $L^2(\mathbb{T})_y$ refers to the copy of $L^2(\mathbb{T})$, we have
	$$U\Big(\pi^r\otimes\pi^s(x,y,z)\Big)U^* = \int_{\Tee}^{\oplus}\pi^{r+sy}\; dy$$
and then apply (\ref{eq:ueq}) to obtain the following quasi-equivalence .
\[
\pi^r\otimes\pi^s \cong \int_{\Tee}^{\oplus}\pi^{|r+sy|}\; dy
\cong \int_{[|r-s|,r+s]}^\oplus \pi^a\; da.
\]
Now we move to the consequences of \eqref{eq-fusion-E(2)}, which implies that for any $A = (A(r))_{r>0} \in L^\infty(\Real^+, rdr; \B(L^2[0,2\pi]))$ we have
	$$\Gamma(A)(r,s) \cong \int^\oplus_{\abs{r-s} \le a \le r+s} A(a)\; da.$$	
By applying functional calculus, we can readily check that the above quasi-equivalence and the condition $\Gamma(W)(W^{-1}\otimes W^{-1})$ being a contraction imply the following inequality.
	\begin{equation}\label{eq-radial-sub-multi}
	\sup_{\abs{r-s} \le a \le r+s, \; a>0} w(a) \le w(r)w(s),\;\; r,s>0.
	\end{equation}
If we put $r = s = a>0$, then we have $w(r) \le w(r)^2$, so that $w(r) \ge 1$ for any $r>0$. Now we assign $w(0) = 1$, then $\tilde{w}(re^{i\theta}) := w(r)$, $r\ge 0, \theta \in [0,2\pi]$, actually becomes a weight function on $\Real^2$. Moreover, it is rather straightforward to check that the condition \eqref{eq-radial-sub-multi} is equivalent to the sub-multiplicativity of the radial weight function $\tilde{w}$ on $\Real^2$.
\end{rem}

\subsubsection{Exponentially growing weights on the dual of $E(2)$ using Laplacian}\label{ssec:exp-weight-E(2)}

Given the examples of exponentially growing weights on the dual of compact Lie groups using Laplacian presented in \ref{def-poly-exp-weights-compact-Lie},
it is natural to expect to obtain exponential weights on the dual of non-compact Lie groups using Laplacian as in the compact case. However, we only understand the situations of $E(2)$ and its simply connected cover $\widetilde{E}(2)$ at the time of this writing.

We consider the case of $E(2)$ in this subsection. From \eqref{eq-Lie-derivatives-E(2)} we can easily see that for $r,s>0$ and $m,n \in \z$ we have
	\begin{align*}
		\lefteqn{\partial(\pi^r \times \pi^s)(S\otimes S + X\otimes X + Y\otimes Y)e_{m,n}}\\
	& = -mn e_{m,n} -\frac{rs}{4}(e_{m-1,n-1} + e_{m-1,n+1} + e_{m+1,n-1} + e_{m+1,n+1})\\
	& \;\;\;\; + \frac{rs}{4}(e_{m-1,n-1} - e_{m-1,n+1} - e_{m+1,n-1} + e_{m+1,n+1})\\
	& = -mn e_{m,n} -\frac{rs}{2}(e_{m-1,n+1} + e_{m+1,n-1}),
	\end{align*}
where $e_{m,n} = e_m \otimes e_n \in \ell^2(\z \times \z)$. Here, the symbol $\pi \times \rho$ for representations $\pi:G \to B(\Hi_\pi)$ and $\rho : H\to B(\Hi_\rho)$ refers to the direct product representation on $G\times H$ given by
	$$(\pi \times \rho)(g,h) := \pi(g) \otimes \rho(h) \in B(\Hi_\pi \otimes \Hi_\rho),\;g,\in G, h\in H.$$
Now we consider the operator $\Gamma(\partial\lambda(\Delta)) = (\Gamma(\partial\lambda(\Delta))(r,s))_{r,s>0}$. Since we have
	\begin{align*}
	\Gamma(\partial\lambda(\Delta))
	& = I \otimes \partial\lambda(\Delta) + \partial\lambda(\Delta) \otimes I + 2\big(\partial\lambda(S)\otimes \partial\lambda(S)\\
	&\;\;\;\; + \partial\lambda(X)\otimes \partial\lambda(X) + \partial\lambda(Y)\otimes \partial\lambda(Y)\big),
	\end{align*}
we get from \eqref{eq-Lie-derivatives-E(2)} and \eqref{eq-Laplacian-E(2)} that
	\begin{align*}
		\Gamma(\partial\lambda(-\Delta))(r,s)e_{m,n}
		& = ((m+n)^2 +r^2+s^2) e_{m,n} + rs(e_{m-1, n+1} + e_{m+1, n-1})\\
		& = A_{r,s}(e_{m,n}) + B_{r,s}(e_{m,n})
	\end{align*}
for $r,s>0$ and $m,n \in \z$. This decomposition $\Gamma(\partial\lambda(-\Delta))(r,s) = A_{r,s} + B_{r,s}$ consists of a positive multiplication operator $A_{r,s}$ and a bounded self-adjoint operator $B_{r,s}$ with $B_{r,s} \le 2rs I$. Moreover, we can easily see that $A_{r,s}$ and $B_{r,s}$ are strongly commuting. We define the operator
	$$C(r,s) := \sqrt{A_{r,s}+2rsI}$$
or equivalently $C(r,s)$ is given by	
	$$C(r,s)e_{m,n} = \sqrt{(m+n)^2 + (r+s)^2} e_{m,n}.$$
Note that $C(r,s)$ is strongly commuting with $\Gamma(\partial\lambda(-\Delta))(r,s)$, so that functional calculus leads us to the conclusion that $\Gamma(\exp(t\sqrt{\partial\lambda(-\Delta)})) \exp(-tC)$ is a contraction, where $C = (C(r,s))_{r,s>0}$. Now we compare $C$ and $I\otimes \sqrt{\partial\lambda(-\Delta)} + \sqrt{\partial\lambda(-\Delta)} \otimes I$, where the latter is given by
	$$(I\otimes \sqrt{\partial\lambda(-\Delta)} + \sqrt{\partial\lambda(-\Delta)} \otimes I)(r,s)(e_{m,n}) = (\sqrt{m^2 + r^2} + \sqrt{n^2 + s^2}) e_{m,n}.$$
Since both of the operators are multiplication operators, they are strongly commuting. Moreover, we can easily check that
	$$\sqrt{(m+n)^2 + (r+s)^2} \le \sqrt{m^2 + r^2} + \sqrt{n^2 + s^2}$$
so that the operator
	$$\exp(tC) (\exp(-t(I\otimes \sqrt{\partial\lambda(-\Delta)} + \sqrt{\partial\lambda(-\Delta)} \otimes I)) $$
is a contraction again by functional calculus. By composition we can conclude that
	\begin{align*}
	\lefteqn{\Gamma(\exp(t\sqrt{\partial\lambda(-\Delta)}))(\exp(-t\sqrt{\partial\lambda(-\Delta)}) \otimes \exp(-t\sqrt{\partial\lambda(-\Delta)}))=}\\
	& \Gamma(\exp(t\sqrt{\partial\lambda(-\Delta)})) \exp(-tC) \cdot \exp(tC) \exp(-t (I\otimes \sqrt{\partial\lambda(-\Delta)} + \sqrt{\partial\lambda(-\Delta)} \otimes I))
	\end{align*}
is a contraction. One might be worried about the domain problem for the composition, but for the latter three operators, $\exp(-tC)$, $ \exp(tC)$ and $\exp(-t(I\otimes \sqrt{-\Delta} + \sqrt{-\Delta} \otimes I))$, the space of finitely supported sequences on $\z^2$ plays the role of common invariant dense subspace in $\ell^2(\z^2)$ for each parameter $(r,s)$. Moreover, we can easily check that $e_{m,n}$ is an analytic vector of $\Gamma(\sqrt{\partial\lambda(-\Delta)})(r,s)$, and hence in the domain of $\Gamma(\exp(t\sqrt{\partial\lambda(-\Delta)}))(r,s)$. Indeed, it is straightforward to check that for $T = \Gamma(\partial\lambda(-\Delta))(r,s)$ we have
	$$\|T^k e_{m,n}\| \le (C_{m,n})^{2k}(2k)!$$
for some constant $C_{m,n}>0$. (See \cite{Sch, Sch2} for the details.)

Note that the 2-cocycle condition for $\exp(t\sqrt{\partial\lambda(-\Delta)})$ is automatic since it is boundedly invertible. Thus, we have proved that the operator $\exp(t\sqrt{\partial\lambda(-\Delta)})$, $t>0$ is a weight on the dual of $E(2).$

\begin{defn}\label{def-exp-weight-E(2)}
The weight $\exp(t\sqrt{\partial\lambda(-\Delta)})$ is called the exponential weight on the dual of $E(2)$ of order $t>0$.
\end{defn}

\subsection{Description of the spectrum of the Euclidean motion group $E(2)$}\label{sec-E(2)}

In this section, we present a full characterization of the spectrum of Beurling-Fourier algebras of $E(2)$ associated with  two important types of weights: the weights extended from the subgroups $H = H_S$ or $H_{X,Y}$, and  the exponential weights coming from Laplacian of $E(2)$. We start this section by introducing  appropriate dense subalgebras and subspaces of the Beurling-Fourier algebra, which we will use in analysis of both cases.

\subsubsection{A dense subalgebra of $A(E(2), W)$ and its companions}

We will follow the same philosophy as in  the cases of $\mathbb{H}$ and $\mathbb{H}_r$, namely finding a dense subalgebra of $A(E(2), W)$ which will substitute the algebra Trig$(G)$ in the case of compact groups. As before this subalgebra alone cannot finish the job till the end, so we need two more companion spaces as well.

\begin{defn}\label{def-companion-spaces-E(2)}
We define the space
	$$\A_0 := \Big\{ f\in C^\infty_c(\Real^2) : f_n \equiv 0, |n| > N \; \text{for some $N\in \n$} \Big\}$$
where $f_n$, $n\in \z$, is the function given by
	$$f_n(r) := \int_\tor f(re^{i\theta})e^{-in\theta}d\theta, \;\; r\ge 0$$
and
	$$\A_{00} := \Big\{f\in \A_0 : f_n \in C^\infty_c(0,\infty)\; \text{for all $n\in \n$}\Big\}.$$
We also define
	$$\B_0 := \Big\{f\in C^\infty_c(\Real^2): \|\partial^\alpha f\|_\rho<\infty, \forall \rho>0, \forall \,\text{multi-index}\; \alpha \Big\}$$
where $\|\cdot \|_\rho$, $\rho>0$, is the norm given by
	$$\|f\|_\rho := \sum_{n\in \z} \rho^{|n|}\big(\int_{\Real_+}\rho^r |f_n(r)|rdr\big).$$	
Finally, we define $\A:= \F^{\Real^2 \times \z}(\A_0 \otimes c_{00}(\z))$ and $\B:= \F^{\Real^2 \times \z}(\B_0 \otimes c_{00}(\z))$, which are the images of the algebraic tensor product $\A_0 \otimes c_{00}(\z)$ and $\B_0 \otimes c_{00}(\z)$, respectively, under the Fourier transform on $\Real^2 \times \z$. We will fix the symbols $\A$, $\B$, $\A_0$, $\B_0$, $\A_{00}$ throughout this section.
\end{defn}

\begin{rem}\label{rem-spaces-for-E(2)}
	\begin{enumerate}
	\item We clearly have the inclusion $\A_{00}\subseteq \A_0 \subseteq \B_0$.
	
	\item The space $\A_{00}$ is nothing but the algebraic tensor product $c_{00}(\z)\otimes C^\infty_c(0,\infty)$, which is equipped with a canonical locally convex topology as follows. We say that $f(m) \to f \;\; \text{as $m\to \infty$ in}\;\; \A_{00}$
if there is a finite set $I\subseteq \z$ with $\text{\rm supp}f(m) := \{n\in \z : f(m)_n \neq 0\} \subseteq I$, $\forall m \ge 1$ and
	$$f(m)_n \to f_n \;\; \text{in $C^\infty_c(0,\infty)$ as}\;\; m\to \infty, \;\forall n\in \z$$
where 	we consider the usual locally convex topology on $C^\infty_c(0,\infty)$. The topology on $\A_0$ is defined in an identical manner, and the topology on $\A$ is induced from that of $\A_0$.

	\item The space $\B_0$ is equipped with a canonical locally convex topology as follows. We say that $f(m) \to f \;\; \text{as $m\to \infty$ in}\;\; \B_0$
if there is a $K$-ball, $B_K \subseteq \Real^2$ with $\text{\rm supp}f(m) \subseteq B_K$, $\forall m \ge 1$ and
	$$\|\partial^\alpha(f(m) - f)\|_\rho \to 0\;\; \text{as}\;\; m\to \infty, \;\forall \rho>0, \;\forall \text{multi-index $\alpha$}.$$
	The topology on $\B$ is induced from that of $\B_0$.
	\end{enumerate}	
\end{rem}

Before we proceed to important properties of the above spaces $\A$, $\B$, $\A_0$, $\B_0$ and $\A_{00}$ we need some preparations. We begin with a detailed understanding of group Fourier transforms on $E(2)$, which is actually an integral operator with precise description of kernel function at each parameter.

\begin{prop}
For $h\in C^\infty_c(\Real^2)$ and $g\in {\rm Trig}(\tor) = \F^\z(c_{00}(\z))$ we have
	\begin{equation}\label{eq-Fourier-E(2)}
	\F^{E(2)}(\widehat{h}^{\Real^2} \otimes g)(r)F(\theta) = 2\pi \int_\tor h(re^{i\theta})g(\theta-t)F(t)dt,\;\; F\in L^2(\tor).
	\end{equation}
In other words, the operator $\F^{E(2)}(\widehat{h}^{\Real^2} \otimes g)(r)$, $r\ge 0$ is the integral operator on $L^2(\tor)$ with the kernel $K(\theta,t) = 2\pi h(re^{i\theta})g(\theta-t)$.
\end{prop}
\begin{proof}
This can be obtained by a straightforward calculation.
\end{proof}

We also need a trace class norm estimate of integral operators acting on finite intervals. Basically, we get trace class operators when the kernels are smooth enough.

\begin{lem}\label{lem-integralOp} (\cite[p.120-121]{GK})\\
Let $A$ be the integral operator acting on $L^2([a,b])$ with the kernel function $K$, i.e.
	$$Af(t) = \int^b_a K(t,s) f(s)ds.$$
Then there is a universal constant $C>0$ such that
	$$\|A\|_{S^1(L^2[a,b])} \le C(\|K\|_{L^2([a,b]^2)} + \|\partial_s K\|_{L^2([a,b]^2)}).$$
\end{lem}

We begin with the properties of the spaces $\A_{00}$, $\A_0$ and $\B_0$ defined on $\Real^2 \times \z$.

\begin{prop}\label{prop-space-A_0-B_0} The spaces $\A_0$, $\A_{00}$ and $\B_0$ satisfy the following.
	\begin{enumerate}
		\item The space $\A_0$ is an algebra with respect to convolution on $\Real^2$.
		\item The space $\A_{00}$ is continuously and densely embedded in $\B_0$.
	\end{enumerate}
\end{prop}
\begin{proof}

(1) It is clear that every element of $\A_0$ is a linear combination of functions of the form
 $f(re^{i\theta}) = g(r)e^{in\theta}$, $n \in \z$ with $g\in C_c[0,\infty)$. For such $f$ we recall the following well-known formula.
	\begin{equation}\label{eq-polar-separation}
	\widehat{f}^{\Real^2}(Re^{i\psi}) = \frac{e^{in\psi}}{2\pi}\int_{\Real_+}\Big(\int_\tor e^{in\theta}e^{-iRr\cos\theta}d\theta\Big)g(r)rdr = G(R)e^{in\psi}.
	\end{equation}
This implies that $\widehat{f}^{\Real^2}$ is of the same form, namely the variables in polar coordinates are separated.
Thus, we can conclude that for $f_1, f_2 \in \A_0$, we know that the function $\widehat{f_1*f_2}^{\Real^2} = 2\pi \widehat{f_1}^{\Real^2}\cdot \widehat{f_2}^{\Real^2}$ is a finite linear combination of the functions whose variables in polar coordinates are separated. Recall that $C^\infty_c(\Real^2)$ is closed under $\Real^2$-convolution, so that we can now conclude that $\A_0$ is also closed under $\Real^2$-convolution.

\vspace{0.3cm}

(2) Let $f\in \B_0$ with supp$f\subseteq B_K$. We fix $\rho>0$ and a multi-index $\alpha$. We first pick $N$ such that
	$$\sum_{|n|>N} \rho^{|n|} \int_{\Real_+}\rho^r |(\partial^\alpha f)_n(r)|rdr <\eps.$$
Next we pick $\delta>0$ such that
	$$\sum_{|n|\le N} \rho^{|n|} \int^\delta_0 \rho^r|(\partial^\alpha f)_n(r)|rdr <\eps.$$
We record here a useful formula for $f\in C^\infty_c(\Real^2)$ and $r\in (0,\infty)$, which comes directly from the identity $\displaystyle \partial_x = \cos \theta \cdot \partial_r - \frac{\sin \theta}{r}\partial_\theta,\; \partial_y = \sin \theta \cdot \partial_r + \frac{\cos \theta}{r}\partial_\theta$, which is true on $\Real^2\backslash \{(0,0)\}$.
	\begin{align}\label{eq-partial-freq}
	(\partial_x f)_n(r) & = \tfrac{1}{2}(f'_{n-1}(r) - \frac{f_{n-1}(r)}{r}(n-1) + f'_{n+1}(r) + \frac{f_{n+1}(r)}{r}(n+1)),\\
	(\partial_y f)_n(r) & = \frac{i}{2}(-f'_{n-1}(r) + \frac{f_{n-1}(r)}{r}(n-1) + f'_{n+1}(r) + \frac{f_{n+1}(r)}{r}(n+1)) \nonumber
	\end{align}
for $r>0$. By applying the above formula repeatedly we can see that $(\partial^\alpha f)_n(r)$, $r>0$ and $|n|\leq N$, is a linear combinations of the functions
	$$\frac{f^{(k)}_l(r)}{r^m},\; 0\le k\le |\alpha|,\; |l| \le N+|\alpha|,\; 0\le m\le |\alpha|$$
which is a finite collection. Finally, we choose $\tilde{f}_n \in C^\infty_c(0,\infty)$, $|n|\le N$ such that supp$\tilde{f}_n \subseteq [\delta, K]$ and
	$$\sum_{|n|\le N}\rho^{|n|}\int^K_\delta \rho^r|(\partial^\alpha \tilde{f})_n(r) - (\partial^\alpha f)_n(r)|rdr <\eps$$
where $\tilde{f}(re^{i\theta}) = \sum_{|n|\le N} \tilde{f}_n(r)e^{in\theta}$ in $\A_{00}$ and $\|\partial^\alpha \tilde{f} - \partial^\alpha f\|_\rho <3\eps$. This explain the density we wanted, and the continuity of the embedding is also immediate.
\end{proof}

Now we move to the spaces $\A$ and $\B$ defined on $\widehat{\Real^2 \times \z}\cong \Real^2 \times \tor$.

\begin{prop}\label{prop-density-E(2)} The spaces $\A$ and $\B$ satisfy the following.
	\begin{enumerate}
	\item The space $\B$ is continuously and densely embedded in $A(E(2),W)$.
	
	\item The space $\A$ is an algebra with respect to pointwise multiplication, i.e. it is a subalgebra of $A(E(2))$.
	\end{enumerate}
\end{prop}
\begin{proof}
(1) (When $W$ is extended from the subgroup $H_{X,Y}$)
Let $w: \widehat{H_{X,Y}}\cong \Real^2 \to (0,\infty)$ be the associated weight function. Consider an element $\widehat{h}^{\Real^2} \otimes g$ of $\B$ with $h\in \B_0$ and $g\in {\rm Trig}(\tor)$.
Combining \eqref{eq-WeightsH} and \eqref{eq-Fourier-E(2)} we get
	$$W(r)\F^{E(2)}(\widehat{h}^{\Real^2} \otimes g)(r)F(\theta) = \int_\tor w(-re^{i\theta})h(re^{i\theta})g(\theta-t)F(t)dt$$
which is an integral operator with the kernel
	$K(\theta, t) = w(-re^{i\theta})h(re^{i\theta})g(\theta-t)$. By Lemma \ref{lem-integralOp} we have
		$$\|W(r)\F^{E(2)}(\widehat{h}^{\Real^2} \otimes g)(r)\|_{S^1(L^2(\tor))} \le C(\|K\|_{L^2(\tor^2)} + \|\partial_t K\|_{L^2(\tor^2)}).$$
By translation invariance we have
	$$\|K\|_{L^2(\tor^2)} = \|g\|_{L^2(\tor)} \cdot \big(\int_\tor |w(-re^{i\theta})h(re^{i\theta})|^2 d\theta \big)^{1/2}.$$
We have a similar expression for $\|\partial_t K\|_2$, so that we get
	\begin{align}\label{eq-ineq-R2}
	\lefteqn{\|\widehat{h}^{\Real^2} \otimes g\|_{A(E(2), W)}}\\
	& = \int_{\Real^+} \|W(r)\F^{E(2)}(\widehat{h}^{\Real^2} \otimes g)(r)\|_{S^1(L^2(\tor))} rdr\nonumber \\ & \le
	C (\|g\|_{L^2(\tor)} + \|g'\|_{L^2(\tor)})\int_{\Real^+}\Big(\int_\tor |w(-re^{i\theta})h(re^{i\theta})|^2d\theta\Big)^{1/2}rdr. \nonumber
	\end{align}
The sub-multiplicativity of $w$ implies that $|w(-re^{i\theta})| \le \rho^r$, $r>0$ for some $\rho>0$ by Proposition \ref{prop-at-most-exp}, so that we have
	\begin{align}\label{eq-estimate1}
	\Big(\int_\tor |w(-re^{i\theta})|^2|h(re^{i\theta})|^2 d\theta\Big)^{\tfrac{1}{2}}
	& \le \rho^r (\int_\tor |h(re^{i\theta})|^2 d\theta)^{\tfrac{1}{2}}\\
	& = \rho^r(\sum_{n\in \z}|h_n(r)|^2)^{\tfrac{1}{2}} \nonumber\\
	& \le \rho^r(\sum_{n\in \z}|h_n(r)|). \nonumber
	\end{align}
Consequently, we have
	\begin{equation}\label{eq-why0}
	\|\widehat{h}^{\Real^2} \otimes g\|_{A(E(2), W)} \le C (\|g\|_{L^2(\tor)} + \|g'\|_{L^2(\tor)})\int_{\Real^+} \rho^r(\sum_{n\in \z}|h_n(r)|)rdr,
	\end{equation}	
which shows that $\B \subseteq A(E(2), W)$.

For the density, we choose $g(\theta) = e^{il\theta}$, $h(re^{i\theta}) = h_1(r)e^{im\theta}$, $h_1\in C^\infty_c(0,\infty)$, $l, m\in \z$. Then we have
	$$\F^{E(2)}(\widehat{h}^{\Real^2} \otimes g)(r)F(\theta) = 2\pi h_1(r)e^{im\theta}(g*F)(\theta)$$
so that we have
	$$[W(r)\F^{E(2)}(\widehat{h}^{\Real^2} \otimes g)(r)e_n](\theta) = 2\pi w(\textcolor{blue}{-}re^{i\theta})h_1(r)\la e_n, e_l \ra e_{l+m}(\theta)$$
where $e_n(\theta) = e^{in\theta}$. Since our choices of $h_1\in C^\infty_c(0,\infty)$ and $l, m\in \z$ are arbitrary, we can easily see the density when $w$ is radial. For general $w$ we note that
	$$L^1(\Real^+,rdr;S^1(L^2(\tor))) \cong L^1(\Real^+,rdr;L^2(\tor)) \otimes_\gamma L^2(\tor)$$
where $\otimes_\gamma$ is the projective tensor product of Banach spaces. Thus, it is enough to check $\{w(re^{i\theta})h_1(r)e_n(\theta) : h_1\in C^\infty_c(0,\infty), n\in \z\}$ is dense in $L^1(\Real^+,rdr;L^2(\tor))$. Indeed, we consider $F \in L^\infty(\Real^+,rdr;L^2(\tor))$ such that
	$$\int_{\Real^+}\int_{\tor}F(r,\theta)w(-re^{i\theta})h_1(r)e_n(\theta)\, rd\theta dr = 0$$
for any $h_1\in C^\infty_c(0,\infty), n\in \z$. Clearly $F(r,\theta)w(-re^{i\theta}) = 0$ and consequently $F(r,\theta)=0$ for almost every $(r,\theta)$. This proves the density.

\vspace{0.3cm}
(When $W$ is extended from the subgroup $H_S$)
Let $w: \widehat{H_S} \cong \z \to (0,\infty)$ be the associated weight function. Let $F \in \text{Trig}(\tor)$.
From \eqref{eq-WeightsK} and \eqref{eq-Fourier-E(2)} we get
	\begin{align*}
	\lefteqn{W(r)\F^{E(2)}(\widehat{h}^{\Real^2} \otimes g)(r)F(\theta)}\\
	& = 2\pi \int_\tor \Big(\int_\tor \sum_{n\in \z} w(n) e^{in(\theta-\beta)} h(re^{i\beta})g(\beta - \alpha)d\beta\Big) F(\alpha) d\alpha,
	\end{align*}
which is an integral operator with the kernel
	\begin{align*}
	K(\theta, \alpha)
	& = 2\pi \int_\tor \sum_{n\in \z} w(n) e^{in(\theta-\beta)} h(re^{i\beta})g(\beta - \alpha)d\beta\\
	& = 2\pi \int_\tor \sum_{m,n\in \z} w(n) e^{in(\theta-\beta)} h(re^{i\beta}) \widehat{g}^\tor(m)e^{im(\beta-\alpha)}d\beta,
	\end{align*}
where we use $g(\theta) = \sum_{m\in \z}\widehat{g}^\tor(m)e^{im\theta}$, a finite sum. We use Lemma \ref{lem-integralOp} again, so that we need to estimate $\|K\|_{L^2(\tor^2)}$ as follows.
	\begin{align}\label{K-estimates-E(2)}
	\|K\|^2_{L^2(\tor^2)}
	& = (2\pi)^2 \int_{\tor^2}\Big| \int_\tor \sum_{m,n\in \z} w(n) e^{in(\theta-\beta)} h(re^{i\beta}) \widehat{g}^\tor(m)e^{im(\beta-\alpha)}d\beta \Big|^2 d\alpha d\theta \nonumber\\
	& = (2\pi)^2 \int_{\tor^2}\Big| \sum_{m,n\in \z} w(n) \widehat{h}^\tor_r(n-m)\widehat{g}^\tor(m) e^{in\theta}e^{-im\alpha} \Big|^2 d\alpha d\theta \\
	& = (2\pi)^2 \sum_{m,n\in \z}w(n)^2|\widehat{h}^\tor_r(n-m)|^2|\widehat{g}^\tor(m)|^2 \nonumber\\
	& = (2\pi)^2 \sum_{m,n\in \z}w(n+m)^2|\widehat{h}^\tor_r(n)|^2|\widehat{g}^\tor(m)|^2 \nonumber\\
	& \le (2\pi)^2 \sum_{m\in \z} w(m)^2|\widehat{g}^\tor(m)|^2 \cdot \sum_{n\in \z}w(n)^2|\widehat{h}^\tor_r(n)|^2 \nonumber\\
	& = (2\pi)^2 \sum_{m\in \z} w(m)^2|\widehat{g}^\tor(m)|^2 \cdot \sum_{n\in \z}w(n)^2|h_n(r)|^2 \nonumber\\
	& \le (2\pi)^2 \left(\sum_{m\in \z} w(m)|\widehat{g}^\tor(m)| \right)^2 \cdot \left(\sum_{n\in \z}w(n)|h_n(r)|\right)^2 \nonumber
	\end{align}
where $\widehat{h}^\tor_r(n) = \int_\tor h(re^{i\theta})e^{-in\theta}d\theta = h_n(r)$, the $n$-th frequency radial part and $h_r(\theta) = h(re^{i\theta})$, and we used submultiplicativity of $w$ in the first inequality.
We have a similar estimate for $\|\partial_\alpha K\|_{L^2(\tor^2)}$ involving $g'$ instead of $g$, so that we get
	\begin{align}\label{eq-why}
	\|\widehat{h}^{\Real^2} \otimes g\|_{A(E(2), W)}
	& \le 2\pi C (\|g\|_{A(\tor, w)} + \|g'\|_{A(\tor, w)}) \cdot \left(\int_{\Real_+}\sum_{n\in \z}w(n)|h_n(r)|\, rdr\right) \nonumber \\
	& \le 2\pi C (\|g\|_{A(\tor, w)} + \|g'\|_{A(\tor, w)}) \cdot \left(\sum_{n\in \z}\rho^{|n|}\int_{\Real_+}|h_n(r)|\, rdr\right),
	\end{align}
where we use the fact that $|w(n)| \le \rho^{|n|}$, $n\in \z$ for some $\rho>0$ by sub-multiplicativity of $w$. This implies the inclusion $\B \subseteq A(E(2), W)$.

The density can be similarly explained by the following formula.
	$$W(r)\F^{E(2)}(\widehat{h}^{\Real^2} \otimes g)(r)e_n = 2\pi w(m+l)h_1(r)\la e_n, e_m \ra e_{l+m}$$
for $g(\theta) = e^{il\theta}$ and $h(re^{i\theta}) = h_1(r)e^{im\theta}$, $h_1 \in C^\infty_c(0,\infty)$, $l, m\in \z$.

\vspace{0.3cm}
(When $W_t=\exp(t\sqrt{\partial\lambda(-\Delta)})$ is an exponential weight as in Definition \ref{def-exp-weight-E(2)}.)
Recall from \eqref{eq-Laplacian-E(2)} that $W_t(r) e_n = e^{t\sqrt{n^2+r^2}}e_n$, $n\in \z$, so that we can follow similar calculations as in the case of weights coming from the subgroup $H_S$ with the weight function $w: \z \to (0,\infty)$ replaced with $w_r: \z \to (0,\infty)$ given by $w_r(n) = e^{t\sqrt{n^2+r^2}}$. Observe that $w_r$ is also submultiplicative and $w_r(n) \le e^{t|n|}e^{tr}$, $n\in \z$. Now the estimate \eqref{K-estimates-E(2)} becomes
	\begin{align*}
	\|K\|_{L^2(\tor^2)}
	& \le 2\pi \Big(\sum_{m \in \z} w_r(m)|\widehat{g}^\tor(m)|\Big) \cdot \Big(\sum_{n \in \z} w_r(n)|h_n(r)|\Big)\\
	& \le 2\pi \Big(\sum_{m \in \z} e^{t|m|}|\widehat{g}^\tor(m)|\Big) \cdot e^{2tr} \cdot \Big(\sum_{n \in \z} e^{t|n|}|h_n(r)|\Big).
	\end{align*}
We have a similar estimate for $\|\partial_\alpha K\|_{L^2(\tor^2)}$ involving $g'$ instead of $g$, so that we get
	\begin{align*}
	\lefteqn{\|\widehat{h}^{\Real^2} \otimes g\|_{A(E(2), W_t)}}\\
	& \le 2\pi (\|g\|_{A(\tor, u)} + \|g'\|_{A(\tor, u)}) \cdot \int_{\Real_+}e^{2tr} \sum_{n \in \z} e^{t|n|}|h_n(r)|\, rdr\\
	& = 2\pi (\|g\|_{A(\tor, u)} + \|g'\|_{A(\tor, u)}) \cdot \sum_{n \in \z} e^{t|n|}\int_{\Real_+}e^{2tr} |h_n(r)|\, rdr,
	\end{align*}
where $u$ is the weight function on $\z$ given by $u(n) = e^{t|n|}$, $n\in \z$. This implies the inclusion $\B \subseteq A(E(2), W_t)$.
For the density we could simply repeat the same argument as in the case of weights coming from the subgroup $H_S$.

\vspace{0.3cm}
(2) This is trivial from Proposition \ref{prop-space-A_0-B_0} and part (1).
\end{proof}

\begin{rem}
The estimates \eqref{eq-why0} and \eqref{eq-why} are the reasons for the choice of the spaces $\A_0$ and $\B_0$. Generally speaking, we need super-exponential decay of $\widehat{h_r}^\tor(n)$ with respect to both of $r>0$ and $n\in \z$. The reason why we need both of the spaces $\A_0$ and $\B_0$ will be clarified in the next section. Note also that the proof of the above proposition tells us that the space $\mc B$ can be used as the subspace $\mc S$ in \ref{ssec:separable-type I}.
\end{rem}

\begin{prop}
	Every element of $\A$ is an entire vector for $\lambda$.
\end{prop}
\begin{proof}
We assume that $g(\theta) = e^{il\theta}$, $h(re^{i\theta}) = h_1(r)e^{im\theta}$, $h_1 \in C_c[0,\infty)$, $l, m\in \z$ with $h\in C^\infty_c(\Real^2)$. Then we have
	$$\F^{E(2)}(\widehat{h}^{\Real^2} \otimes g)(r)F(\theta) = h_1(r)e^{im\theta}\widehat{F}^\tor(l)e^{il\theta}.$$
Now we take $(x,y,z=e^{is}) \in \Comp^2\times \Comp^*$ from the complexification. We know that the analytic extension $\pi^r_\Comp$ of $\pi^r$ has the same formula, so that we have
	\begin{align*}
	\lefteqn{\pi^r_\Comp(x,y,z)\F^{E(2)}(\widehat{h}^{\Real^2} \otimes g)(r)F(\theta)}\\
	& = h_1(r)e^{ir(x\cos \theta + y \sin \theta)}\widehat{F}^\tor(l)e^{i(l+m)\theta}e^{-i(l+m)s}.
	\end{align*}
Thus, the operator $\pi^r_\Comp(x,y,z)\F^{E(2)}(\widehat{h}^{\Real^2} \otimes g)(r)$ is a rank 1 operator so that we have
	\begin{align*}
	\lefteqn{\|\pi^r_\Comp(x,y,z)\F^{E(2)}(\widehat{h}^{\Real^2} \otimes g)(r)\|^2_2}\\
	& = |h_1(r)|^2 e^{2(l+m)\text{Im}s} \int_\tor e^{-2r(\text{Im}x\cos \theta + \text{Im}y \sin \theta)}d\theta\\
	& = |h_1(r)|^2 e^{2(l+m)\text{Im}s} \int_\tor e^{2r\sqrt{(\text{Im}x)^2 +(\text{Im}y)^2} \cos \theta} d\theta.
	\end{align*}
Since $h_1$ is compactly supported it is clear that the integral
	$$\int_{\Real^+}\sup_{|x|, |y|, |z|\le M}\|\pi^r_\Comp(x,y,z)\F^{E(2)}(\widehat{h}^{\Real^2} \otimes g)(r)\|^2_2 \;rdr$$
is finite. Now we just need to observe that $(x,y,z) \in \Om_t$ implies $|x|, |y|, |z|\le M$ for some $M>0$ from \eqref{eq-exp-E(2)}. This proves the conclusion we wanted by (1) of Theorem \ref{thm-Goodman}.
\end{proof}

\subsubsection{Solving Cauchy functional equation for $E(2)$}
Unlike in the Euclidean cases, the Cauchy functional equation for $E(2)$ is far more involved. We divide it into several pieces and tackle them one by one.

We begin with the case of the Cauchy functional equation on $\A_0$.
	$$({\rm CFE}_{\A_0})\;\; T\in \A^\dagger_0 \; \text{such that}\; \la T, f*g\ra = \la T, f\ra \cdot \la T, g \ra,\; \forall f,g \in \A_0.$$
Here $f*g$ denotes the convolution in $\Real^2$ and Proposition \ref{prop-space-A_0-B_0} ensures that $f*g \in \A_0$. We would like to follow the steps in the proof of Theorem \ref{thm-Cauchy-Rn}. For [step 1] we need the algebra $\A_0$ to be closed under partial derivatives.
\begin{prop}
The algebra $\A_0$ is invariant under partial derivatives $\partial_x$ and $\partial_y$ of $\Real^2$.
\end{prop}
\begin{proof}
The formula \eqref{eq-partial-freq} immediately tells us that for $f\in \A_0$ with $f_n \equiv 0$, $|n|>N \in \n$, we have $(\partial_x f)_n(r) = (\partial_y f)_n(r) = 0$, $r\in (0,\infty)$ for $|n|>N+1$. The point $r=0$ is trivial by looking at the integral form, i.e. for example,
	$$(\partial_x f)_n(0) = \int_{\tor} \partial_x f(0) e^{-in\theta}d\theta = 0, \; n\ne 0.$$
The case for $\partial_y$ is the same.	
\end{proof}

Now we could perform [step 1] of Theorem \ref{thm-Cauchy-Rn} for $({\rm CFE}_{\A_0})$.
\begin{prop}\label{prop-CFE-intermediate}
Let $T\in \A^\dagger_0$ be a solution of $({\rm CFE}_{\A_0})$. Then we have
	$$\partial^\dagger_x T = c_1 T,\;\; \partial^\dagger_y T = c_2 T$$
for some $c_1, c_2\in \Comp$. Here, $\A^\dagger_0$ and $\partial^\dagger_x$ refer to the algebraic dual space and adjoint map respectively.
\end{prop}
\begin{proof}
For $f,g\in \A_0$ we have
	$$\la T, (\partial_x f)*g \ra = \la T, \partial_x (f*g) \ra =\la T, f*(\partial_x g) \ra.$$
Let $T$ be a nonzero solution of $({\rm CFE}_{\A_0})$. Then,
	$$\la T, \partial_x f\ra \la T, g\ra = \la T, f\ra \la T, \partial_x g\ra.$$
If we choose $g$ so that $\la T, g\ra \ne 0$, then we get the conclusion with $c_1 = \la T, \partial_x g\ra / \la T, g\ra$. The argument for the case of $\partial_y$ is identical.
\end{proof}

In order to carry out [step 2] of Theorem \ref{thm-Cauchy-Rn}, the space we are working on needs to be closed under multiplication by exponential functions like
	$$\exp(c_1 x+ c_2 y) = \exp(\tfrac{1}{2}(c_1-ic_2)z + \tfrac{1}{2}(c_1+ic_2)\bar{z})$$
where $z=x+iy$ and $c_1, c_2\in \Comp$. Clearly, the space $\A_0$ is not closed under the multiplication by exponential functions, and that is why we need a bigger space $\B_0$.

\begin{prop}\label{prop-space-B}
	\begin{enumerate}
		\item The partial derivatives $\partial_x$ and $\partial_y$ of $\Real^2$ are continuous maps on $\B_0$.
		
		\item The maps $f\mapsto e^{cz}f(x,y)$ and $f\mapsto e^{c\bar{z}}f(x,y)$ are continuous on $\B_0$ for any $c\in \Comp$, where we denote $z=x+iy$.
	\end{enumerate}
\end{prop}
\begin{proof}
(1) This is automatic from the definition.

\vspace{0.3cm}
(2)  We only consider the case of $e^z$ since other cases can be done similarly. We denote $\Phi(f)(x,y) := e^zf(x,y)$. Since the sum $e^z = e^{re^{i\theta}} = \sum_{k\ge 0}\frac{r^k}{k!}e^{ik\theta}$ is absolutely convergent uniformly on $\theta$, we can easily conclude that
	$$\Phi(f)_n(r) = \sum_{k\ge 0} \frac{r^k}{k!}f_{n-k}(r).$$
Thus we have
	\begin{align*}
	\|\Phi(f)\|_\rho
	& = \sum_{n\in \z} \rho^{|n|} \int_{\Real_+} \rho^r |\sum_{k\ge 0} \frac{r^k}{k!}f_{n-k}(r)|rdr\\
	& \le \sum_{k\ge 0} \frac{1}{k!} (\sum_{n\in \z} \rho^{|n|} \int_{\Real_+} \rho^r |r^k f_{n-k}(r)|rdr)\\
	& = \sum_{k\ge 0} \frac{1}{k!} (\sum_{n\in \z} \rho^{|k+n|} \int_{\Real_+} \rho^r |r^k f_n(r)|rdr)\\
	& \le \sum_{k\ge 0} \frac{\rho^k K^k}{k!} (\sum_{n\in \z} \rho^{|n|} \int_{\Real_+} \rho^r |f_n(r)|rdr) \\
	& = \exp(\rho K)\|f\|_\rho,
	\end{align*}
where supp$f \subseteq B_K$. For the derivatives of $\Phi(f)$ we recall the Leibniz rule to get
	$$\partial_x (\Phi(f)) = \Phi(f) + \Phi(\partial_x f).$$
We have a similar identity for $\partial_y(\Phi(f))$. These identities tell us that for any multi-index $\alpha$ we have $\|\partial^\alpha\Phi(f)\|_\rho \le C_{\alpha, K} (\sum_{|\beta|\le |\alpha|}\|\partial^\beta f\|_\rho)$ for supp$f \subseteq B_K$, which explains the continuity of the map $\Phi$ on $\B_0$.
\end{proof}

Now we move to the Cauchy functional equation on $\B_0$.
	$$({\rm CFE}_{\B_0})\;\; S\in \B^*_0 \; \text{such that $S|_{\A_0}$ is a solution of $({\rm CFE}_{\A_0})$}.$$
\begin{rem}
Note that it is not clear whether $\B_0$ is also an algebra with respect to the $\Real^2$-convolution.
\end{rem}

\begin{prop}\label{prop-CFE-intermediate-2}
Let $S \in \B^*_0$ be a solution of $({\rm CFE}_{\B_0})$. Then $S$ is actually an exponential function of the form $\exp(-c_1 x - c_2 y)$, $c_1, c_2 \in \Comp$. In other words, for any $f\in \B_0$ we have
	$$\la S, f \ra = \int_{\Real^2}\exp(-c_1 x - c_2 y)f(x,y)\, dxdy.$$
\end{prop}
\begin{proof}
Let $S \in \B^*_0$ be a solution of $({\rm CFE}_{\B_0})$. Then, clearly $T = S|_{\A_0} \in \A^\dagger_0$ is a solution of $({\rm CFE}_{\A_0})$, so that by Proposition \ref{prop-CFE-intermediate} we have
	$$\partial^\dagger_x T = c_1 T,\;\; \partial^\dagger_y T = c_2 T$$
for some $c_1, c_2\in \Comp$. We introduce two operations on $\B^*_0$, namely the adjoint of partial derivatives and the multiplication with respect to exponential functions. More precisely, we define $\partial^*_x S$ by
	$$\partial^*_x S := S\circ \partial_x$$
and similarly for $\partial^*_y S$. The element $e^{c_1 x + c_2 y}S\in \B^*_0$ is defined by
	$$\la e^{c_1 x + c_2 y}S, f\ra := \la S, e^{c_1 x + c_2 y}f \ra,\; f\in \B_0$$
where we use continuity of the multiplication with respect to exponential functions. Since $\A_0$ is dense in $\B_0$, we can conclude that $\partial^*_x S = c_1 S$ and $\partial^*_y S = c_2 S$. By applying Leibniz rule for the smooth functions we can easily obtain that
	$$\partial^*_x (e^{c_1 x + c_2 y}S) = \partial^*_y (e^{c_1 x + c_2 y}S) = 0.$$
We now focus on the element $\tilde{S} = e^{c_1 x + c_2 y}S \in \B^*_0$, whose both partial derivatives are vanishing. We will get the conclusion by restricting down to the subspace $\A_{00}$. We first consider the continuous embedding
	$$J_n : C^\infty_c(0,\infty) \hookrightarrow \B_0,\; h\mapsto f(re^{i\theta}) = \frac{h(r)}{r}e^{-in\theta},\; n\in \z$$
which allows us the decomposition
	$$\A_{00} = {\rm span}\big(\cup_{n\in \z}J_n(C^\infty_c(0,\infty))\big).$$	
Thus, all the information on $\tilde{S}|_{\A_{00}}$ is encoded in the sequence
	$$(\tilde{S} \circ J_n)_{n\in \z} \subseteq C^\infty_c(0,\infty)^*.$$	
Note that it is straightforward to check that the adjoint map of $J_n$,  $J^*_n : \B^*_0 \to C^\infty_c(0,\infty)^*$ extends the map $f\in C^\infty_c(\Real^2) \mapsto f_n \in C_c[0,\infty)$.

For $g\in C^\infty_c(0,\infty)$ we have
	\begin{align*}
	0  = \la \partial^*_x \tilde{S}, g(r) \cos\theta \ra = \la \tilde{S}, \partial_x(g(r) \cos\theta) \ra = \la \tilde{S}, g'(r)\cos^2\theta + \frac{g(r)}{r}\sin^2\theta \ra,\\
	0  = \la \partial^*_y \tilde{S}, g(r) \sin\theta \ra = \la \tilde{S}, \partial_y(g(r) \sin\theta) \ra = \la \tilde{S}, g'(r)\sin^2\theta + \frac{g(r)}{r}\cos^2\theta \ra.
	\end{align*}
By summing them we get
	$$0 = \la \tilde{S}, g'(r) + \frac{g(r)}{r} \ra = \la\la \tilde{S}\circ J_0, rg'(r)+g(r) \ra\ra$$
where $\la\la \cdot, \cdot \ra\ra$ is the usual $(C^\infty_c(0,\infty)^*, C^\infty_c(0,\infty))$ duality bracket. Now this is the moment we summon the usual distribution theory. By integration by parts we have
	$$0=\la\la \tilde{S}\circ J_0, rg'(r) + g(r)\ra\ra = -\la\la D_r \tilde{S}\circ J_0, rg(r)\ra\ra$$
where $D_r$ is the derivative with respect to the $r$-variable. Since $rg(r)$ covers all functions in $C^\infty_c(0,\infty)$, we know that $\tilde{S}\circ J_0$ is a constant function. This takes care of radial part of $\tilde{S}$.

For higher frequencies we consider
	$$0 = \la \partial^*_x \tilde{S}, g(r)k(\theta)\sin\theta  \ra = \la \partial^*_y \tilde{S}, g(r)  k(\theta)\cos\theta \ra$$
for any $g\in C^\infty_c(0,\infty)$ and $k\in \text{Trig}(\tor)$, which eventually tells us that
	$$0 = \la \tilde{S}, \frac{g(r)}{r}k'(\theta)\ra.$$
Since $k'(\theta)$ covers all $\text{Trig}(\tor)$ except constant functions, we know that $\tilde{S}\circ J_n \equiv 0$ for any $n\ne 0$. Recalling that $\A_{00} ={\rm span}( \cup_{n\in \z}J_n(C^\infty_c(0,\infty)))$, we can conclude that $\tilde{S}$ acts on $\A_{00}$ as a constant function, and by the density we can actually conclude that $\tilde{S}$ is a constant function on $\B_0$, which leads us to the fact that $S$ is actually an exponential function of the form $\exp(-c_1 x - c_2 y)$ as a distribution acting on $\B_0$.
\end{proof}
	
Finally we consider the Cauchy functional equation on $\B_0 \otimes c_{00}(\z)$, where we endow a canonical locally convex topology in the same way as in (2) of Remark \ref{rem-spaces-for-E(2)}.
	\begin{align*}({\rm CFE_{E(2)}})&\;\; v\in (\B_0 \otimes c_{00}(\z))^* \;\; \text{satisfies}\\ & \;\; \la v, (\sum_{n\in \z} f_n \otimes \delta_n)*(\sum_{m\in \z} g_m \otimes \delta_m)\ra = \la v, \sum_{n\in \z} f_n \otimes \delta_n \ra \cdot \la v, \sum_{m\in \z} g_m \otimes \delta_m \ra\\
	&\;\;\text{for $f_n, g_m \in \A_0$, $n, m\in \z$.}
	\end{align*}
Here $*$ implies the convolution in $\Real^2 \times \z$, so that we have
	$$(f \otimes \delta_n)*(g \otimes \delta_m) = f*g \otimes \delta_{n+m}$$
for $f,g\in \A_0$.

\begin{prop}\label{prop-CFE-E(2)}
Let $v \in (\B_0 \otimes c_{00}(\z))^*$ be a solution of $({\rm CFE_{E(2)}})$. Then $v$ is actually an exponential function of the form $\exp(cn+ c_1 x + c_2 y)$, $c, c_1, c_2 \in \Comp$. In other words, for any $f\in \B_0$ we have
	$$\la v,  f\ra = \sum_{n\in \z}\int_{\Real^2}f(x,y,n)e^{c_1x+c_2y+cn}dxdy$$
\end{prop}
\begin{proof}
Instead of differential operators we need the difference operator
	$$D: \B_0 \otimes c_{00}(\z)\to \B_0 \otimes c_{00}(\z),\;\; \sum_{n\in \z} f_n \otimes \delta_n \mapsto \sum_{n\in \z} (f_n-f_{n-1}) \otimes \delta_n.$$
It is straightforward to check that
	\begin{equation}\label{eq-difference-conv}
	D(\sum_{n\in \z} f_n \otimes \delta_n)*(\sum_{m\in \z} g_m \otimes \delta_m) = (\sum_{n\in \z} f_n \otimes \delta_n)*D(\sum_{m\in \z} g_m \otimes \delta_m).
	\end{equation}
Note that $(\B_0 \otimes c_{00}(\z))^* \cong \prod_{n\in \z}\B^*_0$, so that $v\in (\B_0 \otimes c_{00}(\z))^*$ can be  written $v = (v_n)_{n\in \z}$ with $v_n\in \B^*_0$, $n\in \z$. The duality $((\B_0 \otimes c_{00}(\z))^*, \B_0 \otimes c_{00}(\z))$ is given by
	$$\la v,  \sum_{n\in \z} f_n \otimes \delta_n\ra = \sum_{n\in \z} \la v_n, f_n\ra.$$
Then it is also straightforward to check
	\begin{equation}\label{eq-difference-adj}
	 (D^*v)_n = v_n - v_{n+1}, \ n\in {\mathbb Z}.
	\end{equation}
Now suppose that $v$ is a solution for $({\rm CFE_{E(2)}})$. We repeat the 3 steps in Theorem \ref{thm-Cauchy-Rn}. Then, \eqref{eq-difference-conv} tells us
	$$\la D^* v, \sum_{n\in \z} f_n \otimes \delta_n \ra = \frac{\la D^* v, \sum_{m\in \z} g_m \otimes \delta_m \ra}{\la v, \sum_{m\in \z} g_m \otimes \delta_m \ra}\la v, \sum_{n\in \z} f_n \otimes \delta_n\ra$$
which means that
	$$D^* v = d v$$
for some constant $d\in \Comp$. Now \eqref{eq-difference-adj} immediately implies that
	$$v_n = (1-d)^nv_0,\;\; n\in \z.$$
We note that $(f\otimes\delta_0)*(g\otimes\delta_0) =(f*g) \otimes\delta_0$, $f,g\in \A_0$ so that
	\begin{align*}
	\la v_0, f*g \ra & = \la v, f*g \otimes \delta_0 \ra\\
	& = \la v, (f \otimes \delta_0)*(g \otimes \delta_0) \ra \\
	& = \la v, f \otimes \delta_0\ra \cdot \la v, g \otimes \delta_0 \ra\\
	& = \la v_0, f\ra \cdot \la v_0, g \ra,
	\end{align*}
for any $f,g\in \A_0$. In other words, $v_0\in \B^*_0$ is a solution of $({\rm CFE}_{\B_0})$, so that $v_0$ is actually an exponential function by Proposition \ref{prop-CFE-intermediate-2}, i.e. there are $c_1, c_2 \in \Comp$ such that
	$$v_0(f) = \int_{\Real^2}f(x,y)e^{c_1x + c_2y}dxdy.$$
Finally, we get
	$$\la v,  \sum_{n\in \z} f_n \otimes \delta_n\ra = \sum_n \la v_n, f_n\ra = \sum_{n\in \z}\int_{\Real^2}f_n(x,y)e^{c_1x+c_2y+c_3n}dxdy$$
where $e^{c_3} = 1-d$.	Note that the possibility of $d=1$ can be easily excluded. Indeed, if $d=1$, then we have $v_n \equiv 0$, $n\ne 0$ and we have
	$$\la v_0, f*g \ra = \la v, (f\otimes \delta_1)*(g \otimes \delta_{-1}) \ra = \la v_1, f\ra \la v_{-1}, g\ra = 0$$ for any $f*g\in \B_0$, which means that $v_0\equiv 0$.

\end{proof}

\subsubsection{Realization of ${\rm Spec}A(E(2),W)$ in $E(2)_\Comp$}
We now give a realization of the spectrum ${\rm Spec}A(E(2),W)$ in the complexification $E(2)_\Comp$.

\begin{prop}\label{prop-realization-spec-E(2)}
Every character $\varphi \in {\rm Spec}A(E(2), W)$ is uniquely determined by a point $(x,y,z) \in E(2)_\Comp \cong \Comp^2\times \Comp^*$, which is nothing but  the evaluation at the point $(x,y,z)$ on $\B$ (and consequently on $\A$).

\end{prop}
\begin{proof}

Let $\varphi\in \text{Spec} A(E(2), W)$. Consider a continuous composition $\psi = \varphi \circ (2\pi \F^{\Real^2 \times \tor})$ as follows.
\[
\xymatrix{
A(E(2),W) \ar[d]_{\varphi}
& \B \ar@{_{(}->}[l]
& \A \ar@{_{(}->}[l] \ar[lld]_{\varphi|_\A} &
& \A_0 \otimes c_{00}(\z) \ar[ll]_(.6){2\pi \F^{\Real^2 \times \z}} \ar@/^/[lllld]^{\psi}
\\
\mathbb{C} & & &
}
\]

Then Proposition \ref{prop-CFE-E(2)} tells us that there are $c_1, c_2, c_3 \in \Comp$ such that for $h\in \B_0 \subseteq  C^\infty_c(\Real^2)$ and $a = (a_n) \in c_{00}(\z)$ we have
	\begin{align*}
	\varphi(\widehat{h}^{\Real^2} \otimes g) & = \varphi \circ \F^{\Real^2\times \z} (h \otimes a)\\
	& = \frac{1}{2\pi} \left(\sum_n a_ne^{c_3n}\right) \left(\int_{\Real^2}h(s,t)e^{c_1s+c_2t}dsdt\right)\\
	& = g(z)\widehat{h}^{\Real^2}(ic_1,ic_2),
	\end{align*}
where $g$ is the trigonometric polynomial $g = \F^\z (a)$, and we set $e^{c_3} = z$, $c_1 = -i x$ and $c_2 = -i y$. Note that the factor $\frac{1}{2\pi}$ in the second equality appears as a result of our choice for the Fourier transform on $\Real^n$. Indeed, only the scaled operator
$\varphi \circ (2\pi\F^{\Real^2\times \z})$ is multiplicative with respect to convolution.
In the last equality we used the Paley-Wiener theorem saying that the Fourier transform of $h\in C^\infty_c(\Real^2)$ extends holomorphically to $\Comp^2$ and the fact that $g$ is a trigonometric polynomial, so that it extends holomorphically to $\Comp$.
\end{proof}

As in the case of $\Hee$ and $\Hee_r$, the above embedding respects the Cartan decomposition \eqref{eq-Cartan-E(2)} as follows.

\begin{prop}\label{prop:cartan-E(2)}
We have ${\rm Spec}A(E(2),W)\subseteq E(2)\exp(i\fe(2))$ in the sense that for any $\varphi\in {\rm Spec}A(E(2),W)$ there are uniquely determined $g\in E(2)$ and $X'\in \fe(2)$ such that
	$$\varphi = \lambda(g) \overline{\lambda_\Comp(\exp(iX'))W^{-1}}W.$$
\end{prop}
\begin{proof}
Let $\varphi \in {\rm Spec}A(E(2),W)$ be the character associated to the point $(\alpha, z_1, z_2) \in E(2)_\Comp$.
For $f \in \A$ we have by (2) of Theorem \ref{thm-Goodman} and Proposition \ref{prop-realization-spec-E(2)} that
	$$\varphi(f) = f_\Comp(\alpha, z_1, z_2) = \int_{\Real^+} \text{\rm Tr}(\pi^r_\Comp((\alpha, z_1, z_2)^{-1})\widehat{f}^{E(2)}(r))\,rdr$$
where $f_\Comp$ is the analytic continuation of $f$.
Now we repeat the same argument of Proposition \ref{Prop-respect-Cartan-heis}. For $X'\in \mathfrak{g}_\Comp$ we have
	$$\lambda_\Comp(\exp(X'))W^{-1} \stackrel{\F^{E(2)}}{\sim} \big(\pi^r_\Comp(\exp(X'))W^{-1}(r)\big)_{r>0}$$
on $W\D$, where the space $\D$ (which is different from the one introduced in the proof of Proposition \ref{prop-embeddings-Heisenberg}) is given by
	$$\D := \text{span}\Big\{h \otimes P_{mn}: h\in C^\infty_c(\Real^+), m,n\in \z\Big\}$$
and $P_{mn}$ is the rank 1 operator $P_{mn} \cong e_m \otimes e_n$. Note that the density of $W\D$ in $L^2(\Real^+, rdr; S^2(L^2(\tor)))$ can be done in a similar way as in the proof of Proposition \ref{prop-density-E(2)}. Now we have $\lambda_\Comp(\exp(X'))W^{-1}$ is bounded if and only if $\big(\pi^r_\Comp(\exp(X'))W^{-1}(r)\big)_{r>0}$ is bounded, i.e. $\pi^r_\Comp(\exp(X'))W^{-1}(r)$ is uniformly bounded with respect to  $r>0$. This leads us to the conclusion that if we choose $X'\in \fe(2)_{{\mathbb C}}$ to satisfy $\exp(-X')=(\alpha,z_1,z_2)$, then
	$$\varphi = \overline{\lambda_\Comp(\exp(X'))W^{-1}}W \in VN(E(2),W^{-1}).$$
Finally we recall the Cartan decomposition \eqref{eq-Cartan-E(2)} and the fact that $\pi^r_\Comp$ is a (local) representation on $\D^\infty_\Comp(\lambda)$, which clearly contains the space $\D$. Combining these facts with the above observations we get
	\begin{align*}
	\lefteqn{{\rm Spec}A(E(2), W)\subseteq} \\
	& \Big\{ \lambda(g) \overline{\lambda_\Comp(\exp(iX'))W^{-1}}W: g\in E(2),\; X'\in \fe(2),\; \lambda_\Comp(\exp(iX'))W^{-1}\; \text{is bounded} \Big\}
	\end{align*}
as claimed.
\end{proof}

\subsubsection{Description of ${\rm Spec}A(E(2),W)$ when $W$ is extended from subgroups}

	\begin{thm}\label{thm-E(2)-Spec}
	Let $\mathfrak{h}$ be the Lie subalgebra of $\fe(2)$ corresponding to the subgroups $H = H_S$ or $H_{X,Y}$. Suppose that $W_H$ is a weight on the dual of $H$ and $W=\iota(W_H)$ is the extended weight on the dual of $E(2)$.
	Then, we have
		$${\rm Spec}A(E(2), W) \cong \Big\{ g\cdot \exp(iX'): g\in E(2),\; X'\in \mathfrak{h},\; \exp(iX') \in {\rm Spec} A(H,W_H)\Big\}.$$
	\end{thm}

\begin{proof}
We can basically follow the same arguments as in the  proof of Theorem \ref{thm-Heisenberg-Spec}. Note that we use Theorem \ref{thm-unbdd-Lie-noncompact-E(2)} below as a replacement for Theorem \ref{thm-unbdd-Lie-noncompact}.

\end{proof}

\begin{thm}\label{thm-unbdd-Lie-noncompact-E(2)}
Let $\mathfrak{h}$ be the Lie subalgebra of $\fe(2)$ corresponding to the subgroups $H = H_S$ or $H_{X,Y}$.
Suppose that $W_H$ is a bounded below weight on the dual of $H$ and $W=\iota(W_H)$ is the extended weight on the dual of $E(2)$. Then for any $X'\in \fe(2) \backslash \mathfrak{h}$ the operator $\exp (i\partial \lambda(X'))W^{-1}$ is unbounded whenever it is densely defined.
\end{thm}
\begin{proof}
We can  basically repeat the proof of Theorem \ref{thm-unbdd-Lie-noncompact}. Note that we only need to check the case $H = H_{X,Y}$, as Theorem \ref{thm-unbdd-Lie-compact} and Remark \ref{rem-unbdd-Lie-compact} take care of the case of the compact subgroup $H_{S}$.
The only additional point is that we actually have a concrete description of the maps $\theta$ and $\alpha$ as follows. Let $X' = aS+bX+cY$, $a,b,c\in \Real$ with $a\ne 0$. Then for $s, x, y\in \Real$ we can readily check that
	$$\exp(sX')(x,y,1) = (x\cos as - y\sin as, x\sin as + y\cos as, 1)\exp(sX')$$
so that taking $E((x,y,1), t) = (x,y,1) \exp(tX')$ we get
	$$\theta(s, (x,y), t) = (x\cos as + y\sin as, -x\sin as + y\cos as).$$
The above shows us that $\theta$ has polynomial growth in $(x,y)$-variables independent of $s$ and $t$. The case of the map $\alpha$ is easier, namely $\alpha(s, (x,y), t) \equiv -s+t.$
\end{proof}

\begin{rem}\label{rem-1D-subgp-E(2)}
The statements of Theorem \ref{thm-unbdd-Lie-noncompact-E(2)} and Theorem \ref{thm-E(2)-Spec} hold true for $H=H_Y$ and $H=H_r$, $r\in\mathbb R^{\geq 0}$, in place of $H_{X,Y}$, by applying similar arguments.   
\end{rem}

\begin{ex}
For $W$ extended from the subgroup $H=H_{X,Y}$ and $X' = x'X + y'Y \in \mathfrak{h}$, $x',y'\in \Real$, the condition $\exp(iX') \in {\rm Spec} A(H, W_H)$ is equivalent to the existence of a constant $C>0$ such that
		$$e^{ax'}e^{by'} \le C w(a,b),\; \text{for almost all}\; (a,b)\in \Real^2$$
by \eqref{eq-YZ-subgroup}, where $w: \widehat{H_{X,Y}} \cong \Real^2 \to (0,\infty)$ is the weight function. In particular, for the specific weight function $w(re^{i\theta}) = \beta^r$, $(r,\theta) \in \Real_+\times \tor$ for some $\beta \ge 1$ we have
	\begin{align*}
	\lefteqn{{\rm Spec} A(E(2), W)}\\
	& \cong \Big\{g\cdot (ix',iy',1) \in E(2)_\Comp : g\in E(2),\ x',y'\in \Real,\ (x')^2 + (y')^2 \le (\log \beta)^2\Big\}\\
	& \cong \Big\{(x,y,z) \in E(2)_\Comp : ({\rm Im}x)^2 + ({\rm Im}y)^2 \le (\log \beta)^2,\, |z|=1\Big\}.
	\end{align*}
For $W$ extended from the subgroup $H=H_S$ and $X' = s'S \in \mathfrak{h}$ the condition $\exp(iX') \in {\rm Spec} A(H, W_H)$ is equivalent to the existence of a constant $C>0$ such that
		$$e^{s' n} \le C w(n),\; n\in \z,$$
where $w: \widehat{H_S} \cong \z \to (0,\infty)$ is the weight function. In particular, for the specific weight function $w(n) = \beta^{|n|}$, $n\in \z$ for some $\beta \ge 1$ we have
	\begin{align*}
	\lefteqn{{\rm Spec} A(E(2), W)}\\
	& \cong \Big\{g\cdot (0,0,e^{-s'}) \in E(2)_\Comp : g\in E(2),\ s'\in \Real,\ |s'| \le \log \beta\Big\}\\
	& = \Big\{(x,y,z) \in E(2)_\Comp : {\rm Im}x = {\rm Im}y=0,\; \frac{1}{\beta}\le |z| \le \beta\Big\}.
	\end{align*}	
\end{ex}

\subsubsection{Description of ${\rm Spec}A(E(2),W)$ for exponentially growing weights coming from the Laplacian}

In this subsection we consider the case of exponentially growing weights coming from Laplacian of $E(2)$ whose proof depends on the detailed structure of the weights and the group Fourier transforms.

\begin{thm}\label{thm-E(2)-Spec-Laplacian}
For the exponential weight $W_t = \exp(t\sqrt{\partial \lambda (-\Delta)})$, $t>0$ from Definition \ref{def-exp-weight-E(2)} we have
	$${\rm Spec} A(E(2),W_t) \cong \Big\{(x,y,z) \in E(2)_\Comp: ({\rm Im}x)^2 + ({\rm Im}y)^2 + (\log |z|)^2 \le t^2\Big\}.$$
\end{thm}
\begin{proof}
By Proposition \ref{prop-realization-spec-E(2)} and Proposition \ref{prop:cartan-E(2)}, the spectrum of $A(E(2),W_t)$ can be embedded in $E(2)_\Comp$ so that the embedding respects the Cartan decomposition. In that case, the spectrum will exactly be the collection of those elements of $(x,y,z)\in E(2)_\Comp$ so that $\lambda_\Comp(x,y,z)W_t^{-1}$ is bounded.
Thus, we need to check the uniform boundedness of the family $(\pi^r_\Comp(x,y,z)(W_t(r))^{-1})_{r>0}$. Note that from \eqref{eq-Laplacian-E(2)} we have
	$$W_t(r)e_n = w(n,r)e_n,\; w(n,r) = \exp(t\sqrt{n^2 +r^2}), \; n\in \z, r>0.$$
For $F(\theta) = \sum_n a_n e^{in\theta}$ we have
	\begin{align*}
	\lefteqn{\pi^r_\Comp(x,y,z)(W_t(r))^{-1}F(\theta)}\\
	& = e^{ir(x\cos\theta + y\sin\theta)}\sum_n a_n w(n,r)^{-1} e^{in(\theta-s)}\\
	& = e^{ir(\text{Re}x\cos\theta + \text{Re}y\sin\theta)}e^{rA\cos(\theta-\beta)}\sum_n a_n w(n,r)^{-1}e^{-ins} e^{in\theta},
	\end{align*}
where $z=e^{is}$, $A = \sqrt{(\text{Im}x)^2 +(\text{Im}y)^2}$ and $\beta \in [0,2\pi)$ is the constant determined by $\cos\beta = \frac{-\text{Im}x}{A}$, $\sin\beta = \frac{-\text{Im}y}{A}$. Thus, we have the decomposition
	$$\pi^r_\Comp(x,y,z)(W_t(r))^{-1} = S\circ T$$
where
	$$Te_n = w(n,r)^{-1}e^{-ins}e_n, \;n\in \z$$
and
	$$SF(\theta) = e^{ir(\text{Re}x\cos\theta + \text{Re}y\sin\theta)}e^{rA\cos(\theta-\beta)}F(\theta), \;F \in L^2(\tor).$$
We can easily determine the norms of $T$ and $S$, namely
	$$\|T\| = \sup_{n\in \z} \exp(n\text{Im}s - t\sqrt{n^2+r^2}),\;\; \|S\| = \exp(rA).$$
Thus we get
	$$\sup_{r>0}\|\pi^r_\Comp(x,y,z)(W_t(r))^{-1}\| \le \sup_{n\in \z, r>0}\exp(n\text{Im}s +rA - t\sqrt{n^2+r^2})<\infty$$
when $({\rm Im}s)^2 + A^2 = ({\rm Im}s)^2 + ({\rm Im}x)^2 + ({\rm Im}y)^2 \le t^2$, since we have
	$$n\text{Im}s +rA \le t\sqrt{n^2+r^2}.$$
For the converse direction we take $F = e_n$, then we have
	$$\pi^r_\Comp(x,y,z)(W_t(r))^{-1}F(\theta) = e^{ir(\text{Re}x\cos\theta + \text{Re}y\sin\theta)} e^{rA\cos(\theta-\beta)}w(n,r)^{-1}e^{-ins}e^{in\theta}.$$
Since we assume that $(\pi^r_\Comp(x,y,z)(W_t(r))^{-1})_{r>0}$ is uniformly bounded,  we have
  \begin{align*}
  \exp(t\sqrt{n^2+r^2} - n\text{Im}s)
  & \gtrsim \left( \int_\tor \exp(2rA\cos(\theta-\beta)) d\theta \right)^{\tfrac{1}{2}}\\
  & = \left( \int_\tor e^{2rA\cos\theta} d\theta\right)^{\tfrac{1}{2}}.
  \end{align*}
Now we need to handle the integral. Let $0<\eps<1$ and $K_\eps = \{ \theta \in [0,2\pi]: \cos\theta \ge 1-\eps \}$. Then for any $A>0$ we have
	$$\int^{2\pi}_0 e^{2rA\cos\theta}d\theta \ge \int_{K_\eps}e^{2rA(1-\eps)}d\theta + \int_{[0,2\pi]\backslash K_\eps} e^{2rA\cos\theta}d\theta \ge C_\eps e^{2rA(1-\eps)},$$
where $C_\eps = m(K_\eps)$, the Lebesgue measure of $K_\eps$.
Thus, we have
	$$\exp(t\sqrt{n^2+r^2} - n\text{Im}s) \ge D_\eps \exp(rA(1-\eps))$$
for another constant $D_\eps >0$ depending only on $\eps$. As this holds for any $r>0$ and $n\in \z$ we get easily that
	$$\exp(t\sqrt{n^2+r^2} - n\text{Im}s) \ge D^{\frac{1}{m}}_\eps \exp(rA(1-\eps))$$
for any $m\in \z$ and hence letting $m\to \infty$, we get
	$$\exp(t\sqrt{n^2+r^2} - n\text{Im}s) \ge \exp(rA(1-\eps))$$
for any $0<\eps<1$. Thus, we have	
	$$t\sqrt{n^2+r^2} \ge n\text{Im}s + rA$$
for any $n\in \z$ and $r>0$. This implies that
	$$({\rm Im}s)^2 + X^2 \le t^2$$
which is the conclusion we wanted since $\log |z| = - {\rm Im}s$.	
\end{proof}


\section{The simply connected cover $\widetilde{E}(2)$ of the Euclidean motion group}\label{chap-E-infty(2)}

The simply connected cover $\widetilde{E}(2)$ of the Euclidean motion group $E(2)$ on $\Real^2$ is
	$$\widetilde{E}(2) = \Real^2 \rtimes \Real$$
with the group law
	$$(x,y,t)\cdot (x',y',t') = ((x,y)^T + \rho(t) (x',y')^T, t+t')$$
where $\rho(t) = \begin{bmatrix} \cos t & -\sin t \\ \sin t & \cos t\end{bmatrix}$ and we use the notation $(x,y,t) = (\begin{bmatrix} x \\ y\end{bmatrix}, t) = ((x,y)^T,t)$.

The representation theory of $\widetilde{E}(2)$ can be described as follows. For any $r>0$ and $z\in \tor$ we consider the Hilbert space
	$$H_{r,z} = \{F \in L^2_\text{loc}(\Real): F(\theta +2\pi) = zF(\theta)\; \text{for almost every $\theta \in \Real$}\}$$
with the inner product
	$$\la F, G \ra := \frac{1}{2\pi}\int^{2\pi}_0 F(\theta) \overline{G(\theta)}d\theta$$
where $L^2_\text{loc}(\Real)$ refers to the space of locally square integrable functions. We define an irreducible unitary representation $\pi^{r,z}$ acting on $\Hi_{r,z}$ given by
	$$\pi^{r,z}(x,y,t)F(\theta) = e^{i r(x \cos \theta + y \sin \theta)} F(\theta - t),\; F\in \Hi_{r,z}.$$
Note that we have a canonical isometry $\Hi_{r,z} \to L^2([0,2\pi], \frac{1}{2\pi}d\theta),\; F \mapsto F|_{[0,2\pi]}$, but the main difference between those two spaces are coming from the periodic behaviors.

For $f\in L^1(\widetilde{E}(2))$ we define the group Fourier transform on $\widetilde{E}(2)$ by
	\begin{align*}
	\F^{\widetilde{E}(2)}(f)
	& = (\F^{\widetilde{E}(2)}(f)(r,z))_{r>0, z\in \tor}\\
	& = (\widehat{f}^{\widetilde{E}(2)}(r,z))_{r>0, z\in \tor} \in L^\infty(\Real^+\times \tor, dz \,rdr; \B(\Hi_{r,z}))
	\end{align*}
and
	$$\widehat{f}^{\widetilde{E}(2)}(r,z) = \int_{\widetilde{E}(2)} f(g)\pi^{r,z}(g) dg = \int_{\Real^3} f(x,y,t) \pi^{r,z}(x,y,t) \,dxdydt.$$
We note here that the Haar measure is $dg = dxdydt$, the Lebesgue measure on $\Real^3$.
		
The representations $(\pi^{r,z})_{r>0, z\in \tor}$ is the whole family of irreducible unitary representations appearing in the Plancherel picture.
	\begin{prop}
	For $f\in L^1(\widetilde{E}(2)) \cap L^2(\widetilde{E}(2))$ we have
		$$(2\pi)^2 \|f\|^2_2 = \int^\infty_0\int_\tor \|\widehat{f}^{\widetilde{E}(2)}(r,z)\|^2_2dz\, rdr.$$
	\end{prop}
\begin{proof}
For $f\in L^1 \cap L^2$, $F\in \Hi_{r,z}$ and $s\in [0,2\pi]$ we have
	\begin{align}\label{eq-Fourier-E(2)-cover}
        \widehat{f}^{\widetilde{E}(2)}(r,z)F(s)
	& = \int_{\Real^3} f(x,y,t) \pi^{r,z}(x,y,t)F(s) \,dxdydt \nonumber \\
	& =\int_{\Real^2}\int_\Real f(x,y,t) e^{ir(x\cos s + y\sin s)}F(s-t)dt dx dy\nonumber \\
	& =2\pi \int_\Real \check{f}_{12}(r\cos s, r\sin s, s-t)F(t)dt \nonumber \\
	& =2\pi \int^{2\pi}_0\sum_{n\in \z} \check{f}_{12}(r\cos s, r\sin s, s-t-2\pi n)z^n F(t) dt,
	\end{align}
where $\check{f}_{12}$ means we take $\Real^2$-inverse Fourier transform for the first and the second variables. Thus, we have an integral operator, so that
	\begin{align}\label{eq-Plancherel-cover}
	\lefteqn{\|\widehat{f}^{\widetilde{E}(2)}(r,z)\|^2_{HS}}\nonumber \\
	& =(2\pi)^2 \int^{2\pi}_0\int^{2\pi}_0\Big|\sum_{n\in \z} \check{f}_{12}(r\cos s, r\sin s, s-t-2\pi n)z^n \Big|^2 ds dt.
	\end{align}
By applying the Plancherel theorem on $\z$ we get
	\begin{align}\label{eq-Plancherel-cover2}
	\int_\tor \|\widehat{f}^{\widetilde{E}(2)}(r,z)\|^2_{HS} dz
	& = (2\pi)^2 \int^{2\pi}_0\int^{2\pi}_0\sum_{n\in \z} |\check{f}_{12}(r\cos s, r\sin s, s-t-2\pi n)|^2 ds dt \nonumber \\
	& = (2\pi)^2\int^{2\pi}_0\int_\Real |\check{f}_{12}(r\cos s, r\sin s, s-t)|^2 dt ds \nonumber \\
	& = (2\pi)^2 \int^{2\pi}_0\int_\Real |\check{f}_{12}(r\cos s, r\sin s, t)|^2 dt ds.
	\end{align}
Thus, we have
	$$\int^\infty_0 \int_\tor \|\widehat{f}^{\widetilde{E}(2)}(r,z)\|^2_{HS} dz\, rdr = (2\pi)^2\|f\|^2_2.$$	
\end{proof}

Now we have the quasi-equivalence
	$$\lambda \cong \int^\oplus_{\Real^+ \times \tor} \pi^{r,z} \frac{rdrdz}{(2\pi)^2}$$
telling us that
	$$VN(\widetilde{E}(2)) \cong L^\infty(\Real^+\times \tor, \frac{rdr dz}{(2\pi)^2}; \B(\Hi_{r,z}))$$
and
	$$A(\widetilde{E}(2)) \cong L^1(\Real^+\times \tor, \frac{rdr dz}{(2\pi)^2}; S^1(\Hi_{r,z})).$$
Here we record another incarnation of the above Plancherel theorem for later use.
	\begin{prop}
	We have an onto isometry
		$$\Phi: L^2(\widetilde{E}(2)) \to L^2(\Real^+\times \tor \times [0,2\pi] \times [0,2\pi], rdr dz ds dt)$$
	satisfying
		$$\Phi(f)(r,z,s,t) = \sum_{n\in \z}\check{f}_{12}(r\cos s, r\sin s, s-t-2\pi n)z^n$$
	for $f\in \mathcal{S}(\Real^3)$ regarded as a function on $\widetilde{E}(2)$. We also have an onto isometry
		\begin{equation}\label{eq-isometry-variant}\Psi: L^2(\Real^+\times [0,2\pi] \times \Real, RdR\, d\theta\, dx) \to  L^2(\Real^+\times \tor \times [0,2\pi] \times [0,2\pi], rdr dz ds dt)
		\end{equation}
	satisfying
		$$\Psi(h_1 \otimes h_2 \otimes g)(r,z,s,t) = h_1(r) h_2(s) \sum_{n\in \z}g(s-t-2\pi n)z^n$$
	for $h_1 \in C^\infty_c((0,\infty))$, $h_2 \in {\rm Trig}(\tor)$, $g\in \F^\Real(C^\infty_c(\Real))$.
	\end{prop}
\begin{proof}
The first isometry $\Phi$ comes directly from \eqref{eq-Plancherel-cover} and \eqref{eq-Plancherel-cover2}. For the second isometry $\Psi$ we only need to observe that for $f = \widehat{h}^{\Real^2}\otimes g$, $h\in C^\infty_c(\Real^2)$, $g\in \F^\Real(C^\infty_c(\Real))$ with $h(Re^{i\theta}) = h_1(R)h_2(\theta)$, $R>0, \theta\in [0,2\pi]$ we have
	$$\check{f}_{12}(r\cos s, r\sin s, y) = h(re^{is})g(y) = h_1(r)h_2(s)g(y)$$
for $r>0, s\in [0,2\pi], y\in \Real.$

\end{proof}

The Lie algebra of $\widetilde{E}(2)$ is $\mathfrak{e}(2)$ with the exponential map
	$$\exp: \mathfrak{e}(2) \to \widetilde{E}(2)$$
given by
	\begin{align}
	\lefteqn{\exp(s S + xX + yY)}\\
	& = (\tfrac{1}{s}(\sin s)x + \tfrac{1}{s}(\cos s-1)y, \tfrac{1}{s}(1-\cos s)x + \tfrac{1}{s}(\sin s) y,  s), \nonumber
	\end{align}
where we take the limit $s \to 0$ for $s=0$.

We consider a complexification of $\widetilde{E}(2)$ given by $\Comp^2 \rtimes \Comp$ with the same group law, which we denote by $\widetilde{E}(2)_\Comp$. Due to the simple connectivity it is easy to check that $\widetilde{E}(2)_\Comp$ with the inclusion $\widetilde{E}(2) \hookrightarrow \widetilde{E}(2)_\Comp$ is the universal complexification. Moreover, we clearly have the following Cartan decomposition
	\begin{equation}\label{eq-Cartan-E(2)-cover}
	\widetilde{E}(2)_\Comp \cong \widetilde{E}(2) \cdot \exp(i \, \mathfrak{e}(2)).
	\end{equation}

	\begin{rem}\label{rem-universal-comp-E(2)-cover}
		The complexification $E(2)_\Comp$ with the canonical inclusion is actually the universal complexification of $E(2)$. Indeed, we can easily check the universal property using the covering maps $\widetilde{E}(2) \to E(2)$ and $\widetilde{E}(2)_\Comp \to E(2)_\Comp$.
	\end{rem}

For $r>0$, $z\in \tor$ and $F\in \Hi_{r,z} \cong L^2(\tor)$ we can easily check that
	\begin{align*}
		\begin{cases}\partial \pi^{r,z}(S) F = - F',\\ (\partial \pi^{r,z}(X)F)(\theta) = ir \cos \theta \cdot F(\theta),\\ (\partial \pi^{r,z}(Y)F)(\theta) = ir \sin \theta \cdot F(\theta).
		\end{cases}
	\end{align*}
When we write them as operators on $\ell^2(\z)$ we get
	\begin{align}\label{eq-Lie-derivatives-E(2)cover}
		\begin{cases}\partial \pi^r(S) e_n = -ine_n,\\ \partial \pi^r(X)e_n = \frac{ir}{2} (e_{n-1} + e_{n+1}),\\ \partial \pi^r(Y) e_n =\frac{r}{2} (e_{n+1} - e_{n-1}),
		\end{cases}
	\end{align}
where $\{e_n: n\in \z\}$ is the canonical orthonormal basis of $\ell^2(\z)$.

It is clear to see that trigonometric polynomials in $L^2(\tor)$ are entire vectors for $\pi^{r,z}$ for any $r>0$, $z\in \tor$.

\subsection{Weights on the dual of $\widetilde{E}(2)$}

Let us first identify all closed Lie subgroups of $\widetilde{E}(2)$.
	\begin{prop}\label{prop-subgp-structure-E(2)-cover}
	The proper closed Lie subgroups of $\widetilde{E}(2)$ are $H_S = \{(0, 0, z): z\in \Real\} \cong \Real$, $H_Y=\{(0,y,0): y\in \Real\}\cong \Real$,   $H_{r}=\{(x,rx,0):x\in\Real\}\cong \Real$ for every $r\in{\mathbb R}^{\geq 0}$, and $H_{X,Y} =\{(x,y,0): x,y\in \Real\} \cong \Real^2$ up to automorphism.
	\end{prop}
\begin{proof}
The description of $\aut(\fe(2))$ and the classification (up to isomorphism) of all one and two dimensional subspaces of $\mathfrak{e}(2)$ have been given in Proposition \ref{prop-subgp-structure-E(2)}. Since $\widetilde{E}(2)$ is simply connected, this gives us the classification for all the one and two dimensional subgroups up to automorphism.
\end{proof}

By Proposition \ref{prop-subgp-structure-E(2)-cover}  and Theorem \ref{thm-automorphism-principle}, we only need to consider the weight $W$ extended from the subgroups $H = H_S\cong \Real$ or $H_{X,Y}\cong \Real^2$. More precisely
for a weight function $w: \widehat{H_S} \cong \Real \to (0,\infty)$ or $w: \widehat{H_{X,Y}} \cong \Real^2 \to (0,\infty)$, we consider the extended weight $W = \iota(\widetilde{M}_w) = (W(r,z))_{r>0, z\in \tor}$, which is given as follows.

\vspace{0.3cm}
(The case of $H_S$)
	\begin{equation}\label{eq-WeightsK-cover}
	(\F^\tor\circ W(r,z) \circ (\F^\tor)^{-1})e_n = w(n)e_n,\;\; n\in \z.
	\end{equation}
In other words, $W(r,z)$ is a Fourier multiplier on $H_{r,z} \cong L^2(\tor) \cong \ell^2(\z)$ with respect to the symbol $w|_\z$, which is independent of the parameters $r$ and $z$.

\vspace{0.3cm}
(The case of $H_{X,Y}$)
	\begin{equation}\label{eq-WeightsXY-cover}
	W(r,z)F(\theta) = w(-r\cos\theta, -r\sin\theta) F(\theta),\;\; F\in H_{r,z}.
	\end{equation}
We get central weights if the above weight function $w$ is radial as in the case of $E(2)$.

Finally, we present exponentially growing weights on the dual of $\widetilde{E}(2)$ using Laplacian.
We observe that the operator $\partial\pi^{r,z}(S)$ (respectively $\partial\pi^{r,z}(X)$, $\partial\pi^{r,z}(Y)$) for $\widetilde{E}(2)$ is the same as $\partial\pi^r(S)$ (respectively  $\partial\pi^{r}(X)$, $\partial\pi^{r}(Y)$) for $E(2)$, which is independent of $z\in \tor$. Thus, we can define exponentially growing weights on the dual of $\widetilde{E}(2)$ with exactly the same argument as in \ref{ssec:exp-weight-E(2)}.
The weight $\exp(t\sqrt{\partial\lambda(-\Delta)})$ is called the exponential weight on the dual of $\widetilde{E}(2)$ of order $t>0$.

\subsection{Description of ${\rm Spec}A(\widetilde{E}(2),W)$}\label{sec-spec-E(2)-cover}
In this section, we present a full characterization of the spectrum of Beurling-Fourier algebras of $\widetilde{E}(2)$ associated with the weight $W$ extended from the subgroup $H_{X,Y}$ with the weight function $w: \widehat{H_{X,Y}} \cong \Real^2 \to (0,\infty)$. As in the case of the Euclidean motion group, we start this section by introducing appropriate dense subalgebras and subspaces of the Beurling-Fourier algebra, which we will use in analysis of both cases.

\begin{rem}
For the case of the weights extended from the subgroup $H_S$ and the exponentially growing weights using Laplacian, we were not able to find appropriate subalgebras, so that we were unable to determine the associated spectrum of $A(\widetilde{E}(2),W)$.
More precisely, for these weights, we were not able to prove that the subalgebra $\B$, given in Definition \ref{defn-subalg-E(2)-cover}, lies inside $A(\widetilde{E}(2), W)$. Specifically, for the case of the weights extended from the subgroup $H_S$, we believe that the main difficulty for getting good enough norm estimates lies in the formula \eqref{eq-integral-kernel-simply-connected} below, where we have to put additional factors $w(n)$ in front of $z^n$. The other case has a similar problem.
\end{rem}

\subsubsection{A dense subalgebra of $A(\widetilde{E}(2), W)$ and its companions}

We basically follow the same strategy as in the case of $E(2)$. Fortunately, we could re-use most of the companion spaces.

\begin{defn}\label{defn-subalg-E(2)-cover}
We recall the spaces $\A_0$, $\A_{00}$ and $\B_0$ from Definition \ref{def-companion-spaces-E(2)}. The spaces $\A$ and $\B$ are modified as follows.
	$$\A:= \F^{\Real^2 \times \Real}(\A_0 \otimes C^\infty_c(\Real))\;\; \text{and}\;\; \B:= \F^{\Real^2 \times \Real}(\B_0 \otimes C^\infty_c(\Real)).$$
We define an extra subspace $\tilde{\A}$ of $\A$ as follows.
	$$\tilde{\A} := \F^{\Real^2 \times \Real}(\A_{00} \otimes C^\infty_c(\Real)).$$
\end{defn}

We now examine the spaces $\A$, $\B$ and $\tilde{\A}$ defined on $\widehat{\Real^2 \times \Real}\cong \Real^2 \times \Real$.

\begin{prop}\label{prop-density-E(2)-cover} The spaces $\A$, $\B$ and $\tilde{\A}$ satisfy the following.
	\begin{enumerate}
	\item The space $\A$ is an algebra with respect to pointwise multiplication on $A(\widetilde{E}(2))$.

	\item The space $\tilde{\A}$ is closed under $\widetilde{E}(2)$-convolution.

	\item The space $\B$ is continuously and densely embedded in $A(\widetilde{E}(2), W)$.
	\end{enumerate}
\end{prop}
\begin{proof}
(1) This is clear.

\vspace{0.3cm}

(2) We pick any $k_1, k_2\in C^\infty_c(0,\infty)$, $g_1, g_2 \in \F^\Real(C^\infty_c(\Real))$ and $n,m \in \z$ and set $f_j = \widehat{h}^{\Real^2}_j \otimes g_j$, $j=1,2$ with
	$$h_1(re^{is}) = k_1(r)e^{ins},\;\; h_2(re^{is}) = k_2(r)e^{ims},\; r>0,\; s\in [0,2\pi].$$
Note from \eqref{eq-polar-separation} that $\widehat{h}^{\Real^2}_1$ is of the same form as $h_1$, namely the variables in polar coordinates are separated with the same frequency, so that we have
	$$\widehat{h}^{\Real^2}_1(\rho(\theta)(x,y)^T) = e^{in\theta}\widehat{h}^{\Real^2}_1(x,y),\; x,y\in \Real^2,\; \theta\in [0,2\pi].$$
Similarly, we also have $\widehat{h}^{\Real^2}_2(\rho(\theta)(x,y)^T) = e^{im\theta}\widehat{h}^{\Real^2}_2(x,y)$. Then, for $(\tilde{x},\tilde{y},\tilde{t}) \in \widetilde{E}(2)$ we have
	\begin{align*}
		\lefteqn{f_1*_{\widetilde{E}(2)}f_2(\tilde{x},\tilde{y},\tilde{t})}\\
		& = \int_{\widetilde{E}(2)}f_1(x,y,t)f_2((x,y,t)^{-1}(\tilde{x},\tilde{y},\tilde{t}))dxdydt\\
		& = \int_{\Real^3}f_1(x,y,t)f_2(\rho(-t)(\tilde{x}-x,\tilde{y}-y)^T,\tilde{t}-t)dxdydt\\
		& = \int_\Real \left(\int_{\Real^2}\widehat{h}^{\Real^2}_1(x,y) \widehat{h}^{\Real^2}_2(\rho(-t)(\tilde{x}-x,\tilde{y}-y)^T)dxdy \right) g_1(t)g_2(\tilde{t}-t)dt\\
		& = (\widehat{h}^{\Real^2}_1*_{\Real^2}\widehat{h}^{\Real^2}_2)(\tilde{x},\tilde{y}) \int_\Real e^{-imt} g_1(t)g_2(\tilde{t}-t)dt,
	\end{align*}
so that we have $f_1*_{\widetilde{E}(2)}f_2 = [\widehat{h}^{\Real^2}_1*_{\Real^2}\widehat{h}^{\Real^2}_2] \otimes [(e_{-m}g_1)*_\Real g_2]$. This leads us to the conclusion we wanted.

\vspace{0.3cm}

(3) Combining \eqref{eq-WeightsH} and \eqref{eq-Fourier-E(2)-cover} we know that the operator $W(r,z)\F^{\widetilde{E}(2)}(\widehat{h}^{\Real^2} \otimes g)(r,z)$ is an integral operator on $[0,2\pi]$ with the kernel
	\begin{equation}\label{eq-integral-kernel-simply-connected}
		K^{r,z}(s, t) = 2\pi w(-re^{is})h(re^{is}) \cdot  \left(\sum_{n\in \z} g(s - t -2\pi n)z^n \right).
	\end{equation}
Note that the Poisson summation formula says that
	\begin{equation}\label{eq-Poisson}
	\sum_{n\in \z} g(s - t -2\pi n)z^n = \frac{1}{\sqrt{2\pi}}\sum_{k\in \z} \widehat{g}^{\Real}(\frac{\theta}{2\pi}-k) e^{i(s-t)(\frac{\theta}{2\pi}-k)},
	\end{equation}
where $z=e^{i\theta}$. The latter sum is actually a finite sum since $\widehat{g}^{\Real}$ is compactly supported, so that differentiation with respect to $t$ is well-understood. Now we need to estimate $\|K^{r,z}\|_2$ and $\|\partial_tK^{r,z}\|_2$ in view of Lemma \ref{lem-integralOp}. For $\|K^{r,z}\|_2$ we have
	\begin{align*}
		\int_{\tor}\|K^{r,z}\|^2_{L^2(\tor^2)} dz
		& = 2\pi\int^{2\pi}_0\int^{2\pi}_0\int_\tor \Big|(\tilde{w}\cdot h)(re^{is}) \cdot  \sum_{n\in \z} g(s - t -2\pi n)z^n\Big|^2dz \frac{ds}{2\pi} \frac{dt}{2\pi}\\
		& = 2\pi\int^{2\pi}_0\int^{2\pi}_0 |(\tilde{w}\cdot h)(re^{is}) |^2 \cdot \sum_{n\in \z} |g(s - t -2\pi n)|^2 \frac{ds}{2\pi} \frac{dt}{2\pi}\\
		& = \int^{2\pi}_0 |(\tilde{w}\cdot h)(re^{is}) |^2 \int_\Real |g(s - t )|^2 dt \frac{ds}{2\pi}\\
		& = \|g\|^2_{L^2(\Real)} \|(\tilde{w}\cdot h)_r\|^2_{L^2(\tor)},
	\end{align*}
where $\tilde{w}(x,y)=w(-x,-y)$ and  $(\tilde{w}\cdot h)_r(s) = w(-re^{is}) h(re^{is})$. We have a similar expression for $\|\partial_t  K^{r,z}\|_2$, so that Lemma \ref{lem-integralOp} tells us that
	\begin{align}\label{eq-ineq-R22}
	\lefteqn{\|\widehat{h}^{\Real^2} \otimes g\|_{A(\widetilde{E}(2), W)}}\\
	& = \int_{\Real^+} \int_\tor \|W(r,z)\F^{\widetilde{E}(2)}(\widehat{h}^{\Real^2} \otimes g)(r,z)\|_1 \frac{dz\, rdr}{(2\pi)^2}\nonumber \\
	& \le C  \int_{\Real^+} \int_\tor  \big(\|K^{r,z}\|_{L^2(\tor^2)}+\|\partial_tK^{r,z}\|_{L^2(\tor^2)}\big)\frac{dz\, rdr}{(2\pi)^2}\nonumber \\
	&=C  \int_{\Real^+} \big(\|g\|_{L^2(\Real)} +\|g'\|_{L^2(\Real)} \big)\|(\tilde{w}\cdot h)_r\|_{L^2(\tor)}\frac{rdr}{(2\pi)^2}\nonumber \\
	& \leq C(\|g\|_{L^2(\Real)} + \|g'\|_{L^2(\Real)})\int_{\Real^+}\Big(\int_\tor |w(-re^{is})h(re^{is})|^2ds \Big)^{1/2} \frac{rdr}{(2\pi)^2}. \nonumber
	\end{align}
As before we appeal to the sub-multiplicativity of $w$ saying that $|w(-re^{is})| \le \rho^r$, $r>0$, for some $\rho>1$ so that from the estimate \eqref{eq-estimate1} we have
	$$\|\widehat{h}^{\Real^2} \otimes g\|_{A(\widetilde{E}(2), W)} \le C' (\|g\|_{L^2(\Real)} + \|g'\|_{L^2(\Real)})\int_{\Real^+} \rho^r(\sum_{n\in \z}|h_n(r)|)rdr$$	
which shows that $\B \subseteq A(\widetilde{E}(2), W)$.

For the density we note that the density of $\B$ in $A(\widetilde{E}(2), W)$ is the same as the density of the space $W\F^{\widetilde{E}(2)}(\B)$ in $\F^{\widetilde{E}(2)}(A(\widetilde{E}(2)))$. The result of the above (2)  tells us that
	$$\F^{\widetilde{E}(2)}(\tilde{\A}) \F^{\widetilde{E}(2)}(\tilde{\A}) = \F^{\widetilde{E}(2)}(\tilde{\A}*_{\widetilde{E}(2)}\tilde{\A}) \subseteq \F^{\widetilde{E}(2)}(\tilde{\A}) \subseteq \F^{\widetilde{E}(2)}(\B),$$
which means that it is enough to show that both of the spaces $\F^{\widetilde{E}(2)}(\tilde{\A})$ and $W\F^{\widetilde{E}(2)}(\tilde{\A})$ are dense in $\F^{\widetilde{E}(2)}(L^2(\widetilde{E}(2))) \cong L^2(\Real^+\times \tor, \frac{dz \,rdr}{(2\pi)^2}; S^2(\Hi_{r,z}))$, which is immediate from the isometry \eqref{eq-isometry-variant} and the description of kernel function \eqref{eq-integral-kernel-simply-connected} of the associated integral operator.
\end{proof}

Note also that the proof of the above proposition tells us that the space $\mc B$ can be used as the subspace $\mc S$ in \ref{ssec:separable-type I}.

\begin{prop}
	Every element of $\A$ is an entire vector for $\lambda$.
\end{prop}
\begin{proof}
We assume that $h(re^{is}) = h_1(r)e^{ims}$, $h_1 \in C_c[0,\infty)$, $m\in \z$, with $h\in C^\infty_c(\Real^2)$ and $g \in \F^{\Real}(C^\infty_c(\Real))$.
We will show that $f = \widehat{h}^{\Real^2} \otimes g$ is an entire vector for $\lambda$.
By \cite[Cor I.5, Cor I.6]{Pen} it is enough to see that $s\mapsto \la \lambda(\exp(sT))f, f \ra$ extends to an entire function on $\Comp$ for each $T = S, X, Y\in \mathfrak{e}(2)$. For $T=S$ we have $\exp(sS) = (0,0,s)$, so that
	\begin{align*}
	\la \lambda(\exp(sS))f, f \ra
	& = \la \lambda(0,0,s)f, f \ra \\
	& = \int_{\Real^3} f((0,0,s)^{-1}(x,y,t))\overline{f(x,y,t)}dxdydt \\
	& = \int_{\Real^3} f(\rho(-s)(x,y)^T,t-s)\overline{f(x,y,t)}dxdydt \\
	& = \int_{\Real^3}\widehat{h}^{\Real^2}(\rho(-s)(x,y)^T) \overline{\widehat{h}^{\Real^2}(x,y)}g(t-s)\overline{g(t)} dxdydt \\
	& = \|\widehat{h}^{\Real^2}\|^2_2 \cdot e^{-ims}\cdot \tilde{g}*_\Real \bar{g}(s),
	\end{align*}
where $\tilde{g}(t) = g(-t)$, $t\in \Real$. Note that $\F^\Real(\tilde{g}*_\Real \bar{g}) \in C^\infty_c(\Real)$, so that the Paley-Wiener theorem tells us that $\tilde{g}*_\Real \bar{g}$ extends to an entire function on $\Comp$, which gives us the conclusion we wanted.

For $T=X$ we have $\exp(sX) = (s,0,0)$, so that
	\begin{align*}
	\la \lambda(\exp(sX))f, f \ra
	& = \la \lambda(s,0,0)f, f \ra \\
	& = \int_{\Real^3} f((s,0,0)^{-1}(x,y,t))\overline{f(x,y,t)}dxdydt \\
	& = \int_{\Real^3} f(x-s,y,t)\overline{f(x,y,t)}dxdydt \\
	& = \int_{\Real^2}\widehat{h}^{\Real^2}(x-s,y) \overline{\widehat{h}^{\Real^2}(x,y)}dxdy \cdot \| g \|^2_2 \\
	& = \widetilde{\widehat{h}^{\Real^2}}*_{\Real^2}\overline{\widehat{h}^{\Real^2}}(s,0)\cdot \| g \|^2_2.
	\end{align*}
Since $h\in C^\infty_c(\Real^2)$, we get the desired conclusion again by the Paley-Wiener theorem.	The case $T=Y$ is similar, so that we have now that $f = \widehat{h}^{\Real^2} \otimes g$ is an entire vector for $\lambda$.

\end{proof}

\subsubsection{Solving Cauchy functional equation for $\widetilde{E}(2)$ and the final step}

We consider the Cauchy functional equation on $\B_0 \otimes C^\infty_c(\Real)$, where we endow a canonical locally convex topology in the same way as in (2) of Remark \ref{rem-spaces-for-E(2)}.
	\begin{align*}({\rm CFE_{\widetilde{E}(2)}})&\;\; v\in (\B_0 \otimes C^\infty_c(\Real))^* \;\; \text{satisfies}\\ & \;\; \la v, f*g \ra = \la v, f \ra \cdot \la v, g \ra \;\;\text{for any $f, g \in \A_0\otimes C^\infty_c(\Real)$.}
	\end{align*}
Here $*$ implies the convolution in $\Real^2 \times \Real$.

\begin{prop}\label{prop-CFE-simply-E(2)}
Let $v\in (\B_0 \otimes C^\infty_c(\Real))^*$ be a solution of $({\rm CFE_{\widetilde{E}(2)}})$. Then $v$ is actually an exponential function of the form $\exp(c_1 y+ c_2 z + c_3 x)$, $c_1, c_2, c_3 \in \Comp$. In other words, for any $f\in \B_0 \otimes C^\infty_c(\Real)$ we have
	$$\la v,  f\ra = \int_{\Real^3}f(y, z, x)e^{c_1y +c_2z+c_3x}dzdydx.$$
\end{prop}
\begin{proof}
The same proof as in the $\Real^n$ case still works. Note that [Step 1] can be done on the level of $\A_0 \otimes C^\infty_c(\Real)$ whilst [Step 2] can be done on the level of $\B_0 \otimes C^\infty_c(\Real)$.
\end{proof}

We continue to the realization of Spec$A(\widetilde{E}(2), W)$ in $\widetilde{E}(2)_\Comp$, whose proof is similar to the case of $E(2)$.

\begin{prop}
Every character $\varphi \in {\rm Spec}A(\widetilde{E}(2), W)$ is uniquely determined by a point $(x,y,t) \in \widetilde{E}(2)_\Comp \cong \Comp^3$, which is nothing but the evaluation at the point $(x,y,t)$ on $\B$ (and consequently on $\A$).
\end{prop}

\[
\xymatrix{
A(\widetilde{E}(2),W) \ar[d]_{\varphi}
& \B \ar@{_{(}->}[l]
& \A \ar@{_{(}->}[l] \ar[lld]_{\varphi|_\A} &
& \A_0 \otimes C^\infty_c(\Real) \ar[ll]_(.6){(2\pi)^{\frac{3}{2}} \F^{\Real^2 \times \Real}} \ar@/^/[lllld]^{\psi}
\\
\mathbb{C} & & &
}
\]

Here comes our final result.

	\begin{thm}\label{thm-E(2)cover-Spec}
	Let $\mathfrak{h}$ be the Lie subalgebra associated to the subgroup $H = H_{X,Y}$ of $\widetilde{E}(2)$, then we have
		$${\rm Spec}A(\widetilde{E}(2), W) \cong \Big\{ g\cdot \exp(iX): g\in \widetilde{E}(2),\; X\in \mathfrak{h},\; \exp(iX) \in {\rm Spec} A(H,W_H)\Big\}.$$
	\end{thm}
\begin{proof}
A similar proof as in the case of $E(2)$ still works with Theorem \ref{thm-unbdd-Lie-noncompact-E(2)cover} below as the replacement of Theorem \ref{thm-unbdd-Lie-noncompact-E(2)}.
\end{proof}

Note that the proof for the following is the same as the case of $E(2)$.

\begin{thm}\label{thm-unbdd-Lie-noncompact-E(2)cover}
Suppose that $W_H$ is a bounded below weight on the dual of $H$ and $W=\iota(W_H)$ is the extended weight on the dual of $\widetilde{E}(2)$. Then for any $X'\in \fe(2) \backslash \mathfrak{h}$ the operator $\exp (i\partial \lambda(X'))W^{-1}$ is unbounded whenever it is densely defined.
\end{thm}

\begin{rem}\label{rem-1D-subgp-simply-E(2)}
The statements of Theorem \ref{thm-unbdd-Lie-noncompact-E(2)cover} and Theorem \ref{thm-E(2)cover-Spec} hold true for $H=H_Y$ and $H=H_r$, $r\in\mathbb R^{\geq 0}$, in place of $H_{X,Y}$, by applying similar arguments.   
\end{rem}

\section{The spectrum under polynomial weights and regularity of Beurling-Fourier algebras}\label{chap-spec-poly-regularity}

In this section we will demonstrate that a {\it ``polynomially growing''} weight $W$ does not change the spectrum, i.e. Spec$A(G,W)\cong G$ and prove regularity of the associated algebra $A(G,W)$. Recall that a subalgebra $\A$ of $C_0(\Sigma)$ for a locally compact Hausdorff space $\Sigma$ is called {\it regular on $\Sigma$} if for each proper, closed subset $E$ of $\Sigma$  and each $x\in \Sigma\setminus E$ there exists $f\in\A$ with $f(x)=1$ and $f\equiv 0$ on $E$. A commutative Banach algebra $\A$ is {\it regular} if its algebra of Gelfand transforms is regular on ${\rm Spec}\A$.  When the weight is {\it  ``polynomially growing''}, we may have a much simpler substitute of ${\rm Trig}(G)$ for the case of compact $G$, namely the usual test function space $C_c^\infty(G)$. We will show that $C_c^\infty(G)$ is sitting in $A(G,W)$ densely in each case. Then, the problem of determining ${\rm Spec}A(G,W)$ leads us to the problem of understanding ${\rm Spec}\, C_c^\infty(G)$.

\[
\xymatrix{
A(G,W) \ar[d]_{\varphi}
& \;C^\infty_c(G) \ar@{_{(}->}[l] \ar[ld]^{\varphi|_{C^\infty_c(G)}}
\\
\mathbb{C} &
}
\]

The following result takes no extra effort to prove abstractly, and admits an easy standard proof. We say that a subalgebra $\A$ of $C_0(\Sigma)$ is nowhere vanishing on $\Sigma$ if for each $x\in\Sigma$ there is $f\in\A$ with $f(x)\ne 0$; $\A$ is said to separate points of $\Sigma$ if whenever $x\ne y$ in $\Sigma$ then there is $f\in\A$ such that $f(x)\ne f(y)$.

\begin{prop}\label{prop:specresult}
Let $\Sigma$ be a locally compact Hausdorff space and $\fA$ be a conjugate-closed, point separating and nowhere vanishing subalgebra
of $C_0(\Sigma)$ which satisfies for any $f$ in $\fA$:
\begin{equation}\label{eq:speccond}
\text{if }\lam\in\Cee\setminus\wbar{f(\Sigma)}\text{ then }
(f-\lam 1)^{-1}\in\fA+\Cee 1.
\end{equation}
Then $\spec(\fA)\cong \Sigma$ via evaluation maps.
\end{prop}

\begin{proof}
Let $\wtil{\fA}=\fA+\Cee 1$ and $\wtil{\Sigma}$ be $\Sigma$ itself if $\Sigma$ is compact, and the one-point compactification $\Sigma_\infty$ if $\Sigma$ is not compact. 

We first assume that $\Sigma$ is compact and show that $1\in\fA$, which implies $\wtil{\fA}=\fA$. In fact, as $\fA$  vanishes nowhere on $\Sigma$, for any $x\in\Sigma$ there exists $f_x\in\fA$ for which $f_x(x)\ne 0$. Then the family of sets $\{f_x^{-1}(\Cee\setminus\{0\})\}$ is an open cover of ${\Sigma}$, and hence
admits a finite subcover.  Thus there are $f_k=f_{x_k}$ in $\A$, $k=1,\dots,n$, for which
$f=\sum_{k=1}^n|f_k|^2\in\A$ and $f(x)>0$ for $x\in\Sigma$. Therefore $0\in \mathbb C\setminus f(\Sigma)=\mathbb C\setminus\wbar{f(\Sigma)}$  and by the hypothesis $f^{-1}\in\fA+\Cee 1$, giving $1=ff^{-1}\in\fA$.

Let us go back to the general case and let $\fI$ be an ideal in $\wtil{\fA}$.  Suppose
that for any $x$ in $\wtil{\Sigma}$, there exists $f_x$ in $\fI$ for which $f_x(x)\not=0$. The above arguments applied to $\fI$ and $\wtil{\Sigma}$ give a function $f\in \fI$ such that $f(x)>0$ for all $x\in\wtil{\Sigma}$. If $\Sigma$ is compact we obtain as above  that $f^{-1}\in \wtil\fA=\fA$ and hence $\fI=\fA$.
In the case that
$\Sigma$ is not compact, write each $f_k=g_k+\lam_k 1$, where each $g_k\in\fA$, and
it is clear that $f=g+\left(\sum_{k=1}^n|\lam_k|^2\right)1$, where $g\in\fA$, as $\fA$ is an ideal in
$\wtil{\fA}$. We have that $\lambda:=-\sum_{k=1}^n|\lam_k|^2\in \mathbb C\setminus{g(\wtil{\Sigma})}  =  \mathbb C\setminus\wbar{g(\Sigma)}$ since $\wbar{g(\Sigma)} = g(\Sigma) \cup \{0\} = g(\wtil{\Sigma})$. The hypothesis  provides $f^{-1}=(g-\lambda 1)^{-1}\in\wtil{\fA}$, which means that $\fI=\wtil{\fA}$. Thus a proper ideal of $\wtil{\fA}$, must admit a vanishing point
in $\wtil{\Sigma}$.

Now suppose $\psi\in\spec(\fA)$, and let $\til{\psi}$ denote its canonical extension to a multiplicative
character on $\wtil{\fA}$.  Then $\ker\til{\psi}$ is a proper ideal of $\wtil{\fA}$, and hence there is
$x$ in $\wtil{\Sigma}$ for which $f(x)=0$ for every $f$ in $\ker\til{\psi}$.  In particular $[f-{\wtil\psi}(f)1](x)=0$ so
${\wtil\psi}(f)=f(x)$.  In the case that $\Sigma$ is not compact, we observe
that $x\not=\infty$, where $\infty$ is the point in $\Sigma_\infty$ corresponding to the character
$f+\lam 1\mapsto \lam$.  Indeed, since $\psi\not=0$, there is $f$ in $\fA$ for which $\psi(f)\not=0$, so
$f-\psi(f)1\in \ker\til{\psi}$ with $[f-\psi(f)1](\infty)=-\psi(f)\not=0$.  Since $\fA$ is point separating,
the point $x$ implementing $\psi$ is unique. Hence ${\rm Spec}\fA$ sits homeomorphically in $\Sigma$ via evaluation maps. To see that  ${\rm Spec}\fA\simeq \Sigma$ it is enough to note that, as $\fA$ is nowhere vanishing on $\Sigma$, each $\psi_x$, $x\in\Sigma$, given by $\psi_x(f)=f(x)$, $f\in\fA$, is a character.
\end{proof}

\begin{cor}\label{cor:specresult}
We have $\spec(C_c^\infty(G))\cong G$ via evaluation functionals for any Lie group $G$.
\end{cor}
\begin{proof}
Given $f$ in $C_c^\infty(G)$ and $\lam\in\Cee\setminus f(G)$, it is evident that
$(f-\lam 1)^{-1}\in C^\infty(G)$ with
\[
(f-\lam 1)^{-1}+\frac{1}{\lam}1\in C_c^\infty(G).
\]
Further, $C_c^\infty(G)$ is conjugate-closed, point separating and nowhere vanishing on $G$.
\end{proof}

Here comes the main result.

\begin{thm}\label{theo:spec_poly}
Let $G$ be a connected Lie group and $W$ is either (a) $W_m$, the polynomial weight of order $2m$, $m\in \z_+$ from Definition \ref{def-poly-weight-laplacian} or (b) $\iota(\widetilde{M}_w)$, the weight extended from a closed connected abelian Lie subgroup $H$ with the weight function $w: \widehat{H} \to (0,\infty)$ which is polynomially growing and $w^{-1}$ is bounded. Then we have
	\begin{enumerate}
		\item $\spec(A(G,W))\cong G$ via evaluation functionals, and
		
		\item the algebra $A(G,W)$ is regular on $G$.
	\end{enumerate}

\end{thm}

\begin{proof}

First, we note that $C_c^\infty(G)$ is dense in $L^2(G)$. Indeed, $C_c(G)$ is dense in $L^2(G)$, and any element of $C_c(G)$ can be uniformly approximated by a sequence of elements from $C_c^\infty(G)$, all supported on a common compact subset, thanks to the Stone-Weierstrass theorem. Secondly, we recall that $C_c^\infty(G)\subset A(G)$, thanks to either \cite[(3.26)]{Eym} or \cite[(3.8) \& Lemma 3.3]{LT}.

(The case $W=W_m$)
Let us recall the action of $VN(G)$ on $A(G)$ used in Proposition \ref{prop:boundedlyinvertible}.
For $T$ in $VN(G)$ and $u=f\ast\check{h}=\langle\lambda(\cdot)h,\bar{f}\rangle$ in $A(G)$,
where $f,h\in L^2(G)$ and $\check{h}(g)=h(g^{-1})$ for a.e.\ $g\in G$, we have
\begin{equation}\label{eq:Tactsonu}
Tu=T[f\ast\check{h}]=\langle\lambda(\cdot)Th,\bar{f}\rangle=f\ast(Th)^\vee.
\end{equation}
Notice that if $h\in C_c^\infty(G)\subset\D^\infty(\lambda)$, then we may extend
the above notation as follows
\begin{equation}\label{eq:actionofWm}
W[f\ast\check{h}]=f\ast(Wh)^\vee.
\end{equation}
To help avoid confusion in remainder of the proof,
for any subset $\fE$ of $C^\infty_c(G)$, let
$W(\fE)$ denote the image of $\fE\subset L^2(G)$ under the operator $W$;
and we let $W[\fE]$ and $W^{-1}[\fE]$ denote the images of $\fE\subset A(G)$ under the action indicated in
(\ref{eq:actionofWm}), provided it makes sense.

Now we wish to see that $C^\infty_c(G)$ is contained in $A(G,W)$.
We will need to study the action (\ref{eq:Tactsonu})
from a different perspective.  Notice for $g,g'$ in $G$, and $u$ in $A(G)$ that
\begin{equation}\label{eq:lamgaction}
\lambda(g)u(g')=u(g'g).
\end{equation}
Now we let $L^2_r(G)$ be the Hilbert space with respect to the right Haar measure, normalised
so $U:L^2(G)\to L^2_r(G)$, $Uf=\check{f}$ is a unitary.  We then let $\rho=U\lambda(\cdot)U^*$,
so $\rho(g)f(g')=f(g'g)$ for $f$ in $L^2_r(G)$, all $g$ and a.e.\ $g'$.  It follows (\ref{eq:lamgaction})
that
\[
Tu=\rho(T)u\text{ for }u\in A(G)\cap L^2_r(G).
\]
Thus, for $u$ in $C^\infty_c(G)\subset A(G)\cap L^2_r(G)$ we have
\[
u=(I-\partial\rho(\Delta))^{-m}(I-\partial\rho(\Delta))^mu=W^{-1}[C_m(I-\partial\rho(\Delta))^mu].
\]
Since $(I-\partial\rho(\Delta))^mC^\infty_c(G)\subset C^\infty_c(G)\subset A(G)$, the remark following
Proposition \ref{prop:boundedlyinvertible} shows that $C_c^\infty(G)\subset A(G,W)$, as desired.

We now wish to see that $W(C^\infty_c(G))$ is dense in $L^2(G)$.    Indeed, we first observe that $C^\infty_c(G)$ is a core for $W$ by \cite[Theorem 10.1.14]{Sch2}. Then for $T:= W|_{C^\infty_c(G)}$ we know that $\overline{T} = W$, so that ${\rm ker}T^* = {\rm ker}W^* = {\rm ker}W = \{0\}$ (\cite[Theorem 1.8. (ii)]{Sch}) since $W$ is bounded below. This implies that ${\rm ran}(T)^\perp = \{0\}$ by \cite[Proposition 1.6. (ii)]{Sch}, which is the conclusion we wanted.

We can now show that $C^\infty_c(G)$ is dense in $A(G,W)$.  We use (\ref{eq:actionofWm})
and the result of the last paragraph to see that
\[
W[C^\infty_c(G)\ast C^\infty_c(G)^\vee]=C^\infty_c(G)\ast (W(C^\infty_c(G)))^\vee
\]
is dense in $A(G)$.  Again, using the remark following Proposition \ref{prop:boundedlyinvertible}, we deduce that
\[
C^\infty_c(G)\ast C^\infty_c(G)^\vee=W^{-1}[C^\infty_c(G)\ast (W(C^\infty_c(G)))^\vee]
\]
is dense in $A(G,W)$.  This set is clearly contained in $C^\infty_c(G)$, and thus $C^\infty_c(G)$ is dense
in $A(G,W)$.

Now let $\psi\in\spec(A(G,W_m))$. Then $\psi$, being continuous, is determined by its restriction to $C_c^\infty(G)$, which we again denote $\psi$. Corollary \ref{cor:specresult} shows that $\psi$ is evaluation at a point in $G$. Conversely, we note that $W^{-1}_m$ is bounded, so that $A(G,W_m)$ embeds in $A(G)$ continuously. Thus, any evaluation functional at a point in $G$ is bounded on $A(G,W_m)$. For the regularity we only need to recall that $C_c^\infty(G)$ is a regular algebra on $G$, which can be obtained by a smooth Urysohn's Lemma. Hence so too is $A(G,W_m)$.

(The case $W = \iota(\widetilde{M}_w)$) For simplicity we assume that $\widehat{H}\cong \Real^k$. Since $w$ is polynomially growing there is a constant $C>0$ and $m\in \n$ such that
	$$w(x_1, \cdots, x_k) \le C(1+x^2_1 + \cdots + x^2_k)^m,\;\; (x_1, \cdots, x_k)\in \Real^k.$$
This implies that $\iota(\widetilde{M}_w) (1-\partial\lambda(\Delta_H))^{-m}$ is a bounded operator, where $\Delta_H = X^2_1 + \cdots + X^2_k$ is the sublaplacian for a fixed basis $\{X_1, \cdots, X_k\}$ of $\h$, the Lie subalgebra of $\g$ corresponding to $H$. From this point on we can basically follow the same argument as above.
\end{proof}

We end this section with some examples of non-regular Beurling-Fourier algebras.

\begin{thm}\label{the-non-regular}
Let $G$ and $W$ be one of the groups and weights in Section \ref{sec-weights-compact}, Section \ref{sec-Heisenberg}, Section \ref{sec-reduced-Heisenberg}, Section \ref{sec-weights-E(2)} and Section \ref{sec-spec-E(2)-cover}. Suppose that $W$ is boundedly invertible and $G \subsetneq {\rm Spec} A(G,W) $, then $A(G,W)$ is not regular.
\end{thm}
\begin{proof}
Let $\A$ be the subalgebra of $A(G,W)$ used in each section. Note that we have the realization ${\rm Spec} A(G,W) \subseteq G_\Comp$, so that for any $\varphi \in {\rm Spec} A(G,W)$ we can associate a point $x\in G_\Comp$ such that $\varphi(f) = f_\Comp(x)$ for any $f\in \A$, where $f_\Comp$ is the unique analytic extension of $f$ to $G_\Comp$. From the decomposition $G_\Comp \cong G \cdot \exp(i\g)$ we know that the assumption $G \subsetneq {\rm Spec} A(G,W) $ implies that there is $\varphi \in G^+_\Comp \backslash\{e\}$. Now we observe that $\varphi^s \in {\rm Spec} A(G,W)$ for all $0\le s\le 1$.
Indeed, let  $E_\varphi(\cdot)$ be the spectral measure for $\varphi$. It is known that
	$${\rm dom}(\varphi^s)=\{\xi\in L^2(G): \int x^{2s}d \la E_\varphi(x)\xi,\xi \ra<\infty\}.$$   As $x^{2s}\leq x^2+1$ for any $x\in [0,+\infty)$ we have
	$$\int x^{2s}d\la E_\varphi(x)\xi,\xi \ra\leq \int x^2d \la E_\varphi(x)\xi,\xi \ra+\|\xi\|^2<\infty$$
for any $\xi\in\text{dom}(\varphi)$. As $\varphi W^{-1}$ is bounded and densely defined, $W^{-1}L^2(G)\subset\text{dom}(\varphi)$ (see Proposition 2.1) and hence $W^{-1}L^2(G)$ is in the domain of $\varphi^{s}$, giving that $\varphi^{s}W^{-1}$ is bounded by the closed graph theorem. Thus, we know that $\lambda(g)\varphi^{s}W^{-1}$ is bounded for any $g\in G$ and $0\le s \le 1$. In particular, for $X\in \mathfrak{g}$ such that $\varphi = \exp(iX)$ we have
$\varphi^z = \exp(i{\rm Re}z \cdot X) \exp(-{\rm Im}z \cdot X)$, so that
	$$\varphi^z(f) = f_\Comp(\exp(i{\rm Re}z \cdot X) \exp(-{\rm Im}z \cdot X))$$
for any $f\in \A$. Since $z \mapsto \exp(i{\rm Re}z \cdot X) \exp(-{\rm Im}z \cdot X)$ is clearly analytic we get a scalar analytic map $\{z \in \Comp: 0< {\rm Re}z<1\} \to \Comp: z\mapsto \varphi^z(f)$ for any $f\in \A$.

Now we recall the norm density of $\A$ in $A(G,W)$, so that for a $f\in A(G,W)$ we can choose $f_n \in \A \to f$ with $\|f_n\|_{A(G,W)} \le \|f\|_{A(G,W)}$. Then, we know that $\varphi^z(f_n) \to \varphi^z(f)$ uniformly on compacta with respect to $z$, so that the map
	$$\{z \in \Comp: 0< {\rm Re}z<1\} \to \Comp,\; z\mapsto \varphi^z(f)$$
is also analytic for any $f\in A(G,W)$.

Finally, we observe that any compact subset $K\subseteq \{z \in \Comp: 0< {\rm Re}z<1\}$ gives rise to a compact set $\tilde{K} := \{\exp(i{\rm Re}z \cdot X) \exp(-{\rm Im}z \cdot X): z\in K\} \subseteq G_\Comp$. Thus, if there is a $f\in A(G,W)$ such that $f|_{\tilde{K}}\equiv 0$, then we have $\varphi^z(f) = 0$ for all $z\in K$. By analyticity we can conclude that $\varphi^z(f) = 0$ on $\{z \in \Comp: 0< {\rm Re}z<1\}$. This means that we can not separate $\tilde{K}$ and any point in
	$$\{\exp(i{\rm Re}z \cdot X) \exp(-{\rm Im}z \cdot X): 0< {\rm Re}z<1\}\backslash \tilde{K}$$
using functions in $A(G,W)$, i.e. $A(G,W)$ is not regular.
\end{proof}

\begin{ex}\label{ex-quasianalytic1}\rm
Let $w:\mathbb Z^n\to (0,\infty)$ be a weight function and consider the corresponding Beurling Fourier algebra $A(\tor^n,W)$ with $W = \widetilde{M}_w$. By \cite[Example 4.3]{LST}, the Gelfand spectrum of $A(\mathbb T^n,W)$ is
$$\mathbb T^n_w= \left \{z\in\mathbb C^n:\frac{1}{\rho_w(-\mu)} \leq|z^\mu|\leq\rho_w(\mu) \text{ for all }\mu\in\mathbb Z^n \right\}$$
where $\displaystyle \rho_w(\mu)=\lim_{k\to\infty}w(k\mu)^{1/k}$ and $z^\mu=z_1^{\mu_1}\ldots z_n^{\mu_n}$ for $z=(z_1,\ldots,z_n)\in\mathbb C^n$ and $(\mu_1,\ldots,\mu_n)\in \mathbb Z^n$.
If $n=1$ we obtain the annulus of convergence with inner radius $\displaystyle \frac{1}{\rho_w(-1)}$ and outer radius $\rho_w(1)$. Hence if $\rho_w(1)\rho_w(-1)\ne 1$, the algebra $A(\mathbb T^n,W)$ is not regular.
In the case $n\geq 1$ and $w(\mu)=\lambda^\mu$ with $\lambda=(\lambda_1,\ldots,\lambda_n)\in (\mathbb R^{\geq 1})^n$ and $\lambda_j >1$ for some $1\le j \le n$, we get $\mathbb T_w^n\ne\mathbb T$ and hence the corresponding Beurling-Fourier algebra $A(\mathbb T^n,W)$ is not regular.
\end{ex}

\begin{rem}\rm\label{remark:abelian-weight}
Let $G$ be a locally compact abelian group and let $w: G\to (0,\infty)$ be a weight on $G$ such that $w(x)\ge 1$, $x\in G$. Then $A(\widehat G,W)\simeq L^1(G,w)$ for $W = \widetilde{M}_w$. The regularity of the latter algebra was studied by Domar in \cite{domar} (see also \cite{Kan}). In particular, he proved that
 $L^1(G,w)$ is regular iff $w$ is non-quasianalytic, i.e.
	\begin{equation}\label{eq-non-quasianalytic}
	\sum_{n\in\mathbb Z}\frac{\log w(nx)}{1+n^2}<\infty,\  x\in G.
	\end{equation}
 It is easy to see that if $w$ is non-quasianalytic on $\mathbb Z$, then $\rho_w(1)=\rho_w(-1)=1$ and by the above example the spectrum of $A(\mathbb T,W)$ is $\mathbb T$.
\end{rem}

\begin{ex}\label{ex-quasianalytic2}
Let $G$ be either a compact connected Lie group, or the (reduced) Heisenberg group, or the Euclidean motion group on $\Real^2$ or its simply connected cover. Then, we can obtain
a weight $W$ on the dual of $G$ for which $\spec A(G,W)\cong G$,
but $A(G,W)$ is not regular on $G$.
Note that $G$ has a closed subgroup $H$, isomorphic to one of $\Real$ or $\tor$. 
We consider the weight $w(x)=e^{|x|/\log(e+|x|)}$ on $\widehat{H} \cong \Real$ or $\z$. This weight is quasianalytic, in the sense that the series test indicated in \eqref{eq-non-quasianalytic} fails, i.e.\ the series diverges.  However, it satisfies the Shilov property that
\[
\lim_{n\to\infty} w(nx)^{1/n}=1.
\]
We let $W_H=\F_H M_w \F_H^*$ be the weight on the dual of $H$.
This combination of properties implies that $\spec A(H,W_H)=\spec L^1(\widehat{H},w)\cong H$ for $H$ isomorphic
$\Real$ or $\tor$, but that $A(H,W_H)$ is not regular on $H$.

We let $W=\iota(W_H)$, where
$W_H$ is the Shilov but quasianalytic weight, given above. Then, we have $\spec A(G,W)\cong G$ thanks to Theorem \ref{thm-spec-extended-compact}, Remark \ref{rem-1D-subgp-Heis}, Remark \ref{rem-1D-subgp-reduced-Heis}, Remark \ref{rem-1D-subgp-E(2)} and Remark \ref{rem-1D-subgp-simply-E(2)}.
Indeed, for $X\in\h\setminus\{0\}$, $\exp(iX)\notin H$.  However, Proposition \ref{prop-restriction} informs us that
$A(G,W)$ is not regular on $G$.
\end{ex}


\section{Questions}

In this final section we collect relevant questions that we were not able to answer at the time of this writing.

\subsection{Constructing exponentially growing weights on the dual of $G$ using Laplacian}

As is mentioned in Section \ref{ssec:Weights-Laplacian-general} we hope to construct ``exponentially growing'' weights on the dual of $G$ using Laplacian, since Laplacian is one natural candidate to ``measure'' growth rate covering all the directions. For non-compact Lie groups the task was successful only for the Euclidean motion group $E(2)$ and its simply connected cover $\widetilde{E}(2)$.
	\begin{question}
	Can we construct ``exponentially growing'' weights on the dual of connected Lie groups? Can we do it, at least, for the case of the (reduced) Heisenberg group?
	\end{question}

\subsection{Finding an appropriate dense subalgebra $\A$ of $A(G,W)$}
For a non-compact Lie group $G$ finding an appropriate dense subalgebra $\A$ of $A(G,W)$ depends heavily on each example of group $G$.
	\begin{question}
	Is there any unified way of finding an appropriate subalgebra $\A$ of $A(G,W)$ for a suitable choice of a weight $W$ on the dual of $G$?
	\end{question}
We have a more specific question left behind in Section \ref{chap-E-infty(2)}.
	\begin{question}
	Can we find an appropriate dense subalgebra $\A$ of $A(\widetilde{E}(2),W)$ for the extended weight $W= \iota(\widetilde{M}_w)$ on the dual of $\widetilde{E}(2)$, where $w: \widehat{H_S} \cong \Real \to (0,\infty)$ is an exponentially growing weight function?
	\end{question}

\subsection{Groups not admitting entire vectors for the left regular representation}

There are several classes of groups not covered in this paper. The first class we can check would be non-unimodular type I groups and the $ax+b$-group is arguably one of the simplest examples of such groups. Recall that the $ax+b$-group is $\Real \ltimes \Real$ with the group law
	$$(a,b)(a',b') = (a+a', e^{-a'}b+b'),\;a,a',b,b' \in \Real.$$
The $ax+b$-group is a typical example of a non-unimodular group with the left Haar measure $dadb$, the Lebesgue measure on $\Real^2$ and the modular function
	$$\Delta_F(a,b) = e^{-a},\; (a,b) \in F.$$
The second class we can check would be connected semisimple Lie groups such as $SL_2(\Real)$, which are automatically unimodular, but quite far away from the class of solvable Lie groups.

Recall that the density of entire vectors for the left regular representation in Beurling-Fourier algebras played a significant role in the final step of determining Spec$A(G,W)$. See the proofs of Proposition \ref{Prop-respect-Cartan-heis} and Proposition \ref{prop:cartan-E(2)}, for example. Unfortunately, it is already known that the $ax+b$-group and connected semisimple Lie groups do not adimit entire vectors for the left regular epresentation, which gives us an immediate obstacle. Indeed, the left regular representation of the $ax+b$-group has two irreducible components $\pi_\pm$, which does not allow any entire vector by \cite[Theorem 7.2]{Good69}. Moreover, \cite[Theorem 8.1]{Good69} explains the non-existence of entire vectors for the left regular representation of connected semisimple Lie groups. Thus, we have the following question.
	\begin{question}
	Can we determine ${\rm Spec}A(G,W)$ for the group $G=F$, the $ax+b$-group, or $G = SL_2(\Real)$, or in general a connected semisimple Lie group?
	\end{question}

\subsection{Classifying central weights on the dual of Lie groups}
We were able to provide complete lists of central weights on the dual of $G$ for the case of $G = \mathbb{H}$, $\mathbb{H}_r$ and $E(2)$, which leads us to the following question.
	\begin{question}
	Can we determine all central weights on the dual of a Lie group $G$?
	\end{question}

\subsection{Characterizing weights on the dual of Lie groups whose Beurling-Fourier algebras are regular}
We were able to prove that ``polynomially growing weights'' on the dual of connected Lie groups provide regular Beurling-Fourier algebras and some ``exponentially growing weights'' on the dual of connected Lie groups provide irregular Beurling-Fourier algebras. Moreover, Example \ref{ex-quasianalytic2} shows that regularity of Beurling-Fourier algebras are somewhere in the middle, which leads us to the following question.
	\begin{question}
	Can we characterize all the weights on the dual of Lie groups whose Beurling-Fourier algebras are regular as in the abelian case (Example \ref{ex-quasianalytic1})?
	\end{question}

\bibliographystyle{amsalpha}

\end{document}